\documentclass[10pt,reqno]{amsart}
\usepackage[utf8]{inputenc}
\usepackage[T1]{fontenc}
\usepackage[utf8]{inputenc}
\usepackage[english]{babel}
\usepackage[margin=4.1cm]{geometry}
\usepackage{amsmath, amssymb}
\usepackage{amsfonts}
\usepackage{dsfont}
\usepackage{float}
\usepackage{graphicx}
\usepackage{wrapfig}
\usepackage{mathtools}
\usepackage{bbm}
\usepackage{amsthm}
\usepackage{ifthen}
\usepackage{graphicx}
\usepackage[ruled,vlined]{algorithm2e}
\usepackage{xcolor}
\usepackage{mathtools}
\usepackage{empheq}
\usepackage{mathrsfs}
\usepackage{multicol}

\newtheorem{theorem}{Theorem}[section]
\newtheorem{prop}[theorem]{Proposition}
\newtheorem{lemma}[theorem]{Lemma}
\newtheorem{cor}[theorem]{Corollary}

\newtheorem{definition}[theorem]{Definition}

\newtheorem{remark}{Remark}[section]

\newcommand{\bb}[1]{\mathbb{#1}}

\newcommand{\R}{\bb{R}}

\chardef\bslash=`\\ 

\numberwithin{equation}{section}

\newcommand{\N}{\mathbb{N}}
\newcommand{\Z}{\mathbb{Z}}

\newcommand{\Com}{\mathbb{C}}

\newcommand{\re}{\operatorname{Re}}
\newcommand{\ima}{\operatorname{Im}}

\def\bm{\left( \begin{array}{cc}}
\def\endm{\end{array}\right)}

\newcommand{\be}{\begin{equation}}
\newcommand{\ee}{\end{equation}}
\newcommand{\bea}{\begin{eqnarray}}
\newcommand{\eea}{\end{eqnarray}}
\newcommand{\bee}{\begin{eqnarray*}}
\newcommand{\eee}{\end{eqnarray*}}
\newcommand{\ben}{\begin{enumerate}}
\newcommand{\een}{\end{enumerate}}

\newcommand{\ba}{\begin{aligned}}
\newcommand{\ea}{\end{aligned}}
\newcommand{\bpm}{\begin{pmatrix}}
\newcommand{\epm}{\end{pmatrix}}

\usepackage{booktabs}
\usepackage{longtable}
\usepackage{xurl}
\usepackage{hyperref}
\usepackage{chngcntr}
\usepackage{animate} 
\usepackage{tikz}
\usepackage{graphicx}
\numberwithin{figure}{section}

\numberwithin{equation}{section}
\begin{document}
	
	\title[Upper bound for $\Lambda_L$]{Upper Bounds on the de Bruijn-Newman constant for Dirichlet $L$-functions with primitive real-valued characters}
	
    \author[Freire]{Ricardo Freire}  
	\address{Departamento de Ingenier\'{\i}a Matem\'atica, Universidad de Chile, Casilla 170 Correo 3, Santiago, Chile. Current address: Departamento de Matem\'atica, Facultad de Ciencias F\'isicas y Matem\'aticas, Universidad de Concepci\'on, Concepci\'on, Chile}
	\email{rfreire@dim.uchile.cl, ricardofreiremath@gmail.com, rcarlosf@udec.cl}
    
	\author[Gonzalez-Riquelme]{Cristian Gonz\'alez-Riquelme}   
 \address{Facultad de Matem\'atica, P. Universidad Cat\'olica de Chile, Av. Vicu\~na Mackenna 4860, 690444 Santiago, Chile.}
	\email{cristiangonzalezriquelme95@gmail.com, cgonzaler@uc.cl}
	\author[Mu\~noz]{Claudio Mu\~noz}  
	\address{Departamento de Ingenier\'{\i}a Matem\'atica and Centro de Modelamiento Matem\'atico (UMI 2807 CNRS), Universidad de Chile, Casilla 170 Correo 3, Santiago, Chile.}
	\email{cmunoz@dim.uchile.cl}
	
    \author[Poblete]{Felipe Poblete}
	\address{Instituto de Ciencias F\'isicas y Matem\'aticas, Facultad de Ciencias, Universidad Austral de Chile, Valdivia, Chile.}
	\email{felipe.poblete@uach.cl}
	
    \author[Salinas]{Vicente Salinas}  
	\address{Departamento de Ingenier\'{\i}a Matem\'atica, Universidad de Chile, Casilla 170 Correo 3, Santiago, Chile.}
	\email{vsalinas@dim.uchile.cl}
	\thanks{R.F. was partially funded by Chilean research grants FONDECYT 3230256 and ANID Exploraci\'on 13220060. C.G-R. was supported by the Spanish Ministry of Science, Innovation and Universities,  grant PID2023-150984NB-I00. C. G-R. is also supported by the Spanish State Research Agency, through the Severo Ochoa and María de Maeztu Program for Centers and Units of Excellence in R\&D (CEX2020-001084-M) and thanks the CERCA Programme/Generalitat de Catalunya for institutional support. C.M. was partially funded by Chilean research grants ANID 2022 Exploration 13220060, FONDECYT 1191412, 1231250, and Basal CMM FB210005 and MathAmSud WAFFLE 23-MATH-18. F.P.'s work is partially supported by ANID Exploration project 13220060, ANID project FONDECYT 1221076 and MathAmSud WAFFLE 23-MATH-18. V. S. was funded by ANID 2022 Exploration 13220060, ANID Fondecyt 1231250 and Beca ANID-Subdirección de Capital Humano/Doctorado Nacional/2023-21231505.}

	
	\begin{abstract}
		
		Let $\chi_m$ be a primitive real-valued Dirichlet character
of conductor $m\ge3$, let $L(s)=L(s,\chi_m)$ be the
associated Dirichlet $L$-function, and let $\xi_L(s)$
denote its completed form. For each $t \in \R$, we define the entire function
		\[
		H^L_t(z) := \int_{-\infty}^{\infty} e^{tu^2}\, \Phi_L(u)\, e^{izu}\, du,
		\]
		where $\Phi_L$ is the superexponentially decaying density
		\[
		\Phi_L(u) := \frac{1}{2\pi}\int_{-\infty}^{\infty} \xi_L\!\left(\frac{1+is}{2}\right) e^{-isu}\, ds.
		\]
		The family $\{H_t^L\}_{t \in \R}$ describes the heat-flow evolution of $\xi_L$ and solves a backward heat equation. By the work of Newman~\cite{New}, there exists a finite constant $\Lambda_L$, called the generalized de Bruijn--Newman constant, such that all the zeros of $H_t^L$ are real if and only if $t \ge \Lambda_L$. The Generalized Riemann Hypothesis for $L$ is equivalent to the assertion that $\Lambda_L \le 0$. Dobner~\cite{Dob} established the matching lower bound $\Lambda_L \ge 0$, while de Bruijn~\cite{DBru} proved the upper bound $\Lambda_L \le \tfrac12$.
		
		We prove the unconditional uniform bound
$\Lambda_L\le953/3200=0.2978125$
for every primitive real-valued character of conductor
$3\le m\le400$, together with stronger individual bounds.
The proof extends the effective approximation of
Polymath~\cite{Poly} to Dirichlet $L$-functions, retaining
the conductor dependence in the error estimates.
We combine these estimates with Platt's numerical
verification of GRH~\cite{platt} and rigorous interval
computations. We also give conditional bounds under
explicit global initial zero-strip hypotheses, including
$|\operatorname{Im}z|\le3/4$ for the zeros of $H_0^L$.
\end{abstract}
	
	\maketitle
	
	\tableofcontents
	
	\section{Introduction}
	
	\subsection{Setting}
	The Riemann Hypothesis, formulated in 1859 \cite{R59}, asserts that every nontrivial zero of the Riemann zeta function $\zeta$ lies on the critical line $\re(s) = \tfrac12$. A convenient way to study this problem analytically is to embed it into a one-parameter family of entire functions arising from a backward heat flow. Let $\xi : \Com \to \Com$ denote the entire Riemann $\xi$ function
	\[
	\xi(s) := \frac{s(s-1)}{2}\pi^{-s/2}\,\Gamma\!\left(\frac{s}{2}\right)\zeta(s),
	\]
	where $\Gamma$ is the Gamma function and $\zeta$ is the meromorphic extension of the Riemann zeta function, whose only singularity is a simple pole at $s = 1$. The function $\xi$ satisfies the functional equation $\xi(s) = \xi(1-s)$, so its zeros are symmetric with respect to the critical line. Let
	\[
	H_0^{\zeta}(z) := \frac{1}{8}\, \xi\!\left(\frac{1+iz}{2}\right),
	\]
	be the function that transports the zeros of $\xi$ on the line $\re(s) = \tfrac12$ to the real axis $\ima(z) = 0$. The function $H_0^\zeta$ is entire and even, and inherits from $\xi$ the symmetry $\overline{H_0^\zeta(z)} = H_0^\zeta(\bar z)$. With this notation, the Riemann Hypothesis is equivalent to the statement that every zero of $H_0^\zeta$ is real. It is a classical fact \cite{Tit} that $H_0^\zeta$ admits the Fourier representation
	\[
	H_0^{\zeta}(z) = \int_0^{\infty} \Phi_\zeta(u)\cos(zu)\, du,
	\]
	in which the density
	\[
	\Phi_\zeta(u) := \sum_{n \ge 1} \bigl( 2\pi^2 n^4 e^{9u} - 3\pi n^2 e^{5u}\bigr) \exp\!\left(-\pi n^2 e^{4u}\right)
	\]
	decays superexponentially. The sum defining $\Phi_\zeta$ converges absolutely for every real $u$, and the Poisson summation formula (see, e.g., \cite[Ch.~X]{Tit}) yields the parity $\Phi_\zeta(u) = \Phi_\zeta(-u)$, which is itself a reformulation of the functional equation $\xi(s) = \xi(1-s)$. De Bruijn~\cite{DBru} embedded $H_0^\zeta$ into the family of entire functions
	\[
	H_t^{\zeta}(z) := \int_0^{\infty} e^{tu^2}\, \Phi_\zeta(u)\cos(zu)\, du, \qquad t \in \R.
	\]
	As noted by Ki, Kim and Lee~\cite{KKL}, $H_t^\zeta$ is the evolution of $H_0^\zeta$ under the \emph{backward heat equation}
	\[
	\partial_t H_t^\zeta(z) = -\partial_{zz} H_t^\zeta(z).
	\]
	Each member of the family is entire and even, and inherits from $\Phi_\zeta$ the symmetry $\overline{H_t^\zeta(z)} = H_t^\zeta(\bar z)$.
	
	The behavior of the family $\{H_t^\zeta\}_{t \in \R}$ has been the object of extensive study. De Bruijn~\cite{DBru} proved that all zeros of $H_t^\zeta$ are real once $t \ge \tfrac12$, and that the property of having only real zeros is monotone in $t$, so if every zero of $H_{t_0}^\zeta$ is real for some $t_0 < \tfrac12$, then the same holds for every $t > t_0$. Building on this monotonicity, Newman~\cite{New} established the existence of a finite constant $\Lambda^\zeta \le \tfrac12$, the de Bruijn--Newman constant, with the property that the zeros of $H_t^\zeta$ are all real if and only if $t \ge \Lambda^\zeta$. The Riemann Hypothesis is therefore equivalent to the inequality $\Lambda^\zeta \le 0$.
	
	Later work has narrowed the location of $\Lambda^\zeta$. Ki, Kim and Lee~\cite{KKL} sharpened de Bruijn's bound to $\Lambda^\zeta < \tfrac12$. On the lower-bound side, Rodgers and Tao~\cite{RT} showed that $\Lambda^\zeta \ge 0$, building on the description of the time evolution of the zeros developed by Csordas, Smith, and Varga~\cite{CSV} and the use of \emph{dispersive-type} techniques such as energies and virial complex-valued ODE functionals; consequently, the Riemann Hypothesis is equivalent to the equality $\Lambda^\zeta = 0$. Polymath~\cite{Poly} combined effective estimates for
$H_t^\zeta(x+iy)$ with Platt's rigorous verification of the
Riemann Hypothesis~\cite{plattpolymath} to obtain
$\Lambda^\zeta\le0.22$. Platt and Trudgian~\cite{PT21}
subsequently improved this bound to $\Lambda^\zeta\le0.2$.
	
	\subsection{Dirichlet \texorpdfstring{$L$}{L}-functions}
	We extend the program of Polymath~\cite{Poly} to Dirichlet $L$-functions
associated with primitive real-valued characters. The conductor enters
the approximation and its error terms, while the character changes the
finite sums and the Euler-product search for the barrier. We keep this
dependence explicit in the estimates used for the numerical verification.
	
	Let $\chi_m$ be a primitive Dirichlet character of period $m$, and let $L$ denote the associated Dirichlet series
	\begin{equation}\label{L_m_serie0}
		L(s) = \sum_{n=1}^{\infty} \frac{\chi_{m}(n)}{n^{s}},
	\end{equation}
	which converges absolutely for $\re(s) > 1$. Set $a := (1 - \chi_m(-1))/2 \in \{0,1\}$, so that $a = 1$ when $\chi_m$ is odd and $a = 0$ when $\chi_m$ is even. For $m=1$, the Dirichlet series is exactly the Riemann zeta function. In general, $L$ admits a meromorphic continuation to the entire complex plane \cite{MH06}: for the principal character, it has a simple pole at $s = 1$, while for every non-principal character, it is entire. In both cases, $L$ shares with $\zeta$ the critical strip $0 \le \re(s) \le 1$ and the critical line $\re(s) = \tfrac12$. Dirichlet $L$-functions encode important arithmetic information; in particular, they govern the distribution of primes in arithmetic progressions and provide the building blocks of the $L$-functions attached to number fields and modular forms. The Generalized Riemann Hypothesis (GRH) asserts that every nontrivial zero of every Dirichlet $L$-function lies on the critical line $\re(s) = \tfrac12$.

    Dirichlet $L$-functions are the simplest members of the Selberg class
\cite{Selberg,KaczorowskiPerelli}, and the Generalized Riemann Hypothesis
(GRH) for them has the same standing and the same consequences for the
distribution of primes in arithmetic progressions as the Riemann
Hypothesis does for the primes themselves; see \cite{Davenport,
IwaniecKowalski,MH06} for the classical theory and
\cite{Bombieri, Conrey, Sarnak} for broader surveys. For a real primitive
character the critical strip has an additional feature which has no
analogue for $\zeta$: $L(s,\chi_m)$ may have a real zero close to $s=1$,
the possible Landau--Siegel zero \cite{Landau, Siegel35}, whose
non-existence is tied to the class number problem \cite{Goldfeld}. Recently, some works have ruled out the existence of these zeros \cite{OPENAI1,OPENAI2}. In
the heat-flow picture such a zero sits on the imaginary axis of $H^L_0$,
and this is why the real primitive characters, rather than general ones,
are the natural first case for the de Bruijn--Newman program. Numerical
evidence for GRH goes back to Rumely \cite{Rumely}, and the verification
of Platt \cite{platt}, which covers every primitive character of
conductor $m\le 4\cdot 10^5$ up to height $h_m$ and excludes real zeros in
that range, is the input on which our certificates rest. The non-negativity
of the generalized de Bruijn--Newman constant for this family is Dobner's
theorem \cite{Dob}.

	Following \cite{KKL}, we attach to each primitive real-valued
Dirichlet character $\chi_m$ of conductor $m\geq3$ the
one-parameter family of entire functions
	\begin{equation}\label{H_tz}
		\begin{aligned}
			H_t^L(z)
			&:= \int_0^{\infty}
			\exp\!\left(\frac{t}{16}\log^2(x) + \frac{iz}{4}\log(x)\right)
			\varphi_L(x)\,\frac{dx}{x} \\
			&= \int_{-\infty}^{\infty} e^{t u^2 + izu}\,\Phi_L(u)\,du,
		\end{aligned}
	\end{equation}
	where the density $\Phi_L$ is given by
	\begin{equation}\label{Phi_L}
		\begin{aligned}
			& \Phi_L(u) = 4\,\varphi_L(e^{4u}), \qquad \varphi_L(x) = x^{\frac{1}{4} + \frac{a}{2}}\, \psi_L(x), \\
			& \text{and} \qquad \psi_L(x) = \sum_{n=1}^{\infty} n^{a}\, \chi_m(n)\, e^{-\frac{n^{2}\pi x}{m}}.
		\end{aligned}
	\end{equation}
	The case $m=1$ requires separate normalization for
$\zeta$ introduced above and is not included in
\eqref{H_tz}--\eqref{Phi_L}. In the Dirichlet case, GRH is equivalent to the assertion that $H_0^L$ has only real zeros, and a suitable adaptation of the de Bruijn--Newman argument produces a finite constant $\Lambda_L$, the generalized de Bruijn--Newman constant attached to $L$, characterized by the property that $H_t^L$ has only real zeros if and only if $t \ge \Lambda_L$. The bound $\Lambda_L \le \tfrac12$ follows directly from de Bruijn's monotonicity argument. On the lower-bound side, Dobner~\cite{Dob} extended the Rodgers--Tao theorem to a wide class of $L$-functions in the extended Selberg class, establishing $\Lambda_L \ge 0$ in considerable generality, including the family considered in this work.
	
Platt~\cite[, Theorem~10.1]{platt} verified GRH for every primitive character of
modulus $m\le4\times10^5$ up to height
\[
 h_m:=\max\left\{\frac{10^8}{m},\,
 200+A(m)\frac{10^7}{m}\right\},
\]
where $A(m)=7.5$ for even $m$ and $A(m)=3.75$ for odd $m$.
In this range,
\[
 mh_m=\max\{10^8,\,200m+A(m)10^7\}\le1.55\times10^8.
\]
Under $s=(1+iz)/2$, this verification covers
$|\operatorname{Re}z|\le2h_m$.
	For a real primitive character, a possible real zero $\beta$ of $L$
in $(0,1)$ gives a purely imaginary zero $i(1-2\beta)$ of $H_0^L$.
The functional equation supplies the reflected zero $1-\beta$ and hence
the opposite imaginary zero. A Landau--Siegel zero is a real zero
unusually close to $1$; the standard zero-free-region results do not
assert that every real zero must be of this form; see \cite[Ch.~11]{MH06}.
In the conductor range covered by Platt~\cite{platt}, no real zero
occurs. This fact will give positivity on the imaginary axis for
$t\ge0$ and is used in the proof of our zero-free-region criterion.

	\subsection{Main result}
	Our main result gives upper bounds for $\Lambda_L$ for every primitive
real-valued character of conductor $3\le m\le400$. We combine the
zero-free-region criterion and an effective Riemann--Siegel approximation
with Platt's verification of GRH~\cite{platt}. The numerical checks are
described in Section~\ref{sec:numerical_implementation_main}; smaller
initial strips give conditional bounds under the hypotheses stated below.

\begin{theorem}[Upper bound]
\label{MT}
Let $\chi_m$ be a primitive real-valued Dirichlet character of
conductor $3\leq m\leq400$, and let $L$ denote the associated
Dirichlet $L$-function. Let $0<D\leq1$, and assume that every
zero $z$ of $H_0^L$ satisfies $|\operatorname{Im}z|\leq D$.
Then the generalized de Bruijn--Newman constant attached to $L$
satisfies
\[
\Lambda_L\leq
\min\left\{\frac{D^2}{2},\,\Lambda_*(m,\chi_m;D)\right\},
\]
where $\Lambda_*(m,\chi_m;D)$ is defined in~\eqref{eq:lambda_star}
as the infimum of $t_0+\tfrac12 y_0^2$ over the verified admissible triples for this $D$.
\end{theorem}
	
	For $D=1$, the computations in Section~\ref{sec:numerical_implementation_main}
give the uniform bound $\Lambda_L\le953/3200=0.2978125<1/2$ for
$3\le m\le400$, with the stronger individual bounds listed in the
appendix. The additional conclusions for $D<1$ use the respective
initial-strip hypotheses.

\subsection{Strategy of the proof and novelties}

The proof of Theorem~\ref{MT} relies on the following two results,
which may also be of independent interest.
\begin{itemize}
\item[(a)] An upper-bound criterion (Lemma~\ref{lemmabox}),
which reduces the problem of bounding $\Lambda_L$ to the
verification of three explicit zero-free regions for the family
$\{H_t^L\}_{0\leq t\leq t_0}$. The criterion adapts the strategy
of \cite{Poly} to the $L$-function setting under the initial
zero-strip hypothesis.

\item[(b)] An effective Riemann--Siegel approximation for
$H_t^L(x+iy)$ in the range $0\leq t\leq\tfrac12$,
$x\geq480$ and $|y|\leq1$
(Theorem~\ref{EffectiveRSa}), which provides quantitative
control of the error and isolates the dependence on the
conductor $m$. The same analytic estimates also yield information on the
distribution of the zeros of $H_t^L$ for $t>0$.
In Section~\ref{sec:Nyg}, we describe their location and
derive a Riemann--von Mangoldt-type counting formula
(Theorem~\ref{Distributionzero}).
\end{itemize}

From a global point of view, the overall architecture in this paper is the one introduced by Polymath \cite{Poly}, namely a
zero-free-region criterion for the heat-flow family, an effective
Riemann-Siegel approximation with explicit error terms, and a rigorous
numerical verification of a finite number of inequalities. Each of the
three parts has to be rebuilt for Dirichlet $L$-functions, and the places
where the conductor $m$ and the character $\chi_m$ enter are not
cosmetic. We describe them in turn, and then state what was known before.

De Bruijn's strip-contraction theorem gives
$\Lambda_L\le t_0+\tfrac12 y_0^2$ as soon as $H^L_{t_0}$ has no zeros with
$\operatorname{Im} z\ge y_0$. This is reduced to
three regions: an initial rectangle at $t=0$, which is covered by Platt's
verification of GRH up to height $h_m$ \cite{platt}; a moving barrier
in a vertical strip of width $\Delta_D$; and a right tail at the final
time $t_0$. The proof is the zero-dynamics argument of Csordas, Smith, and
Varga \cite{CSV} as organized in \cite{Poly}: the zeros move according
to the differential equation obtained from the Hadamard product, and a comparison of the vertical velocities rules out
a first crossing of the moving boundary. Two features are
specific to $L$-functions. First, $H^L_0$ need not be positive on the
imaginary axis: a real zero $\beta$ of $L$ produces a purely imaginary zero
of $H^L_0$. We use Platt's verification of GRH and his non-vanishing
result at the central point \cite[Theorems~10.1 and~10.2]{platt}
to obtain $H_t^L(iy)>0$, and
we record separately what can be said about zeros on $i\mathbb{R}$
without that input. Second, we work from the outset with a general initial
strip $|\operatorname{Im} z|\le D$, $0<D\le 1$, rather than with the
critical strip only. This costs nothing in the argument and makes the
result directly applicable to any zero-free half-plane for $L$; moreover, a verification carried out at $D=1$ also covers
the same parameter triple for every smaller admissible
value of $D$.

This is where most of the new work lies. In Siegel's formula for
$L(s,\chi_m)$ \cite{Siegel} the kernel $e^{i\pi w^2}/(e^{i\pi
w}-e^{-i\pi w})$ of the zeta case is replaced by
\[
W^L(x)=\frac{\pi}{2m}\sum_{k=1}^{2m}\chi_m(k)\,e^{-i\pi k^2/m}
\cot\Big(\frac{\pi(x-k)}{2m}\Big),
\]
a sum of $2m$ cotangents with Gauss-sum weights, which is $m$-periodic up
to a sign. Deforming the contour through the saddle point, as in Arias de
Reyna \cite{AR}, one obtains the Dirichlet sum truncated at
$M=\lfloor\sqrt{mT/2\pi}\rfloor$, a leading term $U\delta^{-\sigma}
C^{(m)}_0(p)$, the higher Taylor coefficients, and an integral remainder. Three points differ from the zeta
function and, as far as we know, had not been worked out with explicit
constants before.

\begin{itemize}
\item[(i)] The leading coefficient $C^{(m)}_0$ is no longer an
elementary function of $p$. We write
$C^{(m)}_0(p)=e^{-i\pi/8}\mathcal J_m(\delta)$, where
$\mathcal J_m(\delta)$ is a contour integral of a Gaussian against
$W^L$. For fixed $M$, it has the expression
$\mathcal J_{m,M}(\delta)
=\frac{\mathcal R_{m,M}(\delta)-\rho(p)\mathcal K_m(\delta)}
{1-\rho(p)}$ where $\rho(p)=(-1)^m e^{2\pi i p}.$
Here $\mathcal R_{m,M}$ is an incomplete sum of $m$ terms, and
$\mathcal K_m$ is a finite combination of Gauss-type sums. On each fixed branch $M$, the identity
$(1-\rho)\mathcal J_{m,M}=\mathcal R_{m,M}-\rho\mathcal K_m$
extends to the resonant points by the identity theorem. Thus, at
$\rho=1$, the numerator vanishes and the apparent singularity in the
quotient is removable. The uniform bound
below follows from the contour estimate; no separate evaluation of
quadratic Gauss sums is needed. In the numerical checks, $M$ is fixed
on each subregion. At a boundary where $M$ changes, we use both
one-sided values of $C^{(m)}_0$ from the adjacent branches.
\item[(ii)] A uniform bound $|C^{(m)}_0(p)|\le B^\star_m$, with
$B^\star_m/\sqrt m$ explicit in $m$ and smaller than $3.33$ for
$3\le m\le 4\cdot10^5$. It rests on a
non-increasing majorant $\Omega_m(|\operatorname{Im} x|)$ for $|W^L|$ on
the shifted line and on a rearrangement argument
for the Gaussian convolution; this avoids both
estimates for incomplete character sums and numerical maximization.
\item[(iii)] Bounds for the higher coefficients and the remainder in which $\varphi(m)/\sqrt m$ and
the constants of \cite{AR} appear explicitly. The Gaussian
smoothing in $t$ forces us to allow arbitrary negative real parts, which
requires a Stirling estimate uniform on horizontal strips
 and a truncation order $J(u)$ adapted to the
real part $u$.
\end{itemize}

The outcome is the following: for $x\ge 480$, $|y|\le1$ and
$0\le t\le\tfrac12$,
$H^L_t/B^L_t=f^L_t+O_{\le}(e_A+e_B+e_{C,0})$, with every constant being 
explicit and uniform in $3\le m\le 4\cdot10^5$;
We also record the dependence on $D$. We are not aware of
an earlier Riemann-Siegel formula for Dirichlet $L$-functions that is
effective in this sense. The expansions of Siegel \cite{Siegel} and
Deuring \cite{Deuring} are asymptotic, and the numerical verifications of
Rumely \cite{Rumely} and Platt \cite{platt} rely on different
evaluation methods.

The truncation length
$M(x,t)=\lfloor\sqrt{m(x+\pi t/4)/(4\pi)}\rfloor$
depends on $m$, $x$, and $t$, and may vary across a barrier box.
In the range considered, it takes at most seven consecutive integer values. We therefore
divide region~(a) into stripes on which $M$ is constant, using the
continued lateral values of the approximation on their common edges. The Euler product used to locate the
barrier omits the primes dividing $m$. For regions (b) and (c), we evaluate the pointwise error bounds at the relevant endpoints. On each time interval, the factors in the finite sums are bounded separately at the corners; no monotonicity in \(t\) is assumed for the finite sum. All inequalities
are checked in ball arithmetic. The Taylor scheme of
\cite[\S7]{Poly} is adapted with coefficient tables that do not depend
on $t$. As a by-product of the same estimates, we
provide a Riemann-von Mangoldt formula and a localization of the zeros of
$H^L_t$ for $t>0$, uniform in the conductor range.

Last but not least, we record previous and related contributions to the problem considered in this work. For the zeta function, after de Bruijn \cite{DBru} and Newman
\cite{New}, the lower bounds of Csordas, Norfolk and Varga
\cite{CNV}, Odlyzko \cite{Odlyzko} and Saouter, Gourdon and Demichel
\cite{SGD} led to the theorem $\Lambda_\zeta\ge0$ of Rodgers and Tao
\cite{RT}, while \cite{KKL}, \cite{Poly} and
\cite{PT21} gave the upper bounds $\Lambda_\zeta<\tfrac12$,
$\le0.22$ and $\le0.2$. For Dirichlet $L$-functions much less was
available. The heat-flow deformation of quadratic $L$-functions was
studied by Stopple \cite{Stopple}, who obtained a lower bound of the
order $-10^{-7}$ for the constant attached to the family
of quadratic characters with negative fundamental
discriminants, and the Newman conjecture was examined in automorphic and
function-field settings by Andrade, Chang and Miller \cite{ACM} and Chang
et al.\ \cite{CMMRSY}; the differenced $L$-functions of Lagarias
\cite{Lagarias} are a closely related deformation. The non-negativity
$\Lambda_L\ge0$ was established in great generality by Dobner
\cite{Dob}, and the survey of Newman and Wu \cite{NewmanWu} describes
the broader picture. On the upper-bound side, only the trivial bound
$\Lambda_L\le\tfrac12$ from de Bruijn's argument was known.
Theorem~\ref{MT} gives the first non-trivial upper bounds for a
Dirichlet $L$-function. These bounds are weaker than those available for $\zeta$.
In our numerical range, the verification height $h_m=10^8/m$
is much smaller than the $3\cdot10^{12}$ of \cite{PT21}.
We also give a heuristic estimate of how increasing this
height could reduce $t_0$.

\subsection*{Organization of this work} Section~\ref{sec:2} recalls the normalization and functional equation.
Section~\ref{DynSection} gives the dynamics of zeros and the zero-free-region
criterion. The effective approximation and its proof are in
Section~\ref{sec:4}, followed by the fast evaluation scheme in
Section~\ref{sec:numerics}. Sections~\ref{sec:verification} and
\ref{sec:numerical_implementation_main} give the verification
inequalities and the numerical results.
Section~\ref{sec:Nyg} studies the distribution of zeros of
$H_t^L$, and Appendix~\ref{app:certified_m400} lists the
parameter choices and the resulting bounds.

\subsection*{Acknowledgments}
	All authors of this paper were supported by Fondecyt Exploraci\'on 13220060. ANID's support in this direction is warmly acknowledged. 
	
	\section{Preliminaries}\label{sec:2}
	
	\subsection{Dirichlet characters and Dirichlet
\texorpdfstring{$L$}{L}-functions}
	
	We recall the functional equation and the normalization of the heat flow used below.
	
	Throughout the paper, $\gcd(m,n)$ denotes the greatest common divisor of the nonnegative integers $m$ and $n$, and we say that $m$ and $n$ are coprime when $\gcd(m,n) = 1$. For convenience, we restrict the input set to $\N$, although the definitions below extend verbatim to $\Z$.
	
	\begin{definition}
		A Dirichlet character of period $m \in \N$, $m > 0$, is a complex-valued function $\chi_m : \N \to \Com$ satisfying:
		\begin{enumerate}
			\item $\chi_m(n + m) = \chi_m(n)$ for all $n \in \N$.
			\smallskip
			\item $\chi_m(ab) = \chi_m(a)\,\chi_m(b)$ for all $a, b \in \N$.
			\smallskip
			\item For all $n \in \N$,
			\begin{align*}
				& \chi_m(n) \ne 0 \quad \text{if } \gcd(n,m) = 1,\\
				& \chi_m(n) = 0 \quad \text{if } \gcd(n,m) > 1.
			\end{align*}
		\end{enumerate}
	\end{definition}
	
	Dirichlet characters satisfy a number of elementary properties.
	
	\begin{remark}
		The following hold:
		\begin{enumerate}
			\item $\chi_m(1) = 1$, $\chi_m(n^p) = \chi_m(n)^p$ for every $p \in \N$, and $\chi_m(n_1) = \chi_m(n_2)$ whenever $n_1 \equiv n_2 \pmod{m}$.
			\smallskip
			\item If $m > 1$, then $\chi_m(km) = 0$ for every $k \in \N$.
			\smallskip
			\item If $m = 1$, then $\chi_m(n) = 1$ for all $n \in \N$.
		\end{enumerate}
	\end{remark}
	
	\begin{definition}
		A Dirichlet character of period $m$ is \emph{principal} if for all $n \in \N$,
		\begin{align*}
			& \chi_m(n) = 1 \quad \text{if } \gcd(n,m) = 1,\\
			& \chi_m(n) = 0 \quad \text{if } \gcd(n,m) > 1.
		\end{align*}
	\end{definition}

    A principal character is the identity element of the group of Dirichlet characters modulo $m$: if $\chi_m$ is principal and $\widetilde{\chi}_m$ is any Dirichlet character of the same modulus, then
	\(
	\widetilde{\chi}_m(n)\,\chi_m(n) = \chi_m(n)\,\widetilde{\chi}_m(n) = \widetilde{\chi}_m(n)
	\)
	for every $n \in \N$. In particular, the product of two characters of the same modulus is again a character of that modulus. In addition to principal characters, we shall need the notion of a primitive character.
	
	\begin{definition}\label{def_prim}
		The following terminology is used throughout.
		\begin{enumerate}
			\item A character $\chi_m$ of period $m$ has quasiperiod $d$ if $\chi_m(n_1) = \chi_m(n_2)$ for all $n_1, n_2$ coprime to $m$ with $n_1 \equiv n_2 \pmod{d}$.
			\item The smallest positive integer $d$ for which $\chi_m$ has quasiperiod $d$ is the conductor of $\chi_m$.
			\item The character $\chi_m$ is primitive if its period equals its conductor; otherwise, it is imprimitive. Every imprimitive character is induced by a primitive character of smaller period.
		\end{enumerate}
	\end{definition}
	
	We now consider the Dirichlet series induced by a primitive character $\chi_m$ of period $m$, with $a = (1 - \chi_m(-1))/2$, already introduced in \eqref{L_m_serie0}. Thus $a = 1$ when $\chi_m$ is odd and $a = 0$ when $\chi_m$ is even.
	
	\begin{definition}
		Let $\chi_m$ be a Dirichlet character of period $m$. For $\re(s) > 1$, the associated Dirichlet $L$-function is
		\begin{equation}\label{L_m_serie}
			L(s,\chi_m) = \sum_{n=1}^\infty \frac{\chi_m(n)}{n^s}.
		\end{equation}
	\end{definition}
	
	It is well known that \eqref{L_m_serie} admits a meromorphic extension to the complex plane, which is entire when $\chi_m$ is non-principal. We denote the corresponding extension by $L$.
	We define the auxiliary function $M_0^L : \Com \setminus (-\infty, 0] \to \Com \setminus \{0\}$ by
	\begin{equation}\label{M0L}
		M_0^L(s) := \left(\frac{m}{\pi}\right)^{\frac{s+a}{2}} \sqrt{2\pi}\, \exp\!\left( \left(\frac{s+a-1}{2}\right)\log\!\left(\frac{s+a}{2}\right) - \frac{s+a}{2} \right),
	\end{equation}
	so that
	\[
	\log M_0^L(s) = \tfrac12 \log(2\pi) + \frac{s+a}{2}\log\!\left(\frac{m}{\pi}\right) + \left(\frac{s+a-1}{2}\right)\log\!\left(\frac{s+a}{2}\right) - \frac{s+a}{2}.
	\]
	Its logarithmic derivative is
	\begin{equation}\label{alpha_L}
		\begin{aligned}
			\alpha_L(s) := (\log M_0^L)'(s) &= \frac{1}{2}\log\!\left(\frac{m}{\pi}\right) + \frac{1}{2}\log\!\left(\frac{s+a}{2}\right) - \frac{1}{2(s+a)} \\
			&= \frac{1}{2}\log\!\left(\frac{m(s+a)}{2\pi}\right) - \frac{1}{2(s+a)},\\[2pt]
			\alpha_L'(s) &= \frac{1}{2(s+a)} + \frac{1}{2(s+a)^2}.
		\end{aligned}
	\end{equation}
	For each $t \ge 0$, we define the deformation $M_t^L : \Com \setminus (-\infty, 0] \to \Com \setminus \{0\}$ of $M_0^L$ by
	\begin{equation}\label{MtL}
		M_t^L(s) := \exp\!\left(\frac{t}{4}\,\alpha_L(s)^2\right) M_0^L(s).
	\end{equation}
	In the regime of large $x$ we introduce the auxiliary function
	\[
		B_t^L(x + iy) := M_t^L\!\left(\frac{1 + y - ix}{2}\right).
	\]
	The following symmetry properties of $H_t^L$ are used repeatedly in what follows.
	
	\begin{lemma}\label{prop:symmetric}
		Let $\chi_m$ be a primitive real-valued Dirichlet
character of conductor $m\geq3$, and let $t\in\R$.
For $H_t^L$ defined by \eqref{H_tz}, the following hold.
		\begin{enumerate}
			\item (Real-axis symmetry.) For every $z \in \Com$,
			\[
			H_t^L(z) = \overline{H_t^L(\overline{z})},
			\]
			so that the zeros of $H_t^L$ are symmetric with respect to the real axis.
			\item (Double symmetry.) If $H_t^L(x_* + iy_*) = 0$, then
			\[
			H_t^L(\pm x_* \pm i y_*) = H_t^L(\mp x_* \pm i y_*) = 0.
			\]
		\end{enumerate}
	\end{lemma}

\begin{proof}
The density $\Phi_L$ is real and even by \eqref{Phi_L} and
\eqref{eq:phi_inversion_L}. Conjugation and the change of variables
$u\mapsto-u$ in \eqref{H_tz} therefore give
\[
\overline{H_t^L(\overline z)}=H_t^L(z),
\qquad H_t^L(-z)=H_t^L(z).
\]
Thus a zero $z$ is accompanied by $\overline z$, $-z$ and
$-\overline z$, as claimed.
\end{proof}

	\subsection{The heat flow}
	We use the following strip-contraction theorem and the finiteness result for nonreal and multiple zeros.
	
	The following result is essentially contained in Theorem~13 of \cite{DBru}; see also \cite{Dob}.
	
	\begin{theorem}\label{dBru}
		Suppose that $t_0 \in \R$ and $y_0 > 0$ are such that $H_{t_0}^L(x + iy) \ne 0$ for every $x \in \R$ and $y > y_0$. Then, for any $t > t_0$, $H_t^L(x + iy) \ne 0$ for every
		\[
		x \in \R \quad \text{and} \quad y > \max\bigl( y_0^2 - 2(t - t_0),\ 0 \bigr)^{1/2}.
		\]
		In particular, $\Lambda_L \le t_0 + \tfrac{1}{2} y_0^{2}$.
	\end{theorem}
	
We record the finiteness property for nonreal and multiple zeros.
Ki, Kim and Lee~\cite{KKL} establish the corresponding result
for the Riemann zeta family. In the conductor range considered
below, the Dirichlet analogue follows from
Theorem~\ref{Distributionzero} and the local splitting of
multiple zeros in Proposition~\ref{Dynz}.

\begin{theorem}
Let $\chi_m$ be a primitive real-valued Dirichlet character
of conductor $3\le m\le400000$, and let $H_t^L$ be defined
by \eqref{H_tz}. Then, for every $t>0$:
\begin{enumerate}
\item all but finitely many zeros of $H_t^L$ are real and simple;
\item $H_t^L$ has only finitely many multiple zeros.
\end{enumerate}
\end{theorem}

\begin{proof}
First suppose that $0<t\le1/2$.
Theorem~\ref{Distributionzero}(i)--(ii) implies that every
zero with sufficiently large positive real part is real and
simple. Indeed, the location error in (i) is eventually
smaller than the radius of the corresponding disk in (ii),
which contains exactly one zero counted with multiplicity.
Lemma~\ref{prop:symmetric} gives the same conclusion for
zeros with sufficiently large negative real part.

Theorem~\ref{dBru}, applied to the classical initial strip
$|\operatorname{Im}z|\le1$, bounds the imaginary parts of
all zeros. The remaining zeros therefore lie in a compact
set and are finite, since $H_t^L$ is a nonzero entire function.
This proves (1) for $0<t\le1/2$.

Now let $t>1/2$. Theorem~\ref{dBru} implies that all zeros
of $H_t^L$ are real. If a multiple zero occurred at this
time, Proposition~\ref{Dynz}(ii) would produce nonreal zeros
at sufficiently close earlier times, which could still be
chosen greater than $1/2$. This contradicts the same
strip-contraction theorem. Thus all zeros are real and
simple for $t>1/2$.

Finally, every multiple zero belongs to the finite
exceptional set in (1), proving (2).
\end{proof}
	\subsection{Functional equation}
	For a primitive character, the completed $L$-function satisfies the following functional equation; see \cite[Theorem~10.8]{MH06}.
	
	\begin{theorem}[Functional equation]
		The following hold.
		\begin{enumerate}
			\item Assume that the Dirichlet character $\chi_m$ associated to $L$ is primitive (cf.\ Definition~\ref{def_prim}). Then $L$ admits a symmetrized form $\xi_L$ satisfying the functional equation
			\begin{equation}\label{vicho1}
				\xi_L(s) = W(\chi_m)\,\xi_{L_c}(1 - s), \qquad L_c := L(\,\cdot\,, \overline{\chi}_m),
			\end{equation}
			where $\chi_m(-1) = (-1)^a$, the Gauss sum
			\(
			\tau(\chi_m) := \sum_{n=1}^{m} \chi_m(n)\, e^{2\pi i n/m}
			\)
			has modulus $|\tau(\chi_m)| = \sqrt{m}$, $W(\chi_m) = \tau(\chi_m)/(i^a \sqrt{m})$, and
			\begin{equation}\label{vicho2}
				\xi_L(s) = \left(\frac{m}{\pi}\right)^{\frac{s+a}{2}}\,\Gamma\!\left(\frac{s+a}{2}\right)L(s).
			\end{equation}
			When $\chi_m$ is real, $\chi_m = \overline{\chi}_m$ and $L(\chi_m, s) = L(\overline{\chi}_m, s)$, so that $\xi_L$ is symmetric.
			
			\item If the Dirichlet character $\chi_m$ associated to $L$ is imprimitive, then there exists a primitive character $\chi^*_q$ with $q \mid m$ and a corresponding $L$-function $L^*$ such that
			\begin{equation}\label{vicho3}
				L(s) = L^*(s) \prod_{p \mid m}\left(1 - \frac{\chi^*_q(p)}{p^s}\right).
			\end{equation}
		\end{enumerate}
	\end{theorem}
	
	\begin{remark}
		Equation \eqref{vicho1} requires that $L$ be associated with a primitive character $\chi_m$ and will be the key ingredient in the definition of the heat-flow function $H_0^L$ below. Equation \eqref{vicho3} shows that in the imprimitive case the situation is considerably more involved because of the additional Euler factors on the right-hand side.
	\end{remark}
	
	\begin{lemma}\label{lem:rootnumber}
		Let $\chi_m$ be a primitive real-valued Dirichlet character of conductor $m$. Then $W(\chi_m) = +1$ and the functional equation \eqref{vicho1} reduces to
		\begin{equation}\label{eq:funceq_real}
			\xi_L(s) = \xi_L(1 - s).
		\end{equation}
	\end{lemma}
	
	\begin{proof}
		Real primitive characters of conductor $m$ are precisely the Kronecker symbols $\chi_d(\cdot) = \bigl(\frac{d}{\cdot}\bigr)$ with $d$ a fundamental discriminant such that $|d| = m$; see e.g.~\cite[Theorem~9.13]{MH06}. For such characters one has $\tau(\chi_m)^2 = \chi_m(-1)\,m = (-1)^a\,m$, so $\tau(\chi_m) = \varepsilon\,i^a\,\sqrt m$ with $\varepsilon \in \{\pm 1\}$. The classical evaluation of the quadratic Gauss sum (cf.~\cite[Theorem~9.16]{MH06}) gives $\varepsilon = +1$. Hence $W(\chi_m) = \tau(\chi_m)/(i^a\sqrt m) = +1$, and \eqref{vicho1} together with $\overline\chi_m = \chi_m$ yields \eqref{eq:funceq_real}.
	\end{proof}
	
	We now verify that the density in \eqref{Phi_L} gives the
completed $L$-function in \eqref{vicho2}. Throughout this
argument, $\chi_m$ is primitive and real-valued, with
$m\geq3$.

Poisson summation, applied to the Gaussian when $a=0$ and
to its derivative when $a=1$, gives
\begin{equation}\label{eq:theta_transform_L}
 \psi_L(x)
 =\frac{\tau(\chi_m)}{i^a\sqrt m}\,
  x^{-a-\frac12}\psi_L(1/x)
 =x^{-a-\frac12}\psi_L(1/x),
 \qquad x>0.
\end{equation}
Here we used $\overline{\chi_m}=\chi_m$ and
Lemma~\ref{lem:rootnumber}. The same formula holds for
the sums over positive integers because
$n^a\chi_m(n)$ is even in $n$, and the term at zero
vanishes. Multiplying \eqref{eq:theta_transform_L} by
$x^{1/4+a/2}$ yields
\begin{equation}\label{eq:phi_inversion_L}
 \varphi_L(x)=\varphi_L(1/x),\qquad x>0.
\end{equation}
Consequently, $\Phi_L$ is real-valued and even.

The defining series gives
$\psi_L(x)=O_m(e^{-\pi x/m})$ for $x\geq1$.
Together with \eqref{eq:phi_inversion_L}, this implies
\begin{equation}\label{eq:Phi_decay_L}
 |\Phi_L(u)|
 \leq C_m
 \exp\!\left((1+2a)|u|-\frac{\pi}{m}e^{4|u|}\right),
 \qquad u\in\R,
\end{equation}
for some constant $C_m>0$. Differentiating the series
on $u\geq0$ and using evenness shows that every
derivative of $\Phi_L$ also decays faster than any
power of $|u|$. In particular, $\Phi_L$ belongs to
the Schwartz class.

For $\re s>1$, absolute convergence permits integration
term by term:
\begin{align*}
 \int_0^\infty x^{(s+a)/2}\psi_L(x)\,\frac{dx}{x}
 &=
 \sum_{n=1}^\infty n^a\chi_m(n)
 \int_0^\infty
 x^{(s+a)/2}e^{-\pi n^2x/m}\,\frac{dx}{x}\\
 &=
 \left(\frac m\pi\right)^{(s+a)/2}
 \Gamma\!\left(\frac{s+a}{2}\right)
 \sum_{n=1}^\infty\frac{\chi_m(n)}{n^s}\\
 &=\xi_L(s).
\end{align*}
Indeed, the sum of the integrals of the absolute
values is a constant depending on $s,m,a$ times
$\sum_{n\geq1}|\chi_m(n)|n^{-\re s}$, which converges
in this half-plane. On the other hand,
\eqref{eq:theta_transform_L} and the exponential
decay of $\psi_L$ at infinity show that the integral
defines an entire function of $s$. Thus the identity
extends to every $s\in\Com$.

Taking $s=(1+iz)/2$ and then making the change of
variables $x=e^{4u}$, we obtain
\[
 \begin{aligned}
 \xi_L\!\left(\frac{1+iz}{2}\right)
 &=
 \int_0^\infty x^{iz/4}\varphi_L(x)\,\frac{dx}{x}\\
 &=
 \int_{\R}4\varphi_L(e^{4u})e^{izu}\,du
 =H_0^L(z).
 \end{aligned}
\]
Fourier inversion now gives
\[
 \Phi_L(u)
 =\frac1{2\pi}\int_{\R}
 \xi_L\!\left(\frac{1+iv}{2}\right)e^{-ivu}\,dv,
 \qquad u\in\R,
\]
which agrees with the definition in the abstract.

Finally, \eqref{eq:Phi_decay_L} dominates the integrand
in \eqref{H_tz}, together with all its derivatives
in $t$ and $z$, uniformly when these variables range
over compact sets. Hence $H_t^L$ is entire in $z$
for every real $t$, and differentiation under the
integral gives
\[
 \partial_t H_t^L(z)
 =\int_{\R}u^2e^{tu^2+izu}\Phi_L(u)\,du
 =-\partial_z^2H_t^L(z).
\]
	\section{Dynamics of Zeros}\label{DynSection}
	
	We describe the motion of the zeros of $H_t^L$ and prove the zero-free-region criterion used to bound $\Lambda_L$. We also record the local behaviour at multiple zeros and on the imaginary axis.
	
\begin{lemma}\label{lem:positive_imaginary_axis}
		Let $L$ be associated with a real-valued primitive character of conductor
		$3\le m\le 400000$. Then
		\[
		H_t^L(iy)>0\qquad(t\ge0,\ y\in\mathbb R).
		\]
	\end{lemma}
	
	\begin{proof}
		By Platt's verification of GRH and the nonvanishing result on the
		central line, $L$ has no real zero in $0<s<1$; see
		\cite[Theorems~10.1 and~10.2]{platt}. The completed function $\xi_L$
		is real on the real axis and has no real zero. Since it is positive
		for real $s>1$, continuity gives $\xi_L(s)>0$ for every real $s$.
		For $t>0$, the heat-kernel representation \eqref{expH0} gives
		\[
		H_t^L(iy)=\int_{\mathbb R}
		\xi_L\!\left(\frac{1-y}{2}+\sqrt t\,v\right)
		\frac{e^{-v^2}}{\sqrt\pi}\,dv>0.
		\]
		The integral converges by Stirling's formula, since the Gaussian
		factor dominates the growth of $\xi_L$ on the real axis. At $t=0$
		the same conclusion follows from $H_0^L(iy)=\xi_L((1-y)/2)>0$.
	\end{proof}
	
	\subsection{Motion of zeros}
	Following the scheme of \cite{Poly}, we begin with the dynamics of a single zero of $H_t^L$, postponing the discussion of zeros on the imaginary axis.
	
	\begin{prop}\label{Dynz}
		Let $t_0 \in \R$ and let $(z_k(t_0))_{k \in \Z \setminus \{0\}}$ enumerate the nonzero zeros of $H_{t_0}^L$, counting multiplicity, with $z_{-k}(t_0)=-z_k(t_0)$. Write $r_0(t_0)$ for the multiplicity of its zero at the origin, and put $r_0(t_0)=0$ if there is no such zero. Then the following hold.
		\begin{enumerate}
			\item If $z_j(t_0)$ is a simple zero of $H_{t_0}^L$, then there exist a neighborhood $U$ of $z_j(t_0)$, a neighborhood $I$ of $t_0$ in $\R$, and a smooth map $z_j : I \to U$ such that for every $t \in I$, $z_j(t)$ is the unique zero of $H_t^L$ in $U$. Moreover,
			\[
			\frac{dz_j}{dt}(t_0) = 2\left(\frac{r_0(t_0)}{z_j(t_0)}+\sum_{k \ne j}{}'\, \frac{1}{z_j(t_0) - z_k(t_0)}\right),
			\]
			where the sum runs over $k \in \Z\setminus\{0\}$ with $k \ne j$, and the prime indicates that the terms $k$ and $-k$ are summed together (except for the term $k = -j$, which is summed separately) so that the series converges.
			
			\smallskip
			\item Let $P\in\Com$ be a zero of $H_{t_0}^L$ of order
$q\geq2$. Then there is a neighbourhood $U$ of $P$ such that,
for $t$ sufficiently close to $t_0$, $H_t^L$ has exactly $q$
zeros in $U$, counted with multiplicity. For $t\ne t_0$,
these zeros are simple and can be labelled so that
\begin{equation}\label{bifurcacion}
z_j(t)=P+\sqrt{2(t-t_0)}\,\lambda_j+O(|t-t_0|),
\qquad j=1,\ldots,q,
\end{equation}
as $t\to t_0$, where $\lambda_1<\cdots<\lambda_q$ are
the roots of the $q$-th Hermite polynomial
\begin{equation}\label{Hermite}
\mathrm{He}_q(z)
:=(-1)^q e^{z^2/2}\frac{d^q}{dz^q}e^{-z^2/2}
=\sum_{0\leq l\leq q/2}
\frac{q!}{l!\,(q-2l)!}\,
\frac{(-1)^l z^{q-2l}}{2^l}.
\end{equation}
The implied constant in \eqref{bifurcacion} may depend on $L$, $P$, $t_0$ and $q$. 
		\end{enumerate}
	\end{prop}

\begin{proof}
We use the argument of \cite[Proposition~3.1]{Poly}, including
the possible zero at the origin. The decay of the density gives
$\log\max_{|z|\le R}|H_t^L(z)|=O_{L,t}(R\log(R+2))$, so
$H_t^L$ has order at most one. Since it is even, its Hadamard
product can be written as
\[
H_t^L(z)=c_tz^{r_0(t)}\prod_k{}'
\left(1-\frac z{z_k(t)}\right).
\]
At a simple zero, the implicit function theorem and
$\partial_tH_t^L=-\partial_z^2H_t^L$ give
\[
z_j'(t)=\frac{(H_t^L)''(z_j(t))}{(H_t^L)'(z_j(t))}.
\]
Taking the logarithmic
derivative of the product and removing the pole at $z_j$ yields
\[
z_j'(t)=2\left(\frac{r_0(t)}{z_j(t)}
+\sum_{k\ne j}{}'\frac1{z_j(t)-z_k(t)}\right).
\]
The paired series converges locally away from its poles, which
proves~(i).

For~(ii), write $H_{t_0}^L(z)=a_q(z-P)^q+O((z-P)^{q+1})$,
with $a_q\ne0$, and put $r=\sqrt{2(t-t_0)}$.
Taylor expansion and the heat equation give, uniformly on
bounded sets of $w$,
\[
H_t^L(P+rw)=a_qr^q\bigl(\mathrm{He}_q(w)+O(|r|)\bigr).
\]
Rouch\'e's theorem on disjoint small discs about the simple
roots of $\mathrm{He}_q$ \eqref{Hermite} gives one simple zero in each disc.
These zeros satisfy $w_j=\lambda_j+O(|r|)$ and hence
\eqref{bifurcacion}. A second application on a fixed disc
about $P$ counts all $q$ nearby zeros and proves the assertion.
\end{proof}

	We now establish the upper-bound criterion for $\Lambda_L$ in terms of zero-free regions of $H_t^L$.
	
	\begin{remark}[Initial strip hypotheses]
		Let $0<D\leq1$ be such that
		\begin{equation}\label{eq:initial_zero_strip_D}
			H_0^L(x+iy)\ne0
			\qquad\text{for every }x\in\mathbb R\text{ and }|y|>D.
		\end{equation}
The classical critical strip gives $D=1$ unconditionally.
Every application below with $D<1$ assumes the corresponding
global initial-strip hypothesis
\eqref{eq:initial_zero_strip_D}.

Under the change of variables $s=(1+iz)/2$, with $z=x+iy$,
we have $\operatorname{Re}s=(1-y)/2$.
Consequently, the zero-free half-plane
$\operatorname{Re}s>7/8$ asserted in
\cite[Theorem~1.1]{OPENAI1}, together with the functional
equation, would imply the initial strip with $D=3/4$.
The conclusions reported for this value of $D$ are
conditional on that global strip hypothesis.

Similarly, the zero-free half-plane
$\operatorname{Re}s>11/12$ asserted in
\cite[Theorem~1.1]{OPENAI2} would give $D=5/6$.
This latter value is not used in the numerical tables.
		
	\end{remark}
	\begin{lemma}[Zero-free region criterion]\label{lemmabox}Let $L$ be associated with a primitive real-valued Dirichlet
character of conductor $3\leq m\leq400000$, and let $0<D\leq1$
satisfy the initial strip bound~\eqref{eq:initial_zero_strip_D}.
Suppose that $t_0\geq0$, $X>0$, and $0<y_0\leq1$ satisfy
$y_0^2+2t_0\leq D^2$.
Set $\Delta_D:=\sqrt{D^2-y_0^2-2t_0}$.
Assume the following:
		\begin{enumerate}
			\item $H_0^L(x+iy)\ne0$ whenever
			\[
			0\le x\le X,\qquad
			\sqrt{y_0^2+2t_0}\le y\le D.
			\]
			
			\item $H_{t_0}^L(x+iy)\ne0$ whenever
			\[
			x\ge X+\Delta_D,\qquad
			y_0\le y\le\sqrt{D^2-2t_0}.
			\]
			
			\item $H_t^L(x+iy)\ne0$ whenever
			\[
			X\le x\le X+\Delta_D,\qquad
			\sqrt{y_0^2+2(t_0-t)}
			\le y\le\sqrt{D^2-2t},\qquad
			0\le t\le t_0.
			\]
		\end{enumerate}
		Then
		\[
		H_{t_0}^L(x+iy)\ne0
		\qquad\text{for every }x\in\mathbb R\text{ and }y\ge y_0.
		\]
	\end{lemma}
	
	\begin{proof}
    
		If $t_0=0$, the hypotheses cover the region $x\ge0$, $y_0\le y\le D$, and the conclusion follows directly from them, the symmetries, and the initial strip bound. We may therefore assume $t_0>0$.
		Write $Y(t):=\sqrt{y_0^2+2(t_0-t)}$, $R_D(t):=\sqrt{D^2-2t}.$
		We first note that for every $t\in(0,t_0]$, no zero of $H_t^L$
		can lie on the upper boundary $y=R_D(t)$. Indeed, Theorem~\ref{dBru}
		gives $|\ima z|\le R_D(t)$. If equality held, the zero would have
		maximal imaginary part. If it were multiple, Proposition~\ref{Dynz}(ii)
		would produce, at a slightly earlier time, a zero whose imaginary
		part exceeds $R_D(t)$: one of the Hermite roots in
		\eqref{bifurcacion} is positive, and the resulting displacement is
		of order $\sqrt{t-t_1}$, whereas the change in $R_D(t)$ is of order
		$t-t_1$. This contradicts the bound from Theorem~\ref{dBru}.
		
		If the zero on the upper boundary were simple, write it as
$z=x+iR_D(t)$. By Lemma~\ref{lem:positive_imaginary_axis},
$x\ne0$. In the evolution equation of
Proposition~\ref{Dynz}(i), the conjugate zero
$x-iR_D(t)$ contributes $-1/R_D(t)$ to the derivative
of the imaginary part, while the distinct zero
$-x-iR_D(t)$ contributes
$-R_D(t)/(x^2+R_D(t)^2)<0$.
Every remaining zero contributes nonpositively, since
its imaginary part is at most $R_D(t)$. Hence
\[
\frac{d}{dt}\operatorname{Im}z(t)
\leq-\frac1{R_D(t)}
-\frac{R_D(t)}{x^2+R_D(t)^2}
<-\frac1{R_D(t)}.
\]
		Since
		\[
		R_D'(t)=-\frac1{R_D(t)},
		\]
		the zero would lie above the boundary at a slightly earlier time,
		again contradicting Theorem~\ref{dBru}. Thus the upper boundary
		contains no zeros for $t>0$.
		
		By hypotheses (2) and (3), the symmetries in
		Lemma~\ref{prop:symmetric}, and the upper-bound just established,
		it remains to prove that
		\[
		H_t^L(x+iy)\ne0
		\quad\text{for } |x|\le X,\quad y\ge Y(t),\quad 0\le t\le t_0.
		\]
		Hypothesis (1) gives this at $t=0$. If it fails at some later time,
		continuity gives a first time $t_1\in(0,t_0]$ at which a zero lies
		on the boundary of this region. Hypothesis (3) and symmetry exclude
		the vertical sides, and the preceding argument excludes the upper
		boundary $y=R_D(t_1)$. Hence
		\[
		H_{t_1}^L(x+iY(t_1))=0
		\quad\text{for some }-X<x<X.
		\]
		
		If this is a multiple zero, Proposition~\ref{Dynz}(ii) gives, at
		times $t<t_1$ sufficiently close to $t_1$, a zero whose imaginary
		part exceeds $Y(t)$. In fact, the upward displacement is of order
		$\sqrt{t_1-t}$, while $Y(t)-Y(t_1)$ is of order $t_1-t$.
		This contradicts the choice of $t_1$.
		
		We may therefore suppose that the zero is simple. Write
		$z_j(t_1)=x+iY(t_1)$. By Proposition~\ref{Dynz}(i),
		\[
		\frac{d}{dt}\ima z_j(t_1)
		=-2\left(\frac{r_0(t_1)Y(t_1)}{x^2+Y(t_1)^2}+\sum_{k\ne j}'\frac{Y(t_1)-y_k}
		{(x-x_k)^2+(Y(t_1)-y_k)^2}\right),
		\]
		where $z_k=x_k+iy_k$ are the nonzero zeros of $H_{t_1}^L$.
		
		We compare this derivative with
		\[
		Y'(t_1)=-\frac1{Y(t_1)}.
		\]
By Lemma~\ref{lem:positive_imaginary_axis}, $x\ne0$.
The conjugate zero $x-iY(t_1)$ contributes $-1/Y(t_1)$,
while the distinct zero $-x-iY(t_1)$ contributes
$-Y(t_1)/(x^2+Y(t_1)^2)<0$.
Every remaining zero with imaginary part in
$[-Y(t_1),Y(t_1)]$ contributes nonpositively.
The origin term vanishes by the same lemma.
		It remains to consider zeros with $y_k>Y(t_1)$, grouping each with
		its conjugate. Hypothesis (3) and the symmetries imply
		$|x_k|\ge X+\Delta_D$. Since $|x|<X$, we have
		\[
		|x-x_k|>\Delta_D.
		\]
		Also, the global strip bound gives
		\[
		y_k^2\le D^2-2t_1
		=\Delta_D^2+Y(t_1)^2
		<(x-x_k)^2+Y(t_1)^2.
		\]
		Consequently,
		\[
		\frac{Y(t_1)-y_k}{(x-x_k)^2+(Y(t_1)-y_k)^2}
		+
		\frac{Y(t_1)+y_k}{(x-x_k)^2+(Y(t_1)+y_k)^2}
		\ge0.
		\]
		Thus
		\[
		\frac{d}{dt}\ima z_j(t_1)<-\frac1{Y(t_1)}=Y'(t_1).
		\]
		For times just below $t_1$, this places the zero above $Y(t)$,
		contradicting the definition of $t_1$. This proves the lemma.
	\end{proof}
	
	\begin{cor}\label{cor:upperbound}
		Under the hypotheses of Lemma~\ref{lemmabox},
		\[
		\Lambda_L\le t_0+\frac12y_0^2\leq\frac{D^2}{2}.
		\]
	\end{cor}
	
	\begin{proof}
		Lemma~\ref{lemmabox} gives the zero-free region required in
		Theorem~\ref{dBru} and the stated bound follows.
	\end{proof}
	\begin{remark}
		For a primitive character with $3\le m\le400000$, let $h_m$ be the height in Platt's verification of the {\rm GRH} \cite{platt}. With $D=1$, under the change of variables
		\[
		s=\frac{1-y+ix}{2},
		\qquad |\operatorname{Im}s|=\frac{x}{2},
		\]
		condition (1) follows from that verification whenever
		$X\le2h_m$.
	\end{remark}
	
	\begin{remark}
		A real zero $\beta$ of $L$ corresponds to a purely imaginary
		zero $z=i(1-2\beta)$ of $H_0^L$. For $D<1$, the initial strip
		hypothesis excludes such zeros near $s=1$. For $D=1$, condition
		(1) excludes a real zero only if its corresponding zero lies
		in the region tested there. Therefore, it does not by itself rule out every
		possible Siegel zero.
	\end{remark}
	
	\begin{remark}
		The initial strip hypothesis alone gives
		\[
		\Lambda_L\le\frac{D^2}{2}
		\]
		by Theorem~\ref{dBru}. The assumption $t_0+\tfrac12y_0^2\leq\tfrac12D^2$ ensures
		that the upper bound in Corollary~\ref{cor:upperbound} does not
		exceed the initial-strip bound. Then, the improvement is strict when
		this inequality is strict.

This implication cannot simply be reversed to recover an initial strip
from an upper bound on $\Lambda_L$. For example, the even real
polynomials
\[
 P_t(z)=z^4-12tz^2+12t^2+4
\]
satisfy $\partial_tP_t=-\partial_z^2P_t$. The zeros of
$P_0(z)=z^4+4$ are $\pm1\pm i$, while the equation $P_t(z)=0$
is equivalent to $z^2=6t\pm\sqrt{24t^2-4}$. Thus all zeros
are real precisely when $t\ge T:=1/\sqrt6$. Since
$\sqrt{2T}<1$, knowing this threshold would not place the
initial zeros in $|\operatorname{Im}z|\le\sqrt{2T}$.
This example concerns the heat equation. A converse for
Dirichlet $L$-functions would require additional information
specific to that family.
\end{remark}

	\subsection{Multiple zeros}
	The Hermite expansion also describes the splitting at a multiple zero.
	
	\begin{lemma}
		There is no collision involving only real zeros. More precisely, there do not exist $t_0 > 0$ and a repeated real zero $z^*(t_0)$ of $H_{t_0}^L$ of order $q \ge 2$ satisfying \eqref{bifurcacion} such that $\ima z_k^*(t) = 0$ for $t$ sufficiently close to $t_0$ with $t < t_0$, and for every $k \in \{1, \ldots, q\}$.
	\end{lemma}

\begin{proof}
By Proposition~\ref{Dynz}(ii), a positive root $\lambda_j$
of $\mathrm{He}_q$ gives, for $t<t_0$,
\[
\operatorname{Im}z_j(t)
=\sqrt{2}\lambda_j\sqrt{t_0-t}+O(t_0-t)>0
\]
when $t$ is sufficiently close to $t_0$. Thus the branches
before the collision cannot all be real.
\end{proof}

	Here and in the sequel, $D(P, r_0)$ denotes the open disc of radius $r_0 > 0$ centered at $P \in \Com$, and $|A|$ denotes the cardinality of a set $A$.
	
	\begin{lemma}
		Let $T > 0$ and let $P$ be a zero of $H_T^L$ of multiplicity $q \ge 1$. Then there exist $\delta_0, r_0 > 0$ such that
		\[
		\left| \left\{ z \in D(P, r_0) : H_t^L(z) = 0 \right\} \right| = q\qquad(0<|t-T|<\delta_0),
		\]
		and these $q$ zeros are all simple.
	\end{lemma}

\begin{proof}
If $q=1$, the assertion follows from Proposition~\ref{Dynz}(i).
For $q\ge2$, part~(ii) gives exactly $q$ nearby simple zeros
at every sufficiently close time $t\ne T$. Choosing a fixed disc
containing only $P$ at time $T$ and applying Rouch\'e's theorem
keeps this count equal to $q$ and proves the assertion.
\end{proof}

	The double symmetry also determines the local behaviour
of zeros near the imaginary axis.

\begin{lemma}[Bifurcation on the imaginary axis]
Let $L$ be associated with a primitive real-valued
Dirichlet character of conductor $m\geq3$.
Let $T\in\R$, and suppose that $P\in i\R$ is a zero
of $H_T^L$ of multiplicity $q\geq2$.
Then there exist $r_0,\delta_0>0$ such that, for every
$T-\delta_0<t<T$, the function $H_t^L$ has exactly $q$
zeros in $D(P,r_0)$, and all of them are simple and
purely imaginary.
\end{lemma}

\begin{proof}
Choose $r_0>0$ so that $P$ is the only zero of $H_T^L$
in $\overline{D(P,r_0)}$.
By Rouch\'e's theorem, $H_t^L$ has exactly $q$ zeros
in this disc, counted with multiplicity, for $t$
sufficiently close to $T$.
Proposition~\ref{Dynz}(ii) shows that, for $t<T$,
these zeros are simple and satisfy
\[
z_j(t)=P+i\sqrt{2(T-t)}\,\lambda_j+O(T-t),
\qquad j=1,\ldots,q,
\]
where $\lambda_1<\cdots<\lambda_q$ are the real roots
of $\mathrm{He}_q$.

Choose $\rho>0$ such that
$2\rho<\min_{j\ne k}|\lambda_j-\lambda_k|$.
For $t<T$ sufficiently close to $T$, the discs
\[
B_j(t):=
D\!\left(P+i\sqrt{2(T-t)}\,\lambda_j,\,
         \rho\sqrt{2(T-t)}\right),
\qquad j=1,\ldots,q,
\]
are disjoint, lie inside $D(P,r_0)$, and each contains
exactly one of these zeros. Indeed, the error $O(T-t)$
is smaller than their radii for $t$ sufficiently close
to $T$.

The symmetries of $H_t^L$ give
$H_t^L(-\overline z)=\overline{H_t^L(z)}$.
Each disc $B_j(t)$ is invariant under
$z\mapsto-\overline z$, because its centre is purely
imaginary. Hence $-\overline{z_j(t)}$ is also a zero
in $B_j(t)$. Uniqueness in that disc implies
$z_j(t)=-\overline{z_j(t)}$, so $z_j(t)\in i\R$.
Reducing $\delta_0$ if necessary proves the assertion.
\end{proof}
	\begin{figure}
	\centering
	\begin{tikzpicture}[>=latex,scale=1.05,every node/.style={font=\small}]
		\draw[->,line width=0.7pt] (-2.4,0) -- (2.5,0) node[right] {$\operatorname{Re}z$};
		\draw[->,line width=0.7pt] (0,-0.45) -- (0,3.4) node[above] {$\operatorname{Im}z$};
		\coordinate (P) at (0,1.75);
		\draw[red!70!black,dashed,line width=0.8pt] (P) circle (0.62);
		\fill[red!70!black] (P) circle (2.5pt);
		\node[right=2pt] at (P) {$P$};
		\node[red!70!black,above left] at (-0.43,2.418) {$D(P,r_0)$};
		\fill[blue!65!black] (0,1.32) circle (2.2pt);
		\fill[blue!65!black] (0,2.18) circle (2.2pt);
		\node[left=3pt] at (-0.350,1.32) {$z_1(t)$};
		\node[left=3pt] at (-0.35,2.18) {$z_2(t)$};
		\draw[->,blue!65!black,line width=0.8pt] (0.22,1.38) -- (0.22,1.62);
		\draw[->,blue!65!black,line width=0.8pt] (0.22,2.12) -- (0.22,1.88);
		\node[right=5pt] at (0.45,2.85) {$z_1(t),z_2(t)\in i\mathbb R$ for $t<T$};
	\end{tikzpicture}
	\caption{The two branches from a double zero on $i\mathbb R$ remain on the imaginary axis for nearby times $t<T$.}
\end{figure}
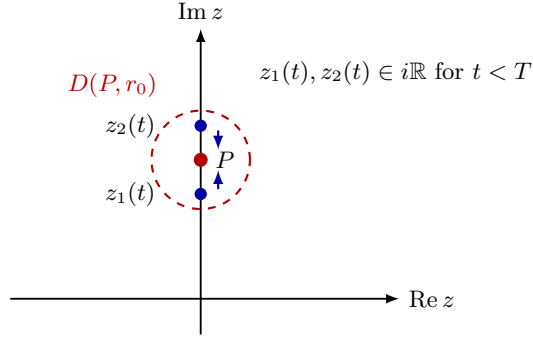
	\subsection{Zeros on the imaginary axis}
	
	We now consider real zeros of $L$ and their relation to the zeros
	of $H_t^L$ on the imaginary axis. For a real-valued primitive
	character, the functional equation implies that a real zero
	$\beta\in(0,1)$ is paired with $1-\beta$. Under the change of
	variables $s=(1+iz)/2$, this pair corresponds to the zeros
	$\pm i(2\beta-1)$ of $H_0^L$. The classical zero-free-region
	theorems ensure simplicity for a zero in the exceptional region
	near $s=1$, but they do not assert that every real zero lies in
	that region. We therefore state the local motion for a simple
	real zero without making a global claim about its trajectory.
	
	\begin{lemma}[Local motion of a simple real zero]
		Let $L$ be associated with a real-valued primitive Dirichlet
		character, and let $\beta\in(0,1)\setminus\{\tfrac12\}$ be a
		simple real zero of $L$. Set
		\[
		y_\beta:=|1-2\beta|.
		\]
		Then $H_0^L$ has simple zeros at $\pm i y_\beta$. There exists
		$\varepsilon>0$ and smooth zero branches $z_\pm(t)$, defined for
		$0\leq t<\varepsilon$, such that
		\[
		z_\pm(0)=\pm i y_\beta,
		\qquad
		H_t^L(z_\pm(t))=0,
		\qquad
		z_\pm(t)\in i\mathbb R.
		\]
		These zeros remain simple for $0\leq t<\varepsilon$.
	\end{lemma}
	
	\begin{proof}
		The completed factor multiplying $L(s)$ is holomorphic and
		nonzero for $0<\operatorname{Re}s<1$. Thus the simple zero
		$\beta$ gives a simple zero of $H_0^L$ at
		$z=i(1-2\beta)$ and the functional equation gives the other zero
		at $-z$.
		
		By Proposition~\ref{Dynz}(i), each simple zero continues locally
		as a unique smooth zero branch. The symmetries of $H_t^L$ imply
		that $H_t^L(-\overline z)=\overline{H_t^L(z)}$. Since
		$-\overline{(\pm i y_\beta)}=\pm i y_\beta$, uniqueness of the
		local branch gives $z_\pm(t)=-\overline{z_\pm(t)}$, so
		$z_\pm(t)\in i\mathbb R$. Simplicity persists after reducing
		$\varepsilon$ if necessary.
	\end{proof}
	
	\begin{remark}
		Suppose that, for some $D\in\{3/4,1\}$,
		\[
		H_0^L(x+iy)\ne0
		\qquad
		(x\in\mathbb R,\ |y|>D).
		\]
		If $L(\beta)=0$ for a real $\beta\in(0,1)$, then
		$z=i(1-2\beta)$ is a zero of $H_0^L$, and hence
		\[
		|1-2\beta|\leq D,
		\qquad
		\frac{1-D}{2}\leq\beta\leq\frac{1+D}{2}.
		\]
		Thus $D=3/4$ excludes real zeros with $\beta>7/8$.
		These bounds do not exclude real zeros inside the resulting
		intervals, nor do they imply that there is at most one such zero.
	\end{remark}
	
	\section{Riemann--Siegel approximation}\label{sec:4}
	
	We derive a finite-sum approximation for $H_t^L(x+iy)$ by deforming the
	contour through its saddle point. The decomposition separates the finite
	Dirichlet sum, the leading coefficient $C_0^{(m)}(p)$, the higher Taylor
	coefficients, and the integral remainder. The estimates below track their
	dependence on $m$ in the range used for the upper bound on $\Lambda_L$.
	
	\subsection{Preliminaries}
	
	The approximation below is proved for $|y|\le1$. In applications with an initial strip parameter $D\le1$, the same estimates apply on the smaller range $|y|\le D$. Only the endpoint constants in the remainder estimate need to be adjusted, as recorded after Proposition~\ref{prop:magnitudes}.

	Throughout this section, we assume that $\chi_m$ is a real-valued primitive
	Dirichlet character. In particular, $\Phi_L$ is even on $\R$ by Lemma
	\ref{prop:symmetric}. Since $e^{tu^2}$ is also even, the representation
	\eqref{H_tz} can be written as the cosine transform
	\[
	H_t^L(z)
	=
	2\int_0^\infty e^{tu^2}\Phi_L(u)\cos(zu)\,du.
	\]
	Following \cite[Section~4]{Poly}, we further have the heat-kernel
	representation
	\begin{equation}\label{expH0}
		H_t^L(z)
		=
		\int_{-\infty}^{\infty}
		H_0^L(z-2i\sqrt{t}\,v)\,
		\frac{e^{-v^2}}{\sqrt{\pi}}\,dv
		=
		\int_{-\infty}^{\infty}
		\xi_L\!\left(\frac{1+iz}{2}+\sqrt{t}\,v\right)
		\frac{e^{-v^2}}{\sqrt{\pi}}\,dv,
	\end{equation}
	which follows from the elementary identity
	\[
	e^{tu^2}\cos(zu)
	=
	\int_{\R}
	\cos\!\bigl((z-2i\sqrt{t}\,v)u\bigr)
	\frac{e^{-v^2}}{\sqrt{\pi}}\,dv,
	\qquad z,u\in\Com,
	\]
	by multiplying by $\Phi_L(u)$ and applying Fubini's theorem.
	
	From \cite{Siegel}, the generalized Riemann--Siegel formula reads
	\begin{equation}\label{xi_L}
		\xi_L(s)
		=
		\left(\frac{m}{\pi}\right)^{(s+a)/2}
		\Gamma\!\left(\frac{s+a}{2}\right)L(s)
		=
		\mu^L(s)+\overline{\mu^L(1-\overline{s})},
	\end{equation}
	where
	\[
	\mu^L(s)
	=
	\left(\frac{m}{\pi}\right)^{(s+a)/2}
	\Gamma\!\left(\frac{s+a}{2}\right)\lambda^L(s)
	\text{ and }
	\lambda^L(s)
	=
	\frac{1}{2\pi i}
	\int_{0\swarrow1}
	x^{-s}e^{\frac{i\pi}{m}x^2}W^L(x)\,dx.
	\]
	The notation $0\swarrow1$ stands for any continuous curve in $\Com$
	forming an angle of $\pi/4$ with the real line and crossing it at a point
	strictly between $0$ and $1$. The symbol $N\swarrow N+1$ is interpreted
	analogously. For any nonnegative integer $N$, complex integration gives
	\[
	\lambda^L(s)
	=
	\sum_{n=1}^{N}\frac{\chi_m(n)}{n^s}
	+
	\frac{1}{2\pi i}
	\int_{N\swarrow N+1}
	e^{\frac{i\pi}{m}x^2}x^{-s}W^L(x)\,dx.
	\]
	Following the notation of \cite{Poly}, set
	\[
	R_{0,0}^L(s):=\mu^L(s),
	\qquad
	R_{0,0}^L(s)
	=
	\sum_{n=1}^{N}r_{0,n}^L(s)+R_{0,N}^L(s),
	\]
	where
	\[
	r_{0,n}^L(s)
	=
	\left(\frac{m}{\pi}\right)^{(s+a)/2}
	\Gamma\!\left(\frac{s+a}{2}\right)
	\chi_m(n)n^{-s},
	\]
	and
	\[
	R_{0,N}^L(s)
	=
	\left(\frac{m}{\pi}\right)^{(s+a)/2}
	\Gamma\!\left(\frac{s+a}{2}\right)
	\frac{1}{2\pi i}
	\int_{N\swarrow N+1}
	x^{-s}e^{\frac{i\pi}{m}x^2}W^L(x)\,dx.
	\]
	Combining the functional equation \eqref{xi_L} with the heat-kernel
	representation \eqref{expH0}, we obtain
	\[
	\begin{aligned}
		H_t^L(z)
		={}&
		\sum_{n=1}^{N}
		r_{t,n}^L\!\left(\frac{1+iz}{2}\right)
		+
		\sum_{n=1}^{N}
		\overline{
			r_{t,n}^L\!\left(\frac{1+i\overline z}{2}\right)
		}
		\\
		&+
		R_{t,N}^L\!\left(\frac{1+iz}{2}\right)
		+
		\overline{
			R_{t,N}^L\!\left(\frac{1+i\overline z}{2}\right)
		},
	\end{aligned}
	\]
	for every $t>0$, every $z\in\Com$ that is not purely imaginary, and every
	nonnegative integer $N$. Moreover, by \eqref{expH0},
	\begin{equation}\label{intrR}
		\begin{aligned}
			r_{t,n}^L(s)
			&=
			\int_{-\infty}^{\infty}
			r_{0,n}^L(s+\sqrt{t}\,v)
			\frac{e^{-v^2}}{\sqrt{\pi}}\,dv,
			\\
			R_{t,N}^L(s)
			&=
			\int_{-\infty}^{\infty}
			R_{0,N}^L(s+\sqrt{t}\,v)
			\frac{e^{-v^2}}{\sqrt{\pi}}\,dv.
		\end{aligned}
	\end{equation}
	
	When no confusion arises we drop the superscript $L$ and write
	$r_{t,n}$ and $R_{t,N}$.
	
	\begin{lemma}\label{lem:GaussianBound}
		Fix $m,N,a$ and a compact interval $I\subset\mathbb R\setminus\{0\}$.
		There is a constant $C>0$ such that, for every $\sigma\in\mathbb R$
		and $T\in I$,
		\begin{equation}\label{r0nexpb}
			|r_{0,n}^L(\sigma+iT)|
			\le \exp\!\left(C(1+|\sigma|)\log(2+|\sigma|)\right),
			\qquad 1\le n\le N,
		\end{equation}
		and
		\begin{equation}\label{R0nexpb}
			|R_{0,N}^L(\sigma+iT)|
			\le \exp\!\left(C(1+|\sigma|)\log(2+|\sigma|)\right).
		\end{equation}
	\end{lemma}
	
	\begin{proof}
		Stirling's formula, together with the reflection formula on the left
		half-plane, gives the asserted growth for the Gamma factor, uniformly
		for $T\in I$. The factors $(m/\pi)^{(s+a)/2}$ and $n^{-s}$ contribute
		at most an exponential in $|\sigma|$, proving \eqref{r0nexpb}.
		
		For the contour integral defining $R_{0,N}^L$, fix
		$c\in(N,N+1)$ and write $x=c+e^{i\pi/4}u$.
		The line stays a positive distance from zero and avoids the poles
		of $W^L$. The kernel $W^L$ is bounded on it, and
		\[
		|e^{i\pi x^2/m}|\le C e^{-c_0u^2},
		\qquad
		|x^{-\sigma-iT}|\le C|x|^{-\sigma}
		\]
		for some $c_0>0$, after increasing $C$.
		For $\sigma\ge0$, the positive lower bound for $|x|$ gives a bound
		$e^{C(1+\sigma)}$ for the integral.
		For $\sigma<0$, we use $|x|\le C(1+|u|)$ and evaluate the Gaussian
		moment of order $|\sigma|$. We can see that it is at most
		$\exp(C(1+|\sigma|)\log(2+|\sigma|))$.
		Then, multiplication by the Gamma factor proves \eqref{R0nexpb}.
		These estimates also show that the Gaussian in \eqref{intrR}
		dominates the growth on horizontal strips in the $s$-plane.
	\end{proof}
	
	By \eqref{r0nexpb} and \eqref{R0nexpb}, we may shift contours in
	\eqref{intrR} (replacing $v$ by
	$v+\tfrac{\sqrt{t}}2\alpha_L^n$) to obtain
	\begin{equation}\label{rtnchange}
		r_{t,n}^L(s)
		=
		\exp\!\left(-\frac{t}{4}(\alpha_L^n)^2\right)
		\int_{-\infty}^{\infty}
		\exp\!\left(-\sqrt{t}\,v\,\alpha_L^n\right)
		r_{0,n}^L\!\left(
		s+\sqrt{t}\,v+\frac{t}{2}\alpha_L^n
		\right)
		\frac{e^{-v^2}}{\sqrt{\pi}}\,dv,
	\end{equation}
	for any complex number $\alpha_L^n$ such that $\ima(s)$ and
	$\ima(s+\tfrac{t}{2}\alpha_L^n)$ have the same sign. Similarly,
	\begin{equation}\label{Rtnchange}
		R_{t,N}^L(s)
		=
		\exp\!\left(-\frac{t}{4}\beta_N^2\right)
		\int_{-\infty}^{\infty}
		\exp\!\left(-\sqrt{t}\,v\,\beta_N\right)
		R_{0,N}^L\!\left(
		s+\sqrt{t}\,v+\frac{t}{2}\beta_N
		\right)
		\frac{e^{-v^2}}{\sqrt{\pi}}\,dv,
	\end{equation}
	for any complex number $\beta_N$ such that $\ima(s)$ and
	$\ima(s+\tfrac{t}{2}\beta_N)$ have the same sign.
	
	Following the spirit of the saddle-point method, the parameters
	$\alpha_L^n$ and $\beta_N$ will be chosen later so that the integrands in
	\eqref{rtnchange} and \eqref{Rtnchange} are nearly stationary at $v=0$,
	which produces sharper estimates for these terms.
	
	\subsection{Elementary estimates}
	
	Before turning to the proof of Theorem~\ref{EffectiveRSa}, we collect a
	short list of elementary inequalities, mostly extracted from \cite{Poly},
	that allow us to handle the error terms uniformly.
	
	\begin{lemma}\label{elementaryball}
		Let $x>0$ be a real number.
		\begin{enumerate}
			\item If $a>0$ and $b\ge0$ satisfy $x>b/a$, then
			\[
			O_\le\!\left(\frac{a}{x}\right)
			+
			O_\le\!\left(\frac{b}{x^2}\right)
			=
			O_\le\!\left(\frac{a}{x-b/a}\right).
			\]
			More generally, if $a>0$ and $b,c\ge0$ satisfy
			$x>\max\{b/a,\sqrt{c/a}\}$, then
			\[
			O_\le\!\left(\frac{a}{x}\right)
			+
			O_\le\!\left(
			\frac{a\max\{b/a,\sqrt{c/a}\}}{x^2}
			\right)
			+
			O_\le\!\left(\frac{c}{x^3}\right)
			=
			O_\le\!\left(
			\frac{a}{x-\max\{b/a,\sqrt{c/a}\}}
			\right).
			\]
			\item If $x>1$, then
			\[
			\log\!\left(1+O_\le\!\left(\frac1x\right)\right)
			=
			O_\le\!\left(\frac1{x-1}\right),
			\]
			or equivalently
			\[
			1+O_\le\!\left(\frac1x\right)
			=
			\exp\!\left(O_\le\!\left(\frac1{x-1}\right)\right).
			\]
			\item If $x>\tfrac12$,
			\[
			\exp\!\left(O_\le\!\left(\frac1x\right)\right)
			=
			1+O_\le\!\left(\frac1{x-0.5}\right).
			\]
			\item For any $x\ge0$,
			\[
			\exp(O_\le(x))
			=
			1+O_\le(e^x-1).
			\]
			\item If $z\in\mathbb C$ satisfies $|\operatorname{Im}z|\ge1$
			or $\operatorname{Re}z\ge1$, then
			\[
			\Gamma(z)=\sqrt{2\pi}
			\exp\!\left(
			(z-\tfrac12)\log z-z+
			O_\le\!\left(
			\frac1{12(\max\{\operatorname{Re}z,|\operatorname{Im}z|\}-1/2)}
			\right)\right).
			\]
			\item If $a,b>0$, $y\ge0$, $x\ge x_0\ge e^{a/b}$ and $x_0>c\ge0$, then
			\[
			\frac{\log^a|x+iy|}{(x-c)^b}
			\le
			\frac{\log^a|x_0+iy|}{(x_0-c)^b}.
			\]
		\end{enumerate}
	\end{lemma}
	
	The estimates (i)--(iv) and (vi) follow as in
	\cite[Lemma~5.1]{Poly}. We record a uniform version of (v), since
	the Gaussian integrals below include arbitrarily negative real parts.
	In the right half-plane, the estimate used in
	\cite[Lemma~5.1(v)]{Poly} bounds the logarithmic Stirling remainder
	by $1/(12(|z|-0.33))$, which implies the stated bound there.
	Define the logarithmic Stirling remainder by
	\[
	r_{\mathrm{St}}(z):=\log\Gamma(z)
	-(z-\tfrac12)\log z+z-\tfrac12\log(2\pi),
	\]
	where $\log z$ is principal and $\log\Gamma$ is the analytic branch on
	$\mathbb C\setminus(-\infty,0]$ that is real on the positive real axis.
	For $\operatorname{Re}z<0$, conjugation reduces to
	$v=\operatorname{Im}z\ge1$. The reflection formula gives
	\[
	r_{\mathrm{St}}(z)=-r_{\mathrm{St}}(-z)-\log(1-e^{2\pi iz}).
	\]
	Here the last logarithm is given by its convergent power series. Consequently,
	\[
	|r_{\mathrm{St}}(z)|\le
	\frac1{12(|z|-0.33)}+\frac1{e^{2\pi v}-1}
	\le\frac1{12(v-1/2)}.
	\]
	For the last inequality, use
	\[
	\frac1{e^{2\pi v}-1}<\frac1{114v^2}
	<\frac{0.17}{12v^2}
	\le\frac1{12(v-1/2)}-\frac1{12(v-0.33)}.
	\]
	Indeed, $v^2/(e^{6v}-1)$ decreases for $v\ge1$ and $e^6>115$.
	In particular, for every real $\sigma$ and $T\ge2$,
	\begin{equation}\label{eq:Stirling_uniform_horizontal}
		\left(\frac m\pi\right)^{(\sigma+iT+a)/2}
		\Gamma\!\left(\frac{\sigma+iT+a}{2}\right)
		=M_0^L(\sigma+iT)
		\exp\!\left(O_\le\!\left(\frac1{6(T-1)}\right)\right).
	\end{equation}

	\begin{theorem}[Effective approximation]
		\label{EffectiveRSa}
		Let $0\le t\le\tfrac12$ and let $x,y\in\R$ satisfy
		$x\ge480$ and $|y|\le1$. Then
		\[
		\frac{H_t^L(x+iy)}{B_t^L(x+iy)}
		=
		f_t^L(x+iy)
		+
		O_\le(e_A+e_B+e_{C,0}),
		\]
		where $f_t^L(x+iy)$ is the explicit main term defined in
		\eqref{ftL_def}. The finite Dirichlet sums are truncated at $M=\lfloor\delta\rfloor,$
		where, in the application to $H_t^L(x+iy)$,
		\[
		T'=\frac{x}{2}+\frac{\pi t}{8},
		\qquad
		\delta=\sqrt{\frac{mT'}{2\pi}}.
		\]
	\end{theorem}

The main term is defined in \eqref{ftL_def}, and the error
terms are specified in Propositions~\ref{prop:eA},
\ref{prop:eB} and \ref{prop:eC}. We first estimate the
Dirichlet terms and the integral remainder, then combine them
in Section~\ref{subsec:setup_sym}.

		\begin{figure}
			\centering
			\includegraphics[width=0.8\linewidth]{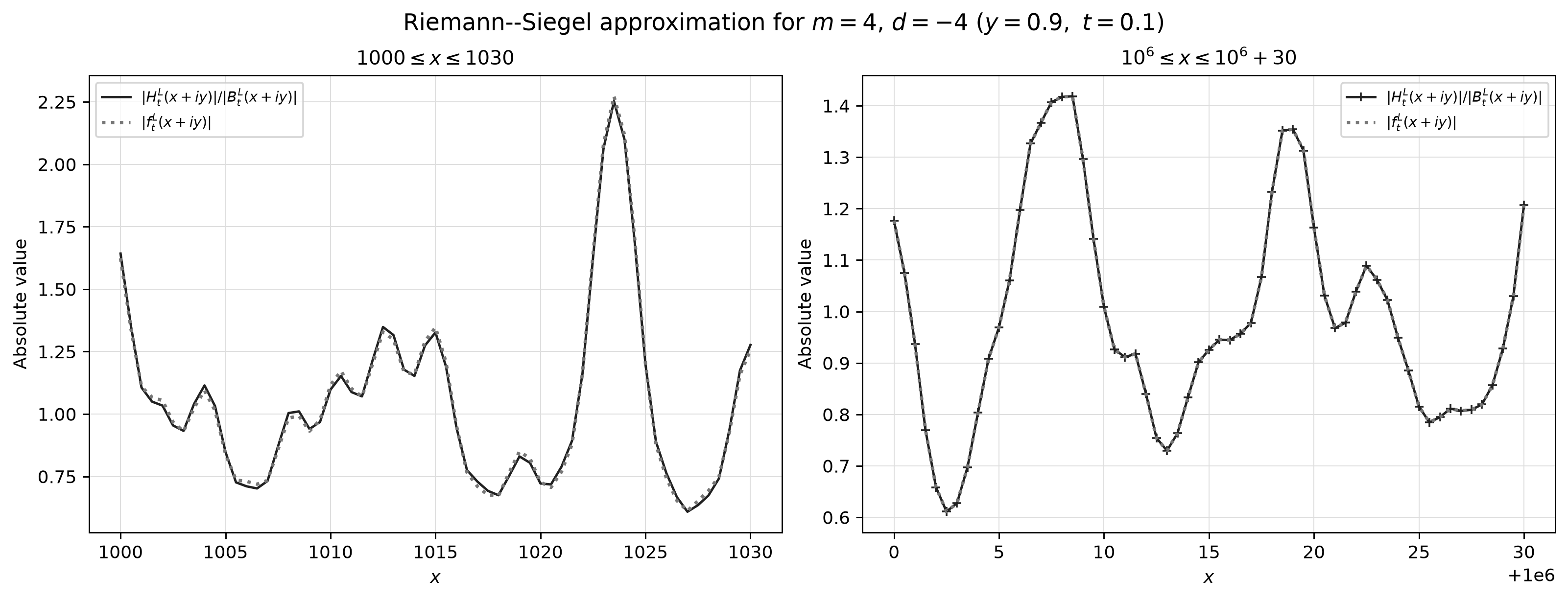}
			\caption{Comparison of $|f_t^L (x + iy)|$ and $|H_t^L (x + iy)|/|B_t^L (x + iy)|$ for $m=4$, $y=0.9$, and $t=0.1$, on $1000\leq x\leq1030$ (left) and $10^6\leq x\leq10^6+30$ (right).}
		\end{figure}

		\begin{figure}
			\centering
			\includegraphics[width=0.8\linewidth]{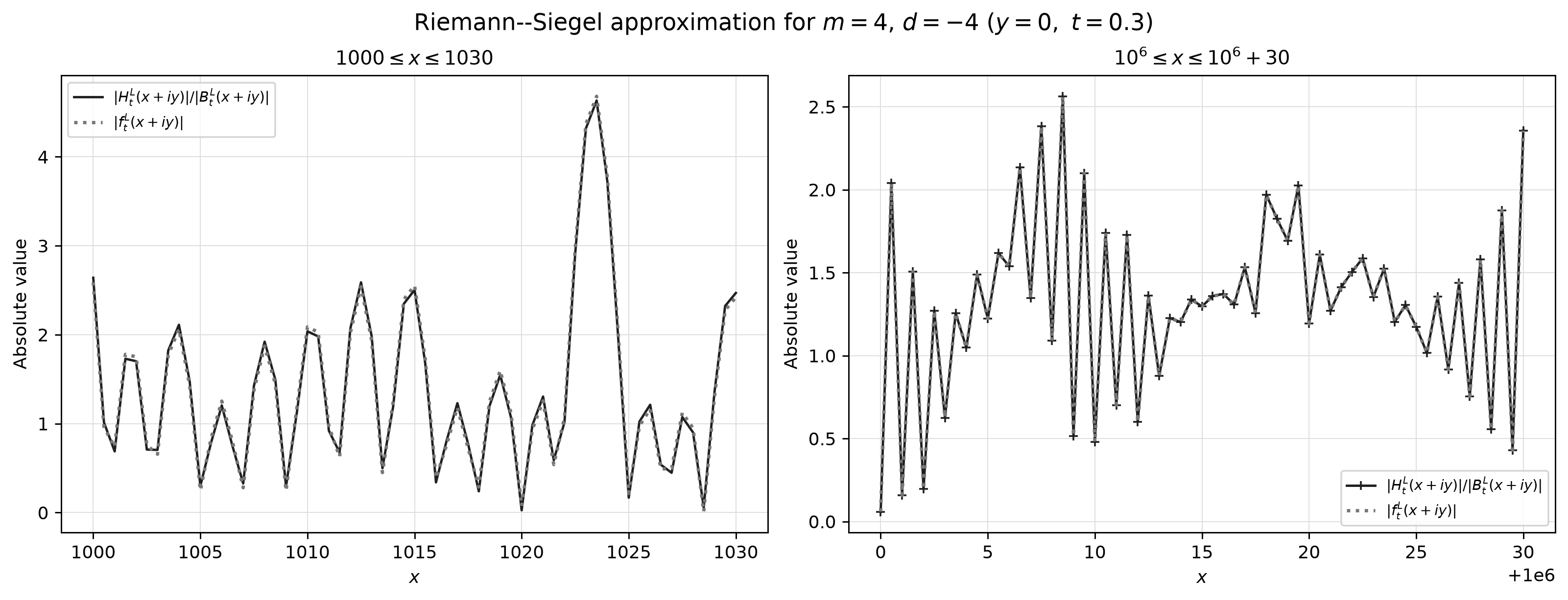}
			\caption{Comparison of $|f_t^L (x + iy)|$ and $|H_t^L (x + iy)|/|B_t^L (x + iy)|$ for $m=4$, $y=0$, and $t=0.3$, on $1000\leq x\leq1030$ (left) and $10^6\leq x\leq10^6+30$ (right).}
		\end{figure}
	
	\subsection{Proof of Theorem~\ref{EffectiveRSa}}
	
	We estimate $r_{t,n}^L$ and $R_{t,N}^L$ separately and then
	combine them in the symmetric decomposition.
	
	\subsubsection{The Dirichlet terms}
	
	By \eqref{alpha_L},
	\[
	\alpha_L'(s)
	=
	\frac{1}{2(s+a)}
	+
	\frac{1}{2(s+a)^2}.
	\]
	If $\ima(s)>1$, then
	\[
	\alpha_L'(s)
	=
	O_\le\!\left(\frac{1}{2\ima(s+a)}\right)
	+
	O_\le\!\left(\frac{1}{2\ima(s+a)^2}\right),
	\]
	and Lemma~\ref{elementaryball}(i) yields
	\begin{equation}\label{alphapl}
		\alpha_L'(s)
		=
		O_\le\!\left(
		\frac{1}{2(\ima(s+a)-1)}
		\right)
		=
		O_\le\!\left(
		\frac{1}{2(\ima(s)-1)}
		\right).
	\end{equation}
	
	Recall
	\[
	b_n^{t,L}
	=
	\exp\!\left(\frac{t}{4}\log^2n\right).
	\]
	\begin{lemma}\label{lem:rtn_estimate}
		Let $\sigma\in\R$, $T>2$, $n\in\N$, and $0<t\le\tfrac12$. Then
		\[
		r_{t,n}^L(\sigma+iT)
		=
		M_t^L(\sigma+iT)
		\frac{b_n^{t,L}\chi(n)}
		{n^{\sigma+iT+\frac t2\alpha_L(\sigma+iT)}}
		\left(1+O_\le(\varepsilon_{t,n}(\sigma+iT))\right),
		\]
		where
		\[
		\varepsilon_{t,n}(\sigma+iT)
		=
		\exp\!\left(
		\frac{
			\frac{t^2}{8}
			|\alpha_L(\sigma+iT)-\log(n)|^2
			+\frac t4+\frac12
		}{
			T-1.25
		}
		\right)-1.
		\]
	\end{lemma}
	
	\begin{proof}
		By Lemma~\ref{elementaryball}(v),
		\[
		r_{0,n}^L(s)
		=
		M_0^L(s)\chi(n)n^{-s}
		\exp\!\left(
		O_\le\!\left(
		\frac{1}{6(\ima(s)-1)}
		\right)
		\right)
		\]
		whenever $\ima(s)>2$.
		
		Set
		\[
		\alpha_L^n
		=
		\alpha_L(\sigma+iT)-\log(n).
		\]
		This is exactly the logarithmic derivative of
		$M_0^L(s)n^{-s}$ at $s=\sigma+iT$.
		
		For $T\ge1$,
		\[
		\ima(\alpha_L^n)
		=
		\frac{T}{2((\sigma+a)^2+T^2)}
		+
		\frac12
		\arg\!\left(\frac{m(s+a)}{2\pi}\right)
		\ge0.
		\]
		In particular,
		$\sigma+iT$ and
		$\sigma+iT+\tfrac{t}{2}\alpha_L^n$ have imaginary parts of the same sign.
		Applying \eqref{rtnchange},
		\begin{equation}\label{5p9}
			\begin{aligned}
				r_{t,n}^L(\sigma+iT)
				={}&
				\exp\!\left(-\frac t4(\alpha_L^n)^2\right)
				\int_{-\infty}^{\infty}
				\exp\!\left(-\sqrt t\,v\,\alpha_L^n\right)
				\\
				&\times
				M_0^L\!\left(
				\sigma+iT+\sqrt t\,v+\frac t2\alpha_L^n
				\right)
				\\
				&\times
				\exp\!\left(
				-\left(
				\sigma+iT+\sqrt t\,v+\frac t2\alpha_L^n
				\right)\log(n)
				\right)\chi(n)
				\\
				&\times
				\exp\!\left(
				O_\le\!\left(
				\frac{
					1
				}{
					6(T-1)
				}
				\right)
				\right)
				\frac{e^{-v^2}}{\sqrt{\pi}}\,dv.
			\end{aligned}
		\end{equation}
		
		Since $\operatorname{Im}\alpha_L^n\ge0$, the shifted argument
		has imaginary part at least $T$. Thus
		\eqref{eq:Stirling_uniform_horizontal} justifies the uniform
		error $1/(6(T-1))$ used in \eqref{5p9}.
		
		By \eqref{alphapl},
		\[
		\alpha_L'(s)
		=
		O_\le\!\left(
		\frac{1}{2(\ima(s)-1)}
		\right)
		\]
		for every $s$ on the line segment joining
		$\sigma+iT$ and
		$\sigma+iT+\sqrt t\,v+\tfrac t2\alpha_L^n$.
		
		Applying Taylor's theorem with remainder to a branch of
		$\log M_0^L$, we obtain
		\[
		\begin{aligned}
			M_0^L\!\left(
			\sigma+iT+\sqrt t\,v+\frac t2\alpha_L^n
			\right)
			={}&
			M_0^L(\sigma+iT)
			\\
			&\times
			\exp\!\left(
			\alpha_L(\sigma+iT)
			\left(
			\sqrt t\,v+\frac t2\alpha_L^n
			\right)
			\right)
			\\
			&\times
			\exp\!\left(
			O_\le\!\left(
			\frac{
				|\sqrt t\,v+\frac t2\alpha_L^n|^2
			}{
				4(T-1)
			}
			\right)
			\right).
		\end{aligned}
		\]
		Combining these estimates, using $
		\alpha_L(\sigma+iT)
		=
		\alpha_L^n+\log(n)$		and
		$|\sqrt t\,v+\tfrac t2\alpha_L^n|^2
		\le
		2tv^2+\frac{t^2}{2}|\alpha_L^n|^2$, 
		we conclude that
		\[
		\begin{aligned}
			r_{t,n}^L(\sigma+iT)
			={}&
			M_0^L(\sigma+iT)
			\exp\!\left(
			\frac t4(\alpha_L^n)^2-(\sigma+iT)\log(n)
			\right)\chi(n)
			\\
			&\times
			\int_{-\infty}^{\infty}
			\exp\!\left(
			O_\le\!\left(
			\frac{
				\frac{tv^2}{2}
				+
				\frac{t^2}{8}|\alpha_L^n|^2
				+\frac12
			}{
				T-1
			}
			\right)
			\right)
			\frac{e^{-v^2}}{\sqrt{\pi}}\,dv.
		\end{aligned}
		\]
		The prefactor simplifies to
		\[
		M_t^L(\sigma+iT)
		\frac{b_n^{t,L}\chi(n)}
		{n^{\sigma+iT+\frac t2\alpha_L(\sigma+iT)}}.
		\]
		It remains to bound the Gaussian integral. With
$A=t^2|\alpha_L^n|^2/8+1/2$, its positive majorant is
\[
\int_{\R}\exp\left(\frac{tv^2/2+A}{T-1}\right)
\frac{e^{-v^2}}{\sqrt\pi}\,dv
=e^{A/(T-1)}\left(1-\frac{t}{2(T-1)}\right)^{-1/2}.
\]
For $z=t/(2(T-1))$, the inequality
$-\log(1-z)\le z/(1-z)$ and $t\le1/2$ give
\[
\left(1-\frac{t}{2(T-1)}\right)^{-1/2}
\le\exp\left(\frac{t}{4(T-1-t/2)}\right)
\le\exp\left(\frac{t}{4(T-5/4)}\right).
\]
Since $A\ge0$, the majorant is at most
$\exp((A+t/4)/(T-5/4))$. Subtracting one, using
Lemma~\ref{elementaryball}(iv), proves the stated error bound.
\end{proof}
	
	\subsubsection{The integral remainder}
	
	In \cite{Poly}, estimates of Arias de Reyna \cite{AR} are used for the
	Riemann--Siegel integral involving the kernel
	\[
	\frac{e^{i\pi w^2}}{e^{i\pi w}-e^{-i\pi w}}.
	\]
	We need an analogous description for the $L$-function kernel
	\[
	\frac{1}{2\pi i}
	\int_{N\swarrow N+1}
	e^{\frac{i\pi}{m}x^2}x^{-s}W^L(x)\,dx,
	\]
	where
	\begin{equation}\label{WL}
		W^L(x)=
		\frac{\pi}{2m}
		\sum_{k=1}^{2m}
		\chi_m(k)e^{-i\pi k^2/m}
		\cot\left(\frac{\pi(x-k)}{2m}\right).
	\end{equation}
	
	\begin{prop}\label{prop:RtN_decomp}
		Let $\sigma\in\mathbb R$, $T>0$, and let $J$ be a nonnegative integer.
		Set		$s=\sigma+iT$, $a_{\rm sp}		=
		\sqrt{\frac{T}{2\pi m}}$, $	\delta
		=
		m a_{\rm sp}$, $N=
		m\lfloor a_{\rm sp}\rfloor$, $K	=
		\lfloor\delta-N\rfloor$, $M=N+K$, $		U
		=
		\exp\!\left\{
		-i\left(
		T\log\delta
		-
		\pi\frac{\delta^2}{m}
		-
		\frac{\pi}{8}
		\right)
		\right\}
		$, and 	$p=
		m-2(\delta-N)$.
		Then
		\[
		0\le K\le m-1,
		\qquad
		N\le M<N+m,
		\qquad
		p\in(-m,m].
		\]
		Moreover, we have the expansion
		\begin{equation}\label{deco_int}
			\ba	\frac{1}{2\pi i}
			& \int_{N\swarrow N+1}
			e^{\frac{i\pi}{m}x^2}x^{-s}W^L(x)\,dx
			\\
			&\quad =
			\sum_{n=N+1}^{M}\frac{\chi_m(n)}{n^s}
			+
			U\delta^{-\sigma}
			\left(
			\sum_{j=0}^{J}
			\frac{C_j^{(m)}(p,\sigma)}{\delta^j}
			+
			RS_J^{(m)}(s)
			\right).
			\ea 
		\end{equation}
		At integral $\delta$, this expansion is interpreted by continuation from the branch with $M=\lfloor\delta\rfloor$. The convention for the two lateral values is specified in the subsection on $C_0^{(m)}$ below.
	\end{prop}
	
	\begin{proof}
		From \eqref{xi_L},
		\[
		L(s)
		=
		R_L(s)+\eta(s)\overline{R_L(1-\overline{s})},
		\]
		where
		\[
		\eta(s)
		=
		\left(\frac{\pi}{m}\right)^{s-\frac12}
		\frac{
			\Gamma\!\left(\frac{1-s+a}{2}\right)
		}{
			\Gamma\!\left(\frac{s+a}{2}\right)
		}
		\text{ and }
		R_L(s)
		=
		\int_{0\swarrow1}
		\frac{x^{-s}e^{\frac{i\pi}{m}x^2}W^L(x)}
		{2\pi i}\,dx.
		\]
		By Cauchy's theorem,
		\[
		R_L(s)
		=
		\sum_{n=1}^{N}\frac{\chi_m(n)}{n^s}
		+
		\frac{1}{2\pi i}
		\int_{N\swarrow N+1}
		e^{\frac{i\pi}{m}x^2}x^{-s}W^L(x)\,dx.
		\]
		The possible poles of $W^L$ are the integers, and the residue at an integer
		$n$ is
		\[
			\operatorname*{Res}_{x=n}W^L(x)
			=\chi_m(n)e^{-i\pi n^2/m}.
		\]
		During the saddle-point deformation, the contour crosses the real axis
		at $x=\delta$. Hence the poles $N+k$, with $
		1\le k\le K$ are crossed. Their residues give
		\[
		\sum_{k=1}^{K}
		\frac{\chi_m(N+k)}{(N+k)^s}.
		\]
		Together with the initial sum this becomes $\sum_{n=1}^{M}\frac{\chi_m(n)}{n^s}$.
		We now isolate the saddle-point contribution. Writing $x=\delta+z$,
		we have
		\[
		x^{-s}e^{\frac{i\pi}{m}x^2}
		=
		e^{-s\log\delta+\frac{i\pi\delta^2}{m}}
		e^{\frac{2\pi i}{m}z^2}
		\widetilde g(\delta,z),
		\]
		where $\widetilde g(\delta,z)
		=
		\exp\!\left(
		-s\log\left(1+\frac{z}{\delta}\right)
		+
		\frac{2\pi i\delta z}{m}
		-
		\frac{\pi i z^2}{m}
		\right)$.
		The factor independent of $z$ may be written as $e^{-s\log\delta+\frac{i\pi\delta^2}{m}}=
		e^{-i\pi/8}U\delta^{-\sigma}$.
		
		Thus the saddle-point contribution can be written as $U\delta^{-\sigma}S$, where
		\[
		S
		=
		\frac{e^{-i\pi/8}}{2\pi i}
		\int_{\mathcal C_\delta^{\rm reg}}
		e^{\frac{2\pi i}{m}(x-\delta)^2}
		\widetilde g(\delta,x-\delta)
		W^L(x)\,dx.
		\]
		No extra sign such as $(-1)^{N/m-1}$ occurs here: the complete kernel
		$W^L$ is kept inside $S$, and the sign predicted by the individual
		cotangent summands is already contained in it through the periodicity
		relation \eqref{eq:WL_shift_m_C0}.
		
		Here $\mathcal C_\delta^{\rm reg}$ is the saddle-point contour through
		$x=\delta$. Set
		\[
		\tau=\frac{1}{2\sqrt{2mT}},
		\qquad p=m-2(\delta-N).
		\]
		Then
		\[
		T=\frac{1}{8m\tau^2},\qquad
		\delta=\frac{1}{4\sqrt\pi\tau},\qquad
		N+\frac m2-\delta=\frac p2.
		\]
		We first parameterize the contour by
		\[
		x=N+\frac m2-\frac{mu}{2},\qquad
		u=\frac pm+t e^{i\pi/4},\quad t\in\mathbb R,
		\]
		with $t$ increasing. Thus $x-\delta=-(mu-p)/2$, and the
		$x$-contour has direction $-e^{i\pi/4}=e^{5i\pi/4}$, a
		steepest-descent direction for $e^{2\pi i(x-\delta)^2/m}$.
		Now put $v=iu$, so $u=-iv$. We have
		\[
		x=N+\frac m2+\frac{imv}{2},\qquad
		x-\delta=\frac i2(mv-ip),\qquad
		\frac{dx}{2\pi i}=\frac m{4\pi}\,dv.
		\]
		The line becomes $v=ip/m+t e^{3i\pi/4}$, traversed from
		southeast to northwest; we denote it by $\nwarrow ip/m$.
		The same substitution applies to any prescribed indentations
		of the regularized contour. Moreover,
		\[
		\frac{2\pi i}{m}(x-\delta)^2
		=-\frac{\pi i}{2m}(mv-ip)^2,
		\]
		and $\widetilde g(\delta,i(mv-ip)/2)
		=g(\tau,\sqrt\pi(mv-ip))$, with $g$ as defined below.
		Consequently,
		\[
		\begin{aligned}
			S
			&=
			\frac{m e^{-i\pi/8}}{4\pi}
			\int_{\nwarrow ip/m}
			e^{-\frac{\pi i}{2m}(mv-ip)^2}
			g\!\left(
			\tau,\sqrt{\pi}(mv-ip)
			\right)
			W^L\left(
			N+\frac m2+\frac{imv}{2}
			\right)\,dv.
		\end{aligned}
		\]
		Here
		\[
		g(\tau,z)
		=
		\exp\!\left(
		-\left(\sigma+\frac{i}{8m\tau^2}\right)
		\log(1+2i\tau z)
		-\frac{z}{4\tau m}
		+\frac{iz^2}{4m}
		\right).
		\]
		The terms of order $\tau^{-1}$ and $\tau^0$ in the exponent cancel:
		they are respectively cancelled by $-z/(4\tau m)$ and $iz^2/(4m)$.
		The exponent therefore extends analytically to $\tau=0$. Hence, for
		fixed $J$,
		\begin{equation}\label{g_kserie}
			g(\tau,z)
			=\sum_{j=0}^{J}P_j(z)\tau^j
			+
			Rg_J(\tau,z).
		\end{equation}
		Since $\tau=\frac{1}{4\sqrt{\pi}\delta}$ and
		$dx/(2\pi i)=m\,dv/(4\pi)$ under
		$x=N+\tfrac m2+\tfrac{imv}{2}$, we obtain
		\begin{equation}\label{S_k2}
			S
			=
			\sum_{j=0}^{J}
			\frac{C_j^{(m)}(p,\sigma)}{\delta^j}+
			RS_J^{(m)}(s),
		\end{equation}
		where
		\begin{equation}\label{C_j}
\begin{aligned}
C_j^{(m)}(p,\sigma):={}&\frac{m e^{-i\pi/8}}{4\pi}
\left(\frac1{4\sqrt\pi}\right)^j
\int_{\nwarrow ip/m}e^{-\frac{\pi i}{2m}(mv-ip)^2}\\
&\quad\times P_j(\sqrt\pi(mv-ip))
W^L\left(N+\frac m2+\frac{imv}2\right)\,dv,
\end{aligned}
\end{equation}
		and
		\begin{equation}\label{RS_new}
\begin{aligned}
RS_J^{(m)}(s):={}&\frac{m e^{-i\pi/8}}{4\pi}
\int_{\nwarrow ip/m}e^{-\frac{\pi i}{2m}(mv-ip)^2}\\
&\quad\times Rg_J(\tau,\sqrt\pi(mv-ip))
W^L\left(N+\frac m2+\frac{imv}2\right)\,dv.
\end{aligned}
\end{equation}
		Combining the residue contribution with \eqref{S_k2} gives
		\eqref{deco_int}.
	\end{proof}

	\begin{theorem}
    \label{thm:Cjk_bound}
		Let $m>1$ be the conductor of a primitive real Dirichlet character,
		$\sigma\in\mathbb R$, $T>0$, and $j\ge1$. Put
		$\delta=\sqrt{mT/(2\pi)}$ and $h_0=3/2-\log2$. Then
		\[
		\frac{|C_j^{(m)}(p,\sigma)|}{\delta^j}
		\le c(\sigma) \, \frac{\Gamma((j+1)/2)}{(\delta b)^j},
		\]
		where
		\begin{equation}\label{eq:cbdef}
        \begin{aligned}
			&
            c(\sigma)=\frac{\varphi(m)}{4\pi}
			\left(\sqrt{\frac{2\pi}{m h_0}}+4\sqrt{2\pi}\right) \times 
			\begin{cases}
				2^{-\sigma},&\sigma\le0,\\
				9^\sigma,&\sigma>0,
			\end{cases}
			\\
            &
			b=\sqrt{\frac{2\pi h_0}{m}} \times 
			\begin{cases}
				1,&\sigma\le0,\\
				8/9,&\sigma>0.
			\end{cases}
        \end{aligned}
		\end{equation}
		For use with the remainder estimate, the following weaker common
		bound is also valid. With $b_1=\sqrt{2\pi/(7.49m)}$,
		\begin{equation}\label{eq:Cjk_common_bound}
			\frac{|C_j^{(m)}(p,\sigma)|}{\delta^j}
			\le\frac{2\varphi(m)}5\,\Gamma((j+1)/2)
			\begin{cases}
				2^{-\sigma}(\delta b_1)^{-j},&\sigma\le0,\\
				9^\sigma(8\delta b_1/9)^{-j},&\sigma>0.
			\end{cases}
		\end{equation}
		Both estimates hold for either continued lateral value at a change
		of branch.
	\end{theorem}
	
	\begin{proof}
		We follow the Cauchy-contour argument in \cite[Theorem~4.1]{AR},
		keeping the two Gaussian moments arising from the kernel $W^L$.
		For $z\ne0$, Cauchy's formula gives
		\[
			P_j(z)=\frac1{2\pi i}\int_C\frac{g(\xi,z)}{\xi^{j+1}}\,d\xi.
		\]
		Under $\zeta=-2iz\xi$,
		\[
		g(\xi,z)=(1-\zeta)^{-\sigma}
		\exp\!\left(-\frac{iz^2}{2m}f(\zeta)\right),
		\qquad
		f(\zeta)=-\frac{\log(1-\zeta)}{\zeta^2}
		-\frac1\zeta-\frac12.
		\]
		The singularity at zero is removable, and the logarithm has its cut
		on $[1,\infty)$. Substitution in \eqref{C_j} yields
		\begin{equation}\label{C_j2}
			\begin{split}
				C_j^{(m)}(p,\sigma)
				={}&\frac{m e^{-i\pi/8}}{8\pi^2 i}\left(\frac1{2i}\right)^j
				\int_{|\zeta|=r}\frac{(1-\zeta)^{-\sigma}}{\zeta^{j+1}}\\
				&\quad\times\int_{\nwarrow ip/m}
				e^{-\frac{i\pi(mv-ip)^2}{2m}(1+f(\zeta))}
				(mv-ip)^j
				W^L\!\left(N+\frac m2+\frac{imv}{2}\right)\,dv\,d\zeta ,
			\end{split}
		\end{equation}
		where $0<r<1$ and the circle is positively oriented. From the identity
		\[
		1+f(\zeta)=\frac12+\int_0^1\frac{t}{1-t\zeta}\,dt,
		\]
		we have that, for $|\zeta|=r<1$,
		\[
		h(\zeta):=1+\operatorname{Re}f(\zeta)
		\ge h_r:=\frac12+\frac1r-\frac{\log(1+r)}{r^2}
		\ge h_0.
		\]
		Indeed, $\operatorname{Re}(1-t\zeta)^{-1}\ge(1+tr)^{-1}$,
		with equality at $\zeta=-r$, and the last integral decreases with $r$.
		Thus the lower bound is uniform up to $r=1$.
		
		Parameterize the outer line by $v=ip/m+e^{-i\pi/4}x$.
		Then $mv-ip=me^{-i\pi/4}x$, and the exponential in
		\eqref{C_j2} has modulus $e^{-\pi m h(\zeta)x^2/2}$.
		For real $a$ and $b\ne0$, $|\cot(a+ib)|\le\coth|b|$.
		Since $\sum_{k=1}^{2m}|\chi_m(k)|=2\varphi(m)$, \eqref{WL} gives
		\begin{equation}\label{eq:Cjk_kernel_majorant}
			\left|W^L\!\left(\delta+\frac{im}{2}e^{-i\pi/4}x\right)\right|
			\le\frac{\varphi(m)}m
			\left(\pi+\frac{4\sqrt2}{|x|}\right),
			\qquad x\ne0.
		\end{equation}
		Here we used $\coth u\le1+u^{-1}$ for $u>0$.
		
		Write $w_\sigma(r)=(1+r)^{-\sigma}$ if $\sigma\le0$, and
		$w_\sigma(r)=(1-r)^{-\sigma}$ if $\sigma>0$.
		Evaluation of the two Gaussian moments in \eqref{C_j2} gives the
		more precise estimate
		\[
			\begin{split}
				\frac{|C_j^{(m)}(p,\sigma)|}{\delta^j}
				\le{}&\frac{\varphi(m)}{4\pi}w_\sigma(r)
				\left(\frac{\sqrt{m/(2\pi h_r)}}{r\delta}\right)^j\\
				&\times\left[
				\sqrt{\frac{2\pi}{m h_r}}\,
				\Gamma\!\left(\frac{j+1}{2}\right)
				+4\sqrt2\,\Gamma\!\left(\frac j2\right)\right].
			\end{split}
		\]
		The same majorants prove absolute convergence and justify Fubini's
		theorem. For $j\ge1$,
		\[
		\frac{\Gamma(j/2)}{\Gamma((j+1)/2)}\le\sqrt\pi.
		\]
		This follows from the monotonicity of
		$\Gamma(u)/\Gamma(u+1/2)$ for $u>0$.
		For $\sigma\le0$, let $r\uparrow1$ in the inequality. This avoids
		putting a Cauchy contour through the logarithmic branch point.
		For $\sigma>0$, choose $r=8/9$.
		Replacing $h_r$ by $h_0$ proves \eqref{eq:cbdef}.
		
		To obtain \eqref{eq:Cjk_common_bound}, note that $m\ge3$ and
		$h_0>4/5$. Since $b_1/b=(7.49h_0)^{-1/2}<1$ in the nonpositive
		case, and the same ratio holds after including $8/9$ in the
		positive case, it is enough to check the first power. Its
		prefactor, divided by $\varphi(m)$, is at most
		\[
		\frac{\sqrt{2\pi/(3h_0)}+4\sqrt{2\pi}}
		{4\pi\sqrt{7.49h_0}}
		<\frac{1.62+4(2.51)}{12(2.44)}<\frac25.
		\]
		These inequalities follow from $3<\pi<22/7$ and $h_0>4/5$.
		Finally, the majorant near $x=0$ is $O(|x|^{j-1})$.
		This also justifies continuation to either endpoint of a branch,
		including when the saddle coincides with a pole of $W^L$.
	\end{proof}
	
	\subsubsection{The leading coefficient $C_0^{(m)}$}
	
	Let $\chi_m$ be a primitive real Dirichlet character of conductor $m>1$.
	We work with the complete kernel
	\[
	W^L(x)=\frac{\pi}{2m}\sum_{k=1}^{2m}
	\chi_m(k)e^{-\pi i k^2/m}
	\cot\!\left(\frac{\pi(x-k)}{2m}\right).
	\]
	Its possible poles are simple and lie at the integers, with 
	$\operatorname*{Res}_{x=n}W^L(x)=\chi_m(n)e^{-\pi i n^2/m}$, when $\chi_m(n)=0$, the singularity is removable. Reindexing the sum gives
	\begin{equation}\label{eq:WL_shift_m_C0}
		W^L(x+m)=(-1)^mW^L(x).
	\end{equation}
	Fix now $N\in m\mathbb Z$ and write $\delta=N+\frac{m-p}{2}$, $E_\delta(x)=\exp\!\left(\frac{2\pi i}{m}(x-\delta)^2\right)$.
	The contour $\swarrow\delta$ is the line
	$x=\delta+e^{i\pi/4}u$, traversed from $u=+\infty$ to $u=-\infty$.
	Initially $\delta\notin\mathbb Z$. At $\delta=d\in\mathbb Z$, the convention
	$M=\lfloor\delta\rfloor=d$ means continuation from $\delta>d$. The other
	lateral value will be recorded separately.
	
	\begin{lemma}
		Consider the function $g$ in  \eqref{g_kserie}. Then $P_0\equiv1$. Moreover, 
		the leading coefficient $C_0^{(m)}$ is independent of $\sigma$ and is given by
		\begin{equation}\label{eq:C0_J_direct}
			C_0^{(m)}(p)=e^{-i\pi/8}\mathcal J_m(\delta),
		\end{equation}
		where 	$\mathcal J_m(\delta)=\frac1{2\pi i}\int_{\swarrow\delta}E_\delta(x)W^L(x)\,dx.$	\end{lemma}
	
	\begin{proof}
		Taylor expansion in the definition of $g$ gives
		\[
		-\frac{i}{8m\tau^2}\log(1+2i\tau z)
		=\frac{z}{4m\tau}-\frac{iz^2}{4m}
		-\frac{\tau z^3}{3m}+O(\tau^2).
		\]
		The first two terms cancel the explicit terms in $\log g$, so that
		$\log g=-\tau(z^3/(3m)+2i\sigma z)+O(\tau^2)$ and $P_0=1$.
		Substituting $P_0=1$ into the saddle-point integral gives
		\eqref{eq:C0_J_direct}. Equivalently, the change of variables
		$x=N+m/2+imv/2$ in \eqref{C_j} has Jacobian
		$dx/(2\pi i)=m\,dv/(4\pi)$ and induces a $v$-contour oriented from
		south-east to north-west. This induced orientation is understood when
		comparing the two formulas.
	\end{proof}
	
	\begin{remark}
		Equation~\eqref{eq:WL_shift_m_C0} gives
		$\mathcal J_m(\delta+m)=(-1)^m\mathcal J_m(\delta)$, hence
		$C_0^{(m)}(p;N)=(-1)^N C_0^{(m)}(p;0)$ for $N\in m\mathbb Z$.
		Thus, the modulus is independent of $N$. No additional external sign
		is included in the normalization \eqref{eq:C0_J_direct}.
		The symbol $\mathcal J_m$ denotes an integral, not the integer $J_m$
		used later in the choice of the truncation parameter.
	\end{remark}
	
	Set $\rho(p)=(-1)^m e^{-4\pi i\delta}=(-1)^m e^{2\pi ip}$. For a fixed integer $M$, define
	\begin{equation}\label{eq:R_direct_def}
		\mathcal R_{m,M}(\delta)=\sum_{n=M+1}^{M+m}
		\chi_m(n)\exp\!\left(\frac{\pi i}{m}(2\delta^2-4\delta n+n^2)\right).
	\end{equation}
	We write $\mathcal R_m=\mathcal R_{m,M}$ when $M=\lfloor\delta\rfloor$
	is understood. Define
\[
\begin{aligned}
\mathcal G_m(r)&=\sum_{k=1}^{2m}\chi_m(k)e^{-\pi i(k^2+rk)/m},\\
\mathcal T_m(\delta)&=\sum_{r=0}^{4m}\beta_r\mathcal G_m(r)
e^{\pi ir\delta/m-\pi ir^2/(8m)},\\
\mathcal K_m(\delta)&=\frac{e^{i\pi/4}}{2\sqrt{2m}}\mathcal T_m(\delta),
\end{aligned}
\]
where $\beta_0=\beta_{4m}=1/2$ and $\beta_r=1$ otherwise.
	
	\begin{prop}
		Assume $M<\delta<M+1$ and $\rho(p)\ne1$. Then
		\[
			\mathcal J_m(\delta)=
			\frac{\mathcal R_{m,M}(\delta)-\rho(p)\mathcal K_m(\delta)}{1-\rho(p)}.
		\]
		In particular,
		\[
				\displaystyle
				C_0^{(m)}(p)=\frac{e^{-i\pi/8}}{1-(-1)^m e^{2\pi ip}}
				\left[\mathcal R_{m,M}(\delta)
				-\frac{e^{i\pi/4}(-1)^m e^{2\pi ip}}{2\sqrt{2m}}
				\mathcal T_m(\delta)\right].
		\]
	\end{prop}
	
	\begin{proof}
		Moving the contour to $\swarrow(\delta+m)$ crosses the integers
		$M+1,\ldots,M+m$. With our orientation the residue theorem gives
		\[
		\frac1{2\pi i}\int_{\swarrow(\delta+m)}E_\delta W^L\,dx
		=\mathcal J_m(\delta)-\mathcal R_{m,M}(\delta).
		\]
		Since $E_\delta(x+m)=e^{-4\pi i\delta}e^{4\pi ix}E_\delta(x)$,
		translation by $m$ also gives
		\begin{equation}\label{eq:J_before_K}
			(1-\rho)\mathcal J_m=\mathcal R_{m,M}+\rho K_m^{\rm int},\qquad
			K_m^{\rm int}=\frac1{2\pi i}\int_{\swarrow\delta}
			E_\delta(x)(e^{4\pi ix}-1)W^L(x)\,dx.
		\end{equation}
		These deformations are justified by Gaussian decay on the ends of the
		contours and on their connecting segments.
		
		For $z=e^{\pi i(x-k)/m}$, the identity
		$\frac12(z^{4m}-1)\frac{z+1}{z-1}
		=\sum_{r=0}^{4m}\beta_r z^r$
		and $\cot(\pi(x-k)/(2m))=i(z+1)/(z-1)$ yield
		\begin{equation}\label{eq:WL_direct_fourier}
			\frac{(e^{4\pi ix}-1)W^L(x)}{2\pi i}
			=\frac1{2m}\sum_{r=0}^{4m}\beta_r\mathcal G_m(r)e^{\pi irx/m}.
		\end{equation}
		On the oriented line,
		\begin{equation}\label{eq:gaussian_swarrow}
			\int_{\swarrow\delta}E_\delta(x)\,dx
			=-e^{i\pi/4}\sqrt{m/2}.
		\end{equation}
		Completing the square gives
		\[
		\frac{2\pi i}{m}(x-\delta)^2+\frac{\pi irx}{m}
		=\frac{2\pi i}{m}(x-\delta+r/4)^2
		+\frac{\pi ir\delta}{m}-\frac{\pi ir^2}{8m}.
		\]
		The resulting Gaussian is entire, so its contour can be translated.
		Equations~\eqref{eq:WL_direct_fourier}--\eqref{eq:gaussian_swarrow} imply
		$K_m^{\rm int}=-\mathcal K_m$. Substitution into
		\eqref{eq:J_before_K} proves the formula.
	\end{proof}

	The resonant points satisfy $\rho(p_0)=1$, or equivalently
	\[
	p_0\in\begin{cases}\mathbb Z,&m\text{ even},\\
		\mathbb Z+1/2,&m\text{ odd}.
	\end{cases}\iff
	4\delta_0-m\in2\mathbb Z,
	\quad \delta_0=N+(m-p_0)/2.
	\]
	Fix $M\in\mathbb Z$ and set $c=M+1/2$. The integral $\mathcal J_{m,M}(\delta)=\frac1{2\pi i}
	\int_{\swarrow c}E_\delta(x)W^L(x)dx$ 
	is entire in $\delta$, the line avoids all poles, and its Gaussian decay
	is locally uniform in complex $\delta$. It agrees with $\mathcal J_m$
	for $M<\delta<M+1$. The identity theorem therefore extends
	\eqref{eq:J_before_K} to $(1-\rho)\mathcal J_{m,M}=\mathcal R_{m,M}-\rho\mathcal K_m$. 
	In particular, on every branch, $\mathcal R_{m,M}(\delta_0)=\mathcal K_m(\delta_0).$
	This proves cancelation with the correct phases and endpoint weights,
	without a separate evaluation of a quadratic Gauss sum.
	
	For implementation, two useful identities are:
	\[
	\mathcal G_m(r+2m)=\mathcal G_m(r),\qquad
	\mathcal G_m(r)=0\quad\text{if }r\not\equiv m\pmod2.
	\]
	Periodicity is immediate, and the second identity follows by pairing $k$
	with $k+m$, whose contributions differ by $(-1)^{m+r}$.
	
	On a fixed branch, put $\mathcal F_{m,M}(p)=\mathcal R_{m,M}(\delta(p))-\mathcal K_m(\delta(p)).$ Then $\mathcal F_{m,M}(p_0)=0$ and	 $C_0^{(m)}(p)=e^{-i\pi/8}
	\left(\mathcal K_m(\delta(p))+
	\frac{\mathcal F_{m,M}(p)}{1-\rho(p)}\right).$
	Since $1-\rho(p)=-2\pi i(p-p_0)+O((p-p_0)^2)$, its continuation at
	resonance is
	\[	\displaystyle
	C_{0,M}^{(m)}(p_0)=e^{-i\pi/8}
	\left(\mathcal K_m(\delta_0)-\frac{\mathcal F'_{m,M}(p_0)}{2\pi i}\right).
	\]
	Here and below a prime denotes differentiation with respect to $p$.
	Writing $A_n(\delta)$ for the summand in \eqref{eq:R_direct_def},
	the required derivatives are explicitly
	\[
	\begin{aligned}
		\frac{d\mathcal R_{m,M}}{dp}
		&=\frac{2\pi i}{m}\sum_{n=M+1}^{M+m}(n-\delta)A_n(\delta),\\
		\frac{d\mathcal K_m}{dp}
		&=-\frac{\pi i}{2m}\frac{e^{i\pi/4}}{2\sqrt{2m}}
		\sum_{r=0}^{4m}r\beta_r\mathcal G_m(r)
		e^{\pi ir\delta/m-\pi ir^2/(8m)}.
	\end{aligned}
	\]	
	\begin{remark}
		A removable singularity on a fixed branch must be distinguished from
		a change of branch. If $\delta(p_b)=d\in\mathbb Z$, moving the contour
		from $\swarrow(d-1/2)$ to $\swarrow(d+1/2)$ crosses only the possible
		pole at $d$. Thus $
		\mathcal J_{m,d}(d)-\mathcal J_{m,d-1}(d)
		=-\chi_m(d)e^{-\pi i d^2/m}.$ Since $d\delta/dp=-1/2$, the left $p$-limit uses branch $d$ and the
		right $p$-limit uses branch $d-1$. Consequently, for either parity of $m$,
		\[
			\displaystyle
			C_{0,-}^{(m)}(p_b)-C_{0,+}^{(m)}(p_b)
			=-e^{-i\pi/8}\chi_m(d)e^{-\pi i d^2/m}.
		\]
		The modulus of this complex difference is $|\chi_m(d)|\le1$. Note that 
		the difference between the two moduli may be smaller.
		For odd $m$, \emph{resonances} are half-integers and are continuous, while block
		changes occur at odd integers. For even $m$, ``block changes'' occur at the
		even-integer resonances, and the odd-integer resonances are continuous.
	\end{remark}
    For the character $\chi_{-4}$, the Dirichlet beta function, take $N=0$
and $-4<p<4$, so that $\delta=2-p/2$ and $M=\lfloor\delta\rfloor$.
Figure~\ref{fig:C0_m4_branch_behavior} plots the real and imaginary parts
and the modulus of $C_0^{(4)}(p)$. The two lateral values may differ: at
$p=\pm2$ their difference has modulus $1$, while at $p=0$ they agree because
$\chi_{-4}(2)=0$.
\begin{figure}[htbp]
\centering
\includegraphics[width=0.5\linewidth]{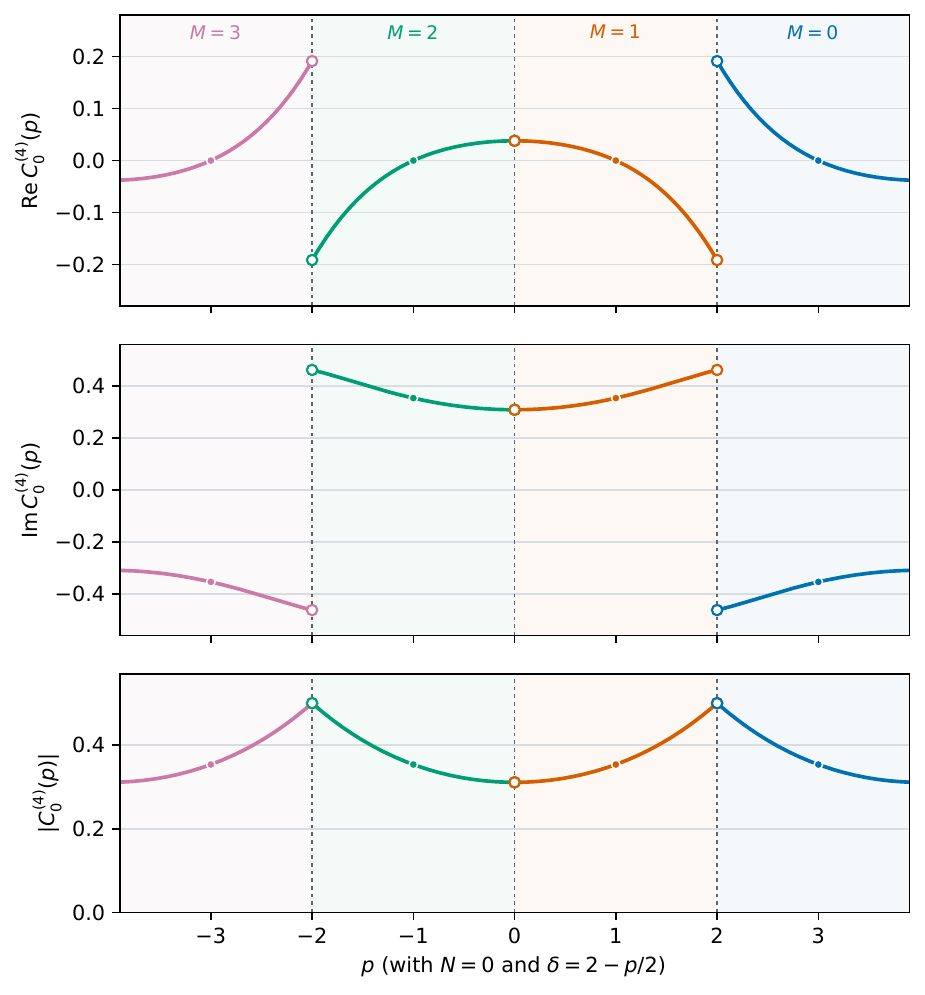}
\caption{The real part, imaginary part and modulus of $C_0^{(4)}(p)$ for
$\chi_{-4}$, with $N=0$ and $-4<p<4$.}
\label{fig:C0_m4_branch_behavior}
\end{figure}

	We keep the dependence on the distance to the real axis inside the
	Gaussian integral. This yields an explicit bound without estimates for
	incomplete character sums or numerical maximization.
	
	\begin{lemma}
		Let $c=M+1/2$. On the line $L_c=\{c+e^{i\pi/4}u:u\in\mathbb R\}$,
		\begin{equation}\label{eq:Omega1_def}
			|W^L(x)|\le\Omega_m(|\operatorname{Im}x|),\qquad
			\Omega_m(y)=\frac{\pi\varphi(m)}m+
			2\sum_{j=0}^{m-1}\frac1{\max\{y,(j+1/2)/\sqrt2\}}.
		\end{equation}
		In particular, $\Omega_m$ is non-increasing.
	\end{lemma}
	
	\begin{proof}
		First, for $|a|\le\pi/2$ and real $b$, away from a pole,
		\begin{equation}\label{eq:C0_sharp_cot}
			|\cot(a+ib)|\le\sqrt{1+\frac1{a^2+b^2}}
			\le1+\frac1{\sqrt{a^2+b^2}}.
		\end{equation}
		Indeed,
		$|\cot(a+ib)|^2=1+\cos(2a)/(\sin^2a+\sinh^2b)$.
		If $\cos(2a)\le0$, this is at most $1$. Otherwise,
		$\sin^2a\ge a^2\cos^2a\ge a^2\cos(2a)$ and
		$\sinh^2b\ge b^2\ge b^2\cos(2a)$, which prove the first inequality.
		
		Write $x=v+iy$, $d=\operatorname{dist}(v-k,2m\mathbb Z)$, and
		$d_k=\operatorname{dist}(c-k,2m\mathbb Z)$. Applying
		\eqref{eq:C0_sharp_cot} after reducing the real argument modulo $\pi$ gives
		\[
		\left|\cot\frac{\pi(x-k)}{2m}\right|
		\le1+\frac{2m}{\pi\sqrt{d^2+y^2}}.
		\]
		On $L_c$, $v=c+y$, and
		$(c+y-k-2m\ell)^2+y^2\ge(c-k-2m\ell)^2/2$ for every integer $\ell$.
		Hence $\sqrt{d^2+y^2}\ge\max\{|y|,d_k/\sqrt2\}$ and
		\[
			|W^L(x)|\le\frac{\pi\varphi(m)}m+
			\sum_{k=1}^{2m}\frac{|\chi_m(k)|}{\max\{|y|,d_k/\sqrt2\}}.
		\]
		Here $\sum_k|\chi_m(k)|=2\varphi(m)$. The distances $d_k$ are
		$1/2,3/2,\ldots,m-1/2$, each occurring twice. Bounding
		$|\chi_m(k)|\le1$ in the last sum proves \eqref{eq:Omega1_def}.
	\end{proof}
	
	\begin{lemma}
		On every branch $M\le\delta\le M+1$, including its continued endpoint
		values, any nonnegative, non-increasing majorant $\Omega$ on $L_c$ gives
		\begin{equation}\label{eq:C0_master_bound}
			|C_0^{(m)}(p)|\le\frac{e^{\pi/(4m)}}{\pi}
			\int_0^\infty e^{-2\pi u^2/m}\Omega(u/\sqrt2)\,du.
		\end{equation}
	\end{lemma}
	
	\begin{proof}
		Move the contour to $L_c$, which crosses no pole in the interior of
		the branch, and continue to its endpoints. Put $a=c-\delta$, so
		$|a|\le1/2$. On this line,
		\[
			|E_\delta(c+e^{i\pi/4}u)|
			=e^{\pi a^2/m}e^{-(2\pi/m)(u+a/\sqrt2)^2}.
		\]
		The convolution of two nonnegative even functions that decrease on
		$[0,\infty)$ is maximal when their centers coincide. In the present
		case, this follows directly by writing $\Omega(|u|/\sqrt2)$ as a
		superposition of indicators of centered intervals: the integral of
		$e^{-(2\pi/m)(u+b)^2}$ over each such interval is largest at $b=0$.
		Tonelli's theorem gives the same conclusion for $\Omega$. Notice that its constant
		part causes no difficulty. The triangle inequality, followed by
		$e^{\pi a^2/m}\le e^{\pi/(4m)}$, now gives \eqref{eq:C0_master_bound}.
	\end{proof}
	
	For $a>0$, define
	\[
	g_m(a)=\int_0^\infty\frac{e^{-2\pi u^2/m}}{\max\{u,a\}}du=\frac{\sqrt m}{2\sqrt2\,a}\operatorname{erf}\left(\sqrt{\frac{2\pi}{m}}a\right)
	+\frac12 E_1\!\left(\frac{2\pi a^2}{m}\right),
	\]
	where $E_1(z)=\int_z^\infty e^{-t}\,dt/t$ for $z>0$.
	
	\begin{theorem}[Uniform bound for $C_0^{(m)}$]
		Define
		\begin{equation}\label{eq:C0_uniform_best}
			\mathcal B_m^\star=e^{\pi/(4m)}
			\left[\frac{\varphi(m)}{2\sqrt{2m}}
			+\frac{2\sqrt2}{\pi}\sum_{j=0}^{m-1}g_m(j+1/2)\right].
		\end{equation}
		For every real $p$, with either lateral convention at a change of branch,
		$|C_0^{(m)}(p)|\le\mathcal B_m^\star$. More explicitly, for $m\ge3$,
		\begin{equation}\label{eq:C0_uniform_log}
			\frac{\mathcal B_m^\star}{\sqrt m}
			\le e^{\pi/(4m)}\left[
			\frac{\varphi(m)}{2\sqrt2\,m}
			+\frac{1+\tfrac12\log(2\pi m)+\tfrac12\gamma+\log2}{\pi}
			+\frac{\sqrt2\,e^{-2\pi m}}{2\pi^2\sqrt m}\right],
		\end{equation}
		where $\gamma$ is Euler's constant. In particular,
		\begin{equation}\label{eq:C0_400k_universal}
			|C_0^{(m)}(p)|<3.33\sqrt m\qquad(3\le m\le400000).
		\end{equation}
	\end{theorem}
	
	\begin{proof}
		Inserting \eqref{eq:Omega1_def} in \eqref{eq:C0_master_bound} gives
		\eqref{eq:C0_uniform_best}. To estimate the sum, put $\alpha=2\pi/m$
		and $A(a)=\int_0^a e^{-\alpha u^2}\,du$. Differentiation gives
		\[
		g_m'(a)=-\frac{A(a)}{a^2},\qquad
		g_m''(a)=\frac{2A(a)-ae^{-\alpha a^2}}{a^3}>0.
		\]
		Thus the midpoint inequality for convex functions gives $\sum_{j=0}^{m-1}g_m(j+1/2)\le\int_0^m g_m(a)\,da.$
		The logarithmic singularity at $a=0$ is integrable. Tonelli's theorem
		evaluates the last integral as
		\[
		\int_0^m e^{-\alpha u^2}\bigl(1+\log(m/u)\bigr)\,du
		+m\int_m^\infty\frac{e^{-\alpha u^2}}u\,du.
		\]
		For $x\ge1$, $0\le x^{-1}-1+\log x\le x$. It follows that this is at most
		\[
		\begin{aligned}
			&\int_0^\infty e^{-\alpha u^2}\bigl(1+\log(m/u)\bigr)\,du
			+\int_m^\infty e^{-\alpha u^2}\frac{u}{m}\,du\\
			&=\frac{\sqrt m}{2\sqrt2}
			\left(1+\frac12\log(2\pi m)+\frac\gamma2+\log2\right)
			+\frac{e^{-2\pi m}}{4\pi}.
		\end{aligned}
		\]
		The Gaussian logarithmic moment follows by differentiating
\[
\int_0^\infty u^{s-1}e^{-\alpha u^2}\,du
=\frac12\alpha^{-s/2}\Gamma(s/2)
\]
at $s=1$ and using $\Gamma'(1/2)/\Gamma(1/2)=-\gamma-2\log2$.
Finally, substitution proves \eqref{eq:C0_uniform_log}.
		
		For clarity, the numerical constant in \eqref{eq:C0_400k_universal}
		needs only elementary bounds, not numerical quadrature. Put
		\[
		A_0=\frac1{2\sqrt2}
		+\frac{1+\tfrac12\log(2\pi)+\tfrac12\gamma+\log2}{\pi},
		\qquad 1.2768728<A_0<1.276873.
		\]
		The last term in the brackets in \eqref{eq:C0_uniform_log} is less than
		$10^{-9}$ for $m\ge3$. Since $\varphi(m)\le m$, it suffices to bound
		\[
		f(x)=e^{\pi/(4x)}\left(\frac{\log x}{2\pi}+A_0+10^{-9}\right).
		\]
		For $3\le x\le8$,
		$f(x)\le e^{\pi/12}(\log8/(2\pi)+A_0+10^{-9})<2.089$.
		This includes the cases $m=3,5,7$.
		For $x\ge8$, the sign of $f'(x)$ is the sign of
		$2x-\pi^2(\log x/(2\pi)+A_0+10^{-9})$. Additionally, this expression is positive
		at $8$ (it exceeds $0.131$), and its derivative is
		$2-\pi/(2x)>0$. Hence, $f$ is increasing on $[8,\infty)$, and
		$f(400000)<3.32986<3.33$, as required.
	\end{proof}

	\begin{theorem}\label{thm:RSJk_bound}
		Let $m>1$ be the conductor of a primitive real Dirichlet character.
		Let $s=\sigma+iT$, with $T>0$, and let $J\ge1$ be an integer.
		Assume that $\sigma\ge0$ or $J+\sigma\ge1$. With
		$\delta=\sqrt{mT/(2\pi)}$, one has
		\[
		|RS_J^{(m)}(p,\sigma)|
		\le c_1(\sigma)\frac{\Gamma((J+2)/2)}{(\delta b_1)^{J+1}},
		\qquad b_1=\sqrt{\frac{2\pi}{7.49m}},
		\]
		where
		\[
		c_1(\sigma)=\frac{2\varphi(m)}5
		\begin{cases}
			2^{-\sigma},&\sigma\le0,\\
			9^\sigma,&\sigma>0.
		\end{cases}
		\]
		At a change of branch, the estimate holds for either continued lateral
		value.
	\end{theorem}
	
	\begin{proof}
		Let us put $k=J+1$, $\tau=(4\sqrt\pi\,\delta)^{-1}$, and
		$z=\sqrt\pi(mv-ip)$. For fixed $z\ne0$, Cauchy's formula gives
		\[
		Rg_J(\tau,z)=\frac{1}{2\pi i}\int_C
		\frac{g(\xi,z)}{\xi-\tau}
		\left(\frac{\tau}{\xi}\right)^k\,d\xi,
		\]
		where $C$ encloses $0$ and $\tau$ in the domain of the logarithm.
		Under $\zeta=-2iz\xi$, the integrand becomes
		\[
		\frac{(-2iz\tau)^k}{2\pi i}
		\frac{(1-\zeta)^{-\sigma}\exp(-iz^2f(\zeta)/(2m))}
		{(\zeta+2iz\tau)\zeta^k}\,d\zeta,
		\]
		with the same $f$ as in the proof of Theorem~\ref{thm:Cjk_bound}.
		
		We parameterize the outer contour by
		$v=ip/m+e^{-i\pi/4}x$, and let
		\[
		\mathcal L=\left\{\frac12+e^{i\pi/4}y:y\in\mathbb R\right\},
		\]
		oriented in the direction of increasing $y$.
		Both $0$ and $-2iz\tau=-im e^{-i\pi/4}x/(2\delta)$ lie on the line
		through zero parallel to $\mathcal L$, to its left.
		The cut $[1,\infty)$ lies to its right.
		The deformation in \cite[proof of Theorem~4.2]{AR} therefore gives
		\begin{equation}\label{eq:RSJ_line_remainder}
			Rg_J(\tau,z)=\frac{(-2iz\tau)^k}{2\pi i}
			\int_{\mathcal L}
			\frac{(1-\zeta)^{-\sigma}\exp(-iz^2f(\zeta)/(2m))}
			{(\zeta+2iz\tau)\zeta^k}\,d\zeta.
		\end{equation}
		Indeed, $f(\zeta)=-1/2+O(1/|\zeta|)$ at infinity in the
		half-plane of closure. The closing integral is
		$O(R^{-k-\sigma})$ and tends to zero because $k+\sigma>0$.
		The distance between the parallel lines gives
		\[
		|\zeta+2iz\tau|\ge\frac1{2\sqrt2}.
		\]
		Write
		\[
		h(\zeta)=1+\operatorname{Re}f(\zeta),\qquad
		G(\zeta)=\frac1{|\zeta|^2h(\zeta)},\qquad
		Q(\zeta)=|1-\zeta|^2G(\zeta).
		\]
		We use the auxiliary estimates established in the proof of
		\cite[Theorem~4.2, following (4.9)]{AR}, with outward rounding:
		\begin{equation}\label{eq:RSJ_AR_line_estimates}
			h>0.43,\qquad G<7.49,\qquad
			\int_{\mathcal L}G(\zeta)\,|d\zeta|<9.58,\qquad Q<5.79.
		\end{equation}
		These are the estimates for precisely this $f$ and this line, notice that 
		they do not depend on $m$. Also
		$|1-\zeta|\ge1/(2\sqrt2)$ on $\mathcal L$.
		
		By \eqref{eq:Cjk_kernel_majorant},
		\[
			\left|W^L\!\left(\delta+\frac{im}{2}e^{-i\pi/4}x\right)\right|
			\le\frac{\varphi(m)}m
			\left(\pi+\frac{4\sqrt2}{|x|}\right).
		\]
		Since $z^2=-i\pi m^2x^2$, the product of the exponential factors
		in \eqref{RS_new} and \eqref{eq:RSJ_line_remainder} has modulus
		$e^{-\pi m h(\zeta)x^2/2}$. Hence
		\[
		\begin{split}
			|RS_J^{(m)}(p,\sigma)|
			\le{}&\frac{\varphi(m)\sqrt2}{4\pi^2}
			\int_{\mathcal L}|1-\zeta|^{-\sigma}|\zeta|^{-k}\\
			&\times\int_{\mathbb R}
			\left(\frac{m|x|}{2\delta}\right)^k
			\left(\pi+\frac{4\sqrt2}{|x|}\right)
			e^{-\pi m h(\zeta)x^2/2}\,dx\,|d\zeta|.
		\end{split}
		\]
		Evaluation of the two Gaussian moments, followed by
		\[
		\frac{\Gamma(k/2)}{\Gamma((k+1)/2)}\le\frac2{\sqrt\pi}
		\qquad(k\ge2),
		\]
		bounds the inner integral by
		\[
		\left(\frac{\sqrt{m/(2\pi)}}{\delta}\right)^k
		\left(\sqrt{\frac{2\pi}{0.43m}}+8\sqrt{\frac2\pi}\right)
		\Gamma\!\left(\frac{k+1}{2}\right)h(\zeta)^{-k/2}.
		\]
		It remains to estimate
		\[
		I_{k,\sigma}=\int_{\mathcal L}
		|1-\zeta|^{-\sigma}G(\zeta)^{k/2}\,|d\zeta|.
		\]
		For $\sigma\ge0$, \eqref{eq:RSJ_AR_line_estimates} gives
		\[
		I_{k,\sigma}\le
		(2\sqrt2)^\sigma\,9.58(7.49)^{(k-2)/2}.
		\]
		For $\sigma<0$, put $r=-\sigma$. The hypothesis gives $k-r\ge2$,
		and
		\[
		|1-\zeta|^rG^{k/2}=Q^{r/2}G^{(k-r)/2}.
		\]
		Thus
		\[
		I_{k,\sigma}\le
		9.58(7.49)^{(k-2)/2}
		\left(\frac{5.79}{7.49}\right)^{r/2}.
		\]
		Keeping $r$ real is useful since it avoids rounding $-\sigma$ to an integer. Combining these estimates gives the stronger bound
		\[
			\begin{split}
				|RS_J^{(m)}(p,\sigma)|
				\le{}&\frac{9.58\,\varphi(m)}{7.49}
				\left(\frac1{2\pi^{3/2}\sqrt{0.43m}}+\frac4{\pi^{5/2}}\right)
				\frac{\Gamma((J+2)/2)}{(\delta b_1)^{J+1}}\\
				&\times
				\begin{cases}
					(5.79/7.49)^{-\sigma/2},&\sigma\le0,\\
					(2\sqrt2)^\sigma,&\sigma>0.
				\end{cases}
			\end{split}
		\]
		For the simpler constant in the statement, use $m\ge3$ and
		\[
		\frac{9.58}{7.49}<1.28,\qquad
		\frac1{2\pi^{3/2}\sqrt{0.43m}}<0.08,\qquad
		\frac4{\pi^{5/2}}<0.23.
		\]
		The last two inequalities follow, for example, from
		$\pi>3.14$, $\sqrt\pi>1.77$, and $\sqrt{1.29}>1.13$.
		Their product is less than $1.28(0.08+0.23)<2/5$.
		Finally, $(2\sqrt2)^\sigma\le9^\sigma$ for $\sigma>0$, and
		$(5.79/7.49)^{-\sigma/2}\le2^{-\sigma}$ for $\sigma\le0$.
		This proves the stated common bound.
		
		The preceding majorants prove absolute convergence and justify
		the interchange of integrals. Near $x=0$ they are
		$O(|x|^{k-1})$, so dominated convergence gives the same estimate
		for both continued lateral values.
	\end{proof}
	
	\begin{theorem}[Riemann--Siegel expansion]
		Let $s=\sigma+iT$ with $T>0$ and $J\ge0$. Then
		\[
		R_L(s)
		=
		\sum_{n=1}^{M}\frac{\chi_m(n)}{n^s}+
		U\delta^{-\sigma}
		\left(
		\sum_{j=0}^{J}
		\frac{C_{j}^{(m)}(p,\sigma)}{\delta^j}
		+
		RS_{J}^{(m)}
		\right),
		\]
		where
		\[
		M=\lfloor\delta\rfloor=N+\lfloor\delta-N\rfloor.
		\]
	\end{theorem}
	
	\begin{proof}
		The result follows from Proposition~\ref{prop:RtN_decomp}.
	\end{proof}

	\begin{prop}[Approximation for $R_{t,N}^L$]
		\label{prop:RtN_rigorous}
		Let $0\le\sigma\le1$, $T\ge240$, and $0<t\le1/2$. Put
		\[
		T'=T+\frac{\pi t}{8},\qquad
		\delta=\sqrt{\frac{mT'}{2\pi}},\qquad
		N=m\left\lfloor\sqrt{\frac{T'}{2\pi m}}\right\rfloor,\qquad
		K=\lfloor\delta\rfloor-N,\qquad M=N+K.
		\]
		Let $p$ and $U$ be the quantities in Proposition~\ref{prop:RtN_decomp},
		evaluated at height $T'$. We use the common constants in
		\eqref{eq:Cjk_common_bound} and Theorem~\ref{thm:RSJk_bound},
		\[
		b_1=\sqrt{\frac{2\pi}{7.49m}},\qquad
		c_1(0)=\frac{2\varphi(m)}5.
		\]
		Then
		\[
		\begin{aligned}
			R_{t,N}^L(\sigma+iT)
			={}&\sum_{k=1}^K r_{t,N+k}^L(\sigma+iT)\\
			&+e^{t\pi^2/64}M_0^L(iT')Ue^{\pi i\sigma/4}
			\left(C_0^{(m)}(p)
			+O_\le\!\left(\widetilde\varepsilon(\sigma,T,t,m)\right)\right),
		\end{aligned}
		\]
		where the following explicit bound is admissible:
		\begin{equation}\label{eq:RtN_epsilon_explicit}
			\begin{split}
				\widetilde\varepsilon(\sigma,T,t,m)
				={}&
				\mathcal B_m^\star
				\left[
				\exp\!\left(\frac{12\sigma^2+13}{24(T-5/4)}\right)-1
				\right]\\
				&+\frac{c_1(0)}{\delta b_1-3/2}
				\left(\frac94\,9^\sigma+\frac54\,2^{-\sigma}\right)
				\exp\!\left(\frac{2}{T-5/4}\right).
			\end{split}
		\end{equation}
		Moreover
		\begin{equation}\label{eq:boundepsilon_explicit}
			|\widetilde\varepsilon(\sigma,T,t,m)|
			\leq{} \mathcal B_m^\star \left(\frac{12\sigma^2+14}{24(T-5/4)}\right)+\frac{c_1(0)}{\delta b_1-3/2}
			\left(\frac94\,9^\sigma+\frac54\,2^{-\sigma}\right)
			\exp\!\left(\frac{2}{T-5/4}\right).
		\end{equation}
		
		Here $\mathcal B_m^\star$ is the bound in
		\eqref{eq:C0_uniform_best}. At integral $\delta$, the continuous
		term is interpreted by continuation from $\delta>M$, with $M$
		fixed, consistently with the residue range $n\le M$.
		The assumption $T\ge240$ corresponds to $x\ge480$ when $T=x/2$.
	\end{prop}
	
	\begin{proof}
		We follow the separation of errors in
		\cite[proof of Proposition~6.3]{Poly}. For use in the proof, put
		\[
		u=\sigma+\sqrt t\,v,\qquad s_v=u+iT',\qquad
		q=\delta b_1=\sqrt{\frac{T'}{7.49}},\qquad
		L=\lfloor q^2\rfloor,\qquad S=T'-1.
		\]
		Thus $q>5$, $L\ge30$, and $S\ge239$.
		Choose the truncation order
		\[
		J(u)=
		\begin{cases}
			1,&u\ge0,\\
			\max\{\lfloor-u\rfloor+3,L\},&u<0.
		\end{cases}
		\]
		For $u<0$ this gives $J(u)+u>2$, so
		Theorem~\ref{thm:RSJk_bound} applies on the whole integration line.
		The cutoff $L=\lfloor q^2\rfloor$ is chosen for the scale
		$q=\delta b_1$ in our coefficient bounds.
		
		The shifted Gaussian identity is
		\begin{equation}\label{eq:RtN_shifted_gaussian}
			R_{t,N}^L(\sigma+iT)
			=e^{t\pi^2/64}\int_{\mathbb R}
			e^{-i\pi\sqrt t\,v/4}R_{0,N}^L(s_v)
			\frac{e^{-v^2}}{\sqrt\pi}\,dv.
		\end{equation}
		Indeed, in \eqref{intrR} translate the contour from $\mathbb R$ to
		$\mathbb R+i\pi\sqrt t/8$. The unsmoothed terms are analytic in the
		intervening strip. Lemma~\ref{lem:GaussianBound}, with
		$T\le\operatorname{Im}s\le T'$, bounds their horizontal growth by
		$\exp(O((1+|\operatorname{Re}s|)\log(2+|\operatorname{Re}s|)))$.
		The Gaussian therefore makes the integrals over the connecting
		segments tend to zero. The same shift applies to each $r_{0,n}^L$.
		
		Write
		\[
		\mathcal G_L(s)=
		\left(\frac m\pi\right)^{(s+a)/2}
		\Gamma\!\left(\frac{s+a}{2}\right).
		\]
		Multiplying \eqref{deco_int} at $s_v$ by $\mathcal G_L(s_v)$ gives
		\[
        \ba 
		R_{0,N}^L(s_v)
		= &~{} \sum_{k=1}^K r_{0,N+k}^L(s_v)
		\\
        &~{} +\mathcal G_L(s_v)U\delta^{-u}
		\left(
		C_0^{(m)}(p)+
		\sum_{j=1}^{J(u)}\frac{C_j^{(m)}(p,u)}{\delta^j}
		+RS_{J(u)}^{(m)}(u+iT')
		\right).
        \ea 
		\]
		The finite sum in \eqref{eq:RtN_shifted_gaussian} becomes exactly
		$\sum_{k=1}^K r_{t,N+k}^L(\sigma+iT)$.
		
		We first control the factor multiplying the continuous part.
		The uniform Stirling estimate
		\eqref{eq:Stirling_uniform_horizontal} gives
		\[
		\mathcal G_L(s_v)=M_0^L(s_v)e^{\theta_v},
		\qquad |\theta_v|\le\frac{1}{6(T'-1)}.
		\]
		Moreover, \eqref{alpha_L} implies
		\[
		\left|\alpha_L(iT')-\log\delta-\frac{i\pi}{4}\right|
		\le\frac{1}{2T'},\qquad
		|\alpha_L'(x+iT')|\le\frac{1}{2(T'-1)}
		\quad(x\in\mathbb R).
		\]
		For the first inequality, the difference is $-1/(2iT')$ if $a=0$.
		If $a=1$, it is
		\[
		\frac12\left(\log(1+w)-\frac{w}{1+w}\right)
		=\frac12\int_0^w\frac{\zeta}{(1+\zeta)^2}\,d\zeta,
		\qquad w=\frac{1}{iT'},
		\]
		whose modulus is at most $1/(4T'^2)$. The derivative estimate
		follows from
		$|\alpha_L'|\le1/(2T')+1/(2T'^2)$.
		Taylor's formula for the analytic logarithm specified in
		\eqref{M0L}, along a horizontal segment, now yields
		\[
        \ba 
		e^{-i\pi\sqrt t\,v/4}\mathcal G_L(s_v)\delta^{-u}
		= &~{} M_0^L(iT')e^{i\pi\sigma/4}e^{\eta_v},
		\\
		|\eta_v|
		\le &~{} \frac{u^2}{4S}+\frac{|u|}{2T'}+\frac1{6S}
		\le\frac{u^2/2+5/12}{S}.
		\ea
        \]
		The final inequality uses $2|u|\le u^2+1$.
		
		Since $C_0^{(m)}(p)$ is independent of $u=\sigma+\sqrt tv$, it follows that the
		normalized continuous contribution is
		\[
		C_0^{(m)}(p)\bigl(1+O_\le(\varepsilon)\bigr)
		+O_\le(\mathcal E),
		\]
		where $\varepsilon=
		\int_{\mathbb R}
		\left(e^{(u^2/2+5/12)/S}-1\right)
		\frac{e^{-v^2}}{\sqrt\pi} dv$
		and
		\[
		\mathcal E=
		\int_{\mathbb R}e^{(u^2/2+5/12)/S}
		\left(
		\sum_{j=1}^{J(u)}\frac{|C_j^{(m)}(p,u)|}{\delta^j}
		+|RS_{J(u)}^{(m)}(u+iT')|
		\right)
		\frac{e^{-(u-\sigma)^2/t}}{\sqrt{\pi t}}du.
		\]
		Completing the square gives
		\[
		\varepsilon=
		\frac{\exp\!\left(5/(12S)+\sigma^2/(2S-t)\right)}
		{\sqrt{1-t/(2S)}}-1
		\le
		\exp\!\left(\frac{12\sigma^2+13}{24(T-5/4)}\right)-1.
		\]
		To obtain the inequality, use $-\log(1-z)\le z/(1-z)$ and
		\[
		\frac{5}{12S}+\frac{\sigma^2}{2S-t}
		-\frac12\log\!\left(1-\frac{t}{2S}\right)
		\le\frac{5/12+\sigma^2/2+t/4}{S-t/2}
		\le\frac{12\sigma^2+13}{24(T-5/4)}.
		\]
		This uses the subtraction of $1$ in the contribution of
		$C_0^{(m)}(p)$. For the rational bound
		\eqref{eq:boundepsilon_explicit}, put $A=12\sigma^2+13$ and
		$\mathfrak d=T-5/4$. Since $A\le25$ and $24\mathfrak d\ge A(A+1)$,
		\[
		e^{A/(24\mathfrak d)}-1
		\le\frac{A}{24\mathfrak d-A}\le\frac{A+1}{24\mathfrak d}.
		\]
		We next estimate $\mathcal E$. For $u\ge0$, take $J=1$ in the
		common coefficient bound \eqref{eq:Cjk_common_bound} and the remainder bound to get
		\[
		\frac{|C_1^{(m)}(p,u)|}{\delta}
		+|RS_1^{(m)}(u+iT')|
		\le c_1(0)9^u
		\left(\frac{9}{8q}+\frac{\sqrt\pi}{2q^2}\right)
		\le \frac{9c_1(0)9^u}{8(q-3/2)}.
		\]
		For $u<0$, both estimates have the factor $c_1(0)2^{-u}$, and
		their sum is at most
		\[
		c_1(0)2^{-u}
		\sum_{j=1}^{J(u)+1}\frac{\Gamma((j+1)/2)}{q^j}.
		\]
		For $j+2\le L+1$ the ratio of terms two indices apart is
		\[
		\frac{\Gamma((j+3)/2)/q^{j+2}}
		{\Gamma((j+1)/2)/q^j}
		=\frac{j+1}{2q^2}\le\frac12.
		\]
		Separate the first odd and even terms and sum the remaining
		geometric majorants. Since $\sqrt\pi/2<1$, this gives
		\[
		\begin{aligned}
			\sum_{j=1}^{L+1}\frac{\Gamma((j+1)/2)}{q^j}
			&\le\frac1q+\frac{\sqrt\pi}{2q^2}
			+\frac2{q^3}+\frac{3\sqrt\pi}{2q^4}\\
			&\le\frac1q+\frac1{q^2}+\frac2{q^3}+\frac3{q^4}
			\le\frac1{q-3/2}.
		\end{aligned}
		\]
		In the last step, compare with the geometric expansion of
		$(q-3/2)^{-1}$. Consequently, for every real $u$,
		\[
			\begin{split}
				&\sum_{j=1}^{J(u)}\frac{|C_j^{(m)}(p,u)|}{\delta^j}
				+|RS_{J(u)}^{(m)}(u+iT')|\\
				&\quad\le
				\frac{c_1(0)}{q-3/2}\left(\frac98\,9^u+2^{-u}\right)
				+c_1(0)2^{-u}
				\sum_{\substack{L+2\le j\le -u+4\\j\in\mathbb Z}}
				\frac{\Gamma((j+1)/2)}{q^j}.
			\end{split}
		\]
		An empty sum is zero. Integrating the three nonnegative terms
		on the right gives $\mathcal E\le\mathcal E_1+\mathcal E_2+\mathcal E_3$,
		as in the corresponding step of \cite{Poly}.
		
		For the first two terms, completing the square gives, for real
		$\lambda$,
		\[
		\begin{aligned}
			&\int_{\mathbb R}
			e^{(u^2/2+5/12)/S+\lambda u}
			\frac{e^{-(u-\sigma)^2/t}}{\sqrt{\pi t}}\,du\\
			&\quad=
			e^{\lambda\sigma+t\lambda^2/4}
			\frac{\exp\!\left(5/(12S)+
				(\sigma+t\lambda/2)^2/(2S-t)\right)}
			{\sqrt{1-t/(2S)}}\\
			&\quad\le
			e^{\lambda\sigma+t\lambda^2/4}
			\exp\!\left(
			\frac{5/12+(\sigma+t\lambda/2)^2/2+t/4}{S-t/2}
			\right).
		\end{aligned}
		\]
		For $\lambda=\log9$, the numerator in the last exponent is at most
		$13/24+625/512<2$, using $\log9<9/4$. For
		$\lambda=-\log2$, it is at most $25/24<2$, since
		$|\sigma-t\log2/2|\le1$.
		The elementary inequalities
		\[
		e^{(\log9)^2/8}<2,\qquad
		e^{(\log2)^2/8}<e^{1/8}\le\frac87
		\]
		therefore imply
		\[
		\mathcal E_1\le
		\frac{9c_1(0)9^\sigma}{4(q-3/2)}
		e^{2/(T-5/4)},\qquad
		\mathcal E_2\le
		\frac{8c_1(0)2^{-\sigma}}{7(q-3/2)}
		e^{2/(T-5/4)}.
		\]
		For the first elementary inequality, it suffices to use
		$(\log9)^2/8<81/128<2/3<\log2$.
		
		Finally, Tonelli's theorem gives
		\[
		\mathcal E_3=
		\frac{c_1(0)}{\sqrt{\pi t}}
		\sum_{j=L+2}^{\infty}\frac{\Gamma((j+1)/2)}{q^j}
		\int_{-\infty}^{4-j}
		e^{(u^2/2+5/12)/S-(u-\sigma)^2/t-u\log2}\,du.
		\]
		Here $j\ge32$, $u\le-28$, and $\sigma\ge0$. In this range
		$(u-\sigma)^2\ge u^2$ and
		$(u^2/2+5/12)/S\le u^2/(2t)$, so the integrand is bounded by
		$e^{-u^2/(2t)-u\log2}$. Since
		$u^2\ge-(j-4)u$ on the integration interval,
		\[
		\begin{aligned}
			\frac1{\sqrt{\pi t}}\int_{-\infty}^{4-j}
			e^{-u^2/(2t)-u\log2}\,du
			&\le
			\frac{2\sqrt t}{\sqrt\pi(j-4-2t\log2)}
			e^{-(j-4)^2/(2t)+(j-4)\log2}\\
			&\le e^{-(j-4)^2+(j-4)}.
		\end{aligned}
		\]
		For $j\ge32$, log-convexity of $\Gamma$ gives
		$\Gamma((j+1)/2)\le j!\le j^j$. Also $\log j\le j/4$, and hence
		\[
		\Gamma((j+1)/2)e^{-(j-4)^2+(j-4)}
		\le e^{-3j^2/4+9j-20}\le1.
		\]
		It follows that
		\[
		\mathcal E_3\le
		c_1(0)\sum_{j=L+2}^{\infty}q^{-j}
		=\frac{c_1(0)q^{-(L+1)}}{q-1}
		\le\frac{c_1(0)}{32(q-3/2)}
		\le
		\frac{c_1(0)2^{-\sigma}}{16(q-3/2)}
		e^{2/(T-5/4)}.
		\]
		We used $q>5$, $L\ge30$, and $2^{-\sigma}\ge1/2$.
		All estimates follow from inequalities; therefore, no numerical evaluation
		of the tail is needed in this case.
		Combining the three bounds and using $8/7+1/16<5/4$ gives
		\[
		\mathcal E\le
		\frac{c_1(0)}{q-3/2}
		\left(\frac94\,9^\sigma+\frac54\,2^{-\sigma}\right)
		e^{2/(T-5/4)}.
		\]
		These finite upper bounds also justify the absolute convergence
		of the integrals used above. Finally,
		$|C_0^{(m)}(p)|\le\mathcal B_m^\star$ and the bound for
		$\varepsilon$ prove \eqref{eq:RtN_epsilon_explicit}.
	\end{proof}
	
	\subsubsection{Symmetric decomposition}
	\label{subsec:setup_sym}
	
	Throughout this section we fix $0<t\le\tfrac12$ and a point
	$x+iy\in\Com$ with $x\ge480$ and $|y|\le1$. Define
	$s_+
	=
	\frac{1+y-ix}{2},
	s_-
	=
	\frac{1-y+ix}{2}.
	$
	Then
	$1-s_-=s_+,
	1-s_+=s_-,$
	and
	$\ima(s_+)=-\frac x2,
	\ima(s_-)=\frac x2.$
	At height
	$T'
	=
	\frac x2+\frac{\pi t}{8},$
	define
	\[
	\delta
	=
	\sqrt{\frac{mT'}{2\pi}},
	\qquad
	N
	=
	m
	\left\lfloor
	\sqrt{\frac{T'}{2\pi m}}
	\right\rfloor,\qquad
	K
	=
	\lfloor\delta-N\rfloor,
	\qquad
	M
	=
	N+K
	=
	\lfloor\delta\rfloor.
	\]
	By \eqref{expH0}, evaluated at $z=x+iy$, we have
	$s=\frac{1+iz}{2}=s_-.$
	
	Applying the generalized Riemann--Siegel formula and changing
	$v\mapsto-v$ in the conjugate term gives
	\begin{equation}\label{eq:symdecomp}
		H_t^L(x+iy)
		=
		\sum_{n=1}^{N}
		r_{t,n}^L(s_-)
		+
		\sum_{n=1}^{N}
		\overline{
			r_{t,n}^L(\overline{s_+})
		}
		+
		R_{t,N}^L(s_-)
		+
		\overline{
			R_{t,N}^L(\overline{s_+})
		}.
	\end{equation}
	
	Because $\chi_m$ is real-valued, the finite Dirichlet terms satisfy
	\[
	\overline{r_{t,n}^L(\overline w)}
	=
	r_{t,n}^L(w).
	\]
	We keep the conjugated integral remainder in the form supplied by
	\eqref{eq:symdecomp}. This is because the reality of $\chi_m$ alone does not imply a reflection
	identity for that contour integral.
	Using Proposition~\ref{prop:RtN_rigorous}, the extra residue contributions
	are exactly
	\[
	\sum_{n=N+1}^{M}r_{t,n}^L(s).
	\]
	Therefore
	\[
		H_t^L(x+iy)
		=
		\sum_{n=1}^{M}
		r_{t,n}^L(s_+)
		+
		\sum_{n=1}^{M}
		r_{t,n}^L(s_-)
		+
		\overline{\widehat R_{t,N}^L(\overline{s_+})}
		+
		\widehat R_{t,N}^L(s_-),
	\]
	where
	\[
	\widehat R_{t,N}^L(s)
	=
	R_{t,N}^L(s)
	-
	\sum_{n=N+1}^{M}r_{t,n}^L(s).
	\]
	To normalize, divide by
	\[
	B_t^L(x+iy)=M_t^L(s_+).
	\]
	Then
	\begin{equation}\label{eq:ABC_decomp}
		\frac{H_t^L(x+iy)}{B_t^L(x+iy)}
		=
		A^L(x+iy)+B^L(x+iy)-C^L(x+iy),
	\end{equation}
	where
	\[
	A^L(x+iy)
	=
	\frac{1}{M_t^L(s_+)}
	\sum_{n=1}^{M}r_{t,n}^L(s_-),
	\]
	\[
	B^L(x+iy)
	=
	\frac{1}{M_t^L(s_+)}
	\sum_{n=1}^{M}r_{t,n}^L(s_+),
	\]
	and
	\[
	C^L(x+iy)
	=
	-\frac{1}{M_t^L(s_+)}
	\left(
	\overline{\widehat R_{t,N}^L(\overline{s_+})}
	+
	\widehat R_{t,N}^L(s_-)
	\right).
	\]
	\subsubsection{Error estimates}
	Define
	$s_*^L(-)
	=
	s_-+\frac t2\alpha_L(s_-)$, 
	$s_*^L(+)
	=
	s_++\frac t2\alpha_L(s_+)$,
	and
	$\gamma^L(x+iy)
	=
	\frac{M_t^L(s_-)}{M_t^L(s_+)}.$
	\begin{prop}\label{prop:eA}
		The conjugate component satisfies
		\[
		A^L(x+iy)
		=
		\gamma^L(x+iy)
		\sum_{n=1}^{M}
		\frac{b_n^{t,L}\chi_m(n)}
		{n^{s_*^L(-)}}
		+
		O_\le(e_A(x+iy)),
		\]
		where
		\[
		e_A(x+iy)
		=
		|\gamma^L(x+iy)|
		\sum_{n=1}^{M}
		\frac{b_n^{t,L}}
		{n^{\re s_*^L(-)}}
		\left[
		\exp\!\left(
		\frac{
			\frac{t^2}{8}
			|\alpha_L(s_-)-\log(n)|^2
			+\frac t4+\frac12
		}{
			x/2-1.25
		}
		\right)-1
		\right].
		\]
	\end{prop}
	
	\begin{proof}
		Apply Lemma~\ref{lem:rtn_estimate} to
		$r_{t,n}^L(s_-)$ for $1\le n\le M$ with $T=x/2$ and divide by
		$M_t^L(s_+)$.
	\end{proof}
	
	\begin{prop}\label{prop:eB}
		The principal component satisfies
		\[
		B^L(x+iy)
		=
		\sum_{n=1}^{M}
		\frac{b_n^{t,L}\chi_m(n)}
		{n^{s_*^L(+)}}
		+
		O_\le(e_B(x+iy)),
		\]
		where
		\[
		e_B(x+iy)
		=
		\sum_{n=1}^{M}
		\frac{b_n^{t,L}}
		{n^{\re s_*^L(+)}}
		\left[
		\exp\!\left(
		\frac{
			\frac{t^2}{8}
			|\alpha_L(s_+)-\log(n)|^2
			+\frac t4+\frac12
		}{
			x/2-1.25
		}
		\right)-1
		\right].
		\]
	\end{prop}
	
	\begin{proof}
		Apply Lemma~\ref{lem:rtn_estimate} to
		$r_{t,n}^L(\overline{s_+})$ for $1\le n\le M$ with $T=x/2$,
		and conjugate the resulting identity. The character is real, and
		$M_t^L(\overline s)=\overline{M_t^L(s)}$,
		$\alpha_L(\overline s)=\overline{\alpha_L(s)}$.
	\end{proof}
	
	\begin{prop}\label{prop:eC}
		Set $T'
		=
		\frac x2+\frac{\pi t}{8}.$
		Then
		\[
		|C^L(x+iy)|
		\le
		e_{C,0}(x+iy),
		\]
		where
		\[
		e_{C,0}(x+iy)
		=
		\frac{
			\exp\!\left(\frac{t\pi^2}{64}\right)
			|M_0^L(iT')|
		}{
			|M_t^L(s_+)|
		}
		\left(
		2\,\mathcal B_m^{\star}
		+
		\widetilde\varepsilon\!\left(\re s_+,\frac x2,t,m\right)
		+
		\widetilde\varepsilon\!\left(\re s_-,\frac x2,t,m\right)
		\right).
		\]
    Here $\mathcal B_m^{\star}$ is any admissible bound for $|C_0^{(m)}(p)|$ of \eqref{eq:C0_uniform_best}; very importantly, the factor $2$
		accounts for the two leading coefficients and cannot be dropped.
	\end{prop}

	\begin{proof}
		We apply Proposition~\ref{prop:RtN_rigorous} directly to 
		$\widehat R_{t,N}^L(s_-)$ (since $\ima(s_-) = x/2 > 0$). 
		For the conjugated term, we apply Proposition~\ref{prop:RtN_rigorous} to
		$\widehat R_{t,N}^L(\overline{s_+})$, whose argument has a positive
imaginary part $x/2$. For the record, we notice that conjugation does not change its modulus. No
		reflection identity for the contour integral is needed. The higher-order corrections are controlled by $\widetilde\varepsilon$.
	\end{proof}
	
	Define
	\begin{equation}\label{ftL_def}
		f_t^L(x+iy)
		=
		\sum_{n=1}^{M}
		\frac{b_n^{t,L}\chi_m(n)}
		{n^{s_*^L(+)}}
		+
		\gamma^L(x+iy)
		\sum_{n=1}^{M}
		\frac{b_n^{t,L}\chi_m(n)}
		{n^{s_*^L(-)}}.
	\end{equation}
	
	Substituting the estimates from Propositions
	\ref{prop:eA}, \ref{prop:eB}, and \ref{prop:eC} into
	\eqref{eq:ABC_decomp}, we obtain
	\[
	\frac{H_t^L(x+iy)}{B_t^L(x+iy)}
	=
	f_t^L(x+iy)
	+
	O_\le
	\left(
	e_A(x+iy)
	+
	e_B(x+iy)
	+
	e_{C,0}(x+iy)
	\right).
	\]
	This proves Theorem~\ref{EffectiveRSa} for $t>0$. The case $t=0$
	follows by taking care with the limit $t\downarrow0$. If $\delta(x,0)$ is an integer,
	the truncation length is constant for all sufficiently small positive $t$. In this case, 
	the coefficient and remainder are taken at the corresponding continued
	endpoint. Proposition~\ref{prop:magnitudes} below gives simpler upper
	bounds for these three errors.
	\qed
	
	\begin{prop}
		\label{prop:magnitudes}
		Assume $x\ge480$, $|y|\le1$, $0\le t\le1/2$, and
		$3\le m\le400000$. Use the notation of the preceding propositions.
		Then:
		\begin{enumerate}
			\item
			\[
			|\gamma^L|
			\le
			\exp\!\left(\left(0.002+\frac{\log m}{3800}\right)|y|\right)
			\left(\frac{xm}{4\pi}\right)^{-y/2}.
			\]
			\item
			\[
			\operatorname{Re}s_*^L(+)
			\ge\frac{1+y}{2}+\frac t4\log\!\left(\frac{mx}{4\pi}\right)
			-\frac{(3+y)t}{2x^2}.
			\]
			\item
			\[
			\kappa:=\frac t2
			\left(\alpha_L(s_-)-\alpha_L(\overline{s_+})\right),
			\qquad
			|\kappa|\le\frac{t|y|}{2(x-2)}.
			\]
			\item
			\[
			\begin{split}
				e_A(x+iy)\le{}&
				|\gamma^L|M^{|\kappa|}
				\sum_{n=1}^{M}n^y\frac{b_n^{t,L}}{n^{\operatorname{Re}s_*^L(+)}}\\
				&\times\left[
				\exp\!\left(
				\frac{\frac{t^2}{16}\log^2\!\left(\frac{mx}{4\pi n^2}\right)
					+\frac{\log m}{2250}+1.3}{x-5/2}
				\right)-1\right].
			\end{split}
			\]
			\item
			\[
			\begin{split}
				e_B(x+iy)\le{}&
				\sum_{n=1}^{M}\frac{b_n^{t,L}}{n^{\operatorname{Re}s_*^L(+)}}\\
				&\times\left[
				\exp\!\left(
				\frac{\frac{t^2}{16}\log^2\!\left(\frac{mx}{4\pi n^2}\right)
					+\frac{\log m}{2250}+1.3}{x-5/2}
				\right)-1\right].
			\end{split}
			\]
			\item
			\[
			\begin{split}
				e_{C,0}(x+iy)\le{}&
				\left(\frac{mT'}{2\pi}\right)^{-(1+y)/4}\\
				&\times\exp\!\left(
				-\frac t{16}\log^2\!\left(\frac{mx}{4\pi}\right)
				+\frac{\frac14\left|\log\!\left(\frac{mx}{4\pi}\right)
					+i\frac\pi2\right|+1.54}{x-5/2}
				\right)\\
				&\times\left[2\mathcal B_m^\star+2\mathcal E_{RS}(x,t)\right],
			\end{split}
			\]
			where
			\begin{equation}\label{eq:ERS_magnitudes}
				\begin{split}
					\mathcal E_{RS}(x,t)={}&
					\mathcal B_m^\star
					\left[\exp\!\left(\frac{25}{12(x-5/2)}\right)-1\right]\\
					&+\frac{167\,c_1(0)}
					{8\left(\sqrt{T'/7.49}-3/2\right)}
					\exp\!\left(\frac4{x-5/2}\right),\\
					&\hspace{1mm}T'=\frac x2+\frac{\pi t}{8},
					\qquad c_1(0)=\frac{2\varphi(m)}5.
				\end{split}
			\end{equation}
		\end{enumerate}
		We keep $\mathcal B_m^\star$ in item~(6). The simpler estimate
		$\mathcal B_m^\star<3.33\sqrt m$ may be used in the stated conductor range.
	\end{prop}

	\begin{proof}
		Put $T=x/2$ and $\ell=\log(mx/(4\pi))$. For $0\le\sigma\le1$,
		the formula for $\alpha_L$, together with \eqref{alphapl}, gives
		\begin{equation}\label{eq:magnitudes_alpha}
			\alpha_L(\sigma+iT)=\frac\ell2+\frac{i\pi}{4}+R,
			\qquad
			|R|\le\frac{1+\sigma}{x-2}\le\frac2{x-2}.
		\end{equation}
		To check the constant, first evaluate at $iT$, where the error
		has modulus at most $1/(2T)$, and then integrate $\alpha_L'$
		along the horizontal segment.
		Writing $w=\sigma+a\in[0,2]$ also gives
		\[
		\operatorname{Re}R
		=\frac14\log\!\left(1+\frac{w^2}{T^2}\right)
		-\frac{w}{2(T^2+w^2)},
		\qquad
		|\operatorname{Re}R|\le\frac4{x^2}.
		\]
		Consequently,
		\[
		\begin{split}
			\left|\frac{d}{d\sigma}\log|M_t^L(\sigma+iT)|
			-\frac\ell2\right|
			&\le\frac4{x^2}+\frac{\ell}{8(x-2)}
			+\frac{\pi}{16(x-2)}+\frac1{2(x-2)^2}\\
			&\le0.002+\frac{\log m}{3800}.
		\end{split}
		\]
		For the last inequality, use $x\ge480$,
		$\log(480/(4\pi))<4$, $\pi<4$, and the fact that
		$\log(x/(4\pi))/(x-2)$ is decreasing on this range.
		Since $|M_t^L(s_+)|=|M_t^L(\overline{s_+})|$, integration from
		$(1+y)/2$ to $(1-y)/2$ proves item~(1), for either sign of $y$.
		The exact real part is
		\[
		\operatorname{Re}\alpha_L(s_+)
		=\frac12\log\!\left(
		\frac{m\sqrt{(1+y+2a)^2+x^2}}{4\pi}\right)
		-\frac{1+y+2a}{(1+y+2a)^2+x^2}.
		\]
		It is at least $\ell/2-(3+y)/x^2$, proving item~(2).
		Since $\overline{s_+}=s_-+y$,
		\[
		\alpha_L(s_-)-\alpha_L(\overline{s_+})
		=-y\int_0^1\alpha_L'(s_-+uy)\,du.
		\]
		On this segment \eqref{alphapl} is at most $1/(x-2)$.
		This proves item~(3).
		
		For items~(4) and~(5), put
		$L_n=\log(mx/(4\pi n^2))$. Since $1\le n\le M\le\delta$,
		\[
		\log(T/T')\le L_n\le\ell,\qquad |L_n|\le\ell.
		\]
		Indeed, $\log(T'/T)\le(T'-T)/T<1<\ell$.
		Equation~\eqref{eq:magnitudes_alpha} gives
		\[
		|\alpha_L(\sigma+iT)-\log n|^2
		\le\frac{L_n^2}{4}+\frac{\pi^2}{16}
		+\frac{2|L_n|+\pi}{x-2}+\frac4{(x-2)^2}.
		\]
		The same bound holds at the conjugate point. Therefore
		\[
		\begin{split}
			&\frac{t^2}{4}|\alpha_L(\sigma+iT)-\log n|^2+\frac t2+1\\
			&\quad\le
			\frac{t^2}{16}L_n^2+\frac54+\frac{\pi^2}{256}
			+\frac{|L_n|+\pi/2}{8(x-2)}
			+\frac1{4(x-2)^2}\\
			&\quad\le
			\frac{t^2}{16}L_n^2+\frac{\log m}{2250}+1.3.
		\end{split}
		\]
		The last inequality follows from the same decreasing logarithmic
		ratio, $x\ge480$, and $\pi^2<10$.
		Insert this in Propositions~\ref{prop:eA} and~\ref{prop:eB}.
		Finally,
		\[
		s_*^L(-)-\overline{s_*^L(+)}=-y+\kappa,\qquad
		|n^{-s_*^L(-)}|
		\le n^yM^{|\kappa|}n^{-\operatorname{Re}s_*^L(+)}.
		\]
		This proves both estimates.
		
		For item~(6), write $\sigma=(1+y)/2$ and
		$d=T'-T=\pi t/8$. Taylor's formula for $\log M_0^L$ from
		$iT'$ to $\sigma+iT$ yields
		\[
		\log\frac{|M_0^L(iT')|}{|M_0^L(s_+)|}
		=-\sigma\log\delta-\frac{\pi^2t}{32}
		+O_\le\!\left(\frac{1.54}{x-2}\right).
		\]
		Indeed, with $h=\sigma-id$, the remainder has modulus at most
		\[
		\frac{|h|}{2T'}+\frac{|h|^2}{4(T-1)}
		\le\frac{|h|+|h|^2/2}{x-2}
		<\frac{1.54}{x-2}.
		\]
		Here $|h|^2\le1+\pi^2/256<133/128$ and $|h|<51/50$.
		Conjugation permits the use of $\sigma+iT=\overline{s_+}$.
		
		By \eqref{eq:magnitudes_alpha},
		\[
		\operatorname{Re}(\alpha_L(s_+)^2)
		\ge\frac{\ell^2}{4}-\frac{\pi^2}{16}
		-\frac{2|\ell+i\pi/2|}{x-2}-\frac4{(x-2)^2}.
		\]
		Combining this with $M_t^L=e^{t\alpha_L^2/4}M_0^L$ gives
		\[
		\begin{split}
			\frac{e^{t\pi^2/64}|M_0^L(iT')|}{|M_t^L(s_+)|}
			\le{}&\delta^{-\sigma}
			\exp\!\left(
			-\frac{t\ell^2}{16}
			+\frac{(t/2)|\ell+i\pi/2|+1.54}{x-2}
			+\frac{t}{(x-2)^2}\right)\\
			\le{}&\delta^{-\sigma}
			\exp\!\left(
			-\frac{t\ell^2}{16}
			+\frac{\frac14|\ell+i\pi/2|+1.54}{x-5/2}\right).
		\end{split}
		\]
		The constant terms involving $t\pi^2$ cancel exactly.
		For the last inequality use $t\le1/2$ and
		$A/(x-2)+(1/2)/(x-2)^2\le A/(x-5/2)$ for $A\ge1$.
		
		It remains to bound the two errors in Proposition~\ref{prop:eC}.
		On $0\le\sigma\le1$, both $\sigma^2$ and
		$(9/4)9^\sigma+(5/4)2^{-\sigma}$ are increasing, while the latter has
		maximum $167/8$. Thus \eqref{eq:RtN_epsilon_explicit}, with
		$T=x/2$, gives
		\[
		\widetilde\varepsilon(\sigma,x/2,t,m)\le\mathcal E_{RS}(x,t).
		\]
		We used $\delta b_1=\sqrt{T'/7.49}$ and retained
		the bound $\mathcal B_m^\star$ for the leading coefficient.
		Proposition~\ref{prop:eC} now proves item~(6).
		The estimate \eqref{eq:boundepsilon_explicit} also permits the
		replacement
		\[
		\exp\!\left(\frac{25}{12(x-5/2)}\right)-1
		\le\frac{26}{12(x-5/2)}
		\]
		if a rational upper bound for the first term of
		$\mathcal E_{RS}$ is preferred.
	\end{proof}

\begin{figure}[htbp]
			\centering
			\includegraphics[width=0.8\linewidth]{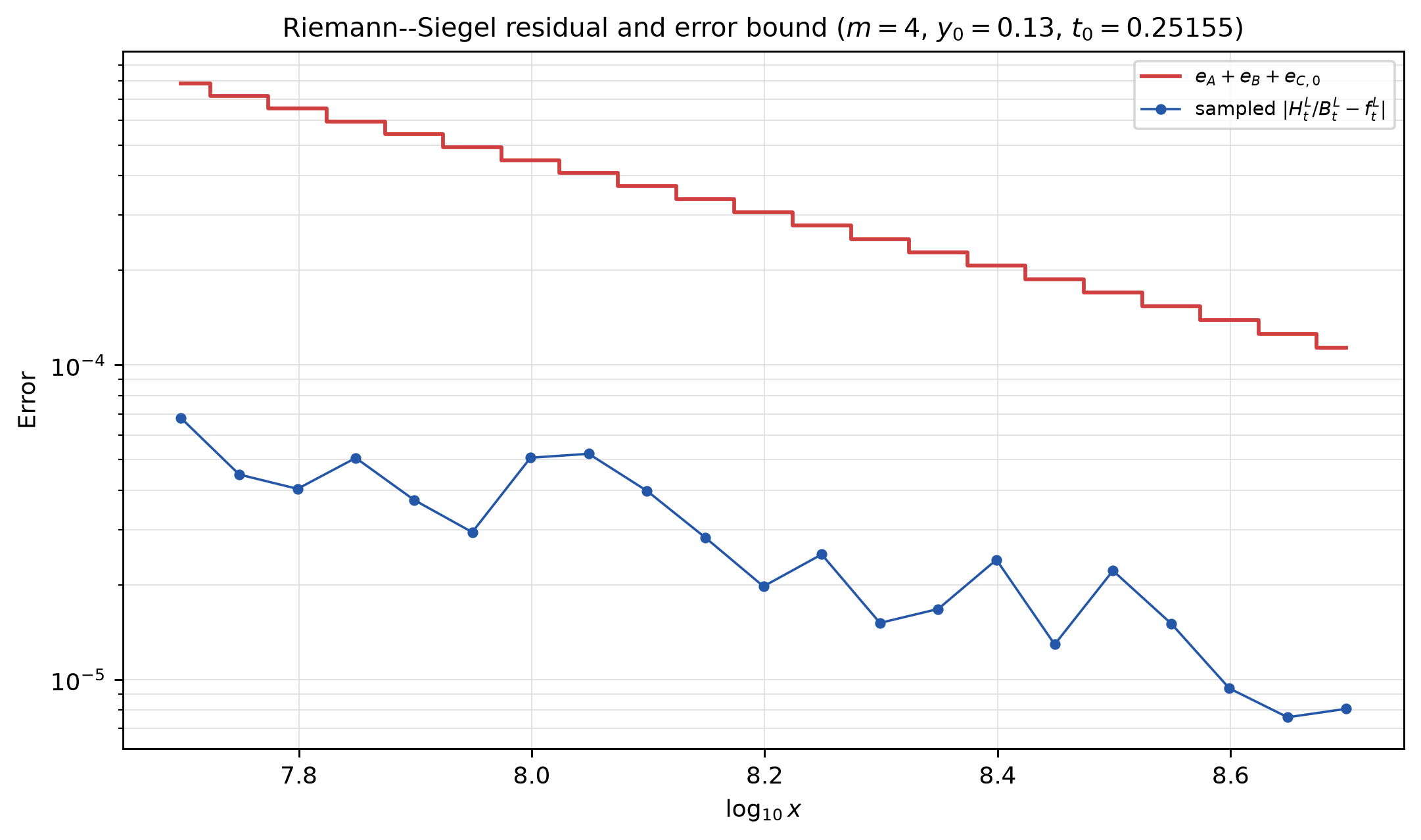}
			\caption{The error bound $e_A+e_B+e_{C,0}$ and the numerical residual $|f_t^L-H_t^L/B_t^L|$ for the Dirichlet beta function ($m=4$, $d=-4$), with $y=y_0=0.13$ and $t=t_0=0.25155$. The residual is sampled at 21 points.}
		\end{figure}

\begin{figure}[htbp]
			\centering
			\includegraphics[width=0.8\linewidth]{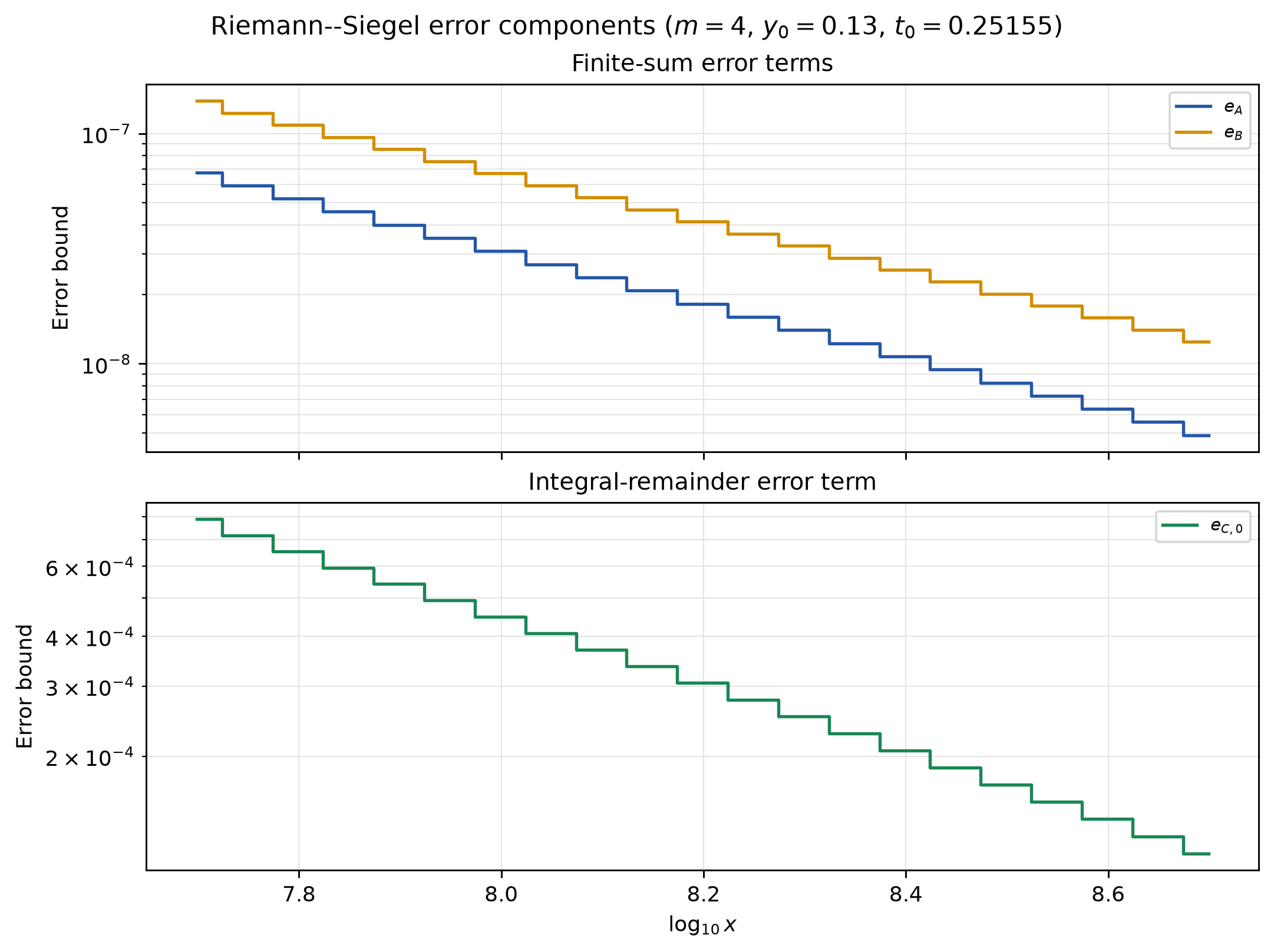}
			\caption{Individual error bounds for $e_A$, $e_B$, and $e_{C,0}$ for the Dirichlet beta function ($m=4$, $d=-4$), with $y=y_0=0.13$ and $t=t_0=0.25155$.}
		\end{figure}
	
	\begin{remark}
Fix $0<y_0\le1$. In the range of
Proposition~\ref{prop:magnitudes}, equation~\eqref{eq:ERS_magnitudes}
gives
\[
 \mathcal E_{RS}(x,t)
 \ll\frac{\sqrt m}{T'}+\frac{\varphi(m)}{\sqrt{T'}},
 \qquad T'=\frac x2+\frac{\pi t}{8}.
\]
The positive correction in the exponential of item~(6) is
uniformly bounded for $x\ge480$ and $3\le m\le400000$.
Consequently, uniformly for $y_0\le y\le1$, we have the upper bounds
\[
\begin{split}
 e_{C,0}(x+iy)
 &\ll e^{-t\ell^2/16}(mT')^{-(1+y_0)/4}
 \left(\sqrt m+\frac{\varphi(m)}{\sqrt{T'}}
       +\frac{\sqrt m}{T'}\right)\\
 &\ll(mT')^{-(1+y_0)/4}
 \left(\sqrt m+\frac{\varphi(m)}{\sqrt{T'}}\right).
\end{split}
\]
\end{remark}

Now we describe a precise dependence on $D$ of the error term $e_{C,0}$.

	\begin{cor}[Dependence on $D$]
		Assume the hypotheses of Proposition~\ref{prop:magnitudes}, except that
		$|y|\le D$ for some $0<D\le1$. Put
		\[
		T'=\frac x2+\frac{\pi t}{8},\qquad
		V=\sqrt{\frac{T'}{7.49}}-\frac32,
		\qquad \ell=\log\!\left(\frac{mx}{4\pi}\right),
		\]
		and
		\[
		S_D:=\frac94\,9^{(1+D)/2}+\frac54\,2^{-(1+D)/2}.
		\]
		Then the bounds for $e_A$ and $e_B$ in Proposition~\ref{prop:magnitudes}
		remain valid, while
		\[
		\begin{aligned}
			e_{C,0}(x+iy)\le{}&
			\left(\frac{mT'}{2\pi}\right)^{-(1+y)/4}
			\exp\!\left(
			-\frac{t\ell^2}{16}
			+\frac{\frac14\sqrt{\ell^2+\pi^2/4}+1.54}{x-5/2}
			\right)\\
			&\times\left[2\mathcal B_m^\star+2\mathcal E_{RS,D}(x,t)\right],
		\end{aligned}
		\]
		where
		\[
		\begin{aligned}
			\mathcal E_{RS,D}(x,t):={}&
			\mathcal B_m^\star
			\left[\exp\!\left(
			\frac{12((1+D)/2)^2+13}{24(x/2-5/4)}\right)-1\right]\\
			&+\frac{c_1(0)S_D}{V}
			\exp\!\left(\frac4{x-5/2}\right).
		\end{aligned}
		\]
		For $D=1$, this is exactly the constant in
		\eqref{eq:ERS_magnitudes}, since $S_1=167/8$.
	\end{cor}
	
	\begin{proof}
		Let us come back to the proof of Proposition~\ref{prop:RtN_rigorous}. In \eqref{eq:RtN_epsilon_explicit}, the two real parts that occur are
		$(1+y)/2$ and $(1-y)/2$. For $|y|\le D$, both lie in
		$[(1-D)/2,(1+D)/2]$. Thus their squares are at most $((1+D)/2)^2$.
		The function
		$F(\sigma)=(9/4)9^\sigma+(5/4)2^{-\sigma}$ is increasing on
		$[0,1]$, since
		$F'(\sigma)=(9/4)(\log 9)9^\sigma-(5/4)(\log 2)2^{-\sigma}>0$.
		Thus $F(\sigma)\leq F((1+D)/2)=S_D$ on this interval.
		Substitution into \eqref{eq:RtN_epsilon_explicit}, followed by the
		prefactor estimate in the proof of Proposition~\ref{prop:magnitudes},
		gives the asserted bound. The estimates for $e_A$ and $e_B$ already
		hold throughout $|y|\le1$, so they also hold on $|y|\le D$.
	\end{proof}

	\section{Fast evaluation of multiple sums}\label{sec:numerics}
	
	The verification of the barrier hypothesis in Lemma~\ref{lemmabox}
	requires evaluating $f_t^L(x+iy)$ at many nearby points.
	Direct evaluation of the two sums in \eqref{ftL_def} takes
	$O(MM_{\mathrm{pts}})$ arithmetic operations, where
	\[
	M=\left\lfloor
	\sqrt{\frac{m}{4\pi}\left(x+\frac{\pi t}{4}\right)}
	\right\rfloor
	\]
	and $M_{\mathrm{pts}}$ is the number of evaluation points.
	We use the Taylor expansion of \cite[\S7]{Poly} to express these
	sums in terms of coefficients that can be computed in advance.
	The character $\chi_m$ enters as a fixed weight; the coefficient
	sums can also be accumulated by residue classes modulo $m$.
	
	Fix $t\ge0$ and a reference point $X+iy_0$.
	In this section $y_0$ denotes the reference ordinate; it need not be
	the lower ordinate of the barrier used in Section~\ref{sec:verification}.
	Consider a group of nearby evaluation points for which the truncation
	length $M\ge1$ is constant. If the evaluation region crosses a change
	of $M$, the points are grouped according to their truncation length.
	
	Recall
	\[
	s_+(x,y)=\frac{1+y-ix}{2},\qquad
	s_-(x,y)=\frac{1-y+ix}{2},
	\]
	and set
	\[
	s_*^L(\pm)=s_\pm(x,y)+\frac t2\alpha_L(s_\pm(x,y)),
	\qquad
	\gamma^L(x+iy)=\frac{M_t^L(s_-(x,y))}{M_t^L(s_+(x,y))}.
	\]
	Then \eqref{ftL_def} reads
	\[
		\ba
		~{} & f_t^L(x+iy)
		=\sum_{n=1}^{M}\frac{\chi_m(n)b_n^{t,L}}{n^{s_*^L(+)}}
		+\gamma^L(x+iy)
		\sum_{n=1}^{M}\frac{\chi_m(n)b_n^{t,L}}{n^{s_*^L(-)}},
		\\
		& b_n^{t,L}=\exp\!\left(\frac t4\log^2n\right).
		\ea
	\]
	In particular, with
	\[
	\kappa^L(x,y)=\frac t2
	\left(\alpha_L(s_-(x,y))-\alpha_L(\overline{s_+(x,y)})\right),
	\]
	we have
	\[
	s_*^L(-)=\overline{s_*^L(+)}-y+\kappa^L(x,y),
	\qquad
	n^{-s_*^L(-)}
	=\frac{n^y}{n^{\overline{s_*^L(+)}+\kappa^L(x,y)}}.
	\]
	To make the coefficients independent of the evaluation point, put
	\[
	s_+^0=s_+(X,y_0),\qquad s_-^0=s_-(X,y_0),
	\]
	and define the exponent shifts
	\[
	\begin{split}
		b(x,y)&=s_+^0-s_*^L(+),\\
		a(x,y)&=s_-^0-s_*^L(-)
		=s_-^0+y-\overline{s_*^L(+)}-\kappa^L(x,y).
	\end{split}
	\]
	Here $a(x,y)$ is an exponent shift, distinct from the constant
	character parity in the definition of $M_t^L$.
	Thus
	\[
	f_t^L(x+iy)
	=\sum_{n=1}^{M}\frac{n^b\chi_m(n)b_n^{t,L}}{n^{s_+^0}}
	+\gamma^L(x+iy)
	\sum_{n=1}^{M}\frac{n^a\chi_m(n)b_n^{t,L}}{n^{s_-^0}}.
	\]
	Choose the integer centre $n_0=\left\lceil\frac M2\right\rceil$, $h=n-n_0$. Then $1-n_0\le h\le M-n_0$ and $|h/n_0|\le1$.
	Writing
	\[
	L_h=\log\!\left(1+\frac h{n_0}\right)
	=\log\!\left(\frac{n_0+h}{n_0}\right),
	\]
	we obtain
	\[
	f_t^L(x+iy)=B(b)+\gamma^L(x+iy)A(a),
	\]
	where
	\[
	\begin{split}
		B(b)&=\sum_{h=1-n_0}^{M-n_0}
		\frac{(n_0+h)^b\chi_m(n_0+h)b_{n_0+h}^{t,L}}
		{(n_0+h)^{s_+^0}},\\
		A(a)&=\sum_{h=1-n_0}^{M-n_0}
		\frac{(n_0+h)^a\chi_m(n_0+h)b_{n_0+h}^{t,L}}
		{(n_0+h)^{s_-^0}}.
	\end{split}
	\]
	Since $\log n=\log n_0+L_h$,
	\begin{equation}\label{eq:factor_first}
		n^b b_n^{t,L}
		=n_0^{b+\frac t4\log n_0}
		\exp\!\left(\frac t4L_h^2\right)
		\exp\!\left(\left(b+\frac t2\log n_0\right)L_h\right).
	\end{equation}
	Expanding the two exponentials gives
	\[
	n^b b_n^{t,L}
	=n_0^{b+\frac t4\log n_0}
	\sum_{i=0}^{\infty}\sum_{j=0}^{\infty}
	\frac{(tL_h^2/4)^i}{i!}\,
	\frac{(b+\frac t2\log n_0)^jL_h^j}{j!}.
	\]
	Both exponential series converge absolutely. Since the sum over $h$
	is finite, they may be interchanged with that sum. Therefore
	\[
	B(b)=n_0^{b+\frac t4\log n_0}
	\sum_{i=0}^{\infty}\sum_{j=0}^{\infty}
	B_{i,j}^L\frac{(b+\frac t2\log n_0)^j}{j!},
	\]
	with
	\[
	B_{i,j}^L:=
	\sum_{h=1-n_0}^{M-n_0}
	\frac{\chi_m(n_0+h)}{(n_0+h)^{s_+^0}}
	\frac{(tL_h^2/4)^i}{i!}L_h^j.
	\]
	Similarly,
	\[
	\ba
	& A(a)=n_0^{a+\frac t4\log n_0}
	\sum_{i=0}^{\infty}\sum_{j=0}^{\infty}
	A_{i,j}^L\frac{(a+\frac t2\log n_0)^j}{j!},
	\\
	& A_{i,j}^L:=
	\sum_{h=1-n_0}^{M-n_0}
	\frac{\chi_m(n_0+h)}{(n_0+h)^{s_-^0}}
	\frac{(tL_h^2/4)^i}{i!}L_h^j.
	\ea
	\]
	These coefficients depend on $t,M$, the character $\chi_m$ and the reference point, but
	not on $(x,y)$.
	
	For an integer $E\ge1$, define the truncated evaluations
	\[
	\begin{split}
		B_E(b)&=n_0^{b+\frac t4\log n_0}
		\sum_{i=0}^{E-1}\sum_{j=0}^{E-1}
		B_{i,j}^L\frac{(b+\frac t2\log n_0)^j}{j!},\\
		A_E(a)&=n_0^{a+\frac t4\log n_0}
		\sum_{i=0}^{E-1}\sum_{j=0}^{E-1}
		A_{i,j}^L\frac{(a+\frac t2\log n_0)^j}{j!}.
	\end{split}
	\]
	Computing both tables takes $O(ME^2)$ arithmetic operations.
	Evaluating these expressions at $M_{\mathrm{pts}}$ points then takes
	$O(M_{\mathrm{pts}}E^2)$ operations, giving a total cost
	$O((M+M_{\mathrm{pts}})E^2)$ for each fixed $t$ and $M$.
	
	The cutoff is chosen to meet the required error tolerance.
	An explicit truncation bound follows by setting
	\[
	\rho_M=\max_{1\le n\le M}\left|\log\frac n{n_0}\right|,
	\qquad \lambda_M=\frac t4\rho_M^2,\qquad
	\mu_M(u)=\left|u+\frac t2\log n_0\right|\rho_M
	\]
	and
	\[
	\mathcal R_E(u;s_0):=
	n_0^{\operatorname{Re}u+\frac t4\log n_0}
	e^{\lambda_M+\mu_M(u)}
	\frac{\lambda_M^E+\mu_M(u)^E}{E!}
	\sum_{n=1}^{M}|\chi_m(n)|n^{-\operatorname{Re}s_0}.
	\]
	Indeed,
	\[
	\left|e^z-\sum_{j=0}^{E-1}\frac{z^j}{j!}\right|
	\le e^{|z|}\frac{|z|^E}{E!}.
	\]
	Apply this to the two factors in \eqref{eq:factor_first}, using
	the bound $e^{|z|}$ for either a full exponential or its truncated
	series. It follows that
	\[
	|B(b)-B_E(b)|\le\mathcal R_E(b;s_+^0),\qquad
	|A(a)-A_E(a)|\le\mathcal R_E(a;s_-^0).
	\]
	Consequently,
	\[
	\left|f_t^L(x+iy)
	-\bigl(B_E(b)+\gamma^L(x+iy)A_E(a)\bigr)\right|
	\le\mathcal R_E(b;s_+^0)
	+|\gamma^L(x+iy)|\mathcal R_E(a;s_-^0).
	\]
	As in \cite[\S7]{Poly}, we take $E=50$ in the range considered below.
	\begin{remark}[Taylor truncation]
		\label{rem:fast_E50}
		We record a range in which the choice $E=50$ gives a sufficiently small
		Taylor remainder. Suppose that
		\[
\begin{gathered}
3\le m\le400000,\qquad x\ge480,\qquad 0\le t\le\frac12,\\
0\le y,y_0\le D\le1,\qquad |x-X|\le1,\qquad 10\le M\le5000.
\end{gathered}
\]
		Here $M$ is fixed when forming each coefficient table. Then
		\[
		\left|f_t^L(x+iy)
		-\bigl(B_{50}(b)+\gamma^L(x+iy)A_{50}(a)\bigr)\right|
		<3\cdot10^{-10}.
		\]
		If $t\le1/5$, the same argument gives the stronger bound $10^{-14}$.
		
		In particular, $M\le5000$ for the heights $x\le2h_m+1$, with $h_m$
		as in \eqref{hm_def}. Indeed, for $m\le400000$ and $t\le1/2$,
		\[
		m\left(x+\frac{\pi t}{4}\right)
		\le\max\{2\cdot10^8,\,400m+1.5\cdot10^8\}
		+m(1+\pi/8)
		<3.106\cdot10^8<4\pi(5000)^2.
		\]
		For the error estimate, put $n_0=\lceil M/2\rceil$ and
		$\delta=\sqrt{m(x+\pi t/4)/(4\pi)}$. Since $10\le M\le5000$, we have $5\le n_0\le2500$ and
		$M/n_0\le2$. Hence
		\[
		\rho_M=\max\{\log n_0,\log(M/n_0)\}=\log n_0<8.
		\]
		Writing $\delta_0=\sqrt{mx/(4\pi)}$, the inequalities
		$M\le\delta<M+1$ and
		$\delta=\delta_0\sqrt{1+\pi t/(4x)}$ give
		\[
		\frac32<\frac{\delta_0}{n_0}<\frac{11}{5}.
		\]
		By \eqref{eq:magnitudes_alpha} and conjugation,
		\[
		\begin{split}
			|\alpha_L(s_\pm)-\log n_0|
			&\le
			\sqrt{\log^2(11/5)+\pi^2/16}+\frac2{x-2}<\frac98,\\
			\operatorname{Re}\alpha_L(s_\pm)-\log n_0
			&\ge\log(3/2)-\frac2{x-2}>0.
		\end{split}
		\]
		The definitions of the shifts $a,b$ therefore imply
		\[
		\ba
		~{} & \left|u+\frac t2\log n_0\right|
		\le\frac{\sqrt{(x-X)^2+(y-y_0)^2}}2+\frac{9t}{16}
		\le\frac{\sqrt2}{2}+\frac9{32}<1,
		\\
		& \operatorname{Re}\left(u+\frac t2\log n_0\right)\le\frac12.
		\ea
		\]
		Thus $\lambda_M\le8$ and $\mu_M(u)\le8$. Since
		$\rho_M=\log n_0$, the factor involving $n_0$ in $\mathcal R_E$
		combines with $e^{\lambda_M}$ as
		\[
		n_0^{\operatorname{Re}u+\frac t4\log n_0}e^{\lambda_M}
		=n_0^{\operatorname{Re}(u+\frac t2\log n_0)}
		\le\sqrt{n_0}.
		\]
		The weight sums in $\mathcal R_E$ are at most $M$, because
		$\operatorname{Re}s_\pm^0\ge0$ and $|\chi_m(n)|\le1$.
		Proposition~\ref{prop:magnitudes} gives $|\gamma^L|<2$.
		The remainder bound now yields
		\[
		3M\sqrt{n_0}\,e^8\,\frac{2\cdot8^{50}}{50!}
		<3\cdot10^{-10}.
		\]
		
		When $t\le1/5$, the shifted-exponent bound improves to
		\[
		\left|u+\frac t2\log n_0\right|
		\le\frac{\sqrt2}{2}+\frac9{80}<\frac56.
		\]
		Thus $\lambda_M\le16/5$ and $\mu_M(u)\le20/3$, and the same argument gives
		\[
		3M\sqrt{n_0}\,e^{20/3}
		\frac{2(20/3)^{50}}{50!}<10^{-14},
		\]
		using $e^{20/3}<800$ and $(20/3)^{50}<1.6\cdot10^{41}$.
		For a fixed-$M$ box, the same estimates apply provided
		$3/2<\delta_0/n_0<11/5$ throughout the box. Outside this range,
		we choose $E$ using the general bound $\mathcal R_E$.
		
		We use this full-strip estimate for every value of $D$. A smaller cap can improve the Taylor remainder, but this refinement is not needed for the stated tolerance.

	\end{remark}
	
	For computations at several times, the coefficient tables can also be
	made independent of $t$. Define
	\[
	\widetilde B_{i,j}^L:=
	\sum_{h=1-n_0}^{M-n_0}
	\frac{\chi_m(n_0+h)}{(n_0+h)^{s_+^0}}
	\frac{(L_h^2/4)^i}{i!}\frac{L_h^j}{j!},
	\]
	and define $\widetilde A_{i,j}^L$ with $s_-^0$ in place of $s_+^0$.
	Then
	\[
	B_E(b)=n_0^{b+\frac t4\log n_0}
	\sum_{j=0}^{E-1}
	\left(\sum_{i=0}^{E-1}t^i\widetilde B_{i,j}^L\right)
	\left(b+\frac t2\log n_0\right)^j,
	\]
	with the analogous expression for $A_E(a)$. These are the same finite
	Taylor sums, so the remainder estimates are unchanged. At reference
	ordinate zero, $s_-^0=\overline{s_+^0}$ and the character is real; thus
	$\widetilde A_{i,j}^L=\overline{\widetilde B_{i,j}^L}$.
	For the barrier computations one may take reference abscissa
	$X_{\rm barrier}+1/2$ and reference ordinate zero. Every point in
	$[X_{\rm barrier},X_{\rm barrier}+1]\times[0,1]$ then satisfies the
	reference-point conditions of Remark~\ref{rem:fast_E50}.
	
	As in \cite[\S7, footnote~5]{Poly}, a coarse truncation bound is
	sufficient when it is well below the margins in the verification.
	For each computed value of $f_{t,M}^L$ used below, we reserve the
	fixed absolute tolerance and require
	\begin{align}
		&\eta_{\mathrm{num}}:=10^{-8},
\notag\\
		&|\widehat f_{t,M}^L(z)-f_{t,M}^L(z)|
		\leq\eta_{\mathrm{num}}.
		\label{eq:numerical_accuracy_requirement}
	\end{align}
	In the range of Remark~\ref{rem:fast_E50}, the Taylor truncation
	contributes less than $3\cdot10^{-10}$, leaving $9.7\cdot10^{-9}$
	for the remaining evaluation errors, including rounding and the
	calculation of the coefficients and prefactors. Outside that range,
	$E$ is chosen using $\mathcal R_E$ to meet the same tolerance.
	To verify this allowance, the finite Taylor sums, their coefficients
	and their prefactors are evaluated with interval arithmetic. If the
	resulting enclosure lies in a disk of radius at most $9.7\cdot10^{-9}$
	about the reported value, adding the Taylor remainder proves
	\eqref{eq:numerical_accuracy_requirement}. The working precision is
	increased if this test fails.
	Alternatively, the subsequent inequalities may be checked directly
	with complex enclosures that already include the Taylor remainder.
	In that case their modulus lower bounds and argument enclosures include
	the evaluation error; any further subtraction of $\eta_{\mathrm{num}}$
	is conservative. Comparisons with direct sums or higher-precision
	calculations provide useful checks, but are not used to bound rounding
	errors. The verification will record the precision, the Taylor remainder
	bound, and the resulting interval inequalities.
	
	\section{Verification of the zero-free regions}
	\label{sec:verification}
	
	In this section we follow the strategy of Polymath~\cite[\S 8]{Poly}, with the Riemann--Siegel truncation length $M$ depending on the conductor. The effective approximation from Section~\ref{sec:4}, together with Platt's GRH verification~\cite{platt}, reduces the proof of $\Lambda_L\le t_0+y_0^2/2$ to the three zero-free regions in Lemma~\ref{lemmabox}.
	
	Throughout the section, $\chi_m$ is a fixed primitive real character with $3\le m\le400000$. Let $0<D\le1$ be an admissible strip bound from Section~\ref{DynSection}, so every zero of $H_0^L$ satisfies $|\operatorname{Im}z|\le D$. We choose $t_0>0$ and $y_0>0$ with $y_0^2+2t_0<D^2$; if only $\Lambda_L\le D^2/2$ is needed, this follows directly from de Bruijn's theorem.
	The cases used below are $D\in\{3/4,1\}$, under their respective strip hypotheses.
	The geometry of the regions depends on $D$. The $E=50$ Taylor remainder is bounded uniformly for $0\leq y\leq1$ as in Remark~\ref{rem:fast_E50}. The remaining Riemann–Siegel majorants are kept in their pointwise form and evaluated using the endpoint monotonicities below; this also permits a full case $D=1$ to be reused on a smaller cap whenever the same pair $(t_0,y_0)$ is used.
	
	\subsection{Selection of parameters}
	
	For the conductor $m$, Platt's verification height~\cite[, Theorem~10.1]{platt} is
	\begin{equation}\label{hm_def}
		h_m=\max\left\{\frac{10^8}{m},\,200+\frac{A(m)10^7}{m}\right\},
		\qquad A(m)=\begin{cases}7.5,&m\text{ even},\\3.75,&m\text{ odd}.\end{cases}
	\end{equation}
	Here $A(m)$ depends on the parity of the conductor, whereas
	$a=(1-\chi_m(-1))/2$ denotes the parity of the character.
	For $m\le400$, both choices of $A(m)$ give $h_m=10^8/m$.
	We use the search width $\Delta_m:=100+4500/m$.
	We search for the barrier location in
	\begin{equation}\label{Xm_search}
		X\in[2h_m-\Delta_m,2h_m].
	\end{equation}
	
	The search is guided by the truncated Euler product
	\[
\operatorname{ep}_r^L(X):=\left|\prod_{j=1}^{r}\left(1-\chi_m(p_j)p_j^{-1+iX/2}\right)^{-1}\right|.
\]
	At $t=0$ and $y=1$, $s_+(x,1)=1-ix/2$, so this product gives a heuristic for the first Dirichlet sum in $f_0^L(x+i)$. For $p\nmid m$, its phases are favourable when
	\[
	\chi_m(p)e^{ix\log p/2}\approx1,
	\qquad\text{equivalently}\qquad
	\left\|\frac{x\log p}{4\pi}-\frac{1-\chi_m(p)}4\right\|_{\mathbb Z}\approx0.
	\]
	We write $X_m^\star$ for a candidate in the search interval with a large value of this product and set $X:=X_m^\star$. The proof only requires the later verification at this $X$.
	
	Using both terms in the definition of $h_m$ gives
	\[
		X\ge x_m:=2h_m-\Delta_m
		=\max\left\{\frac{199995500}{m}-100,\,300+\frac{2A(m)10^7-4500}{m}\right\}.
	\]
	For each fixed value of $A(m)$, the expression on the right decreases
	with $m$. Evaluating the two comparison functions at $m=400000$
	gives $674.98875$ and $487.48875$, respectively. Thus
	\[
		X\ge x_m\ge487.48875>480.
	\]
	
	For the beta function, $m=4$, $h_4=25{,}000{,}000$ and $\Delta_4=1225$, so the search interval is $[49{,}998{,}775,50{,}000{,}000]$. We use the candidate $X_4^\star=49{,}999{,}509.0746$ selected by the truncated Euler product with $r=20$. Figure~\ref{fig:eulerLm4} compares the products for $r=3,5,7,20$ over the search interval and near this candidate.
	\begin{figure}[!ht]
		\centering
		\includegraphics[width=0.8\linewidth]{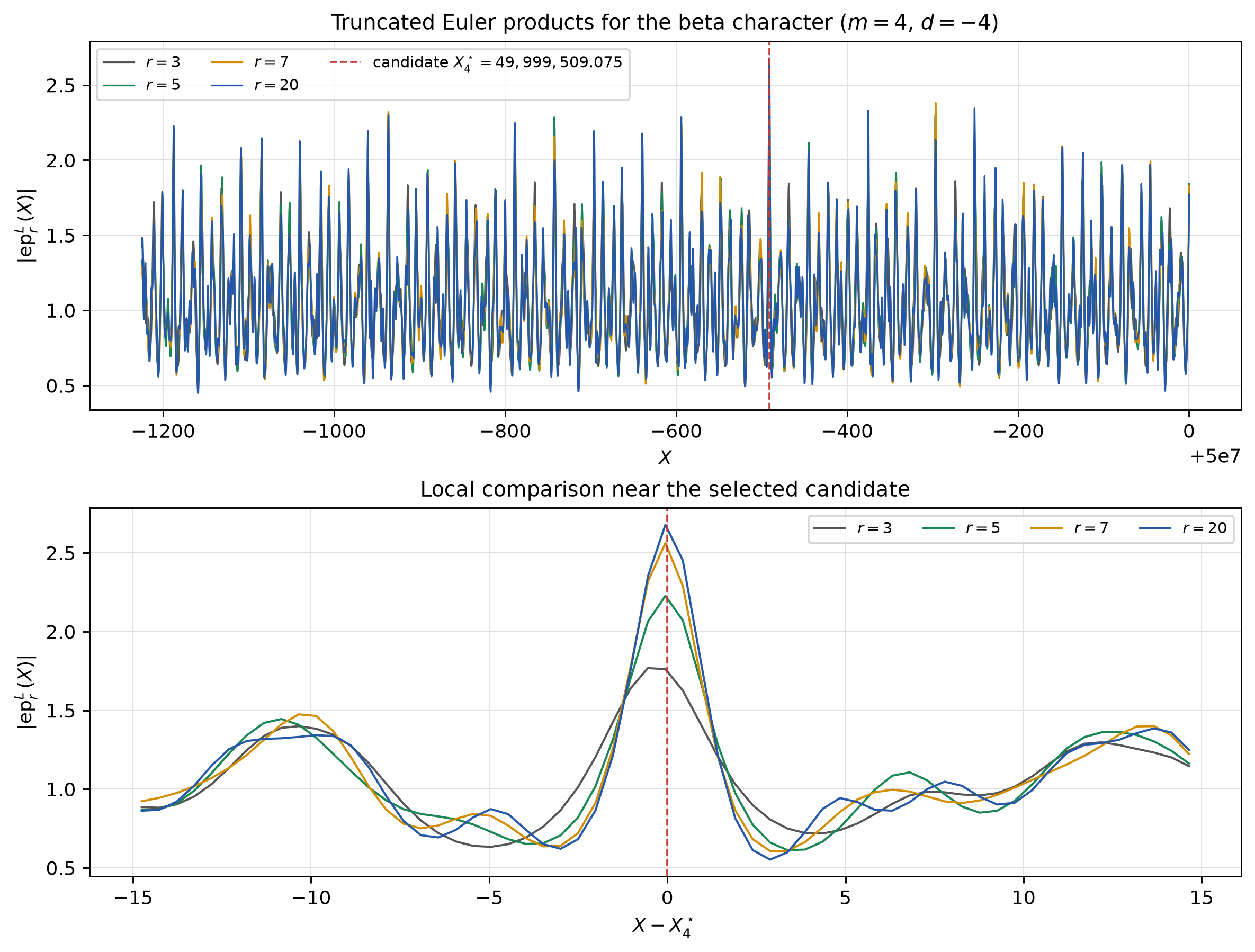}
		\caption{Truncated Euler products for the beta function ($m=4$, $d=-4$), with $r=3,5,7,20$, on the search window (top) and near the candidate $X_4^\star$ selected with $r=20$ (bottom).}
		\label{fig:eulerLm4}
	\end{figure}
	
	\subsection{The initial rectangle}
Hypothesis (1) of Lemma~\ref{lemmabox} requires
$H_0^L(x+iy)\ne0$ for $0\le x\le X$ and
$\sqrt{y_0^2+2t_0}\le y\le D$.
Under $s=(1-y+ix)/2$, a zero in this region would give a
zero of $\xi_L(s)$ satisfying
\[
\begin{aligned}
 \operatorname{Re}s
 &=\frac{1-y}{2}
 \le\frac{1-\sqrt{y_0^2+2t_0}}2<\frac12,\\
 |\operatorname{Im}s|
 &=\frac x2\le\frac X2\le h_m.
\end{aligned}
\]
Platt's verification of GRH up to height $h_m$ excludes such a
zero. It remains to verify hypotheses (2) and (3) of
Lemma~\ref{lemmabox}.

\subsection{The truncation parameters}
	
	For \(x\geq2h_m-\Delta_m\geq480\) and \(0\leq t\leq t_0\), define
\begin{align*}
	 M(x,t)&:=\left\lfloor
	 \sqrt{\frac{m}{4\pi}\left(x+\frac{\pi t}{4}\right)}
	 \right\rfloor,
	 \\
	 \delta(x,t)&:=
	 \sqrt{\frac{m}{4\pi}\left(x+\frac{\pi t}{4}\right)}.
\end{align*}
	Thus \(M(x,t)=\lfloor\delta(x,t)\rfloor\).
	For every integer \(M\geq1\), define
	\[
		x_M(t)
		:=
		\frac{4\pi M^2}{m}-\frac{\pi t}{4}.
	\]
	Then,
	\[
		M(x,t)=M
		\quad\Longleftrightarrow\quad
		x_M(t)\leq x<x_{M+1}(t).
	\]
	The function \(\delta(x,t)\) is increasing in both variables. Hence on
	the enlarged barrier
	\[
	X\leq x\leq X+1,
	\qquad
	0\leq t\leq t_0,
	\]
	we have
	\[
		\delta_0
		\leq
		\delta(x,t)
		\leq
		\delta_1,
	\]
	where
	\[
		\delta_0
		:=
		\sqrt{\frac{mX}{4\pi}},
		\qquad
		\delta_1
		:=
		\sqrt{
			\frac{m}{4\pi}
			\left(
			X+1+\frac{\pi t_0}{4}
			\right)
		}.
	\]
	We define
	\[
		M_0:=\lfloor\delta_0\rfloor.
	\]
	Thus
	\[
	M_0\leq M(x,t)\leq \lfloor\delta_1\rfloor
	\]
	throughout the enlarged barrier.
	
	The variation of $\delta$ on this box is
	\[
		\delta_1-\delta_0
		=\sqrt{\frac{m}{4\pi}}\,
		\frac{1+\pi t_0/4}{\sqrt{X+1+\pi t_0/4}+\sqrt X}
		\leq\frac{(1+\pi t_0/4)\sqrt m}{4\sqrt{\pi X}}
		\leq\frac{(1+\pi/8)\sqrt m}{4\sqrt{\pi x_m}}.
	\]
	For each fixed value of $A(m)$, the last expression increases with
	$m$. Evaluating the two comparison functions at $m=400000$ gives,
	respectively, values below $4.79$ and $5.64$. Hence
	\[
		\delta_1-\delta_0<\begin{cases}4.79,&m\text{ even},\\5.64,&m\text{ odd}.\end{cases}
	\]
	We put
	\[
		J_m:=\lfloor\delta_1\rfloor-\lfloor\delta_0\rfloor.
	\]
	Then $J_m\le5$ for even $m$ and $J_m\le6$ for odd $m$, so
	\[
		M_0\leq M(x,t)\leq M_0+J_m\leq M_0+6.
	\]
	Since the moving barrier has width $\Delta_D\le1$, this same range contains all its active truncation lengths; empty stripes are omitted. We apply the derivative bounds separately on each fixed-$M$ stripe $\mathcal R_{a,M}$.
	
	\begin{remark}
		For the beta function, $m=4$, with $t_0=0.25155$ and
		$X_4^\star=49{,}999{,}509.0746$, we obtain
		$\delta_0=3989.4032$ and $\delta_1=3989.4033$. Thus
		$M_0=\lfloor\delta_0\rfloor=3989$ and
		$J_4=\lfloor\delta_1\rfloor-\lfloor\delta_0\rfloor=0$;
		accordingly, $M(x,t)=3989$ throughout the enlarged barrier.
	\end{remark}

	\medskip
	
	The integer \(M_0\) is determined by the selected barrier location.
	As in \cite[\S 8.2]{Poly}, we now introduce a second truncation parameter
	\(M_1\), whose role is to separate the finite region requiring numerical
	verification from the far-right region that can be treated analytically.
	
	Choose an integer
	\[
		M_1>M_0+J_m
	\]
	and define
	\[
		X_1
		:=
		x_{M_1}(t_0)
		=
		\frac{4\pi M_1^2}{m}
		-\frac{\pi t_0}{4}.
	\]
	Then
	\begin{equation}\label{M1_equivalence}
		M(x,t_0)\geq M_1
		\quad\Longleftrightarrow\quad
		x\geq X_1.
	\end{equation}
	The parameter \(M_1\) is chosen sufficiently large that the analytic
	estimate in region (c) below holds for every
	\(M(x,t_0)\geq M_1\). Thus \(M_1\) is an auxiliary cutoff and is not part
	of the final definition of \(\Lambda_*(m,\chi_m;D)\).
	
	\subsection{The verification regions}
	
	Assume $t_0>0$, $y_0>0$ and $y_0^2+2t_0<D^2$. Put
	\[
		\Delta_D=\sqrt{D^2-y_0^2-2t_0},\qquad
		Y(t)=\sqrt{y_0^2+2(t_0-t)},\qquad U_D(t)=\sqrt{D^2-2t}.
	\]
	Region (a) is the moving barrier in Lemma~\ref{lemmabox}.
	The larger box $[X,X+1]\times[y_0,D]\times[0,t_0]$ is used only
	for convenient bounds on truncation lengths, derivatives and selected
	error terms. In particular, $M_0$ and $J_m$ computed on that box
	contain every active truncation length of the moving barrier.
	The non-vanishing test for (a) is performed on the smaller region below.
	
	\paragraph{Region (a).}
	For each
	\[
	M\in\left\{M_0,\ldots,M_0+J_m\right\}
	\]
	define
	\[
		\mathcal R_{a,M}^L=\left\{(x,y,t):
		X\leq x\leq X+\Delta_D,\quad
		Y(t)\leq y\leq U_D(t),\quad0\leq t\leq t_0,\quad M(x,t)=M\right\}.
	\]
	Then,
	\[
	\mathcal R_a^L
	=
	\bigcup_{M=M_0}^{M_0+J_m}
	\mathcal R_{a,M}^L.
	\]
	On every \(\mathcal R_{a,M}^L\), the truncation length is constant. When \(J_m=0\), this reduces to the single region
	\[
	M(x,t)=M_0,
	\]
	as in the Polymath computation.
	
	\medskip
	
	\paragraph{Region (b).}
	Define
	\[
		\mathcal R_b^L
		:=
		\left\{
		(x,y):
		x\geq X+\Delta_D,
		y_0\leq y\leq D,
		M_0\leq M(x,t_0)<M_1
		\right\}.
	\]
	Equivalently,
	\[
	\mathcal R_b^L
	=
	\bigcup_{M=M_0}^{M_1-1}
	\mathcal R_{b,M}^L,
	\]
	where
	\[
		\mathcal R_{b,M}^L
		:=
		\left\{
		(x,y):
		\max\left\{X+\Delta_D,x_{M}(t_0)\right\}\leq x<x_{M+1}(t_0),
		y_0\leq y\leq D
		\right\}.
	\]
	On each \(\mathcal R_{b,M}^L\), the truncation length is constant.
	
	\medskip
	
	\paragraph{Region (c).}
	Define
	\[
		\mathcal R_c^L
		:=
		\left\{
		(x,y):
		x\geq X_1,
		y_0\leq y\leq D
		\right\}.
	\]
	By \eqref{M1_equivalence}, this is precisely the region in which
	\[
	M(x,t_0)\geq M_1.
	\]
	The regions (b) and (c) cover the enlarged right tail:
	\[
		\left\{
		x\geq X+\Delta_D,
		y_0\leq y\leq D
		\right\}
		=
		\mathcal R_b^L\cup\mathcal R_c^L.
	\]

	\subsection{Bounds for the errors}
	
	Set
	\[
		\mathcal E_t^L(x,y):=e_A(x+iy)+e_B(x+iy)+e_{C,0}(x+iy).
	\]
	By Theorem~\ref{EffectiveRSa},
	\[
		\frac{H_t^L(x+iy)}{B_t^L(x+iy)}
		=f_t^L(x+iy)+O_\leq\!\left(\mathcal E_t^L(x,y)\right).
	\]
	Since $B_t^L\neq0$ in the region under consideration, it is enough to prove
	\[
		|f_t^L(x+iy)|>\mathcal E_t^L(x,y).
	\]
	Throughout this subsection, assume
	\[
	3\leq m\leq400000,\qquad x\geq480,\qquad
	0\leq t\leq\frac12,\qquad 0\leq y\leq D\leq1.
	\]
	The strip cap $D$ is fixed and suppressed in the notation for the
	majorants below. We keep the dependence on $x,y,t$. Put
	\[
	q:=\frac{mx}{4\pi},\qquad \ell:=\log q,\qquad
	\delta_+(x,t):=\sqrt{q\left(1+\frac{\pi t}{4x}\right)}.
	\]
	Then $M\leq\delta_+(x,t)\leq q$. Proposition~\ref{prop:magnitudes}(1)
	and (3) gives
	\[
		|\gamma^L(x+iy)|\leq e^{a_my}q^{-y/2},
		\qquad a_m:=0.002+\frac{\log m}{3800},
	\]
	and
	\[
		|\kappa^L|\leq\frac{ty}{2(x-2)}.
	\]
	Consequently,
	\[
		|\gamma^L(x+iy)|M^{y+|\kappa^L|}\leq G_m(x,y,t),
	\]
	where
	\[
	G_m(x,y,t):=
	\exp\!\left(
	a_my+\frac y2\log\!\left(1+\frac{\pi t}{4x}\right)
	+\frac{ty}{2(x-2)}\log\delta_+(x,t)
	\right).
	\]
	The exponent defining $G_m$ is a nonnegative multiple of $y$, so
	$G_m(x,y,t)\leq G_m(x,D,t)$. For $n\geq1$, define
	\[
	L_n:=\log(q/n^2),\qquad
	H_n(x,t):=\exp\!\left(-\frac t4\log n\,(\ell-\log n)\right),
	\]
	\[
	\eta_n(x,t):=
	\exp\!\left(
	\frac{\frac{t^2}{16}L_n^2+\frac{\log m}{2250}+1.3}{x-5/2}
	\right)-1,
	\qquad
	\theta_y(x,t):=\frac{1-y}{2}+\frac{(3+y)t}{2x^2}.
	\]
	Proposition~\ref{prop:magnitudes}(2) gives
	\[
	\frac{b_n^{t,L}}{n^{\operatorname{Re}s_*^L(+)}}
	\leq n^{-1+\theta_y(x,t)}H_n(x,t).
	\]
	The proof of items~(4)--(5) bounds each error factor by $\eta_n(x,t)$.
	Since $s_*^L(-)-\overline{s_*^L(+)}=-y+\kappa^L$, we have
	$|n^{-s_*^L(-)}|\leq n^{y+|\kappa^L|}n^{-\operatorname{Re}s_*^L(+)}$.
	Keeping this factor inside the sum gives
	\[
		\begin{aligned}
			e_A(x+iy)+e_B(x+iy)\leq{}&
			\sum_{n=1}^{M}n^{-1+\theta_y(x,t)}H_n(x,t)\eta_n(x,t)\\
			&\times\left[
			1+e^{a_my}\left(\frac n{\sqrt q}\right)^y
			n^{ty/(2(x-2))}\right].
		\end{aligned}
	\]
	For $n\leq M$, the expression in square brackets is at most
	$1+G_m(x,D,t)$. Moreover, $H_n\leq1$ and $|L_n|\leq\ell$.
	Set
	\[
		d_m(x,t):=
		\frac{\frac{t^2}{16}\ell^2+\frac{\log m}{2250}+1.3}{x-5/2}.
	\]
	Since $0\leq\theta_y<1$, the integral test yields
	\[
		e_A(x+iy)+e_B(x+iy)\leq\varepsilon_{AB,t}^L(x,y),
	\]
	where
	\begin{equation}\label{eq:epsilon_AB_local}
		\begin{aligned}
			\varepsilon_{AB,t}^L(x,y):={}&
			\bigl(e^{d_m(x,t)}-1\bigr)\bigl(1+G_m(x,D,t)\bigr)\\
			&\times\left[
			1+\frac{\delta_+(x,t)^{\theta_y(x,t)}-1}{\theta_y(x,t)}
			\right].
		\end{aligned}
	\end{equation}
	At $\theta_y=0$ (that is, $t=0$ and $y=1$), the last factor is
	interpreted as $1+\log\delta_+$, the continuous limit of the displayed quotient.
	Indeed,
	\[
	\sum_{n=1}^{M}n^{-1+\theta_y}
	\leq1+\int_1^{\delta_+}u^{-1+\theta_y}\,du.
	\]
	
	For fixed $m,t,D$, the bound \eqref{eq:epsilon_AB_local} is
	non-increasing in $y$ and decreasing in $x$. The assertion for $y$
	follows from $\partial_y\theta_y=-1/2+t/(2x^2)<0$.
	For the assertion concerning $x$, put $a=t^2/16$ and
	$c=\log m/2250+1.3$. With $d=d_m(x,t)$,
	\[
	\frac{d}{d\log x}\log(e^d-1)
	=\frac{d}{1-e^{-d}}
	\left(\frac{2a\ell}{a\ell^2+c}-\frac{x}{x-5/2}\right)<-0.89,
	\]
	because $2a\ell/(a\ell^2+c)\leq\sqrt{a/c}<0.110$.
	Write $I(x)=1+\int_1^{\delta_+}u^{-1+\theta_y}\,du$.
	Since $\partial_x\theta_y\leq0$ and
	$\delta_+^{\theta_y}\leq I(x)$ for $0\leq\theta_y\leq1$,
	\[
	\frac{xI'(x)}{I(x)}
	\leq\frac{\delta_+^{\theta_y}}{I(x)}
	\frac{x}{2(x+\pi t/4)}\leq\frac12.
	\]
	Finally, $G_m(x,D,t)$ is non-increasing in $x$. In particular,
	$\log\delta_+>2$ and
	\[
	\frac{d}{dx}\frac{\log\delta_+}{x-2}
	=\frac{(x-2)/(2(x+\pi t/4))-\log\delta_+}{(x-2)^2}<0.
	\]
	These estimates prove the claimed monotonicity. At fixed $m,x,y,D$,
	$\varepsilon_{AB,t}^L(x,y)$ is non-decreasing in $t$, since each
	factor in \eqref{eq:epsilon_AB_local} is non-decreasing.
	
	For the remainder term we keep the two values of $\sigma$ separately.
	Put
	\[
	T':=\frac x2+\frac{\pi t}{8},\qquad
	\mathfrak d(x):=\frac x2-\frac54,\qquad
	V(x,t):=\sqrt{\frac{T'}{7.49}}-\frac32,
	\]
	\[
	K_m(x):=\exp\!\left(
	\frac{\frac14\sqrt{\ell^2+\pi^2/4}+1.54}{x-5/2}
	\right),\qquad
	\lambda_\pm(x,y):=
	\frac{12((1\pm y)/2)^2+13}{24\mathfrak d(x)},
	\]
	and
	\[
	S(y):=\frac{27}{4}(3^y+3^{-y})
	+\frac{5}{4\sqrt2}(2^{y/2}+2^{-y/2}).
	\]
	Propositions~\ref{prop:RtN_rigorous} and~\ref{prop:eC}, together
	with the prefactor estimate in the proof of
	Proposition~\ref{prop:magnitudes}(6), give
	\[
		e_{C,0}(x+iy)\leq\varepsilon_{C,t}^L(x,y),
	\]
	where
	\begin{equation}\label{eq:epsilon_C_local}
		\begin{aligned}
			\varepsilon_{C,t}^L(x,y):={}&
			K_m(x)\left(\frac{mT'}{2\pi}\right)^{-(1+y)/4}
			e^{-t\ell^2/16}\\
			&\times\left[
			\mathcal B_m^\star
			\left(e^{\lambda_+(x,y)}+e^{\lambda_-(x,y)}\right)
			+\frac{2\varphi(m)}{5V(x,t)}e^{2/\mathfrak d(x)}S(y)
			\right].
		\end{aligned}
	\end{equation}
	Here $\mathcal B_m^\star$ is as in \eqref{eq:C0_uniform_best}.
	It may be replaced throughout by the logarithmic upper bound
	\eqref{eq:C0_uniform_log}, or by $3.33\sqrt m$ in the stated conductor
	range. Both replacements depend only on $m$, so the monotonicities
	below are unchanged.
	The identity
	\[
	\sum_{\pm}
	\left(\frac94\,9^{(1\pm y)/2}+\frac54\,2^{-(1\pm y)/2}\right)=S(y)
	\]
	accounts for the two remainder contributions, counted with multiplicity
	when $y=0$. The leading terms combine as
	$2+(e^{\lambda_+}-1)+(e^{\lambda_-}-1)=e^{\lambda_+}+e^{\lambda_-}$ and
	both contributions from Proposition~\ref{prop:eC} are kept.
	
	The bound \eqref{eq:epsilon_C_local} decreases in $x$, $y$, and $t$,
	with the other variables fixed. For $y$, use
	\[
	\frac{S'(y)}{S(y)}\leq\log3,\qquad
	\frac{d}{dy}\log(e^{\lambda_+}+e^{\lambda_-})
	\leq\frac1{2\mathfrak d(x)}.
	\]
	The logarithmic derivative of \eqref{eq:epsilon_C_local} is therefore
	at most
	\[
	-\frac14\log\left(\frac{mT'}{2\pi}\right)
	+\max\left\{\log3,\frac1{2\mathfrak d(x)}\right\}<0,
	\]
	since $mT'/(2\pi)\ge360/\pi>81$ and
	$1/(2\mathfrak d(x))<\log3$.
	For $t$, the quantities $T'$ and $V(x,t)$ increase and
	$e^{-t\ell^2/16}$ decreases; the remaining factors are independent of $t$.
	For $x$, all factors are non-increasing. To check the only remaining
	factor $K_m(x)$, put
	$N(x)=\frac14\sqrt{\ell^2+\pi^2/4}+1.54$.
	Then $N(x)\geq1.54$ and $N'(x)\leq1/(4x)$, so
	$\frac{d}{dx}(N(x)/(x-5/2))<0$.
	
	We shall use the explicit majorant
	\begin{equation}\label{eq:epsilon_RS_local}
		\varepsilon_t^L(x,y):=
		\varepsilon_{AB,t}^L(x,y)+\varepsilon_{C,t}^L(x,y).
	\end{equation}
	Thus $\mathcal E_t^L(x,y)\leq\varepsilon_t^L(x,y)$, and for fixed
	$m,t$ the latter is non-increasing in $x$ and $y$.
	In particular, for $x\geq x_*$, $y\geq y_*$, and
	$0\leq t_-\leq t\leq t_+\leq1/2$,
	\[
		\mathcal E_t^L(x,y)
		\leq\varepsilon_{AB,t_+}^L(x_*,y_*)
		+\varepsilon_{C,t_-}^L(x_*,y_*).
	\]
	No monotonicity in $t$ is asserted for the sum
	$\varepsilon_t^L$.
	
	The improvement in $\varepsilon_{AB,t}^L$ from the upper ordinate is
	contained in $G_m(x,D,t)$. The type-$C$ bound keeps its dependence on
	the actual ordinate: its maximum on $[y_*,D]$ is still attained at
	$y_*$. Thus, decreasing the upper ordinate alone does not reduce that
	maximum. In particular, replacing a top contour at $y=1$ by one at
	$y=D$ requires checking the top-edge inequality again.
	
	For the finite blocks used below, one can also retain the factors
	$H_n$ and $\eta_n$. For a fixed integer $N\geq1$, define
	\begin{equation}\label{eq:epsilon_AB_finite}
		\varepsilon_{AB,t}^{L,[N]}(x,y):=
		(1+G_m(x,D,t))\sum_{n=1}^{N}
		n^{-1+\theta_y(x,t)}H_n(x,t)\eta_n(x,t).
	\end{equation}
	Whenever $M(x,t)\leq N$, this bounds $e_A+e_B$, since the summands
	are nonnegative. For fixed $N,m,t,D$, it is non-increasing in $x$ and $y$.
	No monotonicity in $t$ is used for this finite sum: the time-dependent
	factors are bounded separately on the barrier boxes below.
	To verify the assertion for $x$, note that
	$\partial_x\log H_n=-t\log n/(4x)\le0$ and
	$\partial_x\theta_y\le0$. It remains to check $\eta_n$.
	With $a,c$ as above and $L=\log(q/n^2)$,
	\[
	\frac{d}{dx}\frac{aL^2+c}{x-5/2}
	=\frac{2aL(1-5/(2x))-aL^2-c}{(x-5/2)^2}<0,
	\]
	because the numerator is at most $a(1-5/(2x))^2-c<0$.
	This argument holds for every $n\geq1$, including the terms added
	when $N>M$.
	It follows that
	\begin{equation}\label{eq:epsilon_RS_finite}
		\varepsilon_t^{L,[N]}(x,y):=
		\varepsilon_{C,t}^L(x,y)
		+\min\left\{\varepsilon_{AB,t}^L(x,y),
		\varepsilon_{AB,t}^{L,[N]}(x,y)\right\}
	\end{equation}
	is a non-increasing majorant in $x,y$ on any region where $M(x,t)\leq N$.
	At a truncation boundary, these estimates apply to either adjacent
	branch by taking its one-sided limit, provided its truncation length
	is at most $N$. Here a continued value always means the limit of a
	branch from a region where that truncation is the natural one; no
	remainder estimate for an arbitrary truncation length is asserted.
	
	\subsection{Region (c)}
	
	We work at the fixed time $t=t_0$, where $0<t_0\leq1/2$,
	and assume $0<y_0<1$. Put $K:=M_0+J_m\geq2$.
	Choose an integer $M_1>K$ and set
	$X_1:=\frac{4\pi M_1^2}{m}-\frac{\pi t_0}{4}.$
	We require $X_1\geq480$. Then $M=M(x,t_0)\geq M_1$
	throughout
	\[
	\mathcal R_c^L
	=
	\{x+iy:x\geq X_1,\ y_0\leq y\leq D\}.
	\]
	The choice of $M_1$, and hence of $X_1$, will be made below.
	
	As in \cite[\S 8.3]{Poly}, we isolate the first term of
	the first Dirichlet sum. Put
	$\sigma_+:=\operatorname{Re}s_*^L(+)$.
	Since $\chi_m(1)=1$ and $b_1^{t_0,L}=1$, we have \[f_{t_0}^L(x+iy)-1=
	\sum_{n=2}^{M}
	\frac{\chi_m(n)b_n^{t_0,L}}{n^{s_*^L(+)}}
	+
	\gamma^L(x+iy)
	\sum_{n=1}^{M}
	\frac{\chi_m(n)b_n^{t_0,L}}{n^{s_*^L(-)}}.\]
	Using $|\chi_m(n)|\leq1$ and $s_*^L(-)=\overline{s_*^L(+)}-y+\kappa^L,$
	we obtain
	\begin{equation}\label{fminusonebound}
		|f_{t_0}^L(x+iy)-1|
		\leq
		\sum_{n=2}^{M}\frac{b_n^{t_0,L}}{n^{\sigma_+}}
		+
		|\gamma^L|
		\sum_{n=1}^{M}
		\frac{n^y b_n^{t_0,L}}{n^{\sigma_+-|\kappa^L|}}.
	\end{equation}
	
	The coefficients $b_n^{t_0,L}
	=
	\exp\left(\frac{t_0}{4}\log^2 n\right)$
	are the same as those in \cite{Poly}. We may therefore
	apply \cite[Lemma 8.2]{Poly}: if $M\geq K\geq1$ and
	$\sigma>(t_0/2)\log M$, then
	\begin{equation}\label{tail_sum_lemma}
		\sum_{n=1}^{M}\frac{b_n^{t_0,L}}{n^\sigma}
		\leq
		\sum_{n=1}^{K}\frac{b_n^{t_0,L}}{n^\sigma}
		+
		\max\left\{
		K^{1-\sigma}b_K^{t_0,L},
		M^{1-\sigma}b_M^{t_0,L}
		\right\}
		\log\frac{M}{K}.
	\end{equation}
	
	We next make the parameters in \eqref{fminusonebound}
	uniform on $\mathcal R_c^L$.
	Proposition~\ref{prop:magnitudes} gives
	$|\kappa^L|
	\leq
	\frac{t_0y}{2(x-2)}
	\leq
	\kappa_1:=\frac{t_0D}{2(X_1-2)}$,
	and
	$|\gamma^L|
	\leq
	\exp\left[
	\left(0.002+\frac{\log m}{3800}\right)y
	\right]
	\left(\frac{mx}{4\pi}\right)^{-y/2}.$
	
	Since
	$M^2\leq
	\frac{mx}{4\pi}
	\left(1+\frac{\pi t_0}{4x}\right),$ 
	it follows, for $1\leq n\leq M$, that
	\begin{equation}\label{gamma_n_tail}
		\begin{aligned}
			|\gamma^L|n^y
			&\leq
			\exp\left[
			\left(0.002+\frac{\log m}{3800}\right)y
			\right]
			\left(1+\frac{\pi t_0}{4x}\right)^{y/2}
			\left(\frac{n}{M}\right)^y
			\\
			&\leq
			\mathfrak G_m M^{-y_0}n^{y_0},
		\end{aligned}
	\end{equation}
	where $
	\mathfrak G_m
	:=
	\exp\left[\left(0.002+\frac{\log m}{3800}\right)D\right]
	\left(1+\frac{\pi t_0}{4X_1}\right)^{D/2}.$
	
	By Proposition~\ref{prop:magnitudes}(ii),
	$\sigma_+
	\geq
	\frac{1+y}{2}
	+
	\frac{t_0}{4}\log\left(\frac{mx}{4\pi}\right)
	-
	\frac{(3+y)t_0}{2x^2}.$ Consequently,
	\begin{equation}\label{sigma_plus_tail}
		\sigma_+
		\geq
		\frac{1+y_0}{2}
		+
		\frac{t_0}{2}\log M-\rho_1,
	\end{equation}
	where $\rho_1:=
	\frac{t_0}{4}
	\log\left(1+\frac{\pi t_0}{4X_1}\right)
	+\frac{(3+D)t_0}{2X_1^2}
	\leq\frac{31t_0}{240X_1}.$ Define
	\[
	\begin{aligned}
		\sigma_m'(M)
		&:=
		\frac{1+y_0}{2}
		+\frac{t_0}{2}\log M-\rho_1,
		\\
		\sigma_m''(M)
		&:=
		\sigma_m'(M)-y_0-\kappa_1
		=
		\frac{1-y_0}{2}
		+\frac{t_0}{2}\log M-\rho_1-\kappa_1.
	\end{aligned}
	\]
	Combining \eqref{gamma_n_tail}
	and \eqref{sigma_plus_tail}, we obtain
	\begin{equation}\label{tail_bound_general}
		|f_{t_0}^L(x+iy)-1|
		\leq
		\sum_{n=2}^{M}
		\frac{b_n^{t_0,L}}{n^{\sigma_m'(M)}}
		+
		\mathfrak G_m M^{-y_0}
		\sum_{n=1}^{M}
		\frac{b_n^{t_0,L}}{n^{\sigma_m''(M)}}.
	\end{equation}
	
	Assume $\rho_1+\kappa_1<\frac{1-y_0}{2}.$
	Then both exponents in \eqref{tail_bound_general}
	exceed $(t_0/2)\log M$, as required in
	\eqref{tail_sum_lemma}.
	Restoring the $n=1$ term in the first sum and applying
	that estimate twice gives
	\[
	|f_{t_0}^L(x+iy)-1|
	\leq A_m(M)+B_m(M)-1,
	\]
	where 
\[
\begin{aligned}
A_m(M):={}&\sum_{n=1}^K\frac{b_n^{t_0,L}}{n^{\sigma_m'(M)}}
+\mathfrak G_m M^{-y_0}\sum_{n=1}^K
\frac{b_n^{t_0,L}}{n^{\sigma_m''(M)}},\\
B_m(M):={}&\log\left(\frac MK\right)
\Bigl[\max\{K^{1-\sigma_m'(M)}b_K^{t_0,L},
M^{1-\sigma_m'(M)}b_M^{t_0,L}\}\\
&\quad+\mathfrak G_m M^{-y_0}
\max\{K^{1-\sigma_m''(M)}b_K^{t_0,L},
M^{1-\sigma_m''(M)}b_M^{t_0,L}\}\Bigr].
\end{aligned}
\]
	
	The function $A_m(M)$ is decreasing in $M$.
	Logarithmic differentiation of the terms defining
	$B_m(M)$ shows that it is non-increasing for
	$M\geq M_1$ provided
	\begin{equation}\label{tail_monotonicity}
		\frac{1}{\log(M_1/K)}
		\leq
		\min\left\{
		\frac{t_0}{2}\log K,\,
		\sigma_m''(M_1)+y_0-1
		\right\}.
	\end{equation}
	To check this directly, put $u=\log M$ and $v=\log K$.
	With $M_1$ fixed, the constants $\rho_1$, $\kappa_1$ and
	$\mathfrak G_m$ are independent of $M$. The four positive terms whose
	maxima define $B_m(M)$ have logarithmic derivatives with respect to
	$u$ equal to
	\[
	\begin{split}
		&\frac1{u-v}-\frac{t_0v}{2},\qquad
		\frac1{u-v}+1-\sigma_m'(M),\\
		&\frac1{u-v}-y_0-\frac{t_0v}{2},\qquad
		\frac1{u-v}+1-y_0-\sigma_m''(M).
	\end{split}
	\]
	Each is nonpositive under \eqref{tail_monotonicity}: $1/(u-v)$
	decreases with $M$, both exponents increase, and
	$\sigma_m''(M)+y_0=\sigma_m'(M)-\kappa_1\le\sigma_m'(M)$.
	The maximum of two non-increasing functions is non-increasing,
	which proves the assertion for $B_m$.
	Hence
	\[
	A_m(M)+B_m(M)
	\leq A_m(M_1)+B_m(M_1)
	\qquad(M\geq M_1).
	\]
	
	By \eqref{eq:epsilon_RS_local} and its monotonicity in $x,y$, we may take
	\[
		\mathcal E_m^{(c)}:=\varepsilon_{t_0}^L(X_1,y_0).
	\]
	Then
	\[
	\sup_{\substack{x\geq X_1\\y_0\leq y\leq D}}
	\mathcal E_{t_0}^L(x,y)\leq\mathcal E_m^{(c)}.
	\]
	
	We choose $M_1$, and hence $X_1$, so that
	\[A_m(M_1)+B_m(M_1)+\mathcal E_m^{(c)}<2
	\]
	holds. For every $x+iy\in\mathcal R_c^L$,
	Theorem~\ref{EffectiveRSa} then yields
	\[
	\begin{aligned}
		\left|
		\frac{H_{t_0}^L(x+iy)}{B_{t_0}^L(x+iy)}
		\right|
		&\geq
		1-|f_{t_0}^L(x+iy)-1|
		-\mathcal E_{t_0}^L(x,y)
		\\
		&\geq
		2-A_m(M_1)-B_m(M_1)-\mathcal E_m^{(c)}
		>0.
	\end{aligned}
	\]
	Thus non-vanishing in region (c) holds for any choice of $M_1$
	satisfying these conditions.
	
	\subsection{Derivative estimates}
We use the notation $s_\pm,s_*^L(\pm),\gamma^L$ already
introduced above. Set
\[
\begin{aligned}
\mathcal A_L(x)&:=\log\left(\frac{mx}{4\pi}\right)
+\frac\pi2+\frac4{x-2},\\
\mathcal C_L(x,t)&:=\frac12+\frac{t}{4(x-2)},
\qquad
\mathcal D_L(x):=\frac\pi2+\frac4{x-2}.
\end{aligned}
\]
\begin{lemma}
		\label{lem:derivative_bounds_L}
		
        Let $3\le m\le400000$ and $0<D\le1$. For a fixed integer $M$,
        let $f_{t,M}^L$ denote the expression in \eqref{ftL_def} with both
        sums truncated at $M$. Assume that
        \[
        x\ge480,\qquad 0\le y\le D,\qquad
        0\le t\le\frac12,\qquad
        1\le M\le\frac{mx}{4\pi}.
        \]
        Then
		\begin{equation}
			\begin{aligned}
				|\partial_z f_{t,M}^L(x+iy)|
				\le{}&
				\mathcal C_L(x,t)
				\sum_{n=1}^{M}
				\frac{b_n^{t,L}}
				{n^{\re s_*^L(+)}}
				\log(n)
				\\
				&+
				|\gamma^L(x+iy)|M^{|\kappa^L|}
				\sum_{n=1}^{M}
				\frac{n^y b_n^{t,L}}
				{n^{\re s_*^L(+)}}
				\left(
				\frac{t}{4(x-2)}\log(n)
				+
				\mathcal C_L(x,t)\mathcal A_L(x)
				\right),
			\end{aligned}
			\label{eq:dz_bound_L}
		\end{equation}
		and
		\begin{equation}\label{eq:dt_bound_L}
\begin{aligned}
|\partial_t f_{t,M}^L(x+iy)|\le{}&
\frac14\sum_{n=1}^M\frac{b_n^{t,L}}{n^{\re s_*^L(+)}}
\log n\bigl(\mathcal A_L(x)-\log n\bigr)\\
&+\frac14|\gamma^L|M^{|\kappa^L|}
\sum_{n=1}^M\frac{n^y b_n^{t,L}}{n^{\re s_*^L(+)}}\\
&\quad\times\left[\log n\bigl(\mathcal A_L(x)-\log n\bigr)
+\left(\mathcal A_L(x)-\frac\pi2\right)\mathcal D_L(x)\right].
\end{aligned}
\end{equation}
	\end{lemma}

\begin{proof}
We follow the product-rule calculation of \cite[Lemma~8.4]{Poly},
with the logarithmic derivative $\alpha_L$. Put $\alpha_\pm=
\alpha_L(s_\pm)$ and $\alpha_\pm'=\alpha_L'(s_\pm)$.
Since $M$ is fixed, the finite sums are holomorphic and
$|\partial_zf_{t,M}^L|=|\partial_xf_{t,M}^L|$.
The formula for $\alpha_L'$ gives
$|\alpha_\pm'|\le1/x+2/x^2\le1/(x-2)$, whence
\[
|\partial_xs_*^L(\pm)|
\le\frac12+\frac{t}{4(x-2)}=\mathcal C_L(x,t).
\]
Also,
\begin{equation}\label{eq:dloggamma_dx_L}
\partial_x\log\gamma^L
=\frac i2(\alpha_-+\alpha_+)
+\frac{it}{4}(\alpha_-\alpha_-'+\alpha_+\alpha_+').
\end{equation}
Write $q=mx/(4\pi)$.
By \eqref{eq:magnitudes_alpha} and conjugation,
$\alpha_\pm=\frac12\log q\mp i\pi/4+R_\pm$, with
$|R_\pm|\le2/(x-2)$. Since $n\le M\le q$,
these estimates imply
\[
\begin{aligned}
|\log n-2\alpha_\pm|&\le\mathcal A_L-\log n,
&|\alpha_\pm|&\le\mathcal A_L/2,\\
|\alpha_-+\alpha_+|&\le\mathcal A_L-\pi/2,
&|\alpha_--\alpha_+|&\le\mathcal D_L.
\end{aligned}
\]
Combining these bounds with \eqref{eq:dloggamma_dx_L} gives
\begin{equation}\label{eq:combined_spatial_L}
|\partial_x\log\gamma^L
-\log n\partial_xs_*^L(-)|
\le\mathcal C_L\mathcal A_L+\frac{t\log n}{4(x-2)}.
\end{equation}
The derivative of a first-sum term has modulus at most its
weight times $\mathcal C_L\log n$; a second-sum term has the
factor in \eqref{eq:combined_spatial_L}.
Moreover, $\operatorname{Re}s_*^L(-)=
\operatorname{Re}s_*^L(+)-y+\operatorname{Re}\kappa^L$ gives
$n^{-\operatorname{Re}s_*^L(-)}\le
n^yM^{|\kappa^L|}n^{-\operatorname{Re}s_*^L(+)}$.
Taking absolute values term by term and using $|\chi_m(n)|\le1$
proves \eqref{eq:dz_bound_L}.

For the time derivative, the logarithmic factors for the two
sums are respectively
\[
\frac14\log n\,(\log n-2\alpha_+),
\qquad
\frac14\left[\log n\,(\log n-2\alpha_-)
+(\alpha_-^2-\alpha_+^2)\right].
\]
Their moduli are bounded by
$\log n\,(\mathcal A_L-\log n)/4$ and
$[\log n\,(\mathcal A_L-\log n)
+(\mathcal A_L-\pi/2)\mathcal D_L]/4$.
The same comparison of weights proves \eqref{eq:dt_bound_L}.
\end{proof}

	The lemma applies on the enclosing boxes used for interpolation,
	even when the natural truncation changes there. Indeed, every active
	truncation length satisfies
	\[
	M\le\sqrt{\frac{m(X+1+\pi t_0/4)}{4\pi}}<\frac{mX}{4\pi}
	\qquad(X\ge480,\ m\ge3).
	\]
	Thus its hypothesis $M\le mx/(4\pi)$ holds throughout the box.
	The simple replacements $n^y\le n^D$ and
	$|\kappa^L|\le tD/(2(x-2))$ are also available there; keeping the
	actual endpoint ordinates generally gives smaller bounds.

\subsection{Region (a)}
	
	We verify the moving barrier required by Lemma~\ref{lemmabox}.
	Put $Y(t)=\sqrt{y_0^2+2(t_0-t)}$, $U_D(t)=\sqrt{D^2-2t}$ and
	$\Delta_D=\sqrt{D^2-y_0^2-2t_0}$. We assume $t_0>0$, $y_0>0$ and
	$y_0^2+2t_0<D^2$, so that $0<Y(t)<U_D(t)$ for $0\leq t\leq t_0$.
	For a fixed integer $M$, let $f_{t,M}^L$ be the expression in
	\eqref{ftL_def} with both sums truncated at $M$, and define
	\[
	R_{a,M}(t)=\left\{x+iy:
	\max\{X,x_M(t)\}\leq x\leq\min\{X+\Delta_D,x_{M+1}(t)\},\quad
	Y(t)\leq y\leq U_D(t)\right\}.
	\]
	Only nonempty rectangles are used. On their interiors,
	$f_{t,M}^L=f_t^L$; on their edges, we use the continued fixed-$M$ values.
	For fixed $t$, $f_{t,M}^L$ is holomorphic in a neighborhood of the rectangle.
	
	The closure of the times at which a stripe has positive width is its
	active interval $[a_M,b_M]$, where
	\[
	a_M=\max\left\{0,\frac{16M^2}{m}-\frac{4(X+\Delta_D)}\pi\right\},
	\qquad
	b_M=\min\left\{t_0,\frac{16(M+1)^2}{m}-\frac{4X}\pi\right\}.
	\]
	We discard intervals with $a_M\geq b_M$ and choose a mesh
	$a_M=\tau_0<\cdots<\tau_J=b_M$ on each remaining stripe.
	Changes in which expression defines a clipped vertical edge do not
	require additional mesh points: the endpoints are Lipschitz functions
	of time, with horizontal speed at most $\pi/4$ almost everywhere.
	All active stripes are included, even when the truncation length varies.
	
	We first bound the error on a time interval.
	Write $I_j=[\tau_j,\tau_{j+1}]$ and $v_j=Y(\tau_{j+1})$.
	For $e_A+e_B$, we bound $\theta_y$ and $\eta_n$ at the upper endpoint
	of the active stripe and $H_n$ at its corners. For $e_{C,0}$, we use
	the lower time endpoint and the lowest ordinate of $I_j$.
	No monotonicity of $t\mapsto\varepsilon_{C,t}^L(X,Y(t))$ is needed.
	
	With the notation of \eqref{eq:epsilon_C_local}, put
	\[
		\begin{split}
			\widehat\varepsilon_{C,t}^L(x,v):={}&
			K_m(x)\left(\frac{mT'}{2\pi}\right)^{-(1+v)/4}
			e^{-t\ell^2/16}\\
			&\times\left[\mathcal B_m^\star
			\left(e^{\lambda_+(x,D)}+e^{\lambda_-(x,v)}\right)
			+\frac{2\varphi(m)}{5V(x,t)}e^{2/\mathfrak d(x)}S(D)\right].
		\end{split}
	\]
	This is an ordinate-uniform envelope, rather than the pointwise
	bound evaluated at $y=v$. For $0\leq v\leq y\leq D$, it bounds
	$\varepsilon_{C,t}^L(x,y)$: replace $\lambda_+(x,y)$ by
	$\lambda_+(x,D)$, $S(y)$ by $S(D)$, and then use
	$\lambda_-(x,y)\leq\lambda_-(x,v)$ and the decreasing power of
	$mT'/(2\pi)>1$. Each factor shows that this envelope is non-increasing
	in $x$, $v$ and $t$ separately. The logarithmic
	upper bound \eqref{eq:C0_uniform_log} may be used for $\mathcal B_m^\star$.
	
	For a stripe active on $[a_M,b_M]$, define
	\[
\begin{aligned}
\delta_{\rm box}&=\sqrt{\frac m{4\pi}
\left(X+1+\frac{\pi b_M}4\right)},
\qquad \Theta=\theta_{y_0}(X,b_M),\\
\widehat G&=\exp\left(a_mD+\frac D2\log\left(1+\frac{\pi b_M}{4X}\right)
+\frac{b_MD\log\delta_{\rm box}}{2(X-2)}\right).
\end{aligned}
\]
	and
	\[
	\widehat H_n=\max_{\substack{x\in\{X,X+1\}\\t\in\{a_M,b_M\}}}
	\exp\left(-\frac t4\log n\left(\log\frac{mx}{4\pi}-\log n\right)\right).
	\]
	A uniform bound for $e_A+e_B$ on the stripe is
	\[
		\widehat\varepsilon_{AB,M}=(1+\widehat G)\sum_{n=1}^{M}
		n^{-1+\Theta}\widehat H_n\eta_n(X,b_M).
	\]
	Indeed, $\theta_y(x,t)\le\Theta$, $\eta_n(x,t)\le\eta_n(X,b_M)$, and $|\gamma^L|M^{y+|\kappa^L|}\le\widehat G$ on the enclosing box. For each $n$, $\log H_n(x,t)$ is affine in $t$ and monotone in $x$, so its maximum occurs at one of the four listed corners.
	
	For $I_j=[\tau_j,\tau_{j+1}]$, set $v_j=Y(\tau_{j+1})$ and take
	\[
		\mathcal E_{a,M,j}=\widehat\varepsilon_{AB,M}
		+\widehat\varepsilon_{C,\tau_j}^L(X,v_j).
	\]
	The first term is uniform over the active stripe; the second uses the earliest time $\tau_j$ and the lowest ordinate $v_j$ on the interval. The effective approximation gives
	\begin{equation}\label{barrier_RS_error}
		\left|H_t^L(z)/B_t^L(z)-f_{t,M}^L(z)\right|
		\leq\mathcal E_{a,M,j},\qquad
		t\in I_j,\quad z\in\partial R_{a,M}(t).
	\end{equation}
	At the last node $b_M$, evaluate the type-$C$ term at $(t,y)=(b_M,Y(b_M))$.
	
	We next control spatial and temporal interpolation.
	Let $D_{z,M}$ bound $|\partial_z f_{t,M}^L|$ on the enclosing box
	$[X,X+1]\times[y_0,D]$ for $a_M\le t\le b_M$.
	For each $I_j=[\tau_j,\tau_{j+1}]$, let $D_{z,M,j}$ and
	$D_{t,M,j}$ bound $|\partial_z f_{t,M}^L|$ and
	$|\partial_t f_{t,M}^L|$ on the same box for $t\in I_j$.
	The truncation remains fixed at $M$. These bounds follow from
	Lemma~\ref{lem:derivative_bounds_L} by evaluating its positive
	factors on the corresponding time interval.
	For example, on $I_j$ we use
	\[
	\mathcal A_L(x)\leq
	\log\frac{m(X+1)}{4\pi}+\frac\pi2+\frac4{X-2},
	\qquad
	\mathcal C_L(x,t)\leq\frac12+\frac{\tau_{j+1}}{4(X-2)}.
	\]
	Put $q_X=mX/(4\pi)$ and let $\delta_{\rm box}$ be as above.
	The two weights in the derivative bounds are bounded by
	\[
	\begin{split}
		W_n(y,t)&:=\exp\left(\frac t4\log^2 n
		-\left(\frac{1+y}{2}+\frac t4\log q_X
		-\frac{(3+y)t}{2X^2}\right)\log n\right),\\
		V_n(y,t)&:=e^{a_my}q_X^{-y/2}
		\delta_{\rm box}^{ty/(2(X-2))}n^yW_n(y,t).
	\end{split}
	\]
	Each logarithm is the bilinear interpolation of its four corner
	values on $[y_0,D]\times I_j$, so its maximum occurs at a corner. The remaining positive factors
	are bounded using $t\le\tau_{j+1}$, $x\ge X$ and the displayed bound
	for $\mathcal A_L$ (and the analogous bound for $\mathcal D_L$).
	Summing these upper bounds gives $D_{z,M,j}$ and $D_{t,M,j}$.
	Applying the same estimates on $[a_M,b_M]$, with $b_M$ in place
	of $\tau_{j+1}$, gives the global spatial bound $D_{z,M}$.
	For a fixed-$M$ sum, $M$ itself may replace $\delta_{\rm box}$
	in the factor accounting for $M^{|\kappa^L|}$.
	This is the same corner rule used in the code. On each time cell,
	the smaller ordinate interval $[Y(\tau_{j+1}),U_D(\tau_j)]$
	may also be used, since it still encloses the moving stripe.
	
	Divide each side into $n_M$ equal segments and denote the resulting
	counterclockwise vertices by $z_\ell(t)$, $1\leq\ell\leq4n_M$.
	All side lengths are at most one, so a boundary point has a nearest
	vertex at distance at most $1/(2n_M)$. Consequently,
	\[
		|f_{t,M}^L(z)-f_{t,M}^L(z_\ell(t))|
		\leq r_{z,M}:=\frac{D_{z,M}}{2n_M}.
	\]
	Both horizontal edges move:
	\[
	Y'(t)=-\frac1{Y(t)},\qquad U_D'(t)=-\frac1{U_D(t)}.
	\]
	Since $U_D(t)\ge Y(t)\ge y_0$ and both ordinates decrease,
	the affine parametrizations of the four sides satisfy, almost
	everywhere on $I_j$,
	\[
	|z_\ell'(t)|\le V_j:=\frac\pi4+\frac1{Y(\tau_{j+1})}
	\le V_0:=\frac\pi4+\frac1{y_0}.
	\]
	The sum of the horizontal and vertical speeds bounds the Euclidean
	speed. The chain rule and integration give
	\[
		|f_{t,M}^L(z_\ell(t))-f_{\tau_j,M}^L(z_\ell(\tau_j))|
		\leq (D_{t,M,j}+V_jD_{z,M,j})(\tau_{j+1}-\tau_j).
	\]
	The global bound $D_{z,M}$ is used for the spatial mesh radius,
	whereas the bounds on $I_j$ are used for motion between time nodes.
	The vertical contribution is retained even when $M$ is constant and
	the horizontal edges of the rectangle have fixed endpoints in $x$.
	
	Let $\widehat f_{j,\ell}^L$ approximate the value at the $j$-th time
	node with error at most $\eta_{\rm num}$. We check
	\begin{equation}\label{barrier_certificate}
		\begin{split}
			\min_\ell|\widehat f_{j,\ell}^L|>{}&
			\mathcal E_{a,M,j}+\eta_{\rm num}+r_{z,M}\\
			&+(D_{t,M,j}+V_jD_{z,M,j})(\tau_{j+1}-\tau_j).
		\end{split}
	\end{equation}
	At the final node, the temporal term is omitted. In the Arb verification,
	each function value is enclosed by a complex ball including the Taylor
	remainder; its modulus lower bound is used in place of
	$|\widehat f_{j,\ell}^L|$. The additional subtraction of
	$\eta_{\rm num}$ in the implementation is conservative.
	
	We then check the winding numbers and the covered domains.
	Using cyclic indices, the polygon winding number is
	\[
		\operatorname{Wind}_n(\widehat f_{\tau_j,M}^L;R_{a,M}(\tau_j))
		=\frac1{2\pi}\sum_{\ell=1}^{4n_M}
		\operatorname{Arg}\frac{\widehat f_{j,\ell+1}^L}{\widehat f_{j,\ell}^L}.
	\]
	The interval calculation encloses every argument increment strictly
	between $-\pi/2$ and $\pi/2$ and encloses the total number of turns
	strictly between $-1/2$ and $1/2$. Since the winding number is an integer,
	this proves
	\[
		\operatorname{Wind}_n(\widehat f_{\tau_j,M}^L;R_{a,M}(\tau_j))=0.
	\]
	Consecutive computed values differ by at most
	$D_{z,M}/n_M+2\eta_{\rm num}$. Thus, on each half-segment, the true
	boundary curve and the corresponding polygon lie in the disk of
	radius $\eta_{\rm num}+r_{z,M}$ about the nearest computed vertex. The strict margin excludes the origin from
	these disks, proving that the two curves have the same winding number.
	Equation~\eqref{barrier_certificate} keeps the boundary image away from
	zero throughout $I_j$, so the winding number stays zero.
	
	The argument principle and \eqref{barrier_RS_error}, via Rouch\'e's
	theorem, imply that $H_t^L/B_t^L$ has no zeros in a nondegenerate
	rectangle, and the strict boundary margin excludes zeros on its edges.
	At a degenerate endpoint the two vertical sides coincide and are
	traversed in opposite directions; their winding number is zero.
	The same boundary estimate and continuity apply when a stripe opens
	or closes. The closures of all active stripes cover the moving barrier.
	Hence
	\[
	H_t^L(x+iy)\neq0\quad
	\left(X\leq x\leq X+\Delta_D,\quad Y(t)\leq y\leq U_D(t),
	\quad0\leq t\leq t_0\right),
	\]
	which is precisely hypothesis (3) of Lemma~\ref{lemmabox}.
	No non-vanishing claim on the full enclosing box is needed here.
	For the retained configurations in Section~\ref{sec:numerical_implementation_main},
	the natural truncation is constant on each enclosing barrier box:
	$M=3989$, except for the two configurations of $\chi_{-395}$,
	where $M=3988$. The argument by stripes allows changes of $M$,
	but none is crossed inside those boxes.
	
	For region (b) we make a separate final-time calculation on the full
	vertical segment $\{X+iy:y_0\leq y\leq D\}$. Its truncation is
	$M(X,t_0)$ throughout. The resulting constants $L_a,\psi_a$ are
	verified on this full segment, as required in
	\eqref{eq:b_left_condition}; data restricted to $y\leq U_D(t_0)$
	are not used for that purpose.

		\begin{figure}
			\centering
			\includegraphics[width=0.8 \linewidth]{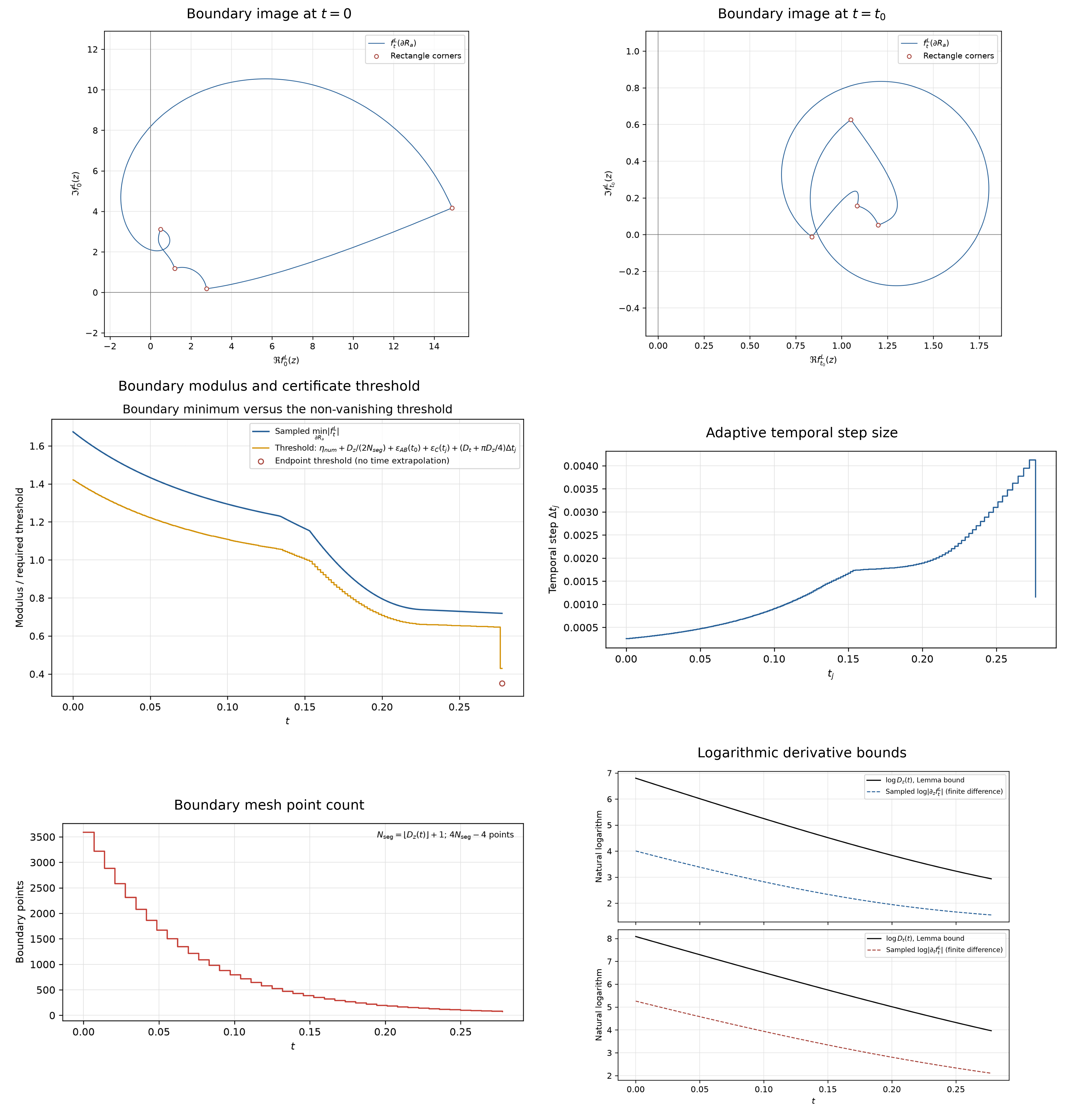}
			\caption{Illustrative barrier computation for the Dirichlet beta
function ($m=4$) on the enclosing box
$X_4^*\le x\le X_4^*+1$, $0.1165\le y\le1$ and
$0\le t\le0.2776$, with $M=3989$. The panels show the boundary
images, sampled boundary minima and verification thresholds,
temporal steps, mesh sizes, and derivative bounds plotted on a
logarithmic scale. These illustrative parameters differ from the
verified choices in Table~\ref{tab:numerical_selected_parameters}.}
		\end{figure}

	\subsection{Region (b)}
	
	We work at the fixed time $t=t_0$, with the same parameters as in
	region (c), and choose $M_1$ so that $X_1>X+1$. It is enough to
	prove non-vanishing on the larger rectangle
	\[
	R_b:=\{x+iy:X\leq x\leq X_1,\ y_0\leq y\leq D\}.
	\]
	Throughout this subsection, $M=M(x,t_0)$ and
	$b_n=b_n^{t_0,L}=\exp((t_0/4)\log^2n)$. We write
	\[
	s=s_*^L(+),\qquad \sigma=\operatorname{Re}s,\qquad
	\mathcal E(z)=e_A(z)+e_B(z)+e_{C,0}(z),
	\]
	where all error terms are evaluated at $t=t_0$.
	Theorem~\ref{EffectiveRSa} gives
	\[
	\left|\frac{H_{t_0}^L(z)}{B_{t_0}^L(z)}-f_{t_0}^L(z)\right|
	\leq\mathcal E(z).
	\]
	At a change of $M$, we use the corresponding continued endpoint
	values. The quotient on the left is holomorphic on a neighbourhood
	of $R_b$, and $B_{t_0}^L$ does not vanish there.
	
	As in \cite[\S8.5]{Poly}, we introduce a finite Euler product.
	Choose $2\leq P\leq M(X,t_0)$ and put
	\[
	\mathcal P=\{p\leq P:p\text{ is prime},\ p\nmid m\},
	\qquad Q_P=\prod_{p\in\mathcal P}p.
	\]
	Define
	\[
		E_{t_0,P}^L(z)
		:=\prod_{p\in\mathcal P}
		\left(1-\frac{\chi_m(p)b_p}{p^{s_*^L(+)}}\right).
	\]
	The choice of $P$ is left to the numerical verification.
	In particular, primes dividing $m$ contribute no factor.
	For brevity, write $E(z)=E_{t_0,P}^L(z)$.
	
	By Proposition~\ref{prop:magnitudes}, a lower bound for $\sigma$
	throughout $R_b$ is
	\[
	\sigma_b:=\frac{1+y_0}{2}
	+\frac{t_0}{4}\log\left(\frac{mX}{4\pi}\right)
	-\frac{(3+y_0)t_0}{2X^2}.
	\]
	Indeed, the lower bound in that proposition is increasing in both
	$x$ and $y$ on the present range. Put
	\[
	u_p:=\frac{b_p}{p^{\sigma_b}},\qquad
	L_P:=\prod_{p\in\mathcal P}(1-u_p),\qquad
	U_P:=\prod_{p\in\mathcal P}(1+u_p),\qquad
	\Theta_P:=\sum_{p\in\mathcal P}\arcsin u_p.
	\]
	We require $u_p<1$ for every $p\in\mathcal P$. Then
	\[
		L_P\leq |E(z)|\leq U_P,\qquad
		|\arg E(z)|\leq\Theta_P,
	\]
	where the argument is the sum of the arguments of the individual
	factors, each lying in the right half-plane. In particular, $E$
	is holomorphic and nonzero on a neighbourhood of $R_b$.
	
	Let $\mathcal L=(-\infty,0]$. It suffices to verify
	\begin{equation}\label{eq:b_cut_condition}
		\operatorname{dist}\bigl(E(z)f_{t_0}^L(z),\mathcal L\bigr)
		>U_P\mathcal E(z)
		\qquad(z\in\partial R_b).
	\end{equation}
	Indeed, this implies that the holomorphic function
	$E H_{t_0}^L/B_{t_0}^L$ avoids $\mathcal L$ on $\partial R_b$.
	Its winding number about the origin is therefore zero, and the
	argument principle gives non-vanishing in the interior as well.
	This argument uses the exact quotient, so changes in the truncation
	length do not affect its holomorphy.
	
	We first record the elementary geometric estimate used on three
	sides of the rectangle. If
	$|f_{t_0}^L(z)|\geq L>0$ and
	$|\arg f_{t_0}^L(z)|\leq\psi$, with $\Theta_P+\psi<\pi$, then
	\[
		\operatorname{dist}\bigl(E(z)f_{t_0}^L(z),\mathcal L\bigr)
		\geq L_P L
		\cos\left(\left(\Theta_P+\psi-\frac\pi2\right)_+\right),
	\]
	where $v_+=\max\{v,0\}$. This follows by considering separately
	the right and left half-planes. In particular, if
	$|f_{t_0}^L(z)-1|\leq Q<1$, we may take
	$L=1-Q$ and $\psi=\arcsin Q$.
	
	On the right edge, the estimate for region (c) gives
	\[
	|f_{t_0}^L(X_1+iy)-1|\leq Q_c,
	\qquad Q_c:=A_m(M_1)+B_m(M_1)-1<1.
	\]
	By \eqref{eq:epsilon_RS_local}, we may take
	\[
		\varepsilon_{\mathrm{right}}:=\varepsilon_{t_0}^L(X_1,y_0).
	\]
	This bounds $\mathcal E(X_1+iy)$ for $y_0\leq y\leq D$.
	We require
	\begin{equation}\label{eq:b_right_condition}
		\begin{split}
			&\Theta_P+\arcsin Q_c<\pi,\\
			&L_P(1-Q_c)
			\cos\left(\left(\Theta_P+\arcsin Q_c-\frac\pi2\right)_+\right)
			>U_P\varepsilon_{\mathrm{right}}.
		\end{split}
	\end{equation}
	The cutoff $M_1$ must satisfy these inequalities as well as the
	conditions for region (c).
	
	On the left edge, use an independent final-time evaluation on the full
	segment $y_0\leq y\leq D$. The final-time boundary of region (a)
	ends at $U_D(t_0)=\sqrt{D^2-2t_0}$ and does not supply this full segment.
	Fix $M=M(X,t_0)$ and let $D_{z,a,M}$ bound
	$|\partial_zf_{t_0,M}^L|$ on this edge. Put
	\[
	h_y:=\frac{D-y_0}{n},\qquad
	y_j:=y_0+jh_y,\qquad z_j:=X+iy_j,
	\qquad j=0,\ldots,n.
	\]
	For every $y\in[y_0,D]$, choose a nearest mesh point $y_j$.
	Then $|y-y_j|\leq h_y/2$.
	Let $\widehat f_j$ be the computed value of $f_{t_0,M}^L(z_j)$,
	with error at most $\eta_{\mathrm{num}}$, and set
	\[
		r_j:=\eta_{\mathrm{num}}+
		\frac{(D-y_0)D_{z,a,M}}{2n}.
	\]
	Thus $|f_{t_0}^L(X+iy)-\widehat f_j|\leq r_j$ whenever
	$y_j$ is a nearest mesh point. Require $r_j<|\widehat f_j|$ for
	all $j$, and define
	\[
		\begin{split}
			L_a&:=\min_{0\leq j\leq n}(|\widehat f_j|-r_j),\\
			\psi_a&:=\max_{0\leq j\leq n}
			\left(|\operatorname{Arg}\widehat f_j|
			+\arcsin\frac{r_j}{|\widehat f_j|}\right).
		\end{split}
	\]
	Provided $\psi_a<\pi$, these give
	\[
	|f_{t_0}^L(X+iy)|\geq L_a>0,\qquad
	|\arg f_{t_0}^L(X+iy)|\leq\psi_a
	\qquad(y_0\leq y\leq D).
	\]
	At a truncation boundary, all values on this edge use the branch
	with truncation $M(X,t_0)$.
	Similarly, set
	\[
		\varepsilon_{\mathrm{left}}:=\varepsilon_{t_0}^L(X,y_0).
	\]
	This bounds $\mathcal E(X+iy)$ for $y_0\leq y\leq D$.
	The required inequalities are
	\begin{equation}\label{eq:b_left_condition}
		\begin{split}
			&\Theta_P+\psi_a<\pi,\\
			&L_P L_a
			\cos\left(\left(\Theta_P+\psi_a-\frac\pi2\right)_+\right)
			>U_P\varepsilon_{\mathrm{left}}.
		\end{split}
	\end{equation}
	The argument bound is an additional check on the barrier
	evaluations; it does not follow from non-vanishing alone.
	
	To treat the horizontal edges, divide the integer range
	$M(X,t_0)\leq M\leq M_1$ into blocks $[M_-,M_+]$.
	For each block put
	\[
	\begin{split}
		x_-&:=\max\{X,x_{M_-}(t_0)\},\qquad
		x_+:=\min\{X_1,x_{M_++1}(t_0)\},\\
		q_-&:=\frac{mx_-}{4\pi},\qquad
		\nu:=\frac{mt_0}{16},\qquad
		a_m:=0.002+\frac{\log m}{3800},
	\end{split}
	\]
	where $x_k(t_0)=4\pi k^2/m-\pi t_0/4$.
	Empty blocks are discarded. Proposition~\ref{prop:magnitudes}
	gives the following bounds throughout the block:
	\[
		\begin{split}
			\sigma&\geq\sigma_-(y)
			:=\frac{1+y}{2}+\frac{t_0}{4}\log q_-
			-\frac{(3+y)t_0}{2x_-^2},\\
			|\gamma^L|&\leq g_-(y):=e^{a_my}q_-^{-y/2},\\
			|\kappa^L|&\leq k_-(y):=\frac{t_0y}{2(x_--2)}.
		\end{split}
	\]
	The block endpoints are interpreted by continuation on the relevant
	branch. Also, for $n\leq M$,
	\begin{equation}\label{eq:b_gamma_n_refined}
		|\gamma^L|n^y
		\leq
		\frac{e^{a_my}n^y}
		{\max\{q_-,n^2-\nu\}^{y/2}}.
	\end{equation}
	Indeed, $mx/(4\pi)\geq q_-$ and
	$mx/(4\pi)\geq M^2-\nu\geq n^2-\nu$.
	
	On the upper edge $y=D$, the triangle inequality and
	$s_*^L(-)=\overline{s_*^L(+)}-D+\kappa^L$ give
	\[
	|f_{t_0}^L(x+iD)-1|\leq Q_{M_-,M_+}^{\mathrm{top}},
	\]
	where
	$ Q_{M_-,M_+}^{\mathrm{top}}:=
	\sum_{n=2}^{M_+}\frac{|\chi_m(n)|b_n}{n^{\sigma_-(D)}}+e^{a_mD}\sum_{n=1}^{M_+}
	\frac{|\chi_m(n)|b_n\,n^{D+k_-(D)-\sigma_-(D)}}
	{\max\{q_-,n^2-\nu\}^{D/2}}.$
	
	For the approximation error on this block, take
	\begin{equation}\label{eq:b_top_error_local}
		\varepsilon_{M_-,M_+}^{\mathrm{top}}
		:=\varepsilon_{t_0}^{L,[M_+]}(x_-,D).
	\end{equation}
	Indeed, $M\leq M_+$ on the block, with endpoint values continued
	from its branches, and the majorant in \eqref{eq:epsilon_RS_finite}
	is non-increasing in $x$.
	Writing $Q=Q_{M_-,M_+}^{\mathrm{top}}$, it suffices
	to verify, on every block,
	\begin{equation}\label{eq:b_top_condition}
		\begin{split}
			&Q<1,\qquad\Theta_P+\arcsin Q<\pi,\\
			&L_P(1-Q)
			\cos\left(\left(\Theta_P+\arcsin Q-\frac\pi2\right)_+\right)
			>U_P\varepsilon_{M_-,M_+}^{\mathrm{top}}.
		\end{split}
	\end{equation}
	
	It remains to treat the lower edge $y=y_0$.
	Fix a block and abbreviate
	\[
	\sigma_-:=\sigma_-(y_0),\qquad
	r:=g_-(y_0),\qquad k_-:=k_-(y_0).
	\]
	We have $0<r<1$ in the present parameter range, since
	$q_->e^{2a_m}$ and $y_0>0$.
	For $d\mid Q_P$, define
	$ \lambda_d:=\prod_{p\mid d}(-b_p)$,
	then
	\[
	E(z)=\sum_{d\mid Q_P}\frac{\chi_m(d)\lambda_d}{d^s}.
	\]
	For each integer $M$ in the block, define the real coefficients
	\[
		\begin{split}
			\beta_n^{(M)}&:=\chi_m(n)
			\sum_{\substack{d\mid\gcd(n,Q_P)\\n/d\leq M}}
			\lambda_d b_{n/d},\\
			\alpha_n^{(M)}&:=\chi_m(n)
			\sum_{\substack{d\mid\gcd(n,Q_P)\\n/d\leq M}}
			\lambda_d(n/d)^{y_0}b_{n/d}.
		\end{split}
	\]
	Thus $\beta_1^{(M)}=\alpha_1^{(M)}=1$, and complete
	multiplicativity of $\chi_m$ gives
	\[
	\begin{split}
		E(z)\sum_{n=1}^M\frac{\chi_m(n)b_n}{n^s}
		&=\sum_{n=1}^{Q_P M}\frac{\beta_n^{(M)}}{n^s},\\
		E(z)\sum_{n=1}^M\frac{\chi_m(n)n^{y_0}b_n}{n^s}
		&=\sum_{n=1}^{Q_P M}\frac{\alpha_n^{(M)}}{n^s}.
	\end{split}
	\]
	In particular, $\beta_p^{(M)}=0$ for $p\in\mathcal P$.
	The restriction $n/d\leq M$ is kept in both coefficients.
	
	Using $|n^{-\kappa^L}-1|\leq n^{|\kappa^L|}-1$, we may write
	\[
	\begin{aligned}
		f_{t_0}^L(z)
		={}&\sum_{n=1}^M\frac{\chi_m(n)b_n}{n^s}
		+\gamma^L\sum_{n=1}^M
		\frac{\chi_m(n)n^{y_0}b_n}{n^{\overline s}}\\
		&+O_{\leq}\left(
		|\gamma^L|\sum_{n=1}^M
		\frac{|\chi_m(n)|n^{y_0}b_n}{n^\sigma}
		(n^{|\kappa^L|}-1)\right).
	\end{aligned}
	\]
	Since $\chi_m$ is real, the second sum is the conjugate of
	the corresponding sum with exponent $s$. If this latter sum is $A$,
	then $|E\overline A|=|E A|$. Thus multiplication by the possibly complex
	factor $E$ preserves the modulus estimate needed here. Consequently,
	\begin{equation}\label{eq:b_two_polynomials}
		\begin{split}
			E(z)f_{t_0}^L(z)
			={}&\sum_{n=1}^{Q_P M}\frac{\beta_n^{(M)}}{n^s}\\
			&+O_{\leq}\left(
			r\left|\sum_{n=1}^{Q_P M}\frac{\alpha_n^{(M)}}{n^s}\right|\right)
			+O_{\leq}(Z_{M_-,M_+}),
		\end{split}
	\end{equation}
	where \eqref{eq:b_gamma_n_refined} permits the uniform choice
	\[
		Z_{M_-,M_+}:=
		U_P e^{a_my_0}\sum_{n=1}^{M_+}
		\frac{|\chi_m(n)|b_n\,n^{y_0-\sigma_-}
			(n^{k_-}-1)}
		{\max\{q_-,n^2-\nu\}^{y_0/2}}.
	\]
	
	We now apply the argument of \cite[Lemma 8.5]{Poly}.
	For fixed $M$, set
	\[
	Y_M(\sigma):=
	\sum_{n=2}^{Q_P M}\frac1{n^\sigma}
	\max\left\{
	|\beta_n^{(M)}-r\alpha_n^{(M)}|,
	\frac{1-r}{1+r}|\beta_n^{(M)}+r\alpha_n^{(M)}|
	\right\}.
	\]
	Then
	\begin{equation}\label{eq:b_improved_triangle}
		\operatorname{dist}\bigl(E(z)f_{t_0}^L(z),\mathcal L\bigr)
		\geq1-r-Y_M(\sigma)-Z_{M_-,M_+}.
	\end{equation}
	For completeness, if $Y_M(\sigma)\geq1-r$, this is immediate.
	Otherwise, write the first two terms of
	\eqref{eq:b_two_polynomials} as
	\[
	\sum_{n=1}^{Q_P M}
	\frac{\beta_n^{(M)}+r\zeta\alpha_n^{(M)}}{n^s},
	\qquad |\zeta|\leq1.
	\]
	For real numbers $\beta,\alpha$, the maximum modulus principle
	and the cosine rule give
	\[
	\left|\frac{\beta+r\zeta\alpha}{1+r\zeta}\right|
	\leq\max\left\{
	\frac{|\beta-r\alpha|}{1-r},
	\frac{|\beta+r\alpha|}{1+r}\right\}.
	\]
	Indeed, on $|\zeta|=1$ the square of the quotient is a
	fractional linear function of $\operatorname{Re}\zeta$ with
	positive denominator, so its maximum occurs at $\zeta=\pm1$.
	Factoring out $1+r\zeta$, the polynomial has the form
	$(1+r\zeta)(1+e)$, where
	$|e|\le Y_M(\sigma)/(1-r)<1$.
	Since $\operatorname{Re}(1+r\zeta)>0$, we have
	$\operatorname{dist}(1+r\zeta,\mathcal L)=|1+r\zeta|$.
	The triangle inequality therefore gives
	\[
	\begin{aligned}
	\operatorname{dist}\bigl((1+r\zeta)(1+e),\mathcal L\bigr)
	&\ge |1+r\zeta|\,(1-|e|)\\
	&\ge |1+r\zeta|\left(1-\frac{Y_M(\sigma)}{1-r}\right)\\
	&\ge 1-r-Y_M(\sigma).
	\end{aligned}
	\]
	Subtracting the remaining error bound $Z_{M_-,M_+}$ proves
	\eqref{eq:b_improved_triangle}.

	To make this estimate uniform in $M$, extend the coefficients by
	zero for $n>Q_P M$ and put
	\[
		\begin{split}
			F_{M_-,M_+}:={}&~{} 1-r \\
			&~{}
			-\sum_{n=2}^{Q_P M_+}\frac1{n^{\sigma_-}}
			\max_{\substack{M\in\mathbb Z\\M_-\leq M\leq M_+}}
			\max\left\{
			|\beta_n^{(M)}-r\alpha_n^{(M)}|,
			\frac{1-r}{1+r}|\beta_n^{(M)}+r\alpha_n^{(M)}|
			\right\}.
		\end{split}
	\]
	For each $n$, the coefficients change only when $M$ crosses an
	integer $n/d$ with $d\mid\gcd(n,Q_P)$. Thus the maximum need only
	be evaluated at $M_-$ and those integers $n/d$ belonging to the
	block. Since $\sigma\geq\sigma_-$, we obtain
	\[
	\operatorname{dist}\bigl(E(z)f_{t_0}^L(z),\mathcal L\bigr)
	\geq F_{M_-,M_+}-Z_{M_-,M_+}.
	\]
	On the lower edge, take
	\begin{equation}\label{eq:b_bottom_error_local}
		\varepsilon_{M_-,M_+}^{\mathrm{bottom}}
		:=\varepsilon_{t_0}^{L,[M_+]}(x_-,y_0).
	\end{equation}
	The remaining condition is
	\begin{equation}\label{eq:b_bottom_condition}
		F_{M_-,M_+}-Z_{M_-,M_+}
		>U_P\varepsilon_{M_-,M_+}^{\mathrm{bottom}}.
	\end{equation}
	
	These choices keep the dependence on $m,t_0$ and the relevant
	boundary coordinates. In \eqref{eq:b_top_error_local} and
	\eqref{eq:b_bottom_error_local}, the simpler choices
	$\varepsilon_{t_0}^L(x_-,D)$ and $\varepsilon_{t_0}^L(x_-,y_0)$
	are also admissible; they avoid the additional finite sum in
	\eqref{eq:epsilon_AB_finite}. Numerical evaluation errors are
	included when checking the finite sums and inequalities.
	At an internal change of branch, the estimates are checked for
	the corresponding lateral values; either valid approximation
	then bounds the same continuous exact quotient.
	
	If \eqref{eq:b_right_condition}, \eqref{eq:b_left_condition},
	\eqref{eq:b_top_condition}, and \eqref{eq:b_bottom_condition}
	hold for the chosen parameters and blocks, then
	\eqref{eq:b_cut_condition} holds on all four sides of $R_b$.
	The argument principle therefore gives
	\[
	H_{t_0}^L(z)\neq0\qquad(z\in R_b).
	\]
	Together with the estimate for region (c), this proves non-vanishing for
	$x\geq X+\Delta_D$ and $y_0\leq y\leq D$, and hence
	on the smaller region in hypothesis (2) of Lemma~\ref{lemmabox}.
	Thus, these finite verification conditions establish hypothesis (2).
	\begin{remark}
		The argument for (b) may instead be carried out up to $y=1$, as in
		\cite[\S8.5]{Poly}. In that case, the estimates on that rectangle and
		on the corresponding tail are evaluated with upper cap $1$, including
		the factors $G_m$, $\mathfrak G_m$, and $\kappa_1$. This proves a
		stronger non-vanishing statement and is compatible with the smaller
		strip required by Lemma~\ref{lemmabox}. 
	\end{remark}
	
	\begin{remark}
		Let $q=y_0^2+2t_0$ and suppose $q<D^2\le1$. Then
		$\Delta_D=\sqrt{D^2-q}\le\Delta_1=\sqrt{1-q}$ and
		$U_D(t)=\sqrt{D^2-2t}\le U_1(t)$. Thus region~(a) for $D$
		is contained in its $D=1$ counterpart. At time $t_0$, the part
		$X+\Delta_D\le x\le X+\Delta_1$ of region~(b) for $D$ lies in
		the final-time slice of region~(a) for $D=1$, since
		$y_0\le y\le\sqrt{D^2-2t_0}\le\sqrt{1-2t_0}$. The remaining
		part, $x\ge X+\Delta_1$, is covered by the $D=1$ checks of
		regions~(b) and~(c). Region~(c) for $D$ is also contained in the
		 tail defined for $D=1$. Therefore, one full ``certificate'' at $D=1$, together
		with the exact check $q<D^2$, certifies the three regions for every
		smaller admissible cap. Direct $D$-specific endpoint checks may be
		used to report sharper margins, but they are not needed for this
		inclusion argument.
	\end{remark}
	
	\subsection{Proof of Theorem~\ref{MT}}

For fixed $m,\chi_m,D$, call $(t_0,y_0,X)$ admissible if it satisfies
the parameter restrictions above, with $X$ in \eqref{Xm_search}, and
the inequalities in regions (a)--(c) are verified for some meshes,
finite blocks and auxiliary cutoff $M_1$. Define
\begin{equation}\label{eq:lambda_star}
\Lambda_*(m,\chi_m;D):=
\inf\left\{t_0+\frac12y_0^2:
(t_0,y_0,X)\text{ is admissible}\right\},
\end{equation}
with $\inf\varnothing=+\infty$.

For any admissible triple, Platt's verification gives hypothesis (1)
of Lemma~\ref{lemmabox}. Regions (a), (b) and (c) give its other two
hypotheses: the fixed-$M$ barrier checks cover every time, the finite
blocks cover all $M(X,t_0)\le M<M_1$, and the monotone tail estimate
covers all $M\ge M_1$. Corollary~\ref{cor:upperbound} therefore gives
$\Lambda_L\le t_0+y_0^2/2$. Taking the infimum and combining this with
the initial-strip bound $\Lambda_L\le D^2/2$ proves Theorem~\ref{MT}.
\qed

\section{Numerical results}\label{sec:numerical_implementation_main}

We apply Lemma~\ref{lemmabox} to the real primitive characters
$\chi_d(n)=(d/n)$ with $3\le |d|\le400$, where $d$ is a fundamental
discriminant and $m=|d|$. Each character is treated separately, including
both signs when they give the same conductor. We consider
$D\in\{3/4,1\}$: the conclusion for $D=3/4$ assumes the corresponding
global strip from Section~\ref{DynSection}, whereas $D=1$ uses the
classical critical strip. We first describe six small-conductor examples
and then summarize the full experiment.

The search supplies $X$ and rational parameters $t_0,y_0$; it does not
prove that these choices minimize $t_0+y_0^2/2$. For a selected triple,
Arb/Acb ball arithmetic~\cite{JohanssonArb} encloses the finite sums,
arguments and error bounds. The region-(a) checks use 256-bit precision;
the region-(b) and (c) verifier records its own working precision.
The figures show sampled values and are used to illustrate the checks.
The positive interval margins, rather than the samples, establish
the numerical inequalities required by the proof.

\subsection{Choice of parameters}

For each character, we search the interval \eqref{Xm_search} using the
truncated Euler product. The candidate $X$ is retained as \emph{an exact decimal
input}. We then screen rational pairs,
first rejecting those for which $y_0^2+2t_0\ge D^2$, and testing the
finite-block and tail bounds before the more expensive barrier check.
The pointwise bounds from Section~\ref{sec:verification} are evaluated
at the endpoints prescribed by their monotonicities.

For the six characters with $3\le m\le8$, the chosen pair is the same
for $D=3/4$ and $D=1$. Table~\ref{tab:numerical_selected_parameters}
lists these six pairs. Here $\lambda_{\mathrm{pair}}=t_0+y_0^2/2$ and
$M_a=3989$ throughout the barrier. The largest value is $0.28097$,
for $d=-7$, so every pair is strictly admissible for both values of $D$
considered here. The table and the comparison figures use these same
refined parameters.

\begin{table}[ht]
\centering\small\setlength{\tabcolsep}{4pt}\renewcommand{\arraystretch}{1.15}
\begin{tabular}{c c c r r r}
\toprule
$(m,d)$ & $t_0$ & $y_0$ & $\lambda_{\mathrm{pair}}$ & $X/10^7$ & $M_1$\\
\midrule
$(3,-3)$ & $\frac{5211}{20000}$ & $\frac{27}{200}$ & 0.269662500 & 6.666576989 & 314182\\
$(4,-4)$ & $\frac{5031}{20000}$ & $\frac{13}{100}$ & 0.260000000 & 4.999950908 & 530969\\
$(5,5)$ & $\frac{1351}{5000}$ & $\frac{267}{2000}$ & 0.279111125 & 3.999939338 & 241678\\
$(7,-7)$ & $\frac{6743}{25000}$ & $\frac{3}{20}$ & 0.280970000 & 2.857074726 & 185906\\
$(8,-8)$ & $\frac{5111}{20000}$ & $\frac{19}{200}$ & 0.260062500 & 2.499946550 & 1166540\\
$(8,8)$ & $\frac{5111}{20000}$ & $\frac{251}{2000}$ & 0.263425125 & 2.499959779 & 408437\\
\bottomrule
\end{tabular}
\vspace{0.2cm}
\caption{Refined parameter choices for the six small-conductor characters, used for $D=3/4$ and $D=1$. The $X$ entries are abbreviated; the full decimal inputs are kept with the verification records.}
\label{tab:numerical_selected_parameters}
\end{table}

\subsection{Numerical verification}

The numerical checks follow Section~\ref{sec:verification}.
PARI/GP is used to construct trial meshes and boundary polygons;
the function values, error bounds and argument increments are then
recomputed with Arb/Acb. The verification includes the Taylor
remainder from Remark~\ref{rem:fast_E50}, the spatial interpolation
error and the variation over each complete time cell. The winding
enclosures isolate zero, and the strict boundary inequalities keep
the exact quotient away from zero throughout the cell.

In region (a), we use the stripewise corner bounds for $e_A+e_B$
from Section~\ref{sec:verification}, with the selected cap $D$.
For $e_{C,0}$, we use
\eqref{eq:C0_uniform_log} in the ordinate-uniform envelope.
On each time cell, this bound is evaluated at the earliest time
and the lowest ordinate, attained at the later endpoint.
For regions (b) and (c), we use the pointwise bound
\eqref{eq:epsilon_C_local} with the coarser choice $3.33\sqrt m$.
Thus the type-$C$ bounds in the different regions need not coincide.

At $t=t_0$, we also check the full left edge $x=X$,
$y_0\le y\le D$ to obtain the modulus and phase bounds needed
in region (b); the shorter moving edge alone would not suffice.
The other checks cover the right edge, the top and bottom of every
finite $M$-block, and the analytic tail. We keep a pair only after
all three regions have passed with the same exact inputs, including
the covered integer ranges and both tail monotonicity conditions.

When a valid verification for the same pair covers $D=1$,
we can reuse it for $D=3/4$ when the pair satisfies its strip constraint. Indeed,
$\Delta_D\le\Delta_1$ and the smaller moving barrier is contained
in the larger one. The part of region (b) between $X+\Delta_D$ and
$X+\Delta_1$ is covered by the final-time slice of the larger barrier;
the remaining part is covered by its regions (b) and (c).
This is an inclusion of verified domains, not a change in $f_t^L$.
A different pair requires a new check, or an explicitly verified
inclusion in a previously verified box.

\subsection{Region (a)}

Table~\ref{tab:numerical_region_a_inherited} gives the enclosing
region-(a) checks for the six selected pairs. All winding enclosures
isolate zero and all endpoint and time-cell margins are positive.
For $d=-7$, the barrier is contained in the previously verified fixed
box with $t_{\mathrm{old}}=0.3$ and $y_{\mathrm{old}}=0.1$; its margin
is therefore a margin for that enclosing box. The other five rows
are direct moving-barrier checks for the pairs in
Table~\ref{tab:numerical_selected_parameters}.

\begin{table}[ht]
\centering\small\setlength{\tabcolsep}{4pt}\renewcommand{\arraystretch}{1.15}
\begin{tabular}{c l r r r}
\toprule
$(m,d)$ & enclosure & time nodes & $N_{\mathrm{seg}}$ & interval margin\\
\midrule
$(3,-3)$ & moving barrier & 283 & 1180 & 0.102334\\
$(4,-4)$ & moving barrier & 382 & 1207 & 0.0772214\\
$(5,5)$ & moving barrier & 247 & 1188 & 0.0908628\\
$(7,-7)$ & fixed box & 233 & 1383 & 0.123669\\
$(8,-8)$ & moving barrier & 421 & 1414 & 0.0638487\\
$(8,8)$ & moving barrier & 391 & 1232 & 0.0764153\\
\bottomrule
\end{tabular}
\vspace{0.2cm}
\caption{Region-(a) ``certificates'' at $D=1$, with $M=3989$. Lower bounds are rounded down. The same enclosure covers $D=3/4$.}
\label{tab:numerical_region_a_inherited}
\end{table}

Figures~\ref{fig:minimum-ft-D-m4} and~\ref{fig:minimum-ft-D-m5}
show the beta function and the even real character of conductor $5$.
For a fixed character the function $f_{t,M}^L$ is the same for both
caps; the boundary on which it is sampled changes with $D$.
We sample 61 times and 21 points per side. These meshes are chosen for
the figures and need not be nested as $D$ varies, so their sampled
minima need not be monotone in $D$.
The lower panels show stepwise Arb error bounds on enclosing time
cells. The bound for $e_A+e_B$ uses the later endpoint, and that for
$e_{C,0}$ uses the earlier endpoint. The sampled minima and the error
maxima are listed in Table~\ref{tab:region_a_sampled_D_minima}.

\begin{table}[ht]
\centering\small\setlength{\tabcolsep}{4pt}\renewcommand{\arraystretch}{1.15}
\begin{tabular}{c c r r r}
\toprule
$(m,d)$ & $D$ & sampled minimum & $\max(e_A+e_B)$ & $\max e_{C,0}$\\
\midrule
$(3,-3)$ & $\frac{3}{4}$ & 2.019646 & \(3.190\cdot10^{-6}\) & 0.0543803\\
$(3,-3)$ & $1$ & 1.076577 & \(3.191\cdot10^{-6}\) & 0.0543947\\
$(4,-4)$ & $\frac{3}{4}$ & 0.943717 & \(4.320\cdot10^{-6}\) & 0.0594610\\
$(4,-4)$ & $1$ & 0.946920 & \(4.321\cdot10^{-6}\) & 0.0594780\\
$(5,5)$ & $\frac{3}{4}$ & 2.016876 & \(5.342\cdot10^{-6}\) & 0.0698363\\
$(5,5)$ & $1$ & 1.213751 & \(5.344\cdot10^{-6}\) & 0.0698738\\
$(7,-7)$ & $\frac{3}{4}$ & 2.345719 & \(7.110\cdot10^{-6}\) & 0.0776015\\
$(7,-7)$ & $1$ & 1.276207 & \(7.112\cdot10^{-6}\) & 0.0776636\\
$(8,-8)$ & $\frac{3}{4}$ & 1.597806 & \(9.635\cdot10^{-6}\) & 0.0956282\\
$(8,-8)$ & $1$ & 0.740726 & \(9.638\cdot10^{-6}\) & 0.0956838\\
$(8,8)$ & $\frac{3}{4}$ & 0.935318 & \(8.763\cdot10^{-6}\) & 0.0842697\\
$(8,8)$ & $1$ & 0.901197 & \(8.766\cdot10^{-6}\) & 0.0843187\\
\bottomrule
\end{tabular}
\vspace{0.2cm}
\caption{Sampled boundary minima and outward-rounded error maxima for the selected pairs. The minima are diagnostics, not interval lower bounds.}
\label{tab:region_a_sampled_D_minima}
\end{table}

\begin{figure}[p]
\centering
\includegraphics[width=0.88\linewidth]{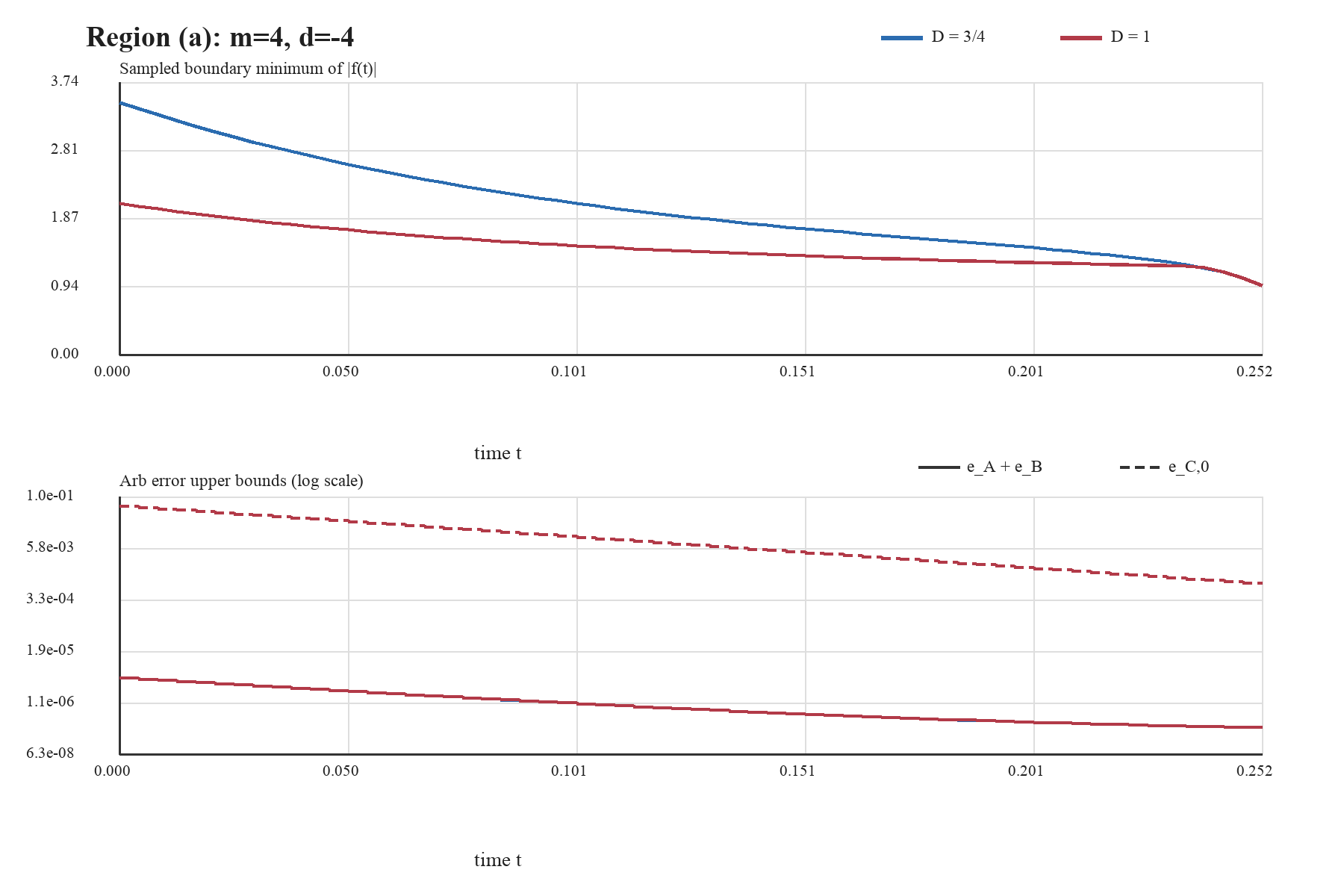}
\caption{Sampled boundary minima and stepwise error bounds for $m=4$, $d=-4$, with the parameters in Table~\ref{tab:numerical_selected_parameters}.}
\label{fig:minimum-ft-D-m4}
\end{figure}

\begin{figure}[p]
\centering
\includegraphics[width=0.88\linewidth]{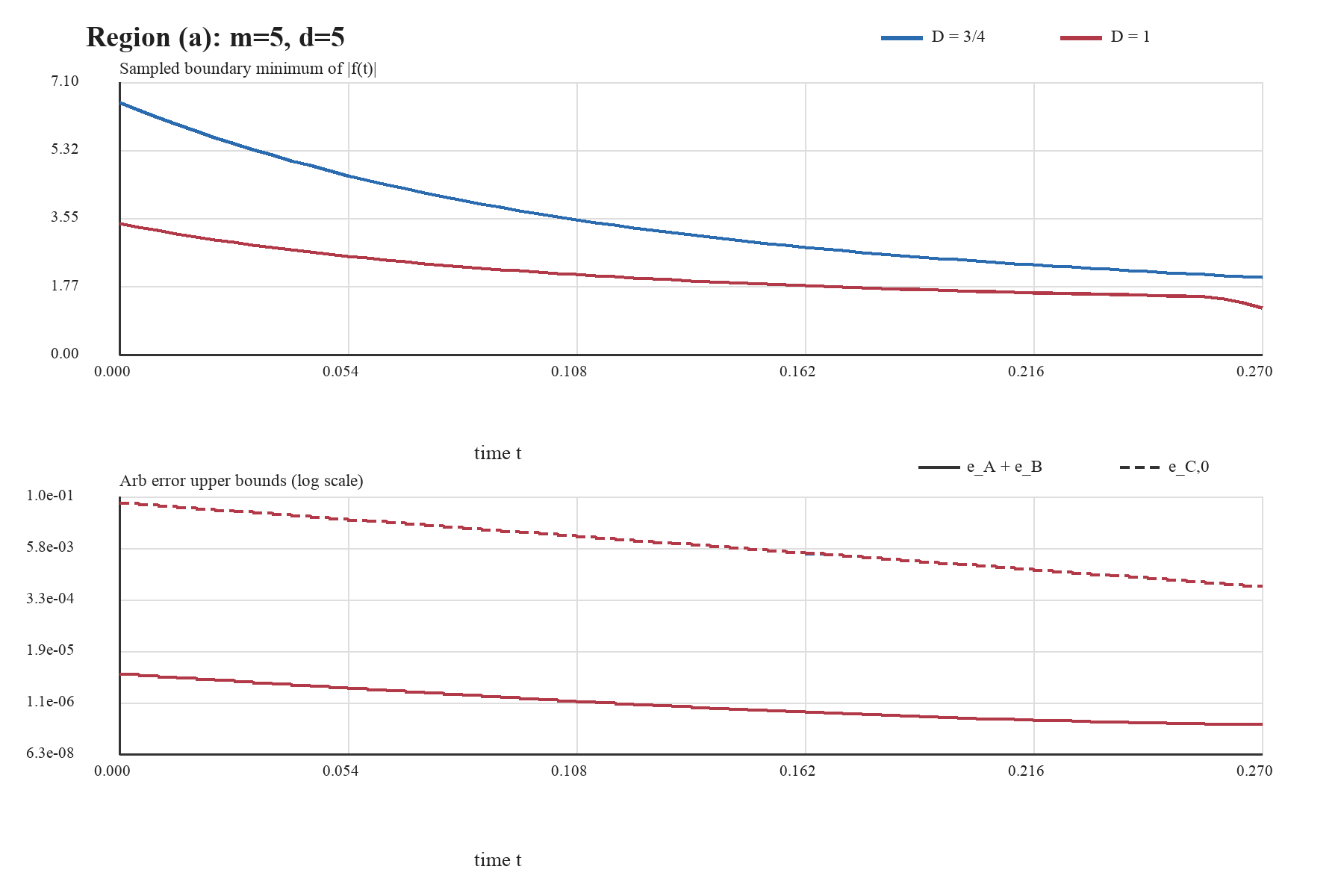}
\caption{Sampled boundary minima and stepwise error bounds for $m=5$, $d=5$, with the parameters in Table~\ref{tab:numerical_selected_parameters}.}
\label{fig:minimum-ft-D-m5}
\end{figure}

\subsection{Regions (b) and (c)}

For each selected pair, the full left edge $x=X$, $y_0\le y\le1$
is checked at $t=t_0$. Its modulus and phase bounds are shown in
Table~\ref{tab:region_b_left_edges_D}. Using the full $D=1$ edge permits
the same valid verification to cover $D=3/4$. The finite blocks then
verify all four boundary inequalities of region (b). Their minimum
clearances and largest errors are recorded in
Table~\ref{tab:region_b_D1_margins}.

\begin{table}[ht]
\centering\small\setlength{\tabcolsep}{4pt}\renewcommand{\arraystretch}{1.15}
\begin{tabular}{c r r r r}
\toprule
$(m,d)$ & $L_a$ lower & $\psi_a$ upper & $e_A+e_B$ upper & $e_{C,0}$ upper\\
\midrule
$(3,-3)$ & 1.43189 & 0.198833 & \(6.087\cdot10^{-6}\) & \(6.177\cdot10^{-4}\)\\
$(4,-4)$ & 1.23428 & 0.529492 & \(7.977\cdot10^{-6}\) & \(7.885\cdot10^{-4}\)\\
$(5,5)$ & 1.52686 & 0.572701 & \(1.056\cdot10^{-5}\) & \(6.723\cdot10^{-4}\)\\
$(7,-7)$ & 1.54203 & 0.0655680 & \(1.403\cdot10^{-5}\) & \(7.534\cdot10^{-4}\)\\
$(8,-8)$ & 1.23185 & 0.0745629 & \(1.805\cdot10^{-5}\) & 0.001185\\
$(8,8)$ & 1.22057 & 0.480456 & \(1.642\cdot10^{-5}\) & 0.001044\\
\bottomrule
\end{tabular}
\vspace{0.2cm}
\caption{Full-height Arb left-edge checks for region (b), at the selected final times.}
\label{tab:region_b_left_edges_D}
\end{table}

\begin{table}[ht]
\centering\small\setlength{\tabcolsep}{4pt}\renewcommand{\arraystretch}{1.15}
\resizebox{\textwidth}{!}{%
\begin{tabular}{c r r r r r r r}
\toprule
$(m,d)$ & blocks & left & right & top & bottom & $\max(e_A+e_B)$ & $\max e_{C,0}$\\
\midrule
$(3,-3)$ & 61 & 0.8327 & 0.02201 & 0.2706 & \(4.657\cdot10^{-4}\) & \(2.005\cdot10^{-7}\) & 0.001186\\
$(4,-4)$ & 66 & 0.8643 & 0.04420 & 0.4819 & 0.001247 & \(2.782\cdot10^{-7}\) & 0.001632\\
$(5,5)$ & 60 & 0.8168 & 0.02573 & 0.1992 & \(2.762\cdot10^{-5}\) & \(3.245\cdot10^{-7}\) & 0.001306\\
$(7,-7)$ & 53 & 0.7543 & 0.02327 & 0.1634 & \(2.470\cdot10^{-5}\) & \(4.479\cdot10^{-7}\) & 0.001456\\
$(8,-8)$ & 77 & 0.8604 & 0.05120 & 0.4882 & \(4.683\cdot10^{-4}\) & \(5.698\cdot10^{-7}\) & 0.002492\\
$(8,8)$ & 66 & 0.9092 & 0.009337 & 0.5204 & \(3.599\cdot10^{-5}\) & \(5.501\cdot10^{-7}\) & 0.002196\\
\bottomrule
\end{tabular}}
\vspace{0.2cm}
\caption{Verified region-(b) clearances and finite-block error bounds for the selected pairs. Positive clearances keep the exact boundary image away from the nonpositive real axis.}
\label{tab:region_b_D1_margins}
\end{table}

The finite blocks cover every integer from $M(X,t_0)=3989$ through
$M_1-1$. At each truncation boundary, the adjacent fixed-$M$ branches
are interpreted by continuation and bound the exact same quotient 
$H_{t_0}^L/B_{t_0}^L$. The analytic tail covers every $M\ge M_1$,
so there is no gap between the finite and infinite parts.
Table~\ref{tab:region_c_D1_margins} gives the tail data: $Q_c$ bounds
$|f_{t_0}^L-1|$, $E_c$ bounds the approximation error, and the
clearance is a lower bound for $1-Q_c-E_c$. Both tail monotonicity
conditions hold in every row.

\begin{table}[ht]
\centering\small\setlength{\tabcolsep}{4pt}\renewcommand{\arraystretch}{1.15}
\begin{tabular}{c r r r r}
\toprule
$(m,d)$ & $M_1$ & $Q_c$ upper & $E_c$ upper & tail clearance lower\\
\midrule
$(3,-3)$ & 314182 & 0.961523 & \(2.571\cdot10^{-7}\) & 0.0384772\\
$(4,-4)$ & 530969 & 0.937036 & \(1.393\cdot10^{-7}\) & 0.0629643\\
$(5,5)$ & 241678 & 0.950710 & \(4.118\cdot10^{-7}\) & 0.0492905\\
$(7,-7)$ & 185906 & 0.950111 & \(8.041\cdot10^{-7}\) & 0.0498886\\
$(8,-8)$ & 1166540 & 0.926977 & \(3.459\cdot10^{-8}\) & 0.0730234\\
$(8,8)$ & 408437 & 0.987428 & \(3.061\cdot10^{-7}\) & 0.0125726\\
\bottomrule
\end{tabular}
\vspace{0.2cm}
\caption{Arb bounds for region (c) at $D=1$. These bounds also cover $D=3/4$.}
\label{tab:region_c_D1_margins}
\end{table}

\subsection{Resulting bounds}

The six pairs provide unconditional improvements on $1/2$ at $D=1$.
They are also admissible under the $D=3/4$ strip hypothesis, with
the same bound. Corollary~\ref{cor:upperbound} gives
\[
\begin{aligned}
\Lambda_{L_{-3}}&\le0.2696625,&\Lambda_{L_{-4}}&\le0.26,\\
\Lambda_{L_5}&\le0.279111125,&\Lambda_{L_{-7}}&\le0.28097,\\
\Lambda_{L_{-8}}&\le0.2600625,&\Lambda_{L_8}&\le0.263425125.
\end{aligned}
\]
The largest of these values is $0.28097<9/32$.

\subsection{The character \texorpdfstring{$\chi_{-11}$}{chi(-11)}}

For this character, we choose $X=18181753.573965154588223$ and
$(t_0,y_0)=(7/25,11/100)$. Thus, $q=y_0^2+2t_0=5721/10000$
and $\lambda_{\mathrm{pair}}=5721/20000=0.28605$.
This pair is admissible for $D=1$, but not for $D=3/4$.
For $D=3/4$, the strip bound $D^2/2=9/32$ is retained; this does not
rule out another pair giving an improvement.

\begin{table}[ht]
\centering\small\setlength{\tabcolsep}{4pt}\renewcommand{\arraystretch}{1.15}
\begin{tabular}{c c c c l}
\toprule
$D$ & $D^2-q$ & admissible pair bound & upper bound for $\Lambda_L$ & result\\
\midrule
$3/4$ & $-6/625$ & -- & 0.28125 & strip bound\\
$1$ & $4279/10000$ & 0.28605 & 0.28605 & verified pair\\
\bottomrule
\end{tabular}
\vspace{0.2cm}
\caption{Conductor-$11$ bounds for the two initial strip hypotheses considered here.}
\end{table}

The $D=1$ check has $M=3989$ throughout region (a), 248 time nodes,
1321 spatial subdivisions, and a verified time-cell margin greater
than $0.0946608343$. Region (b) has 62 finite blocks and a bottom-edge
clearance greater than $5.0770\cdot10^{-5}$. Region (c) starts at
$M_1=314182$; $Q_c<0.932216$ and $E_c<2.647\cdot10^{-7}$ leave
a clearance greater than $0.0677839251$. As for the six preceding examples, the global
Taylor remainder is below $1.496\cdot10^{-10}$.

\subsection{Conductors up to \texorpdfstring{$400$}{400}}

There are 242 nonprincipal real primitive characters with $3\le m\le400$,
at 222 distinct conductors. We treat the two signs of a fundamental
discriminant separately whenever both occur at the same conductor.
The range $201\le m\le400$ was processed in three batches,
201--300, 301--350 and 351--400, containing 62, 32 and 26
characters, respectively.

A verification at $D=1$ can be restricted to a smaller admissible
cap as described above; otherwise, a separate search and verification
are required. Only pairs for which all three regions have passed,
either directly or by a verified inclusion, can enter the  results.

Table~\ref{tab:m400_summary} distinguishes the comparison made in
Appendix~\ref{app:certified_m400} from the separate searches at $D=3/4$.
The appendix uses the chosen pair at $D=1$ for both values of $D$.
Among these 242 reference pairs, 136 give a bound below $9/32$;
the other 106 give a bound above it. Allowing different pairs in the
$D=3/4$ search gives 149 improvements, leaving 93 characters for
which the search keeps the strip bound.

\begin{table}[ht]
\centering\small\setlength{\tabcolsep}{4pt}\renewcommand{\arraystretch}{1.5}
\begin{tabular}{c l r r r r}
\toprule
$D$ & pair choice & improvements & chosen strip & min. $\lambda_{\mathrm{pair}}$ & max. $\lambda_{\mathrm{pair}}$\\
\midrule
$1$ & $D=1$ reference & 242 & 0 & 0.22780 & 0.29782\\
$\frac{3}{4}$ & $D=1$ reference & 136 & 106 & 0.22780 & 0.28097\\
$\frac{3}{4}$ & $D$-specific search & 149 & 93 & 0.22780 & 0.28105\\
\bottomrule
\end{tabular}
\vspace{0.2cm}
\caption{Results for all 242 characters with the two choices of pairs. Improvements are strict relative to $D^2/2$; the last two columns range only over improving pairs and are rounded upwards to five decimal places. The strip bound at $D=3/4$ assumes the corresponding global strip.}
\label{tab:m400_summary}
\end{table}

For $D=1$, every character in this range has a verified improvement
on $1/2$. The largest chosen pair bound is exactly $953/3200$,
so the computations give
\[
\Lambda_{L_d}\le\frac{953}{3200}=0.2978125
\qquad(3\le |d|\le400,\ d\text{ fundamental}).
\]
The individual reference pairs are listed in
Table~\ref{tab:certified_pairs_m400}.

Some searches retain different pairs for different caps.
Table~\ref{tab:m365_D_pairs} gives an example. For $d=365$, the
$D=3/4$ search keeps a pair giving $0.280778125<9/32$, whereas
the $D=1$ reference pair gives $0.284278125>9/32$. The appendix
repeats the latter pair in both rows, so its negative difference
does not rule out the improvement from the separate search.

\begin{table}[ht]
\centering\small\setlength{\tabcolsep}{4pt}\renewcommand{\arraystretch}{1.5}
\begin{tabular}{c c c r}
\toprule
$D$ & $t_0$ & $y_0$ & $\lambda_{\mathrm{pair}}$\\
\midrule
$\frac{3}{4}$ & $\frac{11}{40}$ & $\frac{43}{400}$ & 0.280778125\\
$1$ & $\frac{7}{25}$ & $\frac{37}{400}$ & 0.284278125\\
\bottomrule
\end{tabular}\vspace{0.2cm}
\caption{Retained verified pairs for $\chi_{365}$ under the two initial-strip hypotheses considered here.}
\label{tab:m365_D_pairs}
\end{table}

The 484 rows of Table~\ref{tab:certified_pairs_m400} are comparisons
for 242 characters, rather than separate parameter choices for each cap.
Every row retains the rational parameters of the $D=1$ reference pair,
including when its bound exceeds $9/32$. The column $D^2/2$ gives the
strip bound. The final column gives $9/32-\lambda_{\mathrm{pair}}$
for $D=3/4$ and a dash for $D=1$. Pair bounds are rounded upwards
and differences downwards to five decimal places, using the exact
rational parameters.

The exact inputs and verification records are retained separately
for each character and cap. The separate searches provide
13 additional improvements at $D=3/4$ beyond those obtained by
restricting the reference pairs. These finite searches do not establish
global optimality or a limiting conductor for the method.

The source code, exact input parameters and detailed verification
records are retained in a private repository and can be made
available to referees during peer review.

\clearpage

\section{Distribution of zeros of \texorpdfstring{$H_t^L$}{H(t,L)}}\label{sec:Nyg}

The effective Riemann--Siegel approximation also gives information on the real zeros of $H_t^L$ for $t>0$. In this section, we use it to locate the zeros at large $x$ and derive a Riemann--von Mangoldt-type counting formula with error estimates that are uniform in the stated conductor range. 

\begin{theorem}[Distribution of zeros of $H_t^L$]\label{Distributionzero}
Let $\chi_m$ be a primitive real Dirichlet character with
$3\le m\le400000$, and let $0<t\le1/2$. Let $C>0$ be a
sufficiently large constant and $c>0$ a sufficiently small constant.
For $x\ge4\pi$, put
\[
 g_L(x,t)=\frac{x}{4\pi}\log\frac{mx}{4\pi}
 -\frac{x}{4\pi}+\frac{2a+3}{8}
 +\frac{t}{16}\log\frac{mx}{4\pi}.
\]
For every integer $n\ge C$, let $x_n>4\pi$ be the unique solution
of $g_L(x_n,t)=n$. Write $N_t^L(X)$ for the number of zeros
$\rho$ of $H_t^L$ with $0<\operatorname{Re}\rho\le X$, counted
with multiplicity. Then the following hold.
\begin{enumerate}
\item If $x\ge\exp(C/t)$ and $H_t^L(x+iy)=0$, then $y=0$ and
$x=x_n+O(x^{-ct})$ for some $n$.
\item For each $n\ge\exp(C/t)$, the disk
\[
 \left\{z:\ |z-x_n|\le\frac{c}{\log(mx_n)}\right\}
\]
contains exactly one zero of $H_t^L$, counted with multiplicity.
This zero is real and lies within $O(x_n^{-ct})$ of $x_n$.
\item If $X\ge\exp(C/t)$, then
\[
 N_t^L(X)=g_L(X,t)+O(1).
\]
\item For every $X\ge0$,
\[
 N_t^L(X+1)-N_t^L(X)\ll\log(2+X),
\]
and for every $X\ge4\pi$,
\[
 N_t^L(X)=g_L(X,t)+O(\log(2+X)).
\]
\end{enumerate}
The constants are uniform in $t$ and in the character in the
stated conductor range.
\end{theorem}

	\begin{remark}
		Expanding the definition of $g_L$, part~(iv) reads
		\[
		N_t^L(X)= \frac{X}{4\pi}\log\left(\frac{Xm}{4\pi}\right)-\frac{X}{4\pi}+\frac{2a+3}{8}+\frac{t}{16}\log\left(\frac{Xm}{4\pi}\right)+O(\log(2+X)).
		\]
		The leading terms have the same form as in \cite[Theorem~1.5]{Poly}. The phase function used there has a constant term $11/8$ and the familiar $7/8$ appears in the heuristic averaged expression discussed there. The case $m=1$ is not covered here, since the completed zeta function includes the additional factor $s(s-1)$.
	\end{remark}

	The proof of Theorem~\ref{Distributionzero} relies on the effective Riemann--Siegel approximation of Theorem~\ref{EffectiveRSa}, which sharpens (and makes effective) the results of \cite{KKL}.

	We start with preliminary asymptotic estimates. Throughout this
	section, the conductor is restricted to $3\le m\le400000$.
	Unless otherwise stated, the implied constants are uniform in
	the primitive real character and in $0<t\le1/2$.
	
	\begin{prop}[Preliminary asymptotics]\label{PreAsymp}
		Let $\chi_m$ be a primitive real Dirichlet character with
		$3\le m\le400000$. Let $0<t\le1/2$, $x\ge480$, and
		$-10\le y\le10$, with $s_+=(1+y-ix)/2$ and
		$s_-=(1-y+ix)/2$. Then the following hold.
		\begin{enumerate}
			\item
			If $x\ge\exp(C/t)$ for a sufficiently large constant $C$,
			then
			\[
			H_t^L(x+iy)
			=
			\left(1+O(x^{-ct})\right)M_t^L(s_+)
			+
			\left(1+O(x^{-ct})\right)M_t^L(s_-)
			\]
			for a sufficiently small constant $c>0$.
			\item
			If $3\le y\le4$ and $x\ge C$ for a sufficiently large
			constant $C$, then
			\[
			H_t^L(x+iy)
			=
			\left(1+O_{\leq}(0.7)\right)M_t^L(s_+).
			\]
			
			\smallskip
			\item
			If $x=x_0+O(1)$ for some $x_0\ge480$, then
			\[
			H_t^L(x+iy)=O\left(x_0^{O(1)}|M_t^L(s_+)|\right)= O\left(x_0^{O(1)}|M_t^L(ix_0/2)|\right).
			\]
		\end{enumerate}
		The constants are independent of $t$ and of the character in
		the stated conductor range.
	\end{prop}
	
\begin{proof}
We follow \cite[proof of Proposition~9.1]{Poly}.
We use the estimates underlying Theorem~\ref{EffectiveRSa}
on the enlarged strip $|y|\le10$, with adjusted constants.
The real parts of $s_\pm$ remain bounded, and
Lemma~\ref{lem:rtn_estimate} applies to arbitrary real parts.

For the integral remainder, put $T=x/2$, $T'=T+\pi t/8$,
$\delta=\sqrt{mT'/(2\pi)}$, and $u=\sigma+\sqrt t\,v$.
In the proof of Proposition~\ref{prop:RtN_rigorous}, choose
\[
 J(u)=
 \begin{cases}
 1,&u\ge0,\\
 \lceil1-u\rceil,&u<0.
 \end{cases}
\]
Then $J(u)+u\ge1$ for $u<0$, so
Theorem~\ref{thm:RSJk_bound} applies. Since
$J(u)\le |u|+2$ in this range, the coefficient and remainder
bounds give
\[
 \sum_{j=1}^{J(u)}
 \frac{|C_j^{(m)}(p,u)|}{\delta^j}
 +|RS_{J(u)}^{(m)}(u+iT')|
 \ll \frac{\varphi(m)}{\sqrt{T'}}
 \exp\left(C_1(1+|u|)\log(2+|u|)\right).
\]
The normalized prefactor in the same proof is $e^{\eta_v}$,
where $|\eta_v|\le\frac{u^2/2+5/12}{T'-1}.$

For bounded $\sigma$, its Gaussian average is $1+O(T^{-1})$.
The preceding majorant, multiplied by
$e^{|\eta_v|}e^{-v^2}$, is integrable uniformly for sufficiently
large $T$ and $0<t\le1/2$. Thus the normalized continuous
remainder is
\[
 C_0^{(m)}(p)
 +O\left(
 \frac{\mathcal B_m^\star}{T}
 +\frac{\varphi(m)}{\sqrt T}
 \right).
\]
This extends the required estimate to bounded real
parts. All constants below are uniform in the stated conductor range.

We first consider $0\le y\le10$. Put $q=mx/(4\pi)$,
$\ell=\log q$, and $M=\lfloor\delta\rfloor$. Recall that
$s_*^L(\pm)=s_\pm+\frac t2\alpha_L(s_\pm)$ and
$\gamma^L=M_t^L(s_-)/M_t^L(s_+)$. The symmetric decomposition gives
\begin{equation}\label{eq:PreAsymp_decomposition}
 \frac{H_t^L(x+iy)}{B_t^L(x+iy)}
 =
 \sum_{n=1}^{M}
 \frac{\chi_m(n)b_n^{t,L}}{n^{s_*^L(+)}}+\gamma^L\sum_{n=1}^{M}
 \frac{\chi_m(n)b_n^{t,L}}{n^{s_*^L(-)}}
 +O(e_A+e_B+e_{C,0}),
\end{equation}
where the error bounds use the adjusted constants on this strip.

The explicit formula for $\alpha_L$ gives, uniformly for bounded
$\sigma$,
$\alpha_L(\sigma+ix/2)=\ell/2+\pi i/4+O(x^{-1})$.
Consequently,
 $|\gamma^L|=O(q^{-y/2})$, $\operatorname{Re}s_*^L(\pm)
 =\frac{1\pm y}{2}+\frac t4\ell+O(t/x^2)$,
 $s_*^L(-)=\overline{s_*^L(+)}-y+\kappa^L$ and $\kappa^L=O(t/x)$.

Since $n\le M=O(\sqrt q)$, the factors
$n^{O(t/x^2)}$ are uniformly bounded. The error estimates yield
\[
 \begin{aligned}
 e_A&\ll
 \frac{\ell^2}{x}q^{-y/2}
 \sum_{n=1}^{M}
 b_n^{t,L}n^{-(1-y)/2-(t/4)\ell},\\
 e_B&\ll
 \frac{\ell^2}{x}
 \sum_{n=1}^{M}
 b_n^{t,L}n^{-(1+y)/2-(t/4)\ell},\\
 e_{C,0}&\ll
 q^{-(1+y)/4}e^{-t\ell^2/16}
 \left(\mathcal B_m^\star+\frac{\varphi(m)}{\sqrt x}\right).
 \end{aligned}
\]
Using $q^{-y/2}n^y=O(1)$, we obtain
\begin{equation}\label{eq:PreAsymp_errors}
 e_A+e_B\ll
 \frac{\ell^2}{x}
 \sum_{n=1}^{M}
 b_n^{t,L}n^{-(1+y)/2-(t/4)\ell}.
\end{equation}
The fixed conductor range also gives
$\ell=\log x+O(1)$ and
$\mathcal B_m^\star+\varphi(m)/\sqrt x=O(1)$.

\smallskip
\noindent
Proof of (i).
Removing the $n=1$ term from both sums in
\eqref{eq:PreAsymp_decomposition}, and using
$b_n^{t,L}=\exp((t/4)\log^2n)$, gives, for sufficiently large $x$,
\[
 \frac{H_t^L(x+iy)}{B_t^L(x+iy)}
 =1+\gamma^L+
 O\left(
 \frac{\ell^2}{x}
 +\sum_{n=2}^{M}
 n^{-(1+y)/2-(t/4)(\ell-\log n)}
 +x^{-1/4}
 \right).
\]
Here we absorbed the factor $\ell^2/x$ multiplying the 
sum, since it is bounded for large $x$.

Now
$\log n\le\log M\le
\ell/2+\frac12\log(1+\pi t/(4x))$.
Hence, for sufficiently large $x$,
\[
 \frac{1+y}{2}+\frac t4(\ell-\log n)
 \ge\frac{t}{16}\log x
 \qquad(1\le n\le M).
\]
Put $r=(t/16)\log x$. If $x\ge\exp(C/t)$ with $C$ sufficiently
large, then $r\ge2$, and the integral test gives
\[
 \sum_{n=2}^{M}
 n^{-(1+y)/2-(t/4)(\ell-\log n)}
 \le\sum_{n=2}^{\infty}n^{-r}
 \le2^{-r}+\frac{2^{1-r}}{r-1}
 \le3x^{-(\log2)t/16}.
\]
For sufficiently small $c>0$, we also have
$\ell^2/x+x^{-1/4}=O(x^{-ct})$, uniformly for $0<t\le1/2$.
Therefore
\[
 H_t^L(x+iy)
 =M_t^L(s_+)+M_t^L(s_-)
 +O\left(x^{-ct}|M_t^L(s_+)|\right)
 \qquad(0\le y\le10).
\]
Conjugation gives the corresponding estimate for $y<0$, with
$|M_t^L(s_-)|$ in the error term. This proves (i).

\smallskip
\noindent
Proof of (ii).
Assume $3\le y\le4$. For sufficiently large $x$, we have
$M\le q$, and hence
\[
 b_n^{t,L}n^{-(t/4)\ell}
 =\exp\left(-\frac t4\log n\,(\ell-\log n)\right)
 \le1
 \qquad(1\le n\le M).
\]
The real-part estimate for $s_*^L(+)$ therefore gives
\[
 \left|
 \sum_{n=2}^{M}
 \frac{\chi_m(n)b_n^{t,L}}{n^{s_*^L(+)}}
 \right|
 \le\sum_{n=2}^{\infty}n^{-2+\epsilon_x}
 =\zeta(2)-1+O(x^{-2}),
\]
where $0\le\epsilon_x\ll x^{-2}$. For the second sum,
\[
 \begin{aligned}
 \left|
 \gamma^L\sum_{n=1}^{M}
 \frac{\chi_m(n)b_n^{t,L}}{n^{s_*^L(-)}}
 \right|
 &\ll q^{-y/2}\sum_{n=1}^{M}n^{(y-1)/2}\\
 &\ll q^{-y/2}M^{(y+1)/2}
 \ll q^{(1-y)/4}
 =O(x^{-1/2}).
 \end{aligned}
\]
In this range, \eqref{eq:PreAsymp_errors} gives
$e_A+e_B=O(\log^2x/x)$, while $e_{C,0}=O(x^{-1})$.
Substitution into \eqref{eq:PreAsymp_decomposition} yields
\[
 \left|
 \frac{H_t^L(x+iy)}{B_t^L(x+iy)}-1
 \right|
 \le\zeta(2)-1+O(x^{-1/2}).
\]
Since $\zeta(2)-1=\pi^2/6-1<0.7$, increasing $C$ proves (ii).

\smallskip
\noindent
Proof of (iii).
For $0\le y\le10$, the preceding bounds and
\eqref{eq:PreAsymp_decomposition} give
\[
 H_t^L(x+iy)
 =O\left(x_0^{O(1)}|M_t^L(s_+)|\right),
\]
since $M=O(\sqrt x)$ and
$b_n^{t,L}n^{-(t/4)\ell}\le1$ for $n\le M$.

Moreover, from \eqref{alpha_L} and \eqref{MtL},
\[
 \frac{(M_t^L)'(s)}{M_t^L(s)}
 =\alpha_L(s)+\frac t2\alpha_L(s)\alpha_L'(s)
 =O(\log x_0)
 \qquad\left(s=\frac{ix_0}{2}+O(1)\right).
\]
Integrating a branch of $\log M_t^L$ along the segment joining
$ix_0/2$ to $\overline{s_+}$, and using conjugation, gives
\[
 \log\frac{|M_t^L(s_+)|}{|M_t^L(ix_0/2)|}
 =O(\log x_0).
\]
Thus these factors, including those corresponding to $y$ and $-y$,
are comparable up to powers of $x_0$. Together with
$H_t^L(x-iy)=\overline{H_t^L(x+iy)}$, this proves both estimates
for $|y|\le10$. Any remaining bounded range of $x$ is absorbed
by compactness.
\end{proof}
We next record the corresponding estimates for $M_t^L$.

\begin{lemma}[Behavior of $M_t^L$]\label{lem:Mtbehavior}
Let $0<t\le1/2$, let $x_*>0$ be sufficiently large, and suppose
that $x+iy=x_*+O(1)$. Then
\begin{equation}\label{MtL trans}
\begin{aligned}
 M_t^L\left(\frac{1+y+ix}{2}\right)
 ={}&M_t^L\left(\frac{1+ix_*}{2}\right)\\
 &\times\exp\left(
 (i(x-x_*)+y)
 \left(\frac14\log\frac{mx_*}{4\pi}+\frac{\pi i}{8}\right)
 +O\left(\frac{\log(mx_*)}{x_*}\right)
 \right).
\end{aligned}
\end{equation}
Moreover, there is a continuous branch of
$\arg M_t^L((1+ix_*)/2)$ for all sufficiently large $x_*$ such that
\begin{equation}\label{argML}
\begin{aligned}
 \arg M_t^L\left(\frac{1+ix_*}{2}\right)
 ={}&\frac{x_*}{4}\log\frac{mx_*}{4\pi}-\frac{x_*}{4}
 +\frac{(2a-1)\pi}{8}\\
 &+\frac{t\pi}{16}\log\frac{mx_*}{4\pi}
 +O\left(\frac{\log(mx_*)}{x_*}\right).
\end{aligned}
\end{equation}
The estimates are uniform in $0<t\le1/2$ and in the stated
conductor range, for a fixed bound in $x+iy=x_*+O(1)$.
\end{lemma}

\begin{proof}
For $s=ix_*/2+O(1)$, the explicit formula
\eqref{alpha_L} gives
\[
 \alpha_L(s)
 =\frac12\log\frac{mx_*}{4\pi}+\frac{\pi i}{4}
 +O(x_*^{-1}),
 \qquad
 \alpha_L'(s)=O(x_*^{-1}).
\]
Therefore
\[
 \frac{(M_t^L)'(s)}{M_t^L(s)}
 =\alpha_L(s)+\frac t2\alpha_L(s)\alpha_L'(s)
 =\frac12\log\frac{mx_*}{4\pi}+\frac{\pi i}{4}
 +O\left(\frac{\log(mx_*)}{x_*}\right).
\]
Integrating along the segment joining $(1+ix_*)/2$ to
$(1+y+ix)/2$ gives \eqref{MtL trans}.

For the second assertion, put $s_*=(1+ix_*)/2$ and
$w=(s_*+a)/2$. We use the continuous argument obtained from
the analytic logarithm defining $M_0^L$. Since
\[
 \log w
 =\log\frac{x_*}{4}+\frac{\pi i}{2}
 -\frac{i(1+2a)}{x_*}+O(x_*^{-2}),
\]
taking imaginary parts in \eqref{M0L} yields
\[
 \arg M_0^L(s_*)
 =\frac{x_*}{4}\log\frac{mx_*}{4\pi}
 -\frac{x_*}{4}
 +\frac{(2a-1)\pi}{8}
 +O(x_*^{-1}).
\]
Also, the preceding expansion of $\alpha_L$ gives
\[
 \operatorname{Im}\bigl(\alpha_L(s_*)^2\bigr)
 =\frac{\pi}{4}\log\frac{mx_*}{4\pi}
 +O\left(\frac{\log(mx_*)}{x_*}\right).
\]
By \eqref{MtL}, we can choose the continuous argument so that
\[
 \arg M_t^L(s_*)
 =\arg M_0^L(s_*)+\frac t4\operatorname{Im}(\alpha_L(s_*)^2).
\]
Combining the last two estimates proves \eqref{argML}.
\end{proof}
\begin{proof}[Proof of Theorem~\ref{Distributionzero}]
We follow the argument of \cite[\S9]{Poly}. Let $\Theta_t(v)$ denote
the continuous argument of $M_t^L((1+iv)/2)$ in
Lemma~\ref{lem:Mtbehavior}. By \eqref{argML},
\[
 \Theta_t(v)=\pi g_L(v,t)-\frac\pi2
 +O\left(\frac{\log(mv)}v\right).
\]
Throughout the proof, we may increase $C$ and decrease $c$.

We first prove (ii). Suppose that $x_n\ge\exp(C_1/t)$, where
$C_1$ is sufficiently large, and put
\[
 A_n=M_t^L\left(\frac{1+ix_n}{2}\right),\qquad
 \ell_n=\log\frac{mx_n}{4\pi},\qquad w=z-x_n.
\]
Since $g_L(x_n,t)=n$, the preceding phase estimate gives
$\overline{A_n}/A_n=-1+O(\log(mx_n)/x_n)$.
Proposition~\ref{PreAsymp}(i) and Lemma~\ref{lem:Mtbehavior}
therefore yield, for $|\operatorname{Re}w|\le1$ and $|y|\le1$,
\begin{equation}\label{eq:cosvanish}
 \frac{e^{\pi w/8}H_t^L(z)}{A_n}
 =2i\sin\left(\frac{\ell_n w}{4}\right)
 +O\left(x_n^{-c_0t}e^{\ell_n|y|/4}\right),
 \qquad z=x+iy,
\end{equation}
for a sufficiently small fixed $c_0>0$. Here the errors of size
$\log(mx_n)/x_n$ have been absorbed into $x_n^{-c_0t}$.

On the circle $|w|=c/\log(mx_n)$, the sine term has modulus
$\gg c$, whereas the error is smaller if $C_1$ is sufficiently
large. The sine has exactly one zero in this disk, and
Rouch\'e's theorem gives the same conclusion for $H_t^L$.
The zero is simple and real, since the disk is invariant under
conjugation and $H_t^L(\overline z)=\overline{H_t^L(z)}$.
This proves (ii).

Now suppose that $H_t^L(x+iy)=0$ and $x\ge\exp(C/t)$.
By Theorem~\ref{dBru}, $|y|\le1$. Choose an integer $n$ with
$|g_L(x,t)-n|\le1/2$. Since
\[
 \partial_x g_L(x,t)
 =\frac1{4\pi}\log\frac{mx}{4\pi}+\frac{t}{16x},
\]
we have $x-x_n=O(1/\ell_n)$ and
$|\ell_n(x-x_n)/4|\le\pi/2+o(1)$.
At this zero, comparison of the moduli of the two exponentials
in \eqref{eq:cosvanish} gives
$|y|\ell_n=O(x_n^{-c_0t})$. Substitution back into that equation
then gives
\[
 \sin\left(\frac{\ell_n(z-x_n)}4\right)=O(x_n^{-c_0t}),
 \qquad
 z-x_n=O\left(\frac{x_n^{-c_0t}}{\ell_n}\right).
\]
The choice of $n$ excludes the other zeros of the sine. Thus $z$
lies in one of the disks already treated above and is real.
Decreasing $c$ if necessary proves (i).

For (iii), first consider
$X=x_n+c/\log x_n$, with $c>0$ fixed and sufficiently small,
and $x_n\ge\exp(C_1/t)$.
Put $\ell_n=\log(mx_n/(4\pi))$,
$\ell_X=\log(mX/(4\pi))$, and
$\beta_n=c\ell_n/(2\log x_n)$.
The phase estimate and Lemma~\ref{lem:Mtbehavior} give
\[
 \frac{H_t^L(X+iy)}{B_t^L(X+iy)}
 =1-e^{i\beta_n}e^{-y\ell_X/2}
 +O(X^{-c_0t}),
 \qquad 0\le y\le3.
\]
Here $\beta_n\asymp c$ and $0<\beta_n<\pi/2$.
Since
$\operatorname{Re}(1-e^{i\beta_n}e^{-y\ell_X/2})
\ge1-\cos\beta_n\gg c^2$,
increasing $C_1$ places the quotient in the right half-plane.
Its change in argument on this segment is therefore $O(1)$.

On the upper edge, the corresponding change in argument
of $H_t^L/B_t^L$ is $O(1)$ by
Proposition~\ref{PreAsymp}(ii), with the fixed compact portion
handled by continuity. The left edge is positive by
Lemma~\ref{lem:positive_imaginary_axis}.
Applying the argument principle clockwise to
$0\le\operatorname{Re}z\le X$, $|\operatorname{Im}z|\le3$,
and using conjugation symmetry, we obtain
\[
 N_t^L(X)
 =-\frac1\pi\arg M_t^L\left(\frac{1-iX}{2}\right)+O(1)
 =\frac{\Theta_t(X)}\pi+O(1)
 =g_L(X,t)+O(1).
\]
The arguments are taken continuously along the contour.

For consecutive choices of $X$ of the above form,
the values of $g_L(X,t)$ differ by $O(1)$.
Monotonicity of $N_t^L$ and $g_L$ therefore extends the estimate
to every $X\ge\exp(C/t)$, after increasing $C$. This proves (iii).

Finally, we prove (iv). Put $h=7/2$. By
Proposition~\ref{PreAsymp}(ii) and Cauchy's estimate, the
logarithmic derivative of $H_t^L/B_t^L$ at $X+ih$ is $O(1)$
for sufficiently large $X$. Since
$(B_t^L)'/B_t^L(X+ih)=O(\log(2+X))$ by
\eqref{alpha_L} and \eqref{MtL}, we obtain
\[
 \frac{(H_t^L)'(X+ih)}{H_t^L(X+ih)}
 =O(\log(2+X)).
\]
The bounded range of $X$ follows by compactness, since the
strip bound excludes zeros on this line. Taking imaginary
parts in the Hadamard identity used in the proof of
Proposition~\ref{Dynz}, we conclude that
\[
 \sum_{\rho}
 \frac{h-\operatorname{Im}\rho}
 {(X-\operatorname{Re}\rho)^2+(h-\operatorname{Im}\rho)^2}
 =-\operatorname{Im}
 \frac{(H_t^L)'(X+ih)}{H_t^L(X+ih)}
 \ll\log(2+X),
\]
where the sum includes all zeros, counted with multiplicity.
Since $|\operatorname{Im}\rho|\le1$, every summand is positive,
and each zero with $X<\operatorname{Re}\rho\le X+1$
contributes $\geq\frac{1}{10}$. This proves the first assertion.

For the second assertion, take $X$ sufficiently large and
initially suppose that the vertical line through $X$ contains
no zero. Repeating the rectangle argument in (iii), it remains
to bound the change in argument of $H_t^L$ on
$\{X+iy:0\le y\le3\}$ by $O(\log X)$.
Choose $\theta$ so that $e^{i\theta}H_t^L(X+3i)>0$, and let
$J_X$ be the number of zeros of
$\operatorname{Re}(e^{i\theta}H_t^L(X+iy))$ on this segment.
The absolute value of the net change in argument is at most
$\pi(J_X+1)$.

By conjugation symmetry, these are zeros of the entire function
\[
 F_X(z):=e^{i\theta}H_t^L(z)
        +e^{-i\theta}H_t^L(2X-z).
\]
Moreover, $F_X(X+3i)=2|H_t^L(X+3i)|>0$, and
Proposition~\ref{PreAsymp}(ii)--(iii), together with
Lemma~\ref{lem:Mtbehavior}, gives
\[
 \max_{|z-(X+3i)|\le4}|F_X(z)|
 \ll X^{O(1)}|F_X(X+3i)|.
\]
Jensen's formula therefore bounds the number of zeros of
$F_X$ in $|z-(X+3i)|\le3$ by $O(\log X)$.
Thus $J_X=O(\log X)$, and the rectangle argument yields
\[
 N_t^L(X)=\frac{\Theta_t(X)}{\pi}+O(\log X)
        =g_L(X,t)+O(\log X).
\]
Limits from the right cover the excluded values of $X$.
The bounded range $X\ge4\pi$ follows from the local counting
bound and the boundedness of $g_L$ on compact intervals.
This completes (iv).
\end{proof}	
\begin{remark}
As in the scale discussion of \cite[\S~10]{Poly}, the large-$x$
estimates of Section~\ref{sec:Nyg} suggest that, for fixed
$m$, $\chi_m$, and $D$, the parameter $t_0$ should be on the scale
$C/\log X$. Since a verification up to height $h_m$ in the
$s$-plane gives $X\simeq2h_m$ in the $z$-plane, this suggests
$t_0\sim C/\log(2h_m)$. This is a scale estimate: the other
zero-free conditions must also be verified.

If $h_m$ is multiplied by a factor $R$, the model predicts
\[
 \frac{t_{0,\mathrm{sc}}(Rh_m)}{t_{0,\mathrm{sc}}(h_m)}
 \approx\frac{\log(2h_m)}{\log(2Rh_m)},
 \qquad
 t_{0,\mathrm{sc}}(h):=\frac{C}{\log(2h)}.
\]
For $m=4$, Platt's height is $h_4=25\,000\,000$, and the
parameters $t_{\mathrm{ref}}=0.25155$ and $y_0=0.13$ give
$t_{\mathrm{ref}}+\tfrac12 y_0^2=0.26$. Calibrating this
model at $(h_4,t_{\mathrm{ref}})$ and keeping $y_0=0.13$,
a bound of
$0.1$ would require $t_0=0.1-\tfrac12(0.13)^2=0.09155$.
The calibrated model then predicts
\[
 \frac{h'}{h_4}
 \approx
 (2h_4)^{\,t_{\mathrm{ref}}/0.09155-1}
 \approx10^{13.46},
 \qquad
 h'\approx 7.1\times 10^{20}.
\]
This is only an illustration of the logarithmic scale: the
constant in that scale is implicit, and the estimate gives
neither a necessary nor a sufficient height for the desired
bound. The zero-free regions in Lemma~\ref{lemmabox} would
still have to be verified.
\end{remark}
\begin{cor}
For fixed $0<t\le1/2$, as $X\to\infty$,
\[
 N_t^L(X+1)-N_t^L(X)
 =\frac1{4\pi}\log\frac{mX}{4\pi}+O(1).
\]
Moreover, if $\rho_n(t)$ denotes the unique zero in the disk around
$x_n(t)$ from Theorem~\ref{Distributionzero}(ii), then
\[
 \rho_n(t)=x_n(t)+O\bigl(x_n(t)^{-ct}\bigr),
 \qquad
 x_{n+1}(t)-x_n(t)
 =\frac{4\pi}{\log(mx_n(t)/(4\pi))}\bigl(1+o(1)\bigr).
\]
Consequently, the same leading spacing holds for consecutive zeros
$\rho_n(t)$ in this range.
\end{cor}

\begin{proof}
The first formula follows by subtracting Theorem~\ref{Distributionzero}(iii)
at $X+1$ and $X$, and applying the mean value theorem to $g_L$.
For the spacing, apply the mean value theorem between $x_n(t)$ and
$x_{n+1}(t)$.
\[
 \partial_xg_L(x,t)=\frac1{4\pi}\log\frac{mx}{4\pi}+\frac{t}{16x}.
\]
The zero-location estimate is Theorem~\ref{Distributionzero}(ii).
Since $x_n^{-ct}=o(1/\log x_n)$ for fixed $t>0$, these location errors
are smaller than the spacing and do not change its leading term.
\end{proof}

\begin{remark}
The term $\frac{t}{16}\log(mx/(4\pi))$ also describes the motion
of the centers $x_n(t)$, as in the discussion of
\cite[Remark~9.4]{Poly}. Since $\partial_xg_L>0$, these centers
are smooth in $t$. Differentiating $g_L(x_n(t),t)=n$ and writing
$\ell_n(t):=\log(mx_n(t)/(4\pi))$, we obtain
\[
 \begin{aligned}
 x_n'(t)
 &=-\frac{\ell_n(t)/16}
 {\ell_n(t)/(4\pi)+t/(16x_n(t))}
 =-\frac{\pi}{4}
 \frac{\ell_n(t)}{\ell_n(t)+\pi t/(4x_n(t))}\\
 &=-\frac{\pi}{4}
 +O\!\left(\frac{t}{x_n(t)\ell_n(t)}\right).
 \end{aligned}
\]
Thus the centers move to the left with asymptotic speed $\pi/4$
as $n\to\infty$. This calculation concerns the centers: the
localization estimate in Theorem~\ref{Distributionzero} alone
does not justify differentiating the error to obtain a velocity
estimate for the zeros themselves.
\end{remark}

Since $\Phi_L$ decays superexponentially, the integral defining
$H_t^L$ also converges for negative $t$.
The upper panel of Figure~\ref{fig:m4_zero_gL} illustrates the
motion of real zeros, as in \cite[Fig.~18]{Poly}.
The lower panel plots phase residuals for zeros of the leading
Riemann--Siegel term, following the type of plot in
\cite[Fig.~10]{Poly}. 

\begin{figure}[htbp]
\centering
\includegraphics[width=\textwidth]{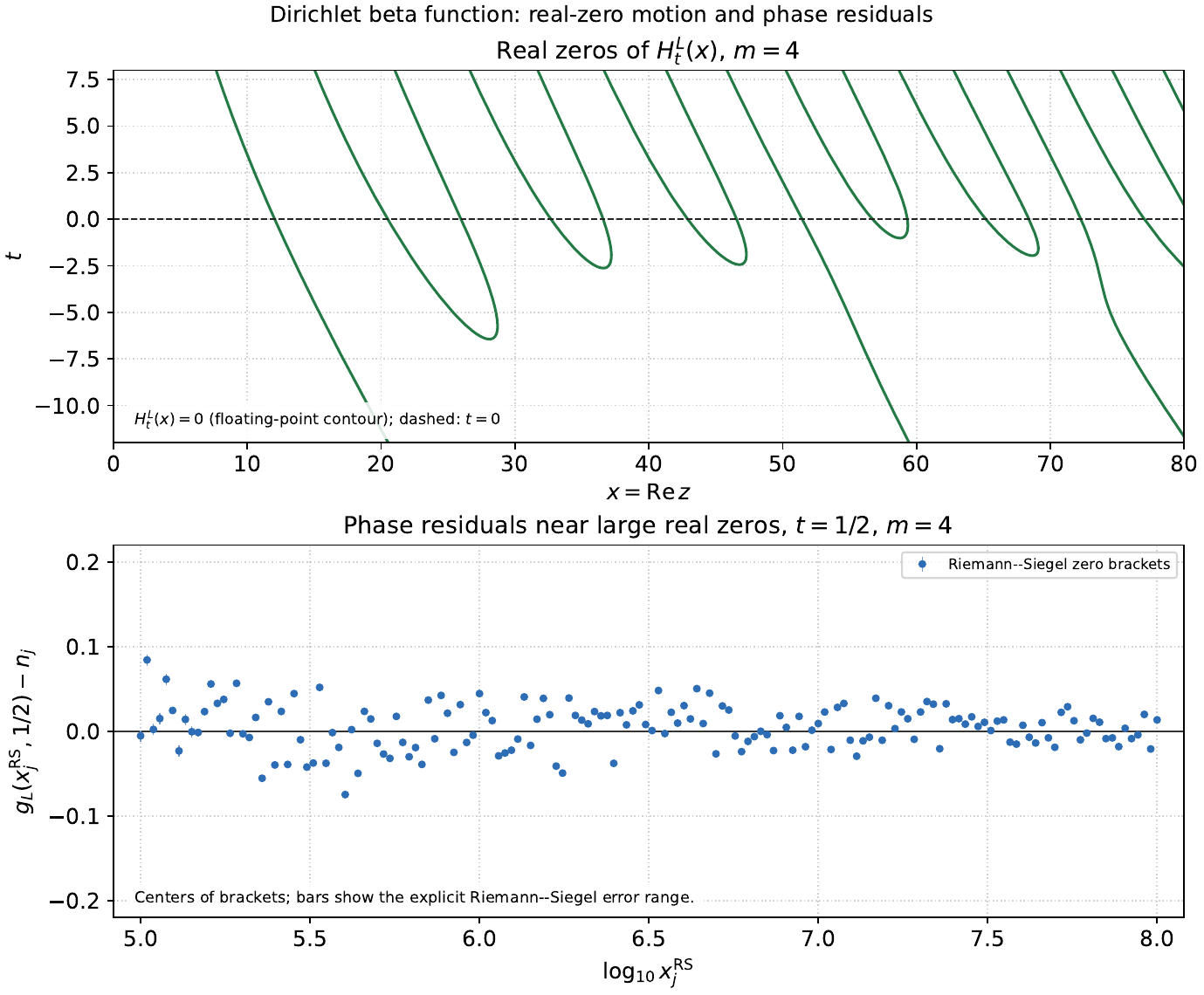}
\caption{Dirichlet beta function ($m=4$). Upper: numerical contour
$H_t^L(x)=0$ for $0\le x\le80$ and $-12\le t\le8$, computed from the
Fourier integral; the dashed line marks $t=0$. Lower: phase residuals for
large $x$. Each plotted center $x_j^{\mathrm{RS}}$ is a zero of the leading
Riemann--Siegel term, and $n_j$ is the nearest integer to
$g_L(x_j^{\mathrm{RS}},1/2)$.}
\label{fig:m4_zero_gL}
\end{figure}

\clearpage
\appendix
\section{Reference pairs and strip bounds}\label{app:certified_m400}

Table~\ref{tab:certified_pairs_m400} lists all 242 real primitive
characters with $3\le m\le400$, at 222 distinct conductors.
The discriminant $d$ identifies the character, so the two characters
at a common conductor are listed separately. Neither $400$ nor $-400$
is a fundamental discriminant.

For each character, the pair $(t_0,y_0)$ retained at $D=1$ is repeated
for $D=3/4$ and $D=1$, with
$\lambda_{\mathrm{pair}}=t_0+y_0^2/2$. Its verification at $D=1$
gives the unconditional bound $\Lambda_L\le\lambda_{\mathrm{pair}}$.
The second comparison bound, $9/32$, assumes the global initial strip
$|\operatorname{Im}z|\le3/4$ from Section~\ref{DynSection}.

The last column gives $9/32-\lambda_{\mathrm{pair}}$ for $D=3/4$
and a dash for $D=1$. It is positive for 136 characters and negative
for 106. A positive difference also gives the strict admissibility
condition $y_0^2+2t_0<9/16$, so the case $D=1$ can be
restricted as explained in Section~\ref{sec:verification}.
When the difference is negative, the displayed pair is retained only
as a reference and  the smaller strip bound is stronger. A different pair
may still improve on it, as the $d=365$ example in
Table~\ref{tab:m365_D_pairs} shows.

The table does not list the separate cap-dependent choices or assert
optimality of the reference pairs. Its 484 comparison rows therefore
give 136 improvements on $9/32$, whereas the separate search in
Table~\ref{tab:m400_summary} gives 149.
The parameters are exact rationals. Pair bounds are rounded upwards
to five decimal places, and differences downwards from their exact
rational values. The values of $D^2/2$ are exact.

{\footnotesize
\setlength{\tabcolsep}{3pt}
\renewcommand{\arraystretch}{1.6}
\begin{longtable}{@{\extracolsep{\fill}}r r c c c r r r@{}}
\caption{The $D=1$ reference pairs compared with the two strip bounds. The signed difference in the last column is shown only for $D=3/4$.}\label{tab:certified_pairs_m400}
\\
\toprule
$m$ & $d$ & $D$ & $t_0$ & $y_0$ & $\lambda_{\mathrm{pair}}$ & $D^2/2$ & $D^2/2-\lambda_{\mathrm{pair}}$\\
\midrule
\endfirsthead
\multicolumn{8}{c}{\tablename\ \thetable\ (continued)}\\
\toprule
$m$ & $d$ & $D$ & $t_0$ & $y_0$ & $\lambda_{\mathrm{pair}}$ & $D^2/2$ & $D^2/2-\lambda_{\mathrm{pair}}$\\
\midrule
\endhead
\midrule
\multicolumn{8}{r}{Continued on the next page}\\
\endfoot
\bottomrule
\endlastfoot
3 & -3 & $\frac{3}{4}$ & $\frac{5211}{20000}$ & $\frac{27}{200}$ & 0.26967 & 0.28125 & 0.01158\\
3 & -3 & $1$ & $\frac{5211}{20000}$ & $\frac{27}{200}$ & 0.26967 & 0.50000 & --\\
4 & -4 & $\frac{3}{4}$ & $\frac{5031}{20000}$ & $\frac{13}{100}$ & 0.26000 & 0.28125 & 0.02125\\
4 & -4 & $1$ & $\frac{5031}{20000}$ & $\frac{13}{100}$ & 0.26000 & 0.50000 & --\\
5 & 5 & $\frac{3}{4}$ & $\frac{1351}{5000}$ & $\frac{267}{2000}$ & 0.27912 & 0.28125 & 0.00213\\
5 & 5 & $1$ & $\frac{1351}{5000}$ & $\frac{267}{2000}$ & 0.27912 & 0.50000 & --\\
7 & -7 & $\frac{3}{4}$ & $\frac{6743}{25000}$ & $\frac{3}{20}$ & 0.28097 & 0.28125 & 0.00028\\
7 & -7 & $1$ & $\frac{6743}{25000}$ & $\frac{3}{20}$ & 0.28097 & 0.50000 & --\\
8 & -8 & $\frac{3}{4}$ & $\frac{5111}{20000}$ & $\frac{19}{200}$ & 0.26007 & 0.28125 & 0.02118\\
8 & -8 & $1$ & $\frac{5111}{20000}$ & $\frac{19}{200}$ & 0.26007 & 0.50000 & --\\
8 & 8 & $\frac{3}{4}$ & $\frac{5111}{20000}$ & $\frac{251}{2000}$ & 0.26343 & 0.28125 & 0.01782\\
8 & 8 & $1$ & $\frac{5111}{20000}$ & $\frac{251}{2000}$ & 0.26343 & 0.50000 & --\\
11 & -11 & $\frac{3}{4}$ & $\frac{7}{25}$ & $\frac{11}{100}$ & 0.28605 & 0.28125 & -0.00480\\
11 & -11 & $1$ & $\frac{7}{25}$ & $\frac{11}{100}$ & 0.28605 & 0.50000 & --\\
12 & 12 & $\frac{3}{4}$ & $\frac{47}{200}$ & $\frac{9}{100}$ & 0.23905 & 0.28125 & 0.04220\\
12 & 12 & $1$ & $\frac{47}{200}$ & $\frac{9}{100}$ & 0.23905 & 0.50000 & --\\
13 & 13 & $\frac{3}{4}$ & $\frac{11}{40}$ & $\frac{4}{25}$ & 0.28780 & 0.28125 & -0.00655\\
13 & 13 & $1$ & $\frac{11}{40}$ & $\frac{4}{25}$ & 0.28780 & 0.50000 & --\\
15 & -15 & $\frac{3}{4}$ & $\frac{49}{200}$ & $\frac{3}{20}$ & 0.25625 & 0.28125 & 0.02500\\
15 & -15 & $1$ & $\frac{49}{200}$ & $\frac{3}{20}$ & 0.25625 & 0.50000 & --\\
17 & 17 & $\frac{3}{4}$ & $\frac{7}{25}$ & $\frac{27}{200}$ & 0.28912 & 0.28125 & -0.00787\\
17 & 17 & $1$ & $\frac{7}{25}$ & $\frac{27}{200}$ & 0.28912 & 0.50000 & --\\
19 & -19 & $\frac{3}{4}$ & $\frac{7}{25}$ & $\frac{7}{50}$ & 0.28980 & 0.28125 & -0.00855\\
19 & -19 & $1$ & $\frac{7}{25}$ & $\frac{7}{50}$ & 0.28980 & 0.50000 & --\\
20 & -20 & $\frac{3}{4}$ & $\frac{47}{200}$ & $\frac{31}{200}$ & 0.24702 & 0.28125 & 0.03423\\
20 & -20 & $1$ & $\frac{47}{200}$ & $\frac{31}{200}$ & 0.24702 & 0.50000 & --\\
21 & 21 & $\frac{3}{4}$ & $\frac{51}{200}$ & $\frac{21}{200}$ & 0.26052 & 0.28125 & 0.02073\\
21 & 21 & $1$ & $\frac{51}{200}$ & $\frac{21}{200}$ & 0.26052 & 0.50000 & --\\
23 & -23 & $\frac{3}{4}$ & $\frac{57}{200}$ & $\frac{21}{200}$ & 0.29052 & 0.28125 & -0.00927\\
23 & -23 & $1$ & $\frac{57}{200}$ & $\frac{21}{200}$ & 0.29052 & 0.50000 & --\\
24 & -24 & $\frac{3}{4}$ & $\frac{23}{100}$ & $\frac{27}{200}$ & 0.23912 & 0.28125 & 0.04213\\
24 & -24 & $1$ & $\frac{23}{100}$ & $\frac{27}{200}$ & 0.23912 & 0.50000 & --\\
24 & 24 & $\frac{3}{4}$ & $\frac{23}{100}$ & $\frac{27}{200}$ & 0.23912 & 0.28125 & 0.04213\\
24 & 24 & $1$ & $\frac{23}{100}$ & $\frac{27}{200}$ & 0.23912 & 0.50000 & --\\
28 & 28 & $\frac{3}{4}$ & $\frac{49}{200}$ & $\frac{11}{100}$ & 0.25105 & 0.28125 & 0.03020\\
28 & 28 & $1$ & $\frac{49}{200}$ & $\frac{11}{100}$ & 0.25105 & 0.50000 & --\\
29 & 29 & $\frac{3}{4}$ & $\frac{7}{25}$ & $\frac{3}{20}$ & 0.29125 & 0.28125 & -0.01000\\
29 & 29 & $1$ & $\frac{7}{25}$ & $\frac{3}{20}$ & 0.29125 & 0.50000 & --\\
31 & -31 & $\frac{3}{4}$ & $\frac{7}{25}$ & $\frac{3}{20}$ & 0.29125 & 0.28125 & -0.01000\\
31 & -31 & $1$ & $\frac{7}{25}$ & $\frac{3}{20}$ & 0.29125 & 0.50000 & --\\
33 & 33 & $\frac{3}{4}$ & $\frac{1}{4}$ & $\frac{4}{25}$ & 0.26280 & 0.28125 & 0.01845\\
33 & 33 & $1$ & $\frac{1}{4}$ & $\frac{4}{25}$ & 0.26280 & 0.50000 & --\\
35 & -35 & $\frac{3}{4}$ & $\frac{13}{50}$ & $\frac{7}{50}$ & 0.26980 & 0.28125 & 0.01145\\
35 & -35 & $1$ & $\frac{13}{50}$ & $\frac{7}{50}$ & 0.26980 & 0.50000 & --\\
37 & 37 & $\frac{3}{4}$ & $\frac{7}{25}$ & $\frac{31}{200}$ & 0.29202 & 0.28125 & -0.01077\\
37 & 37 & $1$ & $\frac{7}{25}$ & $\frac{31}{200}$ & 0.29202 & 0.50000 & --\\
39 & -39 & $\frac{3}{4}$ & $\frac{51}{200}$ & $\frac{27}{200}$ & 0.26412 & 0.28125 & 0.01713\\
39 & -39 & $1$ & $\frac{51}{200}$ & $\frac{27}{200}$ & 0.26412 & 0.50000 & --\\
40 & -40 & $\frac{3}{4}$ & $\frac{6}{25}$ & $\frac{3}{25}$ & 0.24720 & 0.28125 & 0.03405\\
40 & -40 & $1$ & $\frac{6}{25}$ & $\frac{3}{25}$ & 0.24720 & 0.50000 & --\\
40 & 40 & $\frac{3}{4}$ & $\frac{6}{25}$ & $\frac{3}{25}$ & 0.24720 & 0.28125 & 0.03405\\
40 & 40 & $1$ & $\frac{6}{25}$ & $\frac{3}{25}$ & 0.24720 & 0.50000 & --\\
41 & 41 & $\frac{3}{4}$ & $\frac{7}{25}$ & $\frac{31}{200}$ & 0.29202 & 0.28125 & -0.01077\\
41 & 41 & $1$ & $\frac{7}{25}$ & $\frac{31}{200}$ & 0.29202 & 0.50000 & --\\
43 & -43 & $\frac{3}{4}$ & $\frac{7}{25}$ & $\frac{31}{200}$ & 0.29202 & 0.28125 & -0.01077\\
43 & -43 & $1$ & $\frac{7}{25}$ & $\frac{31}{200}$ & 0.29202 & 0.50000 & --\\
44 & 44 & $\frac{3}{4}$ & $\frac{6}{25}$ & $\frac{4}{25}$ & 0.25280 & 0.28125 & 0.02845\\
44 & 44 & $1$ & $\frac{6}{25}$ & $\frac{4}{25}$ & 0.25280 & 0.50000 & --\\
47 & -47 & $\frac{3}{4}$ & $\frac{7}{25}$ & $\frac{4}{25}$ & 0.29280 & 0.28125 & -0.01155\\
47 & -47 & $1$ & $\frac{7}{25}$ & $\frac{4}{25}$ & 0.29280 & 0.50000 & --\\
51 & -51 & $\frac{3}{4}$ & $\frac{13}{50}$ & $\frac{11}{100}$ & 0.26605 & 0.28125 & 0.01520\\
51 & -51 & $1$ & $\frac{13}{50}$ & $\frac{11}{100}$ & 0.26605 & 0.50000 & --\\
52 & -52 & $\frac{3}{4}$ & $\frac{1}{4}$ & $\frac{19}{200}$ & 0.25452 & 0.28125 & 0.02673\\
52 & -52 & $1$ & $\frac{1}{4}$ & $\frac{19}{200}$ & 0.25452 & 0.50000 & --\\
53 & 53 & $\frac{3}{4}$ & $\frac{7}{25}$ & $\frac{4}{25}$ & 0.29280 & 0.28125 & -0.01155\\
53 & 53 & $1$ & $\frac{7}{25}$ & $\frac{4}{25}$ & 0.29280 & 0.50000 & --\\
55 & -55 & $\frac{3}{4}$ & $\frac{13}{50}$ & $\frac{4}{25}$ & 0.27280 & 0.28125 & 0.00845\\
55 & -55 & $1$ & $\frac{13}{50}$ & $\frac{4}{25}$ & 0.27280 & 0.50000 & --\\
56 & -56 & $\frac{3}{4}$ & $\frac{49}{200}$ & $\frac{11}{100}$ & 0.25105 & 0.28125 & 0.03020\\
56 & -56 & $1$ & $\frac{49}{200}$ & $\frac{11}{100}$ & 0.25105 & 0.50000 & --\\
56 & 56 & $\frac{3}{4}$ & $\frac{49}{200}$ & $\frac{11}{100}$ & 0.25105 & 0.28125 & 0.03020\\
56 & 56 & $1$ & $\frac{49}{200}$ & $\frac{11}{100}$ & 0.25105 & 0.50000 & --\\
57 & 57 & $\frac{3}{4}$ & $\frac{13}{50}$ & $\frac{23}{200}$ & 0.26662 & 0.28125 & 0.01463\\
57 & 57 & $1$ & $\frac{13}{50}$ & $\frac{23}{200}$ & 0.26662 & 0.50000 & --\\
59 & -59 & $\frac{3}{4}$ & $\frac{7}{25}$ & $\frac{4}{25}$ & 0.29280 & 0.28125 & -0.01155\\
59 & -59 & $1$ & $\frac{7}{25}$ & $\frac{4}{25}$ & 0.29280 & 0.50000 & --\\
60 & 60 & $\frac{3}{4}$ & $\frac{43}{200}$ & $\frac{4}{25}$ & 0.22780 & 0.28125 & 0.05345\\
60 & 60 & $1$ & $\frac{43}{200}$ & $\frac{4}{25}$ & 0.22780 & 0.50000 & --\\
61 & 61 & $\frac{3}{4}$ & $\frac{7}{25}$ & $\frac{4}{25}$ & 0.29280 & 0.28125 & -0.01155\\
61 & 61 & $1$ & $\frac{7}{25}$ & $\frac{4}{25}$ & 0.29280 & 0.50000 & --\\
65 & 65 & $\frac{3}{4}$ & $\frac{53}{200}$ & $\frac{27}{200}$ & 0.27412 & 0.28125 & 0.00713\\
65 & 65 & $1$ & $\frac{53}{200}$ & $\frac{27}{200}$ & 0.27412 & 0.50000 & --\\
67 & -67 & $\frac{3}{4}$ & $\frac{7}{25}$ & $\frac{4}{25}$ & 0.29280 & 0.28125 & -0.01155\\
67 & -67 & $1$ & $\frac{7}{25}$ & $\frac{4}{25}$ & 0.29280 & 0.50000 & --\\
68 & -68 & $\frac{3}{4}$ & $\frac{49}{200}$ & $\frac{3}{20}$ & 0.25625 & 0.28125 & 0.02500\\
68 & -68 & $1$ & $\frac{49}{200}$ & $\frac{3}{20}$ & 0.25625 & 0.50000 & --\\
69 & 69 & $\frac{3}{4}$ & $\frac{13}{50}$ & $\frac{3}{25}$ & 0.26720 & 0.28125 & 0.01405\\
69 & 69 & $1$ & $\frac{13}{50}$ & $\frac{3}{25}$ & 0.26720 & 0.50000 & --\\
71 & -71 & $\frac{3}{4}$ & $\frac{7}{25}$ & $\frac{4}{25}$ & 0.29280 & 0.28125 & -0.01155\\
71 & -71 & $1$ & $\frac{7}{25}$ & $\frac{4}{25}$ & 0.29280 & 0.50000 & --\\
73 & 73 & $\frac{3}{4}$ & $\frac{57}{200}$ & $\frac{13}{100}$ & 0.29345 & 0.28125 & -0.01220\\
73 & 73 & $1$ & $\frac{57}{200}$ & $\frac{13}{100}$ & 0.29345 & 0.50000 & --\\
76 & 76 & $\frac{3}{4}$ & $\frac{49}{200}$ & $\frac{31}{200}$ & 0.25702 & 0.28125 & 0.02423\\
76 & 76 & $1$ & $\frac{49}{200}$ & $\frac{31}{200}$ & 0.25702 & 0.50000 & --\\
77 & 77 & $\frac{3}{4}$ & $\frac{53}{200}$ & $\frac{31}{200}$ & 0.27702 & 0.28125 & 0.00423\\
77 & 77 & $1$ & $\frac{53}{200}$ & $\frac{31}{200}$ & 0.27702 & 0.50000 & --\\
79 & -79 & $\frac{3}{4}$ & $\frac{57}{200}$ & $\frac{13}{100}$ & 0.29345 & 0.28125 & -0.01220\\
79 & -79 & $1$ & $\frac{57}{200}$ & $\frac{13}{100}$ & 0.29345 & 0.50000 & --\\
83 & -83 & $\frac{3}{4}$ & $\frac{57}{200}$ & $\frac{13}{100}$ & 0.29345 & 0.28125 & -0.01220\\
83 & -83 & $1$ & $\frac{57}{200}$ & $\frac{13}{100}$ & 0.29345 & 0.50000 & --\\
84 & -84 & $\frac{3}{4}$ & $\frac{9}{40}$ & $\frac{11}{100}$ & 0.23105 & 0.28125 & 0.05020\\
84 & -84 & $1$ & $\frac{9}{40}$ & $\frac{11}{100}$ & 0.23105 & 0.50000 & --\\
85 & 85 & $\frac{3}{4}$ & $\frac{669}{2500}$ & $\frac{13}{100}$ & 0.27605 & 0.28125 & 0.00520\\
85 & 85 & $1$ & $\frac{669}{2500}$ & $\frac{13}{100}$ & 0.27605 & 0.50000 & --\\
87 & -87 & $\frac{3}{4}$ & $\frac{13}{50}$ & $\frac{13}{100}$ & 0.26845 & 0.28125 & 0.01280\\
87 & -87 & $1$ & $\frac{13}{50}$ & $\frac{13}{100}$ & 0.26845 & 0.50000 & --\\
88 & -88 & $\frac{3}{4}$ & $\frac{49}{200}$ & $\frac{13}{100}$ & 0.25345 & 0.28125 & 0.02780\\
88 & -88 & $1$ & $\frac{49}{200}$ & $\frac{13}{100}$ & 0.25345 & 0.50000 & --\\
88 & 88 & $\frac{3}{4}$ & $\frac{49}{200}$ & $\frac{13}{100}$ & 0.25345 & 0.28125 & 0.02780\\
88 & 88 & $1$ & $\frac{49}{200}$ & $\frac{13}{100}$ & 0.25345 & 0.50000 & --\\
89 & 89 & $\frac{3}{4}$ & $\frac{57}{200}$ & $\frac{13}{100}$ & 0.29345 & 0.28125 & -0.01220\\
89 & 89 & $1$ & $\frac{57}{200}$ & $\frac{13}{100}$ & 0.29345 & 0.50000 & --\\
91 & -91 & $\frac{3}{4}$ & $\frac{669}{2500}$ & $\frac{29}{200}$ & 0.27812 & 0.28125 & 0.00313\\
91 & -91 & $1$ & $\frac{669}{2500}$ & $\frac{29}{200}$ & 0.27812 & 0.50000 & --\\
92 & 92 & $\frac{3}{4}$ & $\frac{49}{200}$ & $\frac{4}{25}$ & 0.25780 & 0.28125 & 0.02345\\
92 & 92 & $1$ & $\frac{49}{200}$ & $\frac{4}{25}$ & 0.25780 & 0.50000 & --\\
93 & 93 & $\frac{3}{4}$ & $\frac{13}{50}$ & $\frac{13}{100}$ & 0.26845 & 0.28125 & 0.01280\\
93 & 93 & $1$ & $\frac{13}{50}$ & $\frac{13}{100}$ & 0.26845 & 0.50000 & --\\
95 & -95 & $\frac{3}{4}$ & $\frac{669}{2500}$ & $\frac{13}{100}$ & 0.27605 & 0.28125 & 0.00520\\
95 & -95 & $1$ & $\frac{669}{2500}$ & $\frac{13}{100}$ & 0.27605 & 0.50000 & --\\
97 & 97 & $\frac{3}{4}$ & $\frac{57}{200}$ & $\frac{13}{100}$ & 0.29345 & 0.28125 & -0.01220\\
97 & 97 & $1$ & $\frac{57}{200}$ & $\frac{13}{100}$ & 0.29345 & 0.50000 & --\\
101 & 101 & $\frac{3}{4}$ & $\frac{57}{200}$ & $\frac{3}{20}$ & 0.29625 & 0.28125 & -0.01500\\
101 & 101 & $1$ & $\frac{57}{200}$ & $\frac{3}{20}$ & 0.29625 & 0.50000 & --\\
103 & -103 & $\frac{3}{4}$ & $\frac{57}{200}$ & $\frac{3}{20}$ & 0.29625 & 0.28125 & -0.01500\\
103 & -103 & $1$ & $\frac{57}{200}$ & $\frac{3}{20}$ & 0.29625 & 0.50000 & --\\
104 & -104 & $\frac{3}{4}$ & $\frac{1}{4}$ & $\frac{13}{100}$ & 0.25845 & 0.28125 & 0.02280\\
104 & -104 & $1$ & $\frac{1}{4}$ & $\frac{13}{100}$ & 0.25845 & 0.50000 & --\\
104 & 104 & $\frac{3}{4}$ & $\frac{1}{4}$ & $\frac{13}{100}$ & 0.25845 & 0.28125 & 0.02280\\
104 & 104 & $1$ & $\frac{1}{4}$ & $\frac{13}{100}$ & 0.25845 & 0.50000 & --\\
105 & 105 & $\frac{3}{4}$ & $\frac{47}{200}$ & $\frac{4}{25}$ & 0.24780 & 0.28125 & 0.03345\\
105 & 105 & $1$ & $\frac{47}{200}$ & $\frac{4}{25}$ & 0.24780 & 0.50000 & --\\
107 & -107 & $\frac{3}{4}$ & $\frac{57}{200}$ & $\frac{3}{20}$ & 0.29625 & 0.28125 & -0.01500\\
107 & -107 & $1$ & $\frac{57}{200}$ & $\frac{3}{20}$ & 0.29625 & 0.50000 & --\\
109 & 109 & $\frac{3}{4}$ & $\frac{57}{200}$ & $\frac{3}{20}$ & 0.29625 & 0.28125 & -0.01500\\
109 & 109 & $1$ & $\frac{57}{200}$ & $\frac{3}{20}$ & 0.29625 & 0.50000 & --\\
111 & -111 & $\frac{3}{4}$ & $\frac{13}{50}$ & $\frac{4}{25}$ & 0.27280 & 0.28125 & 0.00845\\
111 & -111 & $1$ & $\frac{13}{50}$ & $\frac{4}{25}$ & 0.27280 & 0.50000 & --\\
113 & 113 & $\frac{3}{4}$ & $\frac{29}{100}$ & $\frac{23}{200}$ & 0.29662 & 0.28125 & -0.01537\\
113 & 113 & $1$ & $\frac{29}{100}$ & $\frac{23}{200}$ & 0.29662 & 0.50000 & --\\
115 & -115 & $\frac{3}{4}$ & $\frac{669}{2500}$ & $\frac{4}{25}$ & 0.28040 & 0.28125 & 0.00085\\
115 & -115 & $1$ & $\frac{669}{2500}$ & $\frac{4}{25}$ & 0.28040 & 0.50000 & --\\
116 & -116 & $\frac{3}{4}$ & $\frac{1}{4}$ & $\frac{31}{200}$ & 0.26202 & 0.28125 & 0.01923\\
116 & -116 & $1$ & $\frac{1}{4}$ & $\frac{31}{200}$ & 0.26202 & 0.50000 & --\\
119 & -119 & $\frac{3}{4}$ & $\frac{669}{2500}$ & $\frac{4}{25}$ & 0.28040 & 0.28125 & 0.00085\\
119 & -119 & $1$ & $\frac{669}{2500}$ & $\frac{4}{25}$ & 0.28040 & 0.50000 & --\\
120 & -120 & $\frac{3}{4}$ & $\frac{9}{40}$ & $\frac{13}{100}$ & 0.23345 & 0.28125 & 0.04780\\
120 & -120 & $1$ & $\frac{9}{40}$ & $\frac{13}{100}$ & 0.23345 & 0.50000 & --\\
120 & 120 & $\frac{3}{4}$ & $\frac{9}{40}$ & $\frac{13}{100}$ & 0.23345 & 0.28125 & 0.04780\\
120 & 120 & $1$ & $\frac{9}{40}$ & $\frac{13}{100}$ & 0.23345 & 0.50000 & --\\
123 & -123 & $\frac{3}{4}$ & $\frac{13}{50}$ & $\frac{4}{25}$ & 0.27280 & 0.28125 & 0.00845\\
123 & -123 & $1$ & $\frac{13}{50}$ & $\frac{4}{25}$ & 0.27280 & 0.50000 & --\\
124 & 124 & $\frac{3}{4}$ & $\frac{51}{200}$ & $\frac{3}{25}$ & 0.26220 & 0.28125 & 0.01905\\
124 & 124 & $1$ & $\frac{51}{200}$ & $\frac{3}{25}$ & 0.26220 & 0.50000 & --\\
127 & -127 & $\frac{3}{4}$ & $\frac{29}{100}$ & $\frac{23}{200}$ & 0.29662 & 0.28125 & -0.01537\\
127 & -127 & $1$ & $\frac{29}{100}$ & $\frac{23}{200}$ & 0.29662 & 0.50000 & --\\
129 & 129 & $\frac{3}{4}$ & $\frac{13}{50}$ & $\frac{4}{25}$ & 0.27280 & 0.28125 & 0.00845\\
129 & 129 & $1$ & $\frac{13}{50}$ & $\frac{4}{25}$ & 0.27280 & 0.50000 & --\\
131 & -131 & $\frac{3}{4}$ & $\frac{29}{100}$ & $\frac{23}{200}$ & 0.29662 & 0.28125 & -0.01537\\
131 & -131 & $1$ & $\frac{29}{100}$ & $\frac{23}{200}$ & 0.29662 & 0.50000 & --\\
132 & -132 & $\frac{3}{4}$ & $\frac{9}{40}$ & $\frac{31}{200}$ & 0.23702 & 0.28125 & 0.04423\\
132 & -132 & $1$ & $\frac{9}{40}$ & $\frac{31}{200}$ & 0.23702 & 0.50000 & --\\
133 & 133 & $\frac{3}{4}$ & $\frac{669}{2500}$ & $\frac{4}{25}$ & 0.28040 & 0.28125 & 0.00085\\
133 & 133 & $1$ & $\frac{669}{2500}$ & $\frac{4}{25}$ & 0.28040 & 0.50000 & --\\
136 & -136 & $\frac{3}{4}$ & $\frac{51}{200}$ & $\frac{21}{200}$ & 0.26052 & 0.28125 & 0.02073\\
136 & -136 & $1$ & $\frac{51}{200}$ & $\frac{21}{200}$ & 0.26052 & 0.50000 & --\\
136 & 136 & $\frac{3}{4}$ & $\frac{51}{200}$ & $\frac{21}{200}$ & 0.26052 & 0.28125 & 0.02073\\
136 & 136 & $1$ & $\frac{51}{200}$ & $\frac{21}{200}$ & 0.26052 & 0.50000 & --\\
137 & 137 & $\frac{3}{4}$ & $\frac{29}{100}$ & $\frac{23}{200}$ & 0.29662 & 0.28125 & -0.01537\\
137 & 137 & $1$ & $\frac{29}{100}$ & $\frac{23}{200}$ & 0.29662 & 0.50000 & --\\
139 & -139 & $\frac{3}{4}$ & $\frac{29}{100}$ & $\frac{23}{200}$ & 0.29662 & 0.28125 & -0.01537\\
139 & -139 & $1$ & $\frac{29}{100}$ & $\frac{23}{200}$ & 0.29662 & 0.50000 & --\\
140 & 140 & $\frac{3}{4}$ & $\frac{23}{100}$ & $\frac{27}{200}$ & 0.23912 & 0.28125 & 0.04213\\
140 & 140 & $1$ & $\frac{23}{100}$ & $\frac{27}{200}$ & 0.23912 & 0.50000 & --\\
141 & 141 & $\frac{3}{4}$ & $\frac{53}{200}$ & $\frac{13}{100}$ & 0.27345 & 0.28125 & 0.00780\\
141 & 141 & $1$ & $\frac{53}{200}$ & $\frac{13}{100}$ & 0.27345 & 0.50000 & --\\
143 & -143 & $\frac{3}{4}$ & $\frac{11}{40}$ & $\frac{27}{200}$ & 0.28412 & 0.28125 & -0.00287\\
143 & -143 & $1$ & $\frac{11}{40}$ & $\frac{27}{200}$ & 0.28412 & 0.50000 & --\\
145 & 145 & $\frac{3}{4}$ & $\frac{347}{1250}$ & $\frac{9}{100}$ & 0.28165 & 0.28125 & -0.00040\\
145 & 145 & $1$ & $\frac{347}{1250}$ & $\frac{9}{100}$ & 0.28165 & 0.50000 & --\\
148 & -148 & $\frac{3}{4}$ & $\frac{1}{4}$ & $\frac{4}{25}$ & 0.26280 & 0.28125 & 0.01845\\
148 & -148 & $1$ & $\frac{1}{4}$ & $\frac{4}{25}$ & 0.26280 & 0.50000 & --\\
149 & 149 & $\frac{3}{4}$ & $\frac{57}{200}$ & $\frac{31}{200}$ & 0.29702 & 0.28125 & -0.01577\\
149 & 149 & $1$ & $\frac{57}{200}$ & $\frac{31}{200}$ & 0.29702 & 0.50000 & --\\
151 & -151 & $\frac{3}{4}$ & $\frac{57}{200}$ & $\frac{31}{200}$ & 0.29702 & 0.28125 & -0.01577\\
151 & -151 & $1$ & $\frac{57}{200}$ & $\frac{31}{200}$ & 0.29702 & 0.50000 & --\\
152 & -152 & $\frac{3}{4}$ & $\frac{51}{200}$ & $\frac{11}{100}$ & 0.26105 & 0.28125 & 0.02020\\
152 & -152 & $1$ & $\frac{51}{200}$ & $\frac{11}{100}$ & 0.26105 & 0.50000 & --\\
152 & 152 & $\frac{3}{4}$ & $\frac{51}{200}$ & $\frac{11}{100}$ & 0.26105 & 0.28125 & 0.02020\\
152 & 152 & $1$ & $\frac{51}{200}$ & $\frac{11}{100}$ & 0.26105 & 0.50000 & --\\
155 & -155 & $\frac{3}{4}$ & $\frac{1313}{5000}$ & $\frac{33}{200}$ & 0.27622 & 0.28125 & 0.00503\\
155 & -155 & $1$ & $\frac{1313}{5000}$ & $\frac{33}{200}$ & 0.27622 & 0.50000 & --\\
156 & 156 & $\frac{3}{4}$ & $\frac{23}{100}$ & $\frac{13}{100}$ & 0.23845 & 0.28125 & 0.04280\\
156 & 156 & $1$ & $\frac{23}{100}$ & $\frac{13}{100}$ & 0.23845 & 0.50000 & --\\
157 & 157 & $\frac{3}{4}$ & $\frac{57}{200}$ & $\frac{31}{200}$ & 0.29702 & 0.28125 & -0.01577\\
157 & 157 & $1$ & $\frac{57}{200}$ & $\frac{31}{200}$ & 0.29702 & 0.50000 & --\\
159 & -159 & $\frac{3}{4}$ & $\frac{53}{200}$ & $\frac{13}{100}$ & 0.27345 & 0.28125 & 0.00780\\
159 & -159 & $1$ & $\frac{53}{200}$ & $\frac{13}{100}$ & 0.27345 & 0.50000 & --\\
161 & 161 & $\frac{3}{4}$ & $\frac{1363}{5000}$ & $\frac{27}{200}$ & 0.28172 & 0.28125 & -0.00047\\
161 & 161 & $1$ & $\frac{1363}{5000}$ & $\frac{27}{200}$ & 0.28172 & 0.50000 & --\\
163 & -163 & $\frac{3}{4}$ & $\frac{57}{200}$ & $\frac{31}{200}$ & 0.29702 & 0.28125 & -0.01577\\
163 & -163 & $1$ & $\frac{57}{200}$ & $\frac{31}{200}$ & 0.29702 & 0.50000 & --\\
164 & -164 & $\frac{3}{4}$ & $\frac{51}{200}$ & $\frac{13}{100}$ & 0.26345 & 0.28125 & 0.01780\\
164 & -164 & $1$ & $\frac{51}{200}$ & $\frac{13}{100}$ & 0.26345 & 0.50000 & --\\
165 & 165 & $\frac{3}{4}$ & $\frac{49}{200}$ & $\frac{27}{200}$ & 0.25412 & 0.28125 & 0.02713\\
165 & 165 & $1$ & $\frac{49}{200}$ & $\frac{27}{200}$ & 0.25412 & 0.50000 & --\\
167 & -167 & $\frac{3}{4}$ & $\frac{57}{200}$ & $\frac{31}{200}$ & 0.29702 & 0.28125 & -0.01577\\
167 & -167 & $1$ & $\frac{57}{200}$ & $\frac{31}{200}$ & 0.29702 & 0.50000 & --\\
168 & -168 & $\frac{3}{4}$ & $\frac{9}{40}$ & $\frac{23}{200}$ & 0.23162 & 0.28125 & 0.04963\\
168 & -168 & $1$ & $\frac{9}{40}$ & $\frac{23}{200}$ & 0.23162 & 0.50000 & --\\
168 & 168 & $\frac{3}{4}$ & $\frac{9}{40}$ & $\frac{23}{200}$ & 0.23162 & 0.28125 & 0.04963\\
168 & 168 & $1$ & $\frac{9}{40}$ & $\frac{23}{200}$ & 0.23162 & 0.50000 & --\\
172 & 172 & $\frac{3}{4}$ & $\frac{51}{200}$ & $\frac{13}{100}$ & 0.26345 & 0.28125 & 0.01780\\
172 & 172 & $1$ & $\frac{51}{200}$ & $\frac{13}{100}$ & 0.26345 & 0.50000 & --\\
173 & 173 & $\frac{3}{4}$ & $\frac{57}{200}$ & $\frac{31}{200}$ & 0.29702 & 0.28125 & -0.01577\\
173 & 173 & $1$ & $\frac{57}{200}$ & $\frac{31}{200}$ & 0.29702 & 0.50000 & --\\
177 & 177 & $\frac{3}{4}$ & $\frac{53}{200}$ & $\frac{13}{100}$ & 0.27345 & 0.28125 & 0.00780\\
177 & 177 & $1$ & $\frac{53}{200}$ & $\frac{13}{100}$ & 0.27345 & 0.50000 & --\\
179 & -179 & $\frac{3}{4}$ & $\frac{57}{200}$ & $\frac{31}{200}$ & 0.29702 & 0.28125 & -0.01577\\
179 & -179 & $1$ & $\frac{57}{200}$ & $\frac{31}{200}$ & 0.29702 & 0.50000 & --\\
181 & 181 & $\frac{3}{4}$ & $\frac{57}{200}$ & $\frac{31}{200}$ & 0.29702 & 0.28125 & -0.01577\\
181 & 181 & $1$ & $\frac{57}{200}$ & $\frac{31}{200}$ & 0.29702 & 0.50000 & --\\
183 & -183 & $\frac{3}{4}$ & $\frac{53}{200}$ & $\frac{13}{100}$ & 0.27345 & 0.28125 & 0.00780\\
183 & -183 & $1$ & $\frac{53}{200}$ & $\frac{13}{100}$ & 0.27345 & 0.50000 & --\\
184 & -184 & $\frac{3}{4}$ & $\frac{1}{4}$ & $\frac{31}{200}$ & 0.26202 & 0.28125 & 0.01923\\
184 & -184 & $1$ & $\frac{1}{4}$ & $\frac{31}{200}$ & 0.26202 & 0.50000 & --\\
184 & 184 & $\frac{3}{4}$ & $\frac{1}{4}$ & $\frac{31}{200}$ & 0.26202 & 0.28125 & 0.01923\\
184 & 184 & $1$ & $\frac{1}{4}$ & $\frac{31}{200}$ & 0.26202 & 0.50000 & --\\
185 & 185 & $\frac{3}{4}$ & $\frac{669}{2500}$ & $\frac{17}{100}$ & 0.28205 & 0.28125 & -0.00080\\
185 & 185 & $1$ & $\frac{669}{2500}$ & $\frac{17}{100}$ & 0.28205 & 0.50000 & --\\
187 & -187 & $\frac{3}{4}$ & $\frac{11}{40}$ & $\frac{29}{200}$ & 0.28552 & 0.28125 & -0.00427\\
187 & -187 & $1$ & $\frac{11}{40}$ & $\frac{29}{200}$ & 0.28552 & 0.50000 & --\\
188 & 188 & $\frac{3}{4}$ & $\frac{51}{200}$ & $\frac{13}{100}$ & 0.26345 & 0.28125 & 0.01780\\
188 & 188 & $1$ & $\frac{51}{200}$ & $\frac{13}{100}$ & 0.26345 & 0.50000 & --\\
191 & -191 & $\frac{3}{4}$ & $\frac{57}{200}$ & $\frac{31}{200}$ & 0.29702 & 0.28125 & -0.01577\\
191 & -191 & $1$ & $\frac{57}{200}$ & $\frac{31}{200}$ & 0.29702 & 0.50000 & --\\
193 & 193 & $\frac{3}{4}$ & $\frac{57}{200}$ & $\frac{31}{200}$ & 0.29702 & 0.28125 & -0.01577\\
193 & 193 & $1$ & $\frac{57}{200}$ & $\frac{31}{200}$ & 0.29702 & 0.50000 & --\\
195 & -195 & $\frac{3}{4}$ & $\frac{49}{200}$ & $\frac{29}{200}$ & 0.25552 & 0.28125 & 0.02573\\
195 & -195 & $1$ & $\frac{49}{200}$ & $\frac{29}{200}$ & 0.25552 & 0.50000 & --\\
197 & 197 & $\frac{3}{4}$ & $\frac{57}{200}$ & $\frac{31}{200}$ & 0.29702 & 0.28125 & -0.01577\\
197 & 197 & $1$ & $\frac{57}{200}$ & $\frac{31}{200}$ & 0.29702 & 0.50000 & --\\
199 & -199 & $\frac{3}{4}$ & $\frac{57}{200}$ & $\frac{31}{200}$ & 0.29702 & 0.28125 & -0.01577\\
199 & -199 & $1$ & $\frac{57}{200}$ & $\frac{31}{200}$ & 0.29702 & 0.50000 & --\\
201 & 201 & $\frac{3}{4}$ & $\frac{21}{80}$ & $\frac{3}{20}$ & 0.27375 & 0.28125 & 0.00750\\
201 & 201 & $1$ & $\frac{21}{80}$ & $\frac{3}{20}$ & 0.27375 & 0.50000 & --\\
203 & -203 & $\frac{3}{4}$ & $\frac{109}{400}$ & $\frac{7}{50}$ & 0.28230 & 0.28125 & -0.00105\\
203 & -203 & $1$ & $\frac{109}{400}$ & $\frac{7}{50}$ & 0.28230 & 0.50000 & --\\
204 & 204 & $\frac{3}{4}$ & $\frac{93}{400}$ & $\frac{1}{8}$ & 0.24032 & 0.28125 & 0.04093\\
204 & 204 & $1$ & $\frac{93}{400}$ & $\frac{1}{8}$ & 0.24032 & 0.50000 & --\\
205 & 205 & $\frac{3}{4}$ & $\frac{2751}{10000}$ & $\frac{3}{25}$ & 0.28230 & 0.28125 & -0.00105\\
205 & 205 & $1$ & $\frac{2751}{10000}$ & $\frac{3}{25}$ & 0.28230 & 0.50000 & --\\
209 & 209 & $\frac{3}{4}$ & $\frac{7}{25}$ & $\frac{23}{200}$ & 0.28662 & 0.28125 & -0.00537\\
209 & 209 & $1$ & $\frac{7}{25}$ & $\frac{23}{200}$ & 0.28662 & 0.50000 & --\\
211 & -211 & $\frac{3}{4}$ & $\frac{29}{100}$ & $\frac{3}{25}$ & 0.29720 & 0.28125 & -0.01595\\
211 & -211 & $1$ & $\frac{29}{100}$ & $\frac{3}{25}$ & 0.29720 & 0.50000 & --\\
212 & -212 & $\frac{3}{4}$ & $\frac{101}{400}$ & $\frac{3}{20}$ & 0.26375 & 0.28125 & 0.01750\\
212 & -212 & $1$ & $\frac{101}{400}$ & $\frac{3}{20}$ & 0.26375 & 0.50000 & --\\
213 & 213 & $\frac{3}{4}$ & $\frac{21}{80}$ & $\frac{3}{20}$ & 0.27375 & 0.28125 & 0.00750\\
213 & 213 & $1$ & $\frac{21}{80}$ & $\frac{3}{20}$ & 0.27375 & 0.50000 & --\\
215 & -215 & $\frac{3}{4}$ & $\frac{1363}{5000}$ & $\frac{7}{50}$ & 0.28240 & 0.28125 & -0.00115\\
215 & -215 & $1$ & $\frac{1363}{5000}$ & $\frac{7}{50}$ & 0.28240 & 0.50000 & --\\
217 & 217 & $\frac{3}{4}$ & $\frac{11}{40}$ & $\frac{1}{8}$ & 0.28282 & 0.28125 & -0.00157\\
217 & 217 & $1$ & $\frac{11}{40}$ & $\frac{1}{8}$ & 0.28282 & 0.50000 & --\\
219 & -219 & $\frac{3}{4}$ & $\frac{21}{80}$ & $\frac{3}{20}$ & 0.27375 & 0.28125 & 0.00750\\
219 & -219 & $1$ & $\frac{21}{80}$ & $\frac{3}{20}$ & 0.27375 & 0.50000 & --\\
220 & 220 & $\frac{3}{4}$ & $\frac{93}{400}$ & $\frac{63}{400}$ & 0.24491 & 0.28125 & 0.03634\\
220 & 220 & $1$ & $\frac{93}{400}$ & $\frac{63}{400}$ & 0.24491 & 0.50000 & --\\
221 & 221 & $\frac{3}{4}$ & $\frac{7}{25}$ & $\frac{3}{25}$ & 0.28720 & 0.28125 & -0.00595\\
221 & 221 & $1$ & $\frac{7}{25}$ & $\frac{3}{25}$ & 0.28720 & 0.50000 & --\\
223 & -223 & $\frac{3}{4}$ & $\frac{29}{100}$ & $\frac{3}{25}$ & 0.29720 & 0.28125 & -0.01595\\
223 & -223 & $1$ & $\frac{29}{100}$ & $\frac{3}{25}$ & 0.29720 & 0.50000 & --\\
227 & -227 & $\frac{3}{4}$ & $\frac{29}{100}$ & $\frac{3}{25}$ & 0.29720 & 0.28125 & -0.01595\\
227 & -227 & $1$ & $\frac{29}{100}$ & $\frac{3}{25}$ & 0.29720 & 0.50000 & --\\
228 & -228 & $\frac{3}{4}$ & $\frac{47}{200}$ & $\frac{9}{80}$ & 0.24133 & 0.28125 & 0.03992\\
228 & -228 & $1$ & $\frac{47}{200}$ & $\frac{9}{80}$ & 0.24133 & 0.50000 & --\\
229 & 229 & $\frac{3}{4}$ & $\frac{29}{100}$ & $\frac{3}{25}$ & 0.29720 & 0.28125 & -0.01595\\
229 & 229 & $1$ & $\frac{29}{100}$ & $\frac{3}{25}$ & 0.29720 & 0.50000 & --\\
231 & -231 & $\frac{3}{4}$ & $\frac{49}{200}$ & $\frac{27}{200}$ & 0.25412 & 0.28125 & 0.02713\\
231 & -231 & $1$ & $\frac{49}{200}$ & $\frac{27}{200}$ & 0.25412 & 0.50000 & --\\
232 & -232 & $\frac{3}{4}$ & $\frac{51}{200}$ & $\frac{1}{8}$ & 0.26282 & 0.28125 & 0.01843\\
232 & -232 & $1$ & $\frac{51}{200}$ & $\frac{1}{8}$ & 0.26282 & 0.50000 & --\\
232 & 232 & $\frac{3}{4}$ & $\frac{51}{200}$ & $\frac{1}{8}$ & 0.26282 & 0.28125 & 0.01843\\
232 & 232 & $1$ & $\frac{51}{200}$ & $\frac{1}{8}$ & 0.26282 & 0.50000 & --\\
233 & 233 & $\frac{3}{4}$ & $\frac{29}{100}$ & $\frac{3}{25}$ & 0.29720 & 0.28125 & -0.01595\\
233 & 233 & $1$ & $\frac{29}{100}$ & $\frac{3}{25}$ & 0.29720 & 0.50000 & --\\
235 & -235 & $\frac{3}{4}$ & $\frac{11}{40}$ & $\frac{1}{8}$ & 0.28282 & 0.28125 & -0.00157\\
235 & -235 & $1$ & $\frac{11}{40}$ & $\frac{1}{8}$ & 0.28282 & 0.50000 & --\\
236 & 236 & $\frac{3}{4}$ & $\frac{51}{200}$ & $\frac{27}{200}$ & 0.26412 & 0.28125 & 0.01713\\
236 & 236 & $1$ & $\frac{51}{200}$ & $\frac{27}{200}$ & 0.26412 & 0.50000 & --\\
237 & 237 & $\frac{3}{4}$ & $\frac{53}{200}$ & $\frac{27}{200}$ & 0.27412 & 0.28125 & 0.00713\\
237 & 237 & $1$ & $\frac{53}{200}$ & $\frac{27}{200}$ & 0.27412 & 0.50000 & --\\
239 & -239 & $\frac{3}{4}$ & $\frac{29}{100}$ & $\frac{3}{25}$ & 0.29720 & 0.28125 & -0.01595\\
239 & -239 & $1$ & $\frac{29}{100}$ & $\frac{3}{25}$ & 0.29720 & 0.50000 & --\\
241 & 241 & $\frac{3}{4}$ & $\frac{29}{100}$ & $\frac{3}{25}$ & 0.29720 & 0.28125 & -0.01595\\
241 & 241 & $1$ & $\frac{29}{100}$ & $\frac{3}{25}$ & 0.29720 & 0.50000 & --\\
244 & -244 & $\frac{3}{4}$ & $\frac{51}{200}$ & $\frac{27}{200}$ & 0.26412 & 0.28125 & 0.01713\\
244 & -244 & $1$ & $\frac{51}{200}$ & $\frac{27}{200}$ & 0.26412 & 0.50000 & --\\
247 & -247 & $\frac{3}{4}$ & $\frac{7}{25}$ & $\frac{1}{8}$ & 0.28782 & 0.28125 & -0.00657\\
247 & -247 & $1$ & $\frac{7}{25}$ & $\frac{1}{8}$ & 0.28782 & 0.50000 & --\\
248 & -248 & $\frac{3}{4}$ & $\frac{51}{200}$ & $\frac{1}{8}$ & 0.26282 & 0.28125 & 0.01843\\
248 & -248 & $1$ & $\frac{51}{200}$ & $\frac{1}{8}$ & 0.26282 & 0.50000 & --\\
248 & 248 & $\frac{3}{4}$ & $\frac{51}{200}$ & $\frac{1}{8}$ & 0.26282 & 0.28125 & 0.01843\\
248 & 248 & $1$ & $\frac{51}{200}$ & $\frac{1}{8}$ & 0.26282 & 0.50000 & --\\
249 & 249 & $\frac{3}{4}$ & $\frac{53}{200}$ & $\frac{27}{200}$ & 0.27412 & 0.28125 & 0.00713\\
249 & 249 & $1$ & $\frac{53}{200}$ & $\frac{27}{200}$ & 0.27412 & 0.50000 & --\\
251 & -251 & $\frac{3}{4}$ & $\frac{23}{80}$ & $\frac{7}{50}$ & 0.29730 & 0.28125 & -0.01605\\
251 & -251 & $1$ & $\frac{23}{80}$ & $\frac{7}{50}$ & 0.29730 & 0.50000 & --\\
253 & 253 & $\frac{3}{4}$ & $\frac{111}{400}$ & $\frac{7}{50}$ & 0.28730 & 0.28125 & -0.00605\\
253 & 253 & $1$ & $\frac{111}{400}$ & $\frac{7}{50}$ & 0.28730 & 0.50000 & --\\
255 & -255 & $\frac{3}{4}$ & $\frac{99}{400}$ & $\frac{7}{50}$ & 0.25730 & 0.28125 & 0.02395\\
255 & -255 & $1$ & $\frac{99}{400}$ & $\frac{7}{50}$ & 0.25730 & 0.50000 & --\\
257 & 257 & $\frac{3}{4}$ & $\frac{23}{80}$ & $\frac{7}{50}$ & 0.29730 & 0.28125 & -0.01605\\
257 & 257 & $1$ & $\frac{23}{80}$ & $\frac{7}{50}$ & 0.29730 & 0.50000 & --\\
259 & -259 & $\frac{3}{4}$ & $\frac{109}{400}$ & $\frac{29}{200}$ & 0.28302 & 0.28125 & -0.00177\\
259 & -259 & $1$ & $\frac{109}{400}$ & $\frac{29}{200}$ & 0.28302 & 0.50000 & --\\
260 & -260 & $\frac{3}{4}$ & $\frac{6}{25}$ & $\frac{23}{200}$ & 0.24662 & 0.28125 & 0.03463\\
260 & -260 & $1$ & $\frac{6}{25}$ & $\frac{23}{200}$ & 0.24662 & 0.50000 & --\\
263 & -263 & $\frac{3}{4}$ & $\frac{23}{80}$ & $\frac{7}{50}$ & 0.29730 & 0.28125 & -0.01605\\
263 & -263 & $1$ & $\frac{23}{80}$ & $\frac{7}{50}$ & 0.29730 & 0.50000 & --\\
264 & -264 & $\frac{3}{4}$ & $\frac{47}{200}$ & $\frac{17}{200}$ & 0.23862 & 0.28125 & 0.04263\\
264 & -264 & $1$ & $\frac{47}{200}$ & $\frac{17}{200}$ & 0.23862 & 0.50000 & --\\
264 & 264 & $\frac{3}{4}$ & $\frac{47}{200}$ & $\frac{17}{200}$ & 0.23862 & 0.28125 & 0.04263\\
264 & 264 & $1$ & $\frac{47}{200}$ & $\frac{17}{200}$ & 0.23862 & 0.50000 & --\\
265 & 265 & $\frac{3}{4}$ & $\frac{27}{100}$ & $\frac{4}{25}$ & 0.28280 & 0.28125 & -0.00155\\
265 & 265 & $1$ & $\frac{27}{100}$ & $\frac{4}{25}$ & 0.28280 & 0.50000 & --\\
267 & -267 & $\frac{3}{4}$ & $\frac{53}{200}$ & $\frac{27}{200}$ & 0.27412 & 0.28125 & 0.00713\\
267 & -267 & $1$ & $\frac{53}{200}$ & $\frac{27}{200}$ & 0.27412 & 0.50000 & --\\
268 & 268 & $\frac{3}{4}$ & $\frac{51}{200}$ & $\frac{27}{200}$ & 0.26412 & 0.28125 & 0.01713\\
268 & 268 & $1$ & $\frac{51}{200}$ & $\frac{27}{200}$ & 0.26412 & 0.50000 & --\\
269 & 269 & $\frac{3}{4}$ & $\frac{23}{80}$ & $\frac{7}{50}$ & 0.29730 & 0.28125 & -0.01605\\
269 & 269 & $1$ & $\frac{23}{80}$ & $\frac{7}{50}$ & 0.29730 & 0.50000 & --\\
271 & -271 & $\frac{3}{4}$ & $\frac{23}{80}$ & $\frac{7}{50}$ & 0.29730 & 0.28125 & -0.01605\\
271 & -271 & $1$ & $\frac{23}{80}$ & $\frac{7}{50}$ & 0.29730 & 0.50000 & --\\
273 & 273 & $\frac{3}{4}$ & $\frac{97}{400}$ & $\frac{13}{80}$ & 0.25571 & 0.28125 & 0.02554\\
273 & 273 & $1$ & $\frac{97}{400}$ & $\frac{13}{80}$ & 0.25571 & 0.50000 & --\\
276 & -276 & $\frac{3}{4}$ & $\frac{47}{200}$ & $\frac{3}{25}$ & 0.24220 & 0.28125 & 0.03905\\
276 & -276 & $1$ & $\frac{47}{200}$ & $\frac{3}{25}$ & 0.24220 & 0.50000 & --\\
277 & 277 & $\frac{3}{4}$ & $\frac{23}{80}$ & $\frac{7}{50}$ & 0.29730 & 0.28125 & -0.01605\\
277 & 277 & $1$ & $\frac{23}{80}$ & $\frac{7}{50}$ & 0.29730 & 0.50000 & --\\
280 & -280 & $\frac{3}{4}$ & $\frac{47}{200}$ & $\frac{41}{400}$ & 0.24026 & 0.28125 & 0.04099\\
280 & -280 & $1$ & $\frac{47}{200}$ & $\frac{41}{400}$ & 0.24026 & 0.50000 & --\\
280 & 280 & $\frac{3}{4}$ & $\frac{47}{200}$ & $\frac{41}{400}$ & 0.24026 & 0.28125 & 0.04099\\
280 & 280 & $1$ & $\frac{47}{200}$ & $\frac{41}{400}$ & 0.24026 & 0.50000 & --\\
281 & 281 & $\frac{3}{4}$ & $\frac{23}{80}$ & $\frac{7}{50}$ & 0.29730 & 0.28125 & -0.01605\\
281 & 281 & $1$ & $\frac{23}{80}$ & $\frac{7}{50}$ & 0.29730 & 0.50000 & --\\
283 & -283 & $\frac{3}{4}$ & $\frac{23}{80}$ & $\frac{7}{50}$ & 0.29730 & 0.28125 & -0.01605\\
283 & -283 & $1$ & $\frac{23}{80}$ & $\frac{7}{50}$ & 0.29730 & 0.50000 & --\\
284 & 284 & $\frac{3}{4}$ & $\frac{103}{400}$ & $\frac{47}{400}$ & 0.26441 & 0.28125 & 0.01684\\
284 & 284 & $1$ & $\frac{103}{400}$ & $\frac{47}{400}$ & 0.26441 & 0.50000 & --\\
285 & 285 & $\frac{3}{4}$ & $\frac{51}{200}$ & $\frac{7}{80}$ & 0.25883 & 0.28125 & 0.02242\\
285 & 285 & $1$ & $\frac{51}{200}$ & $\frac{7}{80}$ & 0.25883 & 0.50000 & --\\
287 & -287 & $\frac{3}{4}$ & $\frac{11}{40}$ & $\frac{13}{100}$ & 0.28345 & 0.28125 & -0.00220\\
287 & -287 & $1$ & $\frac{11}{40}$ & $\frac{13}{100}$ & 0.28345 & 0.50000 & --\\
291 & -291 & $\frac{3}{4}$ & $\frac{107}{400}$ & $\frac{47}{400}$ & 0.27441 & 0.28125 & 0.00684\\
291 & -291 & $1$ & $\frac{107}{400}$ & $\frac{47}{400}$ & 0.27441 & 0.50000 & --\\
292 & -292 & $\frac{3}{4}$ & $\frac{103}{400}$ & $\frac{47}{400}$ & 0.26441 & 0.28125 & 0.01684\\
292 & -292 & $1$ & $\frac{103}{400}$ & $\frac{47}{400}$ & 0.26441 & 0.50000 & --\\
293 & 293 & $\frac{3}{4}$ & $\frac{23}{80}$ & $\frac{7}{50}$ & 0.29730 & 0.28125 & -0.01605\\
293 & 293 & $1$ & $\frac{23}{80}$ & $\frac{7}{50}$ & 0.29730 & 0.50000 & --\\
295 & -295 & $\frac{3}{4}$ & $\frac{109}{400}$ & $\frac{29}{200}$ & 0.28302 & 0.28125 & -0.00177\\
295 & -295 & $1$ & $\frac{109}{400}$ & $\frac{29}{200}$ & 0.28302 & 0.50000 & --\\
296 & -296 & $\frac{3}{4}$ & $\frac{51}{200}$ & $\frac{13}{100}$ & 0.26345 & 0.28125 & 0.01780\\
296 & -296 & $1$ & $\frac{51}{200}$ & $\frac{13}{100}$ & 0.26345 & 0.50000 & --\\
296 & 296 & $\frac{3}{4}$ & $\frac{51}{200}$ & $\frac{13}{100}$ & 0.26345 & 0.28125 & 0.01780\\
296 & 296 & $1$ & $\frac{51}{200}$ & $\frac{13}{100}$ & 0.26345 & 0.50000 & --\\
299 & -299 & $\frac{3}{4}$ & $\frac{57}{200}$ & $\frac{19}{200}$ & 0.28952 & 0.28125 & -0.00827\\
299 & -299 & $1$ & $\frac{57}{200}$ & $\frac{19}{200}$ & 0.28952 & 0.50000 & --\\
301 & 301 & $\frac{3}{4}$ & $\frac{7}{25}$ & $\frac{19}{200}$ & 0.28452 & 0.28125 & -0.00327\\
301 & 301 & $1$ & $\frac{7}{25}$ & $\frac{19}{200}$ & 0.28452 & 0.50000 & --\\
303 & -303 & $\frac{3}{4}$ & $\frac{13}{50}$ & $\frac{17}{100}$ & 0.27445 & 0.28125 & 0.00680\\
303 & -303 & $1$ & $\frac{13}{50}$ & $\frac{17}{100}$ & 0.27445 & 0.50000 & --\\
305 & 305 & $\frac{3}{4}$ & $\frac{11}{40}$ & $\frac{13}{100}$ & 0.28345 & 0.28125 & -0.00220\\
305 & 305 & $1$ & $\frac{11}{40}$ & $\frac{13}{100}$ & 0.28345 & 0.50000 & --\\
307 & -307 & $\frac{3}{4}$ & $\frac{29}{100}$ & $\frac{49}{400}$ & 0.29751 & 0.28125 & -0.01626\\
307 & -307 & $1$ & $\frac{29}{100}$ & $\frac{49}{400}$ & 0.29751 & 0.50000 & --\\
308 & -308 & $\frac{3}{4}$ & $\frac{93}{400}$ & $\frac{63}{400}$ & 0.24491 & 0.28125 & 0.03634\\
308 & -308 & $1$ & $\frac{93}{400}$ & $\frac{63}{400}$ & 0.24491 & 0.50000 & --\\
309 & 309 & $\frac{3}{4}$ & $\frac{13}{50}$ & $\frac{17}{100}$ & 0.27445 & 0.28125 & 0.00680\\
309 & 309 & $1$ & $\frac{13}{50}$ & $\frac{17}{100}$ & 0.27445 & 0.50000 & --\\
311 & -311 & $\frac{3}{4}$ & $\frac{29}{100}$ & $\frac{49}{400}$ & 0.29751 & 0.28125 & -0.01626\\
311 & -311 & $1$ & $\frac{29}{100}$ & $\frac{49}{400}$ & 0.29751 & 0.50000 & --\\
312 & -312 & $\frac{3}{4}$ & $\frac{47}{200}$ & $\frac{39}{400}$ & 0.23976 & 0.28125 & 0.04149\\
312 & -312 & $1$ & $\frac{47}{200}$ & $\frac{39}{400}$ & 0.23976 & 0.50000 & --\\
312 & 312 & $\frac{3}{4}$ & $\frac{47}{200}$ & $\frac{39}{400}$ & 0.23976 & 0.28125 & 0.04149\\
312 & 312 & $1$ & $\frac{47}{200}$ & $\frac{39}{400}$ & 0.23976 & 0.50000 & --\\
313 & 313 & $\frac{3}{4}$ & $\frac{29}{100}$ & $\frac{49}{400}$ & 0.29751 & 0.28125 & -0.01626\\
313 & 313 & $1$ & $\frac{29}{100}$ & $\frac{49}{400}$ & 0.29751 & 0.50000 & --\\
316 & 316 & $\frac{3}{4}$ & $\frac{51}{200}$ & $\frac{11}{80}$ & 0.26446 & 0.28125 & 0.01679\\
316 & 316 & $1$ & $\frac{51}{200}$ & $\frac{11}{80}$ & 0.26446 & 0.50000 & --\\
317 & 317 & $\frac{3}{4}$ & $\frac{29}{100}$ & $\frac{49}{400}$ & 0.29751 & 0.28125 & -0.01626\\
317 & 317 & $1$ & $\frac{29}{100}$ & $\frac{49}{400}$ & 0.29751 & 0.50000 & --\\
319 & -319 & $\frac{3}{4}$ & $\frac{111}{400}$ & $\frac{29}{200}$ & 0.28802 & 0.28125 & -0.00677\\
319 & -319 & $1$ & $\frac{111}{400}$ & $\frac{29}{200}$ & 0.28802 & 0.50000 & --\\
321 & 321 & $\frac{3}{4}$ & $\frac{53}{200}$ & $\frac{11}{80}$ & 0.27446 & 0.28125 & 0.00679\\
321 & 321 & $1$ & $\frac{53}{200}$ & $\frac{11}{80}$ & 0.27446 & 0.50000 & --\\
323 & -323 & $\frac{3}{4}$ & $\frac{57}{200}$ & $\frac{41}{400}$ & 0.29026 & 0.28125 & -0.00901\\
323 & -323 & $1$ & $\frac{57}{200}$ & $\frac{41}{400}$ & 0.29026 & 0.50000 & --\\
327 & -327 & $\frac{3}{4}$ & $\frac{53}{200}$ & $\frac{11}{80}$ & 0.27446 & 0.28125 & 0.00679\\
327 & -327 & $1$ & $\frac{53}{200}$ & $\frac{11}{80}$ & 0.27446 & 0.50000 & --\\
328 & -328 & $\frac{3}{4}$ & $\frac{101}{400}$ & $\frac{3}{20}$ & 0.26375 & 0.28125 & 0.01750\\
328 & -328 & $1$ & $\frac{101}{400}$ & $\frac{3}{20}$ & 0.26375 & 0.50000 & --\\
328 & 328 & $\frac{3}{4}$ & $\frac{101}{400}$ & $\frac{3}{20}$ & 0.26375 & 0.28125 & 0.01750\\
328 & 328 & $1$ & $\frac{101}{400}$ & $\frac{3}{20}$ & 0.26375 & 0.50000 & --\\
329 & 329 & $\frac{3}{4}$ & $\frac{111}{400}$ & $\frac{23}{200}$ & 0.28412 & 0.28125 & -0.00287\\
329 & 329 & $1$ & $\frac{111}{400}$ & $\frac{23}{200}$ & 0.28412 & 0.50000 & --\\
331 & -331 & $\frac{3}{4}$ & $\frac{29}{100}$ & $\frac{1}{8}$ & 0.29782 & 0.28125 & -0.01657\\
331 & -331 & $1$ & $\frac{29}{100}$ & $\frac{1}{8}$ & 0.29782 & 0.50000 & --\\
332 & 332 & $\frac{3}{4}$ & $\frac{101}{400}$ & $\frac{31}{200}$ & 0.26452 & 0.28125 & 0.01673\\
332 & 332 & $1$ & $\frac{101}{400}$ & $\frac{31}{200}$ & 0.26452 & 0.50000 & --\\
335 & -335 & $\frac{3}{4}$ & $\frac{11}{40}$ & $\frac{13}{100}$ & 0.28345 & 0.28125 & -0.00220\\
335 & -335 & $1$ & $\frac{11}{40}$ & $\frac{13}{100}$ & 0.28345 & 0.50000 & --\\
337 & 337 & $\frac{3}{4}$ & $\frac{29}{100}$ & $\frac{1}{8}$ & 0.29782 & 0.28125 & -0.01657\\
337 & 337 & $1$ & $\frac{29}{100}$ & $\frac{1}{8}$ & 0.29782 & 0.50000 & --\\
339 & -339 & $\frac{3}{4}$ & $\frac{21}{80}$ & $\frac{31}{200}$ & 0.27452 & 0.28125 & 0.00673\\
339 & -339 & $1$ & $\frac{21}{80}$ & $\frac{31}{200}$ & 0.27452 & 0.50000 & --\\
340 & -340 & $\frac{3}{4}$ & $\frac{47}{200}$ & $\frac{33}{200}$ & 0.24862 & 0.28125 & 0.03263\\
340 & -340 & $1$ & $\frac{47}{200}$ & $\frac{33}{200}$ & 0.24862 & 0.50000 & --\\
341 & 341 & $\frac{3}{4}$ & $\frac{7}{25}$ & $\frac{13}{100}$ & 0.28845 & 0.28125 & -0.00720\\
341 & 341 & $1$ & $\frac{7}{25}$ & $\frac{13}{100}$ & 0.28845 & 0.50000 & --\\
344 & -344 & $\frac{3}{4}$ & $\frac{101}{400}$ & $\frac{3}{20}$ & 0.26375 & 0.28125 & 0.01750\\
344 & -344 & $1$ & $\frac{101}{400}$ & $\frac{3}{20}$ & 0.26375 & 0.50000 & --\\
344 & 344 & $\frac{3}{4}$ & $\frac{101}{400}$ & $\frac{3}{20}$ & 0.26375 & 0.28125 & 0.01750\\
344 & 344 & $1$ & $\frac{101}{400}$ & $\frac{3}{20}$ & 0.26375 & 0.50000 & --\\
345 & 345 & $\frac{3}{4}$ & $\frac{51}{200}$ & $\frac{39}{400}$ & 0.25976 & 0.28125 & 0.02149\\
345 & 345 & $1$ & $\frac{51}{200}$ & $\frac{39}{400}$ & 0.25976 & 0.50000 & --\\
347 & -347 & $\frac{3}{4}$ & $\frac{29}{100}$ & $\frac{1}{8}$ & 0.29782 & 0.28125 & -0.01657\\
347 & -347 & $1$ & $\frac{29}{100}$ & $\frac{1}{8}$ & 0.29782 & 0.50000 & --\\
348 & 348 & $\frac{3}{4}$ & $\frac{93}{400}$ & $\frac{29}{200}$ & 0.24302 & 0.28125 & 0.03823\\
348 & 348 & $1$ & $\frac{93}{400}$ & $\frac{29}{200}$ & 0.24302 & 0.50000 & --\\
349 & 349 & $\frac{3}{4}$ & $\frac{29}{100}$ & $\frac{1}{8}$ & 0.29782 & 0.28125 & -0.01657\\
349 & 349 & $1$ & $\frac{29}{100}$ & $\frac{1}{8}$ & 0.29782 & 0.50000 & --\\
353 & 353 & $\frac{3}{4}$ & $\frac{29}{100}$ & $\frac{1}{8}$ & 0.29782 & 0.28125 & -0.01657\\
353 & 353 & $1$ & $\frac{29}{100}$ & $\frac{1}{8}$ & 0.29782 & 0.50000 & --\\
355 & -355 & $\frac{3}{4}$ & $\frac{7}{25}$ & $\frac{37}{400}$ & 0.28428 & 0.28125 & -0.00303\\
355 & -355 & $1$ & $\frac{7}{25}$ & $\frac{37}{400}$ & 0.28428 & 0.50000 & --\\
356 & -356 & $\frac{3}{4}$ & $\frac{51}{200}$ & $\frac{7}{50}$ & 0.26480 & 0.28125 & 0.01645\\
356 & -356 & $1$ & $\frac{51}{200}$ & $\frac{7}{50}$ & 0.26480 & 0.50000 & --\\
357 & 357 & $\frac{3}{4}$ & $\frac{51}{200}$ & $\frac{9}{100}$ & 0.25905 & 0.28125 & 0.02220\\
357 & 357 & $1$ & $\frac{51}{200}$ & $\frac{9}{100}$ & 0.25905 & 0.50000 & --\\
359 & -359 & $\frac{3}{4}$ & $\frac{29}{100}$ & $\frac{1}{8}$ & 0.29782 & 0.28125 & -0.01657\\
359 & -359 & $1$ & $\frac{29}{100}$ & $\frac{1}{8}$ & 0.29782 & 0.50000 & --\\
364 & 364 & $\frac{3}{4}$ & $\frac{93}{400}$ & $\frac{67}{400}$ & 0.24653 & 0.28125 & 0.03472\\
364 & 364 & $1$ & $\frac{93}{400}$ & $\frac{67}{400}$ & 0.24653 & 0.50000 & --\\
365 & 365 & $\frac{3}{4}$ & $\frac{7}{25}$ & $\frac{37}{400}$ & 0.28428 & 0.28125 & -0.00303\\
365 & 365 & $1$ & $\frac{7}{25}$ & $\frac{37}{400}$ & 0.28428 & 0.50000 & --\\
367 & -367 & $\frac{3}{4}$ & $\frac{29}{100}$ & $\frac{1}{8}$ & 0.29782 & 0.28125 & -0.01657\\
367 & -367 & $1$ & $\frac{29}{100}$ & $\frac{1}{8}$ & 0.29782 & 0.50000 & --\\
371 & -371 & $\frac{3}{4}$ & $\frac{7}{25}$ & $\frac{1}{10}$ & 0.28500 & 0.28125 & -0.00375\\
371 & -371 & $1$ & $\frac{7}{25}$ & $\frac{1}{10}$ & 0.28500 & 0.50000 & --\\
372 & -372 & $\frac{3}{4}$ & $\frac{47}{200}$ & $\frac{13}{100}$ & 0.24345 & 0.28125 & 0.03780\\
372 & -372 & $1$ & $\frac{47}{200}$ & $\frac{13}{100}$ & 0.24345 & 0.50000 & --\\
373 & 373 & $\frac{3}{4}$ & $\frac{29}{100}$ & $\frac{1}{8}$ & 0.29782 & 0.28125 & -0.01657\\
373 & 373 & $1$ & $\frac{29}{100}$ & $\frac{1}{8}$ & 0.29782 & 0.50000 & --\\
376 & -376 & $\frac{3}{4}$ & $\frac{51}{200}$ & $\frac{27}{200}$ & 0.26412 & 0.28125 & 0.01713\\
376 & -376 & $1$ & $\frac{51}{200}$ & $\frac{27}{200}$ & 0.26412 & 0.50000 & --\\
376 & 376 & $\frac{3}{4}$ & $\frac{51}{200}$ & $\frac{27}{200}$ & 0.26412 & 0.28125 & 0.01713\\
376 & 376 & $1$ & $\frac{51}{200}$ & $\frac{27}{200}$ & 0.26412 & 0.50000 & --\\
377 & 377 & $\frac{3}{4}$ & $\frac{57}{200}$ & $\frac{41}{400}$ & 0.29026 & 0.28125 & -0.00901\\
377 & 377 & $1$ & $\frac{57}{200}$ & $\frac{41}{400}$ & 0.29026 & 0.50000 & --\\
379 & -379 & $\frac{3}{4}$ & $\frac{29}{100}$ & $\frac{1}{8}$ & 0.29782 & 0.28125 & -0.01657\\
379 & -379 & $1$ & $\frac{29}{100}$ & $\frac{1}{8}$ & 0.29782 & 0.50000 & --\\
380 & 380 & $\frac{3}{4}$ & $\frac{19}{80}$ & $\frac{61}{400}$ & 0.24913 & 0.28125 & 0.03212\\
380 & 380 & $1$ & $\frac{19}{80}$ & $\frac{61}{400}$ & 0.24913 & 0.50000 & --\\
381 & 381 & $\frac{3}{4}$ & $\frac{53}{200}$ & $\frac{7}{50}$ & 0.27480 & 0.28125 & 0.00645\\
381 & 381 & $1$ & $\frac{53}{200}$ & $\frac{7}{50}$ & 0.27480 & 0.50000 & --\\
383 & -383 & $\frac{3}{4}$ & $\frac{29}{100}$ & $\frac{1}{8}$ & 0.29782 & 0.28125 & -0.01657\\
383 & -383 & $1$ & $\frac{29}{100}$ & $\frac{1}{8}$ & 0.29782 & 0.50000 & --\\
385 & 385 & $\frac{3}{4}$ & $\frac{101}{400}$ & $\frac{61}{400}$ & 0.26413 & 0.28125 & 0.01712\\
385 & 385 & $1$ & $\frac{101}{400}$ & $\frac{61}{400}$ & 0.26413 & 0.50000 & --\\
388 & -388 & $\frac{3}{4}$ & $\frac{51}{200}$ & $\frac{7}{50}$ & 0.26480 & 0.28125 & 0.01645\\
388 & -388 & $1$ & $\frac{51}{200}$ & $\frac{7}{50}$ & 0.26480 & 0.50000 & --\\
389 & 389 & $\frac{3}{4}$ & $\frac{29}{100}$ & $\frac{1}{8}$ & 0.29782 & 0.28125 & -0.01657\\
389 & 389 & $1$ & $\frac{29}{100}$ & $\frac{1}{8}$ & 0.29782 & 0.50000 & --\\
391 & -391 & $\frac{3}{4}$ & $\frac{57}{200}$ & $\frac{11}{100}$ & 0.29105 & 0.28125 & -0.00980\\
391 & -391 & $1$ & $\frac{57}{200}$ & $\frac{11}{100}$ & 0.29105 & 0.50000 & --\\
393 & 393 & $\frac{3}{4}$ & $\frac{53}{200}$ & $\frac{7}{50}$ & 0.27480 & 0.28125 & 0.00645\\
393 & 393 & $1$ & $\frac{53}{200}$ & $\frac{7}{50}$ & 0.27480 & 0.50000 & --\\
395 & -395 & $\frac{3}{4}$ & $\frac{7}{25}$ & $\frac{19}{200}$ & 0.28452 & 0.28125 & -0.00327\\
395 & -395 & $1$ & $\frac{7}{25}$ & $\frac{19}{200}$ & 0.28452 & 0.50000 & --\\
397 & 397 & $\frac{3}{4}$ & $\frac{29}{100}$ & $\frac{1}{8}$ & 0.29782 & 0.28125 & -0.01657\\
397 & 397 & $1$ & $\frac{29}{100}$ & $\frac{1}{8}$ & 0.29782 & 0.50000 & --\\
399 & -399 & $\frac{3}{4}$ & $\frac{51}{200}$ & $\frac{19}{200}$ & 0.25952 & 0.28125 & 0.02173\\
399 & -399 & $1$ & $\frac{51}{200}$ & $\frac{19}{200}$ & 0.25952 & 0.50000 & --\\
\end{longtable}
}\subsection*{Use of AI disclosure}
Most of this work was carried out without assistance from recent AI language models. AI tools were used to improve the writing, refine some error bounds for large $m$, and extend the numerical methods. We also added a comparison with recent OpenAI preprints~\cite{OPENAI1,OPENAI2}, particularly regarding the last column of Table~\ref{tab:certified_pairs_m400}. The authors take full responsibility for the mathematical arguments, numerical computations, and final manuscript.

\bibliographystyle{unsrtnat}

\end{document}